%% file: n_queens.tex
\documentclass[openany]{memo-l}

\usepackage{}

\usepackage[latin1]{inputenc} 
\usepackage{amsthm,amsmath,amssymb,bm,mathtools} 
\usepackage{graphicx}
\usepackage{color}
\usepackage{verbatim}
\usepackage{float}
\usepackage{bbm}
\usepackage{enumerate}

\usepackage{array} 

\newtheorem{theo}{Theorem}[chapter]
\newtheorem{lemma}[theo]{Lemma}
\newtheorem{prop}[theo]{Proposition}
\newtheorem{cor}[theo]{Corollary}
\newtheorem{conj}[theo]{Conjecture}
\newtheorem{claim}[theo]{Claim}

\newtheorem{fact}[theo]{Fact}

\theoremstyle{definition}
\newtheorem{defn}[theo]{Definition}

\newtheorem{alg}[theo]{Algorithm}
\newtheorem{plan}[theo]{Plan}
\newtheorem{set}[theo]{Set-up}

\theoremstyle{remark}
\newtheorem{rem}[theo]{Remark}

\DeclareMathOperator{\supp}{supp}
\DeclareMathOperator{\gr}{gr}
\DeclareMathOperator{\vor}{vor}
\DeclareMathOperator{\bad}{bad}
\DeclareMathOperator{\co}{co}

\newcommand{\mymod}{\mathrm{~mod~}}

\numberwithin{section}{chapter}
\numberwithin{equation}{chapter}

\makeindex

\begin{document}

\frontmatter

\title{The $n$-queens problem}


\author{Candida Bowtell}
\address{Mathematical Institute, University of Oxford, Oxford, UK}
\curraddr{}
\email{bowtell@maths.ox.ac.uk}
\thanks{}

\author{Peter Keevash}
\address{Mathematical Institute, University of Oxford, Oxford, UK}
\curraddr{}
\email{keevash@maths.ox.ac.uk}
\thanks{}

\thanks{Research supported in part by ERC Consolidator Grant 647678 and ERC Advanced Grant 883810.}
\date{}




\begin{abstract}
The famous $n$-queens problem asks how many ways there are to place $n$ queens on an $n \times n$ chessboard so that no two queens can attack one another. The toroidal $n$-queens problem asks the same question where the board is considered on the surface of the torus and was asked by P\'{o}lya in 1918. Let $Q(n)$ denote the number of $n$-queens configurations on the classical board and $T(n)$ the number of toroidal $n$-queens configurations. P\'{o}lya showed that $T(n)>0$ if and only if $n \equiv 1,5 \mymod 6$ and much more recently, in 2017, Luria showed that $T(n)\leq ((1+o(1))ne^{-3})^n$ and conjectured equality when $n \equiv 1,5 \mymod 6$. Our main result is a proof of this conjecture, thus answering P\'{o}lya's question asymptotically. Furthermore, we also show that $Q(n)\geq((1+o(1))ne^{-3})^n$ for all $n$ sufficiently large, which was independently proved by Luria and Simkin. This combined with our main result and an upper bound of Luria completely settles a conjecture of Rivin, Vardi and Zimmmerman from 1994 regarding both $Q(n)$ and $T(n)$. Our proof combines a random greedy algorithm to count `almost' configurations with a complex absorbing strategy that uses ideas from the recently developed methods of randomised algebraic construction and iterative absorption.
\end{abstract}

\maketitle

\tableofcontents


\mainmatter
\include{ch_intro_new}
\include{ch_notation_prelims}

\include{ch_structure}
\include{ch_absorber}
\include{ch_count_new}
\include{ch_vortex}

\input{sec_prob_tools}
\input{sec_1}

\input{sec_functions}

\input{sec_i}
\input{sec_initial_steps}
\include{ch_classical}

\include{bib}
\appendix

\backmatter
\printindex

\end{document}

%% file: ch_intro_new.tex
\chapter{Introduction} \label{intro}

\section{The $n$-queens problem} \label{sec_queens}

The $n$-queens problem has a long and varied history. Stated in its classical form it 
asks, given an $n \times n$ chessboard, how many different ways there are to place $n$ queens on the board, so that no two can attack one another. That is, how many ways can one place $n$ items on the board so that no two share the same row, column, or either diagonal. A trivial upper bound is $n!$, observed by counting the number of choices so that no two queens share a row or column. With the problem being of such natural interest recreationally, as well as mathematically, and dating back a long way, it is somewhat difficult to follow the initial history of the problem from original sources, and various authors of work relating to the problem cite different histories. However we give details of the initial history in line with that of a survey of the problem by Bell and Stevens \cite{survey}, who in particular address some of these mis-citations. The $8$-queens problem was first published in 1848 by Max Bezzel~\cite{berliner}, a chess composer, in the monthly German chess magazine, `Schachzeitung der Berliner Schachgesellschaft' (which from 1872 became known as the `Deutsche Schachzeitung'). In 1850, Nauck~\cite{nauck} solved the problem, finding all $92$ solutions, though not proving that this list was complete. At the same time, Gauss was thinking about the problem as seen, for example, in letters to an astronomer friend Schumacher later published in \cite{peters}. However Gauss only found $72$ of the solutions before becoming aware of Nauck's $92$ solutions. According to Campbell~\cite{campbell}, in these letters Gauss reformulated the $8$-queens problem as an arithmetic one and related it to the representation of complex numbers, though seemingly this got Gauss neither further in enumerating the solutions, nor in proving an upper bound. Again, various different attributions are made as to who first proposed the generalised $n$-queens variation of the problem. One of the earliest references is by Lionnet~\cite{lionnet} in 1869 who posed the following arithmetic representation. Note first that any solution (a placement of $n$ non-attacking queens on the $n \times n$ board) is a permutation of $[n]$, seen by labelling the rows and columns $\{1, \ldots, n\}$. Letting $S_n$ denote the symmetric group of order $n$, the conditions for such a permutation $\sigma=(i_1, \ldots, i_n) \in S_n$ to represent a solution are equivalent to the arithmetic question of whether $i_j-i_k \neq \pm(j-k)$ for all ordered pairs $(j,k) \in [n] \times [n]$ with $j \neq k$. Equivalently, $(i_1, \ldots, i_n) \in S_n$ is a solution if and only if $i_j-j \neq i_k-k$ unless $j=k$ and $i_j+j \neq i_k+k$ unless $j = k$, for every $(j,k) \in [n] \times [n]$. To see this, first note that each square on the board identifies a unique combination of row, column, NW-SE (or {\it forward}) diagonal, and NE-SW (or {\it backward}) diagonal. Then we can think of labelling the rows and columns $\{1, \ldots n\}$ from the bottom left hand corner of the board and then have the diagonals labelled in such a way that the square on the board in row $i$ and column $j$ has forward diagonal $i+j$ and backward diagonal $i-j$. (Then in order to place queens so that no two can attack one another, we require that $i_1 \pm j_1 \neq i_2 \pm j_2$ for two queens placed in rows $i_1$ and $i_2$ and columns $j_1$ and $j_2$ respectively.)

Let $Q(n)$ denote the number of solutions to the $n$-queens problem. That $Q(8)=92$, confirming Nauck's solution, was proven in 1874 by Glaisher~\cite{glaisher}. Pauls~\cite{pauls1, pauls2} also proved this in the same year and, in addition also proved that $Q(n) \geq 1$ if and only if $n \notin \{2,3\}$, establishing that at least one solution exists to the generalised problem whenever $n \geq 4$ (as well as when $n=0$ or $1$). In general, there has been some success in calculating $Q(n)$ exactly for small values of $n$. In particular, sequence $A000170$ in the OEIS gives the value up to $n=27$. Concerning upper and lower bounds for all $n$ or infinite subsequences, progress has been much slower until relatively recently. In particular the precise nature of $Q(n)$ is difficult to understand. This has led to interest in studying several different variants of the problem which seem more straightforward or more mathematically natural, including the toroidal $n$-queens problem, which is the key consideration here. In particular, the toroidal $n$-queens problem views the $n \times n$ chessboard on the surface of a torus, which affects the diagonals. The diagonals `wrap-around' the standard board so that identifying a square in row $i$ and column $j$ which on the standard board is on the diagonals $i+j$ and $i-j$, the square is now considered to be on the diagonals $i \pm j \mymod n$, so that there are in total $n$ of each type of diagonal, as well as $n$ rows and $n$ columns, creating a more regular mathematical structure. Letting $T(n)$ denote the number of solutions to the toroidal $n$-queens problem, it should be clear that $Q(n) \geq T(n)$ for all $n$. In particular, any solution to the toroidal problem is also a solution on the original board of the same dimensions. Due to the close links between the standard and toroidal $n$-queens problems, along with the intractability of the former, many papers have considered the two variants simultaneously and thus much of the work on lower bounds for $Q(n)$ has been via work on $T(n)$. P\'{o}lya~\cite{polya} was the first to formally ask the toroidal $n$-queens problem in 1918, though Lucas~\cite{lucas} had earlier showed that $Q(p)>T(p)>p(p-3)$ whenever $p$ is prime. P\'{o}lya~\cite{polya} also showed that $T(n) \geq 1$ if and only if $n \equiv 1,5 \mymod 6$. It was not until 1994 that Rivin, Vardi and Zimmerman~\cite{UB} were able to show exponential bounds for some values of $n$. In particular, they showed for every prime $p$ that $T(p)>2^{\sqrt(p-1)/2}$. Moreover, if $p$ is a prime such that $(p-1)/2$ is not prime and $d$ is the smallest nontrivial divisor of $(p-1)/2$, then $T(p)>2^{(p-1)/(2d)}$. They also showed that if $n$ is divisible by a prime $\equiv 1 \mymod 4$ then $T(n)>2^{n/5}$, and separately if $\gcd(n,30)=5$, then $Q(n) > 4^{n/5}$. Beyond this, Rivin, Vardi and Zimmerman also conjectured asymptotic super exponential bounds for $Q(n)$ and $T(n)$, in particular,

\begin{conj}[Rivin, Vardi, Zimmerman~\cite{UB}] \label{conj_rvz}
 $$\log Q(n)=\Theta(n\log(n)),$$ and for $n \equiv 1,5 \mymod 6$, 
$$\log T(n)=\Theta(n\log(n)).$$
\end{conj}  

In 2017, Luria \cite{luria} made significant progress towards both of these conjectures. Firstly, Luria showed that the upper bound given by Conjecture \ref{conj_rvz} holds in both cases, and more precisely showed that $T(n)\leq ((1+o(1))\frac{n}{e^3})^n$ and $Q(n) \leq ((1+o(1))\frac{n}{e^{c}})^n$, where $c=-3+2\sqrt{3/5}\cdot\arctan(\sqrt{5/3})> 1.587$. Luria also showed that if $n=2^{2k}+1$ for some $k \in \mathbb{N}$, then $T(n) \geq n^{n/16-o(1)}$, hence proving the conjecture (both for the standard and toroidal problems) for some values of $n$. Luria went on to give a stronger conjecture than that of Rivin, Vardi and Zimmerman in the toroidal case, that when $n \equiv 1,5 \mymod 6$ we have $T(n)= ((1+o(1))\frac{n}{e^3})^n$. In fact, our main result is precisely that this conjecture is true.

\begin{theo} \label{thm_main}
Let $n \equiv 1,5 \mymod 6$. Then
$$T(n) = \left((1+o(1))\frac{n}{e^3}\right)^n.$$
\end{theo}

Theorem \ref{thm_main} both settles the conjectures of Rivin, Vardi and Zimmerman, and of Luria concerning the toroidal problem, as well as answering the original question of P\'{o}lya asymptotically in the exponent. As far we are aware, there are no other known bounds for the toroidal problem for general $n \equiv 1,5 \mymod 6$ or for any infinite subsequence. In particular, until now, there was no non-trivial lower bound on the number of solutions to the toroidal $n$-queens problem that held for all $n \equiv 1,5 \mymod 6$ (i.e. all $n$ for which at least one solution exists), and due to Luria's upper bound, our work here on the lower bound asymptotically settles the problem completely. 

We also prove that the bound on $T(n)$ given above when $n \equiv 1,5 \mod 6$ yields a lower bound for all $n$ in the classical case.

\begin{theo}\label{thm_class}
For $n \in \mathbb{N}$ sufficiently large we have that
$$Q(n) \geq \left((1+o(1))\frac{n}{e^3}\right)^n.$$
Furthermore, this lower bound only considers configurations with at most six pairs of queens that attack toroidally.
\end{theo}

This result, along with Theorem \ref{thm_main} and the upper bound of Luria settles Conjecture \ref{conj_rvz} completely (both the classical and toroidal cases). Independently, very recently, Luria and Simkin~\cite{lslb} also released a preprint obtaining the same lower bound on $Q(n)$, thereby also settling the conjecture in the classical case. Following on from this, Simkin~\cite{simkin} subsequently released a preprint improving both the lower and upper bounds for the classical problem, with the main result as follows:

\begin{theo}[\cite{simkin}]\label{thm_simkin}
There exists a constant $1.94<c<1.9449$ such that
$$Q(n)=\left((1+o(1))\frac{n}{e^c}\right)^n.$$
\end{theo}

This result gets significantly closer to the `truth' for the classical problem and is a big breakthrough concerning the classical case. However neither of these recent results yields headway on the toroidal problem for any $n$ and nor do the methods shed any new light on how to solve the toroidal problem. We defer mention of the methods used for these recent lower bounds to discuss alongside the methods we use to prove Theorem \ref{thm_main}.

The $n$-queens problem (both the classical and toroidal versions) are not only of recreational interest, but their importance goes well beyond this. As well as the algebraic formulation discussed above, there are various combinatorial reformulations that we discuss in the next section that demonstrate its abstract mathematical importance. Beyond this, the problem is useful in the discussion of mathematical optimisation and in design and analysis of algorithms, as well as having various practical applications. One algorithmic application has been the use of the problem as an example problem for various programming techniques. In particular, Dijkstra~\cite{dij} used the $8$-queens problem to demonstrate the importance and use of backtracking algorithms and recursion in programming. In~\cite{applics1}, Erbas, Sarkeshikt and Tanik discuss various algorithms for generating solutions to the problem, as well as noting various reformulations, including those in integer programming and constraint satisfaction. From a computer science perspective, the problem is not only useful in the study of algorithms, but solutions to the $n$-queens problem are also useful for memory storage schemes. One example of this, observed by Shapiro~\cite{shapiro}, is that finding valid periodic skewing schemes, a class of conflict-free storage algorithms for parallel memories, is identical to the problem of finding solutions to the toroidal $n$-queens problem. Yang, Wang, Liu and Chiang~\cite{pixel} discuss the use of $n$-queens solutions in pixel decimation which can be used to improve the speed of block motion, used in various video coding standards. For a more comprehensive list of applications, we direct the reader to the afore mentioned survery of Bell and Stevens~\cite[Section 2]{survey}.

\section{Graphical reformulations of the $n$-queens problem}\label{sec_reform}

Many different approaches to tackling the $n$-queens problem have been made over the years, with its algebraic, algorithmic and combinatorial structure luring in mathematicians and computer scientists from many different areas. One aspect of this is that the problems can be phrased in various different ways which are of much more general mathematical importance and have been studied extensively in their own right. To this end, we mention here a few of the graphical realisations of the problem. 

Firstly, one could represent each of the $n^2$ squares of the chessboard via a vertex of a graph, and form edges between two vertices if they lie in the same row, column or either of the two diagonals. In this way, a single solution to the problem is an independent set of size $n$, and so it is equivalent to counting independent sets of size $n$ in this graph. (Note that the largest independent set in this graph is of size $n$ or smaller, and so for all $n$ such that at least one solution to the $n$-queens problem exists, this corresponds to counting maximum size independent sets.) Enumerating independent sets in graphs and hypergraphs is a fundamental problem in combinatorics and many different tools and techniques have been developed to work on such problems. Some recent tools to have contributed significantly to the literature of results include the (hypergraph) container method, and methods from statistical physics, including the use of polymer models and cluster expansion. One can also view the problem in terms of maximum cliques. In particular, defining the graph on $n^2$ vertices corresponding to the squares of the board and joining pairs of vertices by an edge when they do {\it not} share a row, column or either diagonal, we get the complement graph of the previously defined graph, and a solution to the $n$-queens problem is equivalent to a clique of size $n$ (the maximum size that any clique in this graph will possibly have). 

Another formulation, in fact the key one for us, comes in terms of the following hypergraph which we refer to as {\it the $n$-queens hypergraph, $\mathcal{Q}(n)$}. We let $V(\mathcal{Q}(n))$ consist of a vertex for each of the $n$ rows, $n$ columns, $2n-1$ forward diagonals and $2n-1$ backward diagonals, and $E(\mathcal{Q}(n))$ consist of all quadruples that identify a unique square on the board. In this way two edges intersect if and only if they share a row, column, forward diagonal or backward diagonal, and so two queens which are non-attacking correspond to a pair of disjoint edges in $\mathcal{Q}(n)$. Hence the $n$-queens problem is now to count the number of matchings which cover the row and column parts of $\mathcal{Q}(n)$, a $4$-partite, $4$-uniform hypergraph. From now on we shall name our parts $V^X$, $V^Y$, $V^{X+Y}$ and $V^{X-Y}$, and may also refer to them simply as the $X$, $Y$, $X+Y$ and $X-Y$ parts, respectively. Note that, labelling the vertices corresponding to rows and columns from $0,\ldots,n-1$, the forward diagonal from $0,\ldots,2(n-1)$, and the backward diagonal from $-n,\ldots,0, \ldots, n$, fixing row $i$ and column $j$ dictates the square with diagonals $i+j$ and $i-j$. In particular fixing any row and column pair dictates a unique square on the board, and thus a unique edge in $\mathcal{Q}(n)$, and we can write 
\begin{eqnarray*}
V(\mathcal{Q}(n))&:=& \{i^X:i \in \{0,\ldots,n-1\}\}~ \dot\cup ~ \{i^Y:i \in \{0,\ldots,n-1\}\}  \\ &~& \dot\cup ~ \{i^{X+Y}:i \in \{0,\ldots,2(n-1)\}\}~ \dot\cup ~ \{i^{X-Y}:i \in \{-n,\ldots,n\}\} \\ &=& V^X ~ \dot\cup ~ V^Y~ \dot\cup ~V^{X+Y}~ \dot\cup ~V^{X-Y},\mbox{~and~}\\ 
E(\mathcal{Q}(n)) &:=& \left\{(i^X, j^Y, i+j^{X+Y}, i-j^{X-Y}):i,j \in \{0,\ldots,n-1\}\right\} \\
&\subseteq& V^X \times V^Y \times V^{X+Y} \times V^{X-Y}.
\end{eqnarray*} 
Thus $\mathcal{Q}(n)$ is a $4$-graph such that every vertex in $V^X \cup V^Y$ has degree $n$, and each pair $(i,j)\in V^X \times V^Y$ has pair degree $1$. Whilst the problem of counting matchings covering the $X$ part in $\mathcal{Q}(n)$ is not an unnatural problem, we see that from this graph theoretic perspective, the toroidal problem is even more natural. In particular, the related $4$-partite $4$-graph, hereafter referred to as {\it the $n$-queens toroidal hypergraph, $\mathcal{T}(n)$},\label{toroidal_graph} can be seen as having parts $V^X$, $V^Y$, $V^{X+Y}$ and $V^{X-Y}$ each containing exactly $n$ vertices (or {\it coordinates}) with labels $0,\ldots,n-1$, so that $(i,j, i+j \mymod n, i-j \mymod n) \in V^X \times V^Y \times V^{X+Y} \times V^{X-Y}$ is an edge in $\mathcal{T}(n)$ for every $i,j \in [0,n-1]$. In this context, answering the toroidal $n$-queens problem corresponds to counting {\it perfect} matchings in $\mathcal{T}(n)$. Furthermore, $\mathcal{T}(n)$ is an $n$-regular graph with pair degree exactly $1$ across all pairs from two different parts when $n$ is odd, and in particular when $n \equiv 1,5 \mymod 6$.\footnote{In fact when $n$ is even, we find this is true over all pairs across two parts except for pairs $(a,b) \in V^{X+Y} \times V^{X-Y}$. For these, we find the pair degree is exactly $2$ if $a \equiv b \mod 2$, and $0$ otherwise.} So the toroidal $n$-queens problem (where we only consider $n \equiv 1,5 \mymod 6$) becomes that of counting perfect matchings in an $n$-regular $4$-partite $4$-graph on $4n$ vertices with partite pair degree exactly $1$ for all pairs. The study of matchings in graphs and hypergraphs is so extensive that it seems natural to consider the toroidal $n$-queens problem in this context, and from now on we consider the problem as such, and often refer to counting perfect matchings in $\mathcal{T}(n)$, rather than counting solutions to the toroidal $n$-queens problem. 

Note that we could also view this problem as a {\it rainbow} matching problem in a $3$-partite $3$-graph. In particular, taking three parts with vertices to represent the rows, columns and one of the diagonals, then we may use colours to represent the other diagonal. Then we have an edge of colour $c$ in the graph if and only if the three vertices and colour correspond to a particular position on the board. In this setting, a rainbow matching (a matching where every edge has a different colour) of size $n$ yields a solution to the $n$-queens problem, since two queens can attack each other if and only if they are in the same row, column, or either diagonal, which in this case corresponds to two edges sharing a vertex or a colour. 

\subsection{A very brief history of matching theory}

Matching theory encompasses a wide variety of problems in combinatorics and has diverse applications both within and beyond mathematics. There is a wealth of literature covering many different areas including (but not limited to) the characterisation of graphs with perfect matchings (in particular Hall's Theorem \cite{hall} for bipartite graphs, and Tutte's Theorem \cite{tutte} for a general graph $G$) and efficient algorithms for finding matchings under certain conditions (such as Gale and Shapley's \cite{stablemarriage} Stable Marriage assignment, and Kuhn's `Hungarian' algorithm \cite{hungarian}). The natural interest in matchings in graphs extends readily to hypergraph matchings. Indeed, hypergraph matching theory is at the heart of complexity theory as one of Karp's original 21 NP-complete problems \cite{karp}. This has led to much work looking for sufficient conditions to guarantee a perfect matching in a hypergraph. A key area of this is to look at problems concerning minimum degree conditions, also known as `Dirac-type' problems, so-called as they spring from a classical result of Dirac \cite{dirac} which states that every graph $G$ on $n$ vertices with minimum degree at least $n/2$ contains a Hamilton cycle (and thus, when $n$ is even, a perfect matching, seen by taking every other edge in a Hamilton cycle).  Whilst much progress has been made in determining minimum degree thresholds for perfect matchings in $k$-uniform hypergraphs (or {\it $k$-graphs} for short), there is a wealth of open questions still remaining. In particular, both the asymptotic and exact minimum vertex degree thresholds for a perfect matching in a $k$-uniform hypergraph still remain open for $k \geq 6$. For more on Dirac-type problems we suggest the surveys~\cite{r&r, zhao}, though there has been a lot more progress on such problems since.

Beyond the existence of {\it perfect} matchings, it is interesting to consider, given a family of graphs $\mathcal{H}$ defined by a particular structure, what size matching can we guarantee to find in any $H \in \mathcal{H}$? It is interesting to lower bound the size of a {\it maximum} matching in a class of graphs, even when it is known that a perfect matching cannot be guaranteed. Again, minimum degree conditions are still an area of interest here and there are many results to this end. Another natural family to consider is the family of regular $k$-graphs, of which the toroidal $n$-queens graph is a member. Whilst the result of P\'{o}lya~\cite{polya} determines that $\mathcal{T}(n)$ has a matching of size $n$ if and only if $n \equiv 1,5 \mymod 6$, it was a while before the size of the largest matching was determined for all other $n$. In particular, various authors considered this and gave partial solutions before Monsky~\cite{monsky} settled the problem completely (when combined with P\'{o}lya's result), showing that for every $n$ there exists a partial solution of size $n-2$, and a partial solution of size $n-1$ if and only if $n$ is not divisible by either $3$ or $4$.

More generally, we may wish to find a matching leaving only a $o(1)$ proportion of vertices uncovered, and indeed see how far we can push the $o(1)$ term for a general family of graphs. Probabilistic methods have been discovered to be very successful for proving results in this realm. This was initiated by R\"{o}dl~\cite{rodlnibble} who introduced a semi-random construction method that is now referred to as the {\it R\"{o}dl nibble} to settle a conjecture of Erd\H{o}s and Hanani~\cite{approx_designs} concerning approximate Steiner systems. A {\it Steiner system} with parameters $(n,k,l)$ is a collection $\mathcal{A}$ of $k$-sets from $[n]$ such that every $l$-set from $[n]$ is a subset of exactly one element from $\mathcal{A}$. This is equivalent to a perfect matching in the complete $\binom{k}{l}$-graph on vertex set $\binom{[n]}{l}$. Concerning approximate Steiner systems, let $m(n,k,l)$ be the maximal size of a family $\mathcal{A}$ of $k$-sets from $[n]$ such that every $l$-set is a subset of {\it at most} one element from $\mathcal{A}$. Then R\"{o}dl \cite{rodlnibble} showed that $\lim_{n \rightarrow \infty} \frac{m(n,k,l)}{\binom{n}{l}/\binom{k}{l}}=1$, which implies a matching covering all but a $o(1)$ proportion of the vertices in the $\binom{k}{l}$-graph described above. Frankl and R\"{o}dl~\cite{frankl_rodl} and Pippenger and Spencer~\cite{pip_spencer} observed that the nibble applied in the more general setting of almost regular uniform hypergraphs with small maximum pair degrees. Building on this, again inspired by the nibble technique, Grable~\cite{grable} gave a more precise analysis under slightly stronger conditions concerning the pair degree conditions, finding a matching in an $n$ vertex graph that covers all but $n^{1-\alpha}$ vertices for some constant $\alpha>0$. More recently, related methods have been used to establish even larger matchings in regular hypergraphs with small pair degrees. Alon, Kim and Spencer~\cite{alonnibble} considered simple $d$-regular $n$-vertex $k$-graphs (those in which pair degree is at most 1), and showed that when $k=3$, the graph contains a matching covering all but at most $O(nd^{-1/2}\log^{3/2}(d))$ of the vertices, and when $k>3$ there is a matching covering all but at most $O(nd^{-1/(k-1)})$ vertices. Note that this is a vast improvement in the case where $d$ is close to $n$. The method used here differs from the R\"{o}dl nibble but takes some inspiration from the method, requiring adjustments and martingale inequalities for the more careful analysis. More recently, Bennett and Bohman~\cite{randomgreedy} analysed the case of $d$-regular $k$-graphs on $n$ vertices where $d \rightarrow \infty$ as $n \rightarrow \infty$ and pair degrees are at most $l=\frac{d}{\log^{\omega(1)}(n)}$. They analyse a random greedy matching process using the differential equations method to show that with high probability ({\it whp}) every graph satisfying the conditions above contains a matching covering all but $(l/d)^{(1/(2r-2))+o(1)}$ vertices.

Another natural question concerns {\it counting} matchings. One might ask how many matchings a graph has in total, or of a given size, and in particular, given a (hyper)graph $H$ such that a perfect matching exists, how many different perfect matchings does $H$ contain? Even in $2$-graphs, the counting problem is known to be $\sharp$P-complete \cite{sharp, sharp2}, thus leading to much interest in solving the problem for different types of graphs and hypergraphs. As previously mentioned, our work relates to a problem concerning counting perfect matchings in $\mathcal{T}(n)$. Our counting approach makes use of some of the ideas used in the analysis of large matchings in regular $k$-graphs, specifically the strategy of Bennett and Bohman~\cite{randomgreedy}, leaving us with a subgraph of $\mathcal{T}(n)$ on which we wish to find only a single perfect matching. More details are given throughout.  

\section{Semi-queens}\label{sec_semi_queens}

Before discussing our methods for working on the lower bound of the toroidal $n$-queens problem we discuss another related problem known as the {\it $n$-semi-queens} problem. We define a {\it semi-queen} to be a chess piece that can attack along rows and columns and along the forward diagonal (this could instead be the backward diagonal, but the important point is that all semi-queens on a board can attack along the same diagonal). Then the semi-queens version (and its toroidal counterpart) asks for the number of ways to place $n$ non-attacking semi-queens on an $n \times n$ (toroidal) board. The corresponding graphs, $\mathcal{Q}'(n)$ and $\mathcal{T}'(n)$, can be seen as $3$-partite $3$-uniform hypergraphs, with parts $V^X$, $V^Y$ and $V^{X+Y}$, where every pair $(i,j) \in V^X \times V^Y$ defines an edge $(i,j,i+j) \in V^X \times V^Y \times V^{X+Y}$, and as before, the solution to the problem corresponds to counting the matchings covering $V^X \cup V^Y$. (Note that this is the same graph as discussed in the rainbow matching realisation of the (full) $n$-queens problem, but ignoring the colours given to the edges.) Clearly this is a relaxation of the queens problem, and every solution to the queens problem is a solution to the semi-queens problem. Let $Q'(n)$ and $T'(n)$ denote the number of solutions to the classical and toroidal $n$-semi-queens problems respectively. It is known that $T'(n)>0$ if and only if $n$ is odd. The toroidal semi-queens problem was recently solved by Eberhard, Manners and Mrazovi\'{c} \cite{semi-queens}, who showed that $T'(n)=(1+o(1))\frac{n!^2}{(\sqrt{e}) n^{n-1}}$ using methods from analytic number theory. We mention the semi-queens problem now, not only since it is interesting in its own right, but furthermore in Section \ref{understanding_the_lattice}, in order to prove statements about $\mathcal{T}(n)$, it is helpful to first prove statements about $\mathcal{T}'(n)$. 

For further details on the history of the $n$-queens problem, we refer the reader to the survey \cite{survey}, though this does not have the recent bounds of Luria \cite{luria} and Luria and Simkin~\cite{lslb}, Simkin~\cite{simkin}, nor the bounds of Eberhard, Manners and Mrazovi\'{c}~\cite{semi-queens} in the semi-queens problem. Additionally we draw the reader's attention to the website \cite{website}, where Walter Kosters from Universiteit Leiden maintains a bibliography of papers and results concerning the $n$-queens problem. 

\section{Strategy overview}

Our proof uses a random greedy algorithm to lower bound the number of almost-perfect matchings, and uses an absorbing strategy to show that, with high probability, each of the almost-perfect matchings included in the count extends to at least one distinct perfect matching. It will become clear later precisely what we mean by an almost-perfect matching in this context but, informally, we mean a matching covering all but a $o(1)$ proportion of the vertices. The absorbing strategy combines the techniques of randomised algebraic construction and iterative absorption, both recently developed methods that have independently been used to make big breakthroughs in design theory. The former was introduced by Keevash~\cite{designs1} to prove the Existence Conjecture for combinatorial designs and the latter was introduced and developed for hypergraph decomposition by K\"{u}hn, Osthus and various coauthors (Barber, Glock, Lo, Montgomery)~\cite{itabs1, itabs2, it_abs}, and was used to give a new proof of the Existence Conjecture. Keevash's result was a great breakthrough in combinatorial design theory, answering a long standing open question posed by Steiner in 1853. Since then Keevash \cite{designs2} has generalised this work to the setting of subset sums in lattices with coordinates indexed by labelled faces of simplicial complexes. This includes hypergraph decompositions in partite settings, and the method extends to give approximate counting results for structures such as the latin hypercube (also known as an $r$-dimensional hypercube) and Sudoku squares. The result \cite[Theorem 1.7]{designs2} combined with the analysis of a random greedy algorithm yields lower bounds for these quantities, and matching upper bounds follow from the work of Linial and Luria \cite{linial} who use the entropy method to obtain these bounds. It is explicitly noted in the paper that \cite[Theorem 1.7]{designs2} does not apply to the toroidal $n$-queens problem. Details of the proof strategy used for these counting arguments are presented in more detail in \cite{countingdesigns}, where Keevash applies the method to the problem of counting Steiner Triple Systems, that is Steiner systems with parameters $(n,3,2)$. In particular, Keevash gives a lower bound for the number of Steiner Triple Systems on $n$ vertices, $STS(n)$, by considering the equivalent problem of counting the number of different triangle decompositions of $K_n$. To do so, the proof relies first on the random greedy triangle removal process of Bohman, Frieze and Lubetzky \cite{triangle_removal}, which gives a mechanism for lower bounding the number of different partial triangle decompositions of size $\frac{n^2}{6}-cn^\alpha$ for some constants $c>0$ and $\frac{3}{2}<\alpha<2$. The second part shows that in the remaining graph, with high probability, a triangle decomposition exists, thus in $(1+o(1))$ proportion of cases transforming the partial triangle decomposition into a full triangle decomposition, enabling us to lower bound $STS(n)$. The random greedy triangle removal process shows that, up until a certain point, the graph remains quasi-random, and retains what are referred to as {\it c-typical} conditions, which relate to the degree of each vertex as well as the size of the intersection of neighbourhoods for pairs of vertices being close to what is expected in a random graph of the same density. Thus, after the counting process we are left with a sparse quasi-random graph for which it suffices to prove that a single triangle decomposition exists. Keevash uses the method of randomised algebraic construction in this sparse setting to prove the existence of a triangle decomposition in the remaining graph.
 
We follow a similar strategy to lower bound $T(n)$, with two key differences. Firstly, we take out an absorber $A^*$ {\it before} doing the almost-perfect matching count on $\mathcal{T}(n) \setminus A^*$ until we reach a small subgraph $H$, and then show that with high probability the graph $\mathcal{T}(n)[V(H)\cup A^*]$ has a perfect matching. Secondly, in order to use our carefully chosen absorber $A^*$, we use an iterative process to cover vertices remaining in $H$ which leaves a carefully chosen subset of vertices $L^*$ to be absorbed by $A^*$, as per the iterative absorption strategy.  
Crucially, however, our strategy takes advantage of the algebraic structure embedded in $\mathcal{T}(n)$ to control the iterative process. Additionally, our approach to showing the existence of the required absorber $A^*$ uses ideas from the method of randomised algebraic construction. Since we are concerned with asymptotic thresholds for $T(n)$, from now on we may write $\mathcal{T}$ to mean $\mathcal{T}(n)$ for any $n$ sufficiently large, and do similarly for $\mathcal{T}'$, $\mathcal{Q}$ and $\mathcal{Q}'$.

As previously mentioned, our work proves a lower bound matching the upper bound of Luria \cite{luria}, and we may divide the strategy into three key elements: (1) {\it the absorber}, (2) {\it the random greedy count} and (3) {\it the iterative matching strategy}. Indeed, we may describe the strategy in a way which is comparable to Keevash's strategy in \cite{countingdesigns}. Firstly let $H^* \subseteq \mathcal{T}$ be a subgraph of $\mathcal{T}$ which has a `template' perfect matching. This is our absorber $A^*$ described as a subgraph of $\mathcal{T}$, so that $A^*:=V(H^*)$, with $\mathcal{T}[A^*]$ containing a perfect matching. Then running a random greedy edge removal process on $\mathcal{T}^{-A^*}:=\mathcal{T}[V(\mathcal{T})\setminus A^*]$, with high probability we are left with a subgraph $H \subseteq \mathcal{T}^{-A^*}$ which satisfies various random-like properties (details of which are discussed in Chapter \ref{ch_count}). On $H$ we run an iterative matching strategy (details of which are discussed in Chapter \ref{ch_it_abs}), which builds up an almost-perfect matching $M$ for $H$ in gradual steps, accumulating smaller and smaller disjoint matchings over $O(\log(n))$ steps of an iterative process, eventually leaving only a small collection of vertices $L^*$ (referred to as {\it the leave}) uncovered. We show for any leave $L^*$ satisfying various requirements discussed in Chapter \ref{ch_absorber} that $\mathcal{T}[L^* \cup A^*]$ has a perfect matching. This is done following ideas from Keevash's randomised algebraic construction. First we  consider an integral relaxation of the problem and express $L^*$ as the difference of two collections of edges of bounded size in $\mathcal{T}$. Then, from this, through several steps, we are able to show that in fact $L^*$ can be expressed in terms of the difference of two matchings $M^+$ and $M^-$ such that $M^+$ covers $L^*$ and $H^*\setminus M^-$ has a perfect matching. This yields a perfect matching in $\mathcal{T}[L^* \cup A^*]$, as required. For the random greedy count we use the same strategy as Bennett and Bohman~\cite{randomgreedy}, that is, we analyse a random greedy matching process via the differential equations method. In fact we can follow their strategy very closely, however we are unable to use their result as a black box due to additional properties required in the graph left when we stop the process.

Comparing our strategy to the methods used in \cite{lslb} and \cite{simkin} to obtain the recent lower bounds for the classical $n$-queens problem, we note that \cite{lslb} similarly combines a random greedy counting argument with an absorbing strategy. Luria and Simkin give details of the random greedy matching process directly as building a random greedy non-attacking queens configuration on the toroidal board, however, noticing the relation to perfect matchings they could use details from \cite{randomgreedy} as a black box. The absorbing method they use then takes such a partial configuration and relaxes from the toroidal setting to the classical setting, increasing the number of diagonals, and notes that the remaining unfilled rows and columns can be filled (or absorbed) by making only small switches. This gives a lower bound for classical $n$-queens configurations which are `approximately' toroidal. Luria and Simkin note that the absorbers used in their strategy are difficult to find in the toroidal case due to significantly fewer diagonals, and it seems a much more complex absorbing strategy is required in the toroidal case, as seems to be the case following our methods. Simkin's~\cite{simkin} even more recent result for the classical $n$-queens problem considers the problem as a convex optimization problem in the space of Borel probability measures on the square and uses numerical computations for both the upper and lower bounds. Defining limit objects for $n$-queens configurations referred to as {\it queenons}, for the lower bound Simkin uses a randomised algorithm that constructs queens configurations close to a given queenon. The entropy of this process matches the entropy of the upper bound giving the result. Again, Simkin remarks, as in the preprint with Luria, that the absorption method in this paper takes advantages of the `freedom' of the many unoccupied diagonals in the classical case, and so cannot be used for the toroidal setting. Simkin also remarks that `perhaps [analytic number theory] is required to understand $T(n)$'. Our results show that this is not necessary, at least as far as asymptotics in the exponent are concerned. Again regarding the toroidal problem, Simkin also says `we wonder if the methods of randomised algebraic construction or iterative absorption might be more appropriate'. Our work on the problem and final proof suggest that neither technique on its own is able to tackle the problem, but combining the two methods does indeed give the desired result. In particular, a key aspect of the method of randomised algebraic construction is to find an algebraic template that has powerful absorbing properties, which was elusive to us in this setting and led to considering ideas from the method of iterative absorption. Equally, though our strategy overall takes on an iterative absorbing approach, to build our absorbers required another key element of the randomised algebraic construction strategy, referred to as {\it hole} (see e.g.~\cite{countingdesigns}).

\subsection{Organisation}

We continue the discussion of the toroidal $n$-queens problem from Chapter \ref{ch_structure}. We start with a more in depth outline of the proof of Theorem \ref{thm_main} before developing a deeper understanding of the structure of $\mathcal{T}$, and introducing the notion of {\it zero-sum configurations} as well as some degree-type definitions all of which are vital in our proof strategy. Before this, in Chapter \ref{ch_notation_prelims}, we present some definitions, notation, inequalities and probabilistic tools that are generally useful throughout.  
Chapter \ref{ch_absorber} focuses on the details of the absorbing strategy and also includes the proof of a result associated with the work of Chapters \ref{ch_structure} and \ref{ch_absorber} which does not directly relate to the proof of Theorem \ref{thm_main}, but is key to the proof of Theorem \ref{thm_class}.  
Chapter \ref{ch_count} concerns the random greedy counting process, and some parity modifications that take us to a subgraph $H \subseteq \mathcal{T}$ on which it remains for us to run the iterative matching process to obtain a small leave $L^*$ which can be absorbed by $A^*$ when the necessary divisibility conditions are satisfied.
Chapter \ref{ch_it_abs} concerns the details of this iterative matching process, in which we shift to considering {\it weighted} subhypergraphs of $\mathcal{T}$ and build the matching covering all remaining vertices but some `good' leave $L^*$ over $O(\log(n))$ such hypergraphs.   
Finally, in Chapter \ref{ch_classical}, we give the remaining details for the proof of Theorem \ref{thm_class} and make some concluding remarks.

%% file: ch_notation_prelims.tex
\chapter{Definitions, notation and preliminaries} \label{ch_notation_prelims}

In this section we introduce some fundamental notation, definitions and results that will be required throughout. We write $[n]$ to be the set of integers from $1$ to $n$, that is $[n]:=\{1, 2, ..., n\}$, and $[i,n]:=\{i, i+1, ..., n\}$ to be the set of integers from $i$ to $n$ for any $i \leq n$. In an abuse of notation, we also use $[a,b]$ in the continuous sense such that $[a,b]:=\{c \in \mathbb{R}: a \leq c \leq b\}$, but the meaning should be clear from context.

A {\it $k$-uniform hypergraph} $H:=H(V,E)$ consists of a vertex set $V(H)$ and an edge set $E(H)$ whose elements are $k$-subsets of $V(H)$. We shall often abbreviate $k$-uniform hypergraph to {\it $k$-graph}, and may use `graph' when referring to a hypergraph.  For an edge $e=\{v_1,...,v_k\}$, we may sometimes write $e=v_1...v_k$. Additionally we associate $H$ to its edge set rather than its vertex set, and so may write $e \in H$ or $|H|$ in place of $e \in E(H)$ and $|E(H)|$ respectively. For a hypergraph $H$ and $U \subset V(H)$, the subhypergraph of $H$ {\it induced} by $U$, $H[U]$, is the subhypergraph of $H$ with vertex set $V(H[U])=U$ and edge set consisting of all edges between the vertices of $U$ in $H$. In an abuse of notation, we may sometimes induce a subhypergraph on an edge set, or collection of edges rather than a set of vertices. In this case, it should be read as the subhypergraph induced by the vertices contained in the given collection.

A {\it matching} $M$ in a hypergraph $H$ is a collection of disjoint edges in $H$. That is, $M \subseteq E(H)$ such that $e \cap f = \emptyset$ for all distinct $e,f \in M$. A matching $M$ is a {\it perfect matching} if $\bigcup\limits_{e \in M} e = V(H)$. An {\it almost-perfect matching} is a matching such that all but a $o(1)$ proportion of the vertices are covered. We often use $V(M):=\bigcup_{e \in M} e$ for a matching $M$, or more generally for any collection of edges, or structures that are themselves collections of vertices.

For a set $S$ and $l \in \mathbb{N}$, we write $\binom{S}{l}$ to be the collection of all $l$-sets from $S$. Additionally we write $\binom{S}{\leq l}:= \bigcup_{i \leq l} \binom{S}{l}$.

The {\it degree}, $d_T(H)$ or $d_{\{v_1,...,v_t\}}(H)$, of a set of vertices $T=\{v_1,...,v_t\}$ in a hypergraph $H$ is the number of edges that contain $T$ as a subset. When we refer to the {\it vertex degree} we are considering sets $T=\{v\}$ only of size one and write $d_v(H)$, the {\it pair degree} refers to sets of size two, and the {\it codegree} of a $k$-graph refers to the degree of sets of size $k-1$.
The minimum $t$-degree, $\delta_{t}(H):=\min\{d_T(H):T \subseteq V(H), |T|=t\}$, of $H$ is the minimum of $d_T(H)$ over every subset of vertices $T$ in $H$ of size $t$.

\section{Bachmann--Landau Notation}

In this section we formally define little-$o$, big-$O$, $\omega$, $\Omega$ and $\Theta$ notation, collectively known as Bachmann--Landau notation, which are used throughout in the discussion of asymptotic results.
 
Let $f(n)$ and $g(n)$ be functions of $n$. Then we say that, 

\begin{enumerate}[(i)]
\item $f(n)=o(g(n))$ (as $n \rightarrow \infty$) if for every $\alpha>0$, there exists $n_0 >0$ such that $|f(n)|\leq \alpha |g(n)|$ for every $n \geq n_0$.  
\item We write that $f(n)=O(g(n))$ (as $n \rightarrow \infty$) if there exist $M, n_0 >0$ such that $|f(n)|\leq M|g(n)|$ for every $n \geq n_0$.
\item We write that $f(n)=\omega(g(n))$ (as $n \rightarrow \infty$) if $g(n)=o(f(n))$.
\item We write that $f(n)=\Omega(g(n))$ (as $n \rightarrow \infty$) if $g(n)=O(f(n))$.
\item $f(n)=\Theta(g(n))$ (as $n \rightarrow \infty$) if $f(n)=O(g(n))$ and $f(n)=\Omega(g(n))$.
\end{enumerate}

Throughout we shall also write $f(n) \ll g(n)$ to mean that $f(n)=o(g(n))$ and naturally also $f(n) \gg g(n)$ to mean that $g(n)=o(f(n))$.

\section{Inequalities and Bounds} \label{section1}

It will be convention to omit floor and ceiling signs unless crucial to an argument. We write $x = a \pm b$ to mean that $x \in [a-b,a+b]$ when $b \geq 0$, and $x \in [a+b,a-b]$ when $b<0$. 

The following inequalities and bounds are used in various places throughout.

\begin{prop} \label{log(1+x)}
For all $x \in (-1,\infty]$, $\log(1+x) \leq x$.
\end{prop}

\begin{prop} \label{log(p)}
Take $n$ sufficiently large and let $p(i):=1-\frac{i}{n}$. Then for every $1>\alpha>0$,
$\sum_{i=0}^{n-n^{1-\alpha}-1} \log(p(i))=n(-1+O(n^{-\alpha}\log(n)))$.
\end{prop}

\begin{proof}
We have that
\begin{multline*}
\int_{n^{-\alpha}}^{1} \log (x) \mbox{~d}x \leq \frac{1}{n} \sum_{i=0}^{n-n^{1-\alpha}-1} \log (p(i)) \leq \int_{n^{-\alpha}+n^{-1}}^{1+\frac{1}{n}} \log (x) \mbox{~d}x \\ \leq \int_{n^{-\alpha}+n^{-1}}^{1} \log (x) \mbox{~d}x+O(n^{-2}).
\end{multline*}
Then 
\begin{multline*}
\frac{1}{n} \sum_{i=0}^{n-n^{1-\alpha}-1} \log (p(i)) = \int_{n^{-\alpha}}^{1} \log (x) \mbox{~d}x \pm  
\left(\int_{n^{-\alpha}}^{n^{-\alpha}+n^{-1}} \log (x) \mbox{~d}x + O(n^{-2})\right)
\\ = \int_{n^{-\alpha}}^{1} \log (x) \mbox{~d}x + O(n^{-\alpha}\log(n)) = -1 + O(n^{-\alpha}\log(n)).
\end{multline*}
Rearranging gives the result.
\end{proof}

\subsection{Probabilistic bounds}

Various results rely on probabilistic arguments requiring knowledge of the following bounds. We say that an event $A$ holds {\it with high probability} ({\it whp}) if there exists some $\alpha>0$ such that $\mathbb{P}(A)=1-e^{-\Omega(n^{\alpha})}$ as $n \rightarrow \infty$. Note that this is not the usual definition, but a stronger statement.
In the following lemma, $B(n,p)$ is the binomial distribution with parameters $n$ and $p$.

\begin{lemma}\textup{(Chernoff's inequality, e.g.~\cite[Remark 2.8]{chernoff})} \label{chernoff}
Suppose that $X \in B(n,p)$. Then,
\[\mathbb{P}(|X-\mathbb{E}(X)| \geq t) \leq 2 \exp \left(- \frac{2t^2}{n} \right)\]
for every $t \geq 0$.
\end{lemma}

We state three versions of Azuma's Inequality, the first in terms of permutations, requiring also the definition of what it is for something to be {\it Lipschitz}, the second in terms of martingales, and the third a special case in terms of independent random variables. Note that we write $S_n$ for the symmetric group.
\begin{defn}[$b$-Lipschitz]
For $f:S_n \rightarrow \mathbb{R}$ and $b \geq 0$, we say that $f$ is {\it $b$-Lipschitz} if for any two permutations $\sigma, \sigma' \in S_n$ differing by a transposition, 
\[|f(\sigma)-f(\sigma')|\leq b.\]
\end{defn}
The next lemma follows from \cite[Section 11.1]{molloy} and, in particular, the discussion on page 93.

\begin{lemma}\textup{(Azuma's inequality, e.g.~\cite{molloy})} \label{azuma} 
Suppose $f:S_n \rightarrow \mathbb{R}$ is $b$-Lipschitz, and $\sigma \in S_n$ is a uniformly random permutation. Define $X=f(\sigma)$. Then
\[ \mathbb{P}(|X-\mathbb{E}(X)|>t) \leq 2\exp \left(- \frac{t^2}{2nb^2} \right).\] 
\end{lemma}

The Azuma-Hoeffding inequality which follows is perhaps the most basic concentration inequality for martingales. For random variables $(X_i)_i$ and $(Z_i)_i$ defined in the same probability space, we say that $(X_i)$ is a {\it martingale with respect to $(Z_i)$} if for every $i$ we have that $X_i$ is measurable given $Z_1, \ldots, Z_i$ (i.e. $X_i=f_i(Z_1,\ldots, Z_i)$), $\mathbb{E}(|X_i|)$ is finite and $\mathbb{E}(X_i | Z_1, \ldots, Z_{i-1})=X_{i-1}$.

\begin{lemma}\textup{(Azuma--Hoeffding  inequality)} \label{azuma_mart}
Let $(X_i)_i$ be a martingale with respect to random variables $(Z_i)_i$ in the same probability space. If $|X_i - X_{i-1}| \leq c_i$, then
$$\mathbb{P}(|X_n-X_0| \geq t) \leq 2\exp\left(-\frac{t^2}{2\sum_{i=1}^n c_i^2}\right).$$
\end{lemma}

\begin{cor}\textup{(Bounded difference inequality, also known as McDiarmid's inequality \cite[Theorem 3.1]{colin})} \label{colin}
Let $X=f(X_1, \ldots, X_n)$ where $X_1, \ldots, X_n$ are independent random variables. Suppose that for every $x_1, \ldots, x_n$ and $x_1', \ldots, x_n'$ and every $i \in [n]$ we have
$$|f(x_1, \ldots, x_i, \ldots, x_n)-f(x_1, \ldots, x_i', \ldots, x_n)|\leq c_i.$$
Then
$$\mathbb{P}(|X-\mathbb{E}(X)| \geq t) \leq 2 \exp \left( - \frac{2t^2}{\sum_{i \in [n]} c_i^2} \right ).$$ 
\end{cor}

We finish this section with Bernstein's inequality which enables us to bound sums of independent random variables using the second moment which can often give stronger concentration when McDiarmid's inequality is not sufficient.

\begin{lemma}[Bernstein's inequality, e.g.~\cite{concentration1}]\label{bernstein}
Let $X = \sum_{i=1}^n X_i$ be the sum of independent random variables such that $|X_i|<b$ for all $i \in [n]$. Then
$$\mathbb{P}(|X-\mathbb{E}(X)|>t)<2\exp\left(-\frac{t^2}{2(bt/3 + \sum_{i=1}^n \mathbb{E}(X_i^2))}\right).$$
\end{lemma}

%% file: ch_structure.tex
\chapter{Key structural properties of the toroidal $n$-queens hypergraph} \label{ch_structure}

\section{Proof Overview} \label{sec_overview}

We consider $\mathcal{T}$ (see page \pageref{toroidal_graph} for the definition) on vertex set with coordinate representatives $[-\frac{n-1}{2}, \frac{n-1}{2}]$ in each part, and in this way view the edge $(0,0,0,0)$ as being at the {\it centre} of the graph $\mathcal{T}$. (When $n$ is even, we shall consider $\mathcal{T}(n)$ on vertex set with coordinate representatives $[-\frac{n}{2}+1, \frac{n}{2}]$. For the most part this is not important, and one can focus on the case when $n$ is odd, however we need to have an understanding of $\mathcal{T}(n)$ when $n$ is even for Chapter \ref{ch_classical}.) A key element of our strategy is to ensure that the leave $L^*$ remaining at the end of the iterative matching process is contained entirely on coordinates which are `close' to the centre of $\mathcal{T}$ (i.e. only on vertices whose coordinates have small moduli, where here by small $\leq n^{10^{-5}}$ will be sufficient). We start the process by reserving a small (here meaning $o(n)$ but $\gg |L^*|$) 
set of vertices $A^*$, our {\it absorber}, such that $\mathcal{T}[A^*]$ contains a perfect matching and, for every set $L^*$ that might remain as the leave after the iterative matching process, $\mathcal{T}[A^* \cup L^*]$ contains a perfect matching. We build the absorber following ideas from the method of Randomised Algebraic Construction. Details of this process and which such leave $L^*$ can be absorbed by $A^*$ are given in Chapter \ref{ch_absorber}. We then run a random greedy edge removal process on $\mathcal{T}^{-A^*}:=\mathcal{T}[V(\mathcal{T})\setminus A^*]$. This process is analysed via the differential equations method and we follow very closely the strategy of Bennett and Bohman \cite{randomgreedy}. We are unable to use their result directly as a black box, since when we terminate the process we shall require concentration around expected values for certain properties in the remaining graph additional to those considered in \cite{randomgreedy}. We describe the details of this random greedy edge removal process in Chapter \ref{ch_count}, and the simplicity of the algorithm allows for a straight forward counting argument which will count distinct matchings in $\mathcal{T}^{-A^*}$. Furthermore, we shall stop the random greedy edge removal process on a graph $H_{\gr} \subseteq \mathcal{T}$ when some $O(n^{1-10^{-25}})$ vertices remain from $\mathcal{T}^{-A^*}$, and show that with high probability $H_{\gr}^{+A^*}:=\mathcal{T}[V(H_{\gr}) \cup A^*]$ has a perfect matching, i.e. in almost all cases the matching found via the random greedy edge removal process extends to a perfect matching for $\mathcal{T}$. (Our count will ensure that each matching counted in the random greedy edge removal process extends in such a way that for each matching $M^G_i$ counted in the process, and the perfect matching $M^{A^*}_i$ found in $H_{\gr}^{+A^*}$, we have that $\{M^G_i \cup M^{A^*}_i\}_i$ is a collection of distinct perfect matchings in $\mathcal{T}$.) Now, in order to show that $H_{\gr}^{+A^*}$ contains a perfect matching, it suffices to show that $H_{\gr}$ contains a matching leaving only some small collection of vertices $L^*$ which are close to the centre of $\mathcal{T}$ uncovered, which then, by construction, can be absorbed by $A^*$. In order that the collection of vertices $L^*$ satisfies all necessary parity requirements we start by constructing $H \subseteq H_{\gr}$ such that $H_{\gr}[V(H_{\gr})\setminus V(H)]$ contains a perfect matching and $H$ satisfies some additional parity requirements. (We defer details of the necessary parity requirements to Section \ref{sec_parity}. So it remains to show that with high probability $H^{+A^*}:=\mathcal{T}[V(H) \cup A^*]$ has a perfect matching. This is where our strategy based on the idea of iterative absorption comes in to play. In particular, our absorber is built only to deal with a leave that has a specific structure. This is necessary since we are unable to build an absorber which has the capacity to absorb {\it any} leave of bounded size, however reducing the number of possible configurations that the leave can take requires less `flexibility' from the absorber, which thus makes such an absorber easier to find. The classical `iterative absorption' method (see e.g. \cite{minimalistdesigns} for a nice illustration of the method) splits the absorbing process up into many steps by, at each step, organising a `partial absorbing procedure' until what remains is structured sufficiently to be absorbed in a final absorption step. Our strategy for finding a perfect matching in $H^{+A^*}$ uses this method, with our `partial absorbing procedure' at each step being a random greedy cover process that is feasible due to the construction of the process, but does not rely on any fixed `reserve' or `absorber'. Only the final step uses the absorber $A^*$ which is reserved at the beginning of the process. The graph $H$ from which we start the iterative matching process has certain {\it quasi-random} properties. Part of this definition, for us, relates to the density of vertices throughout the graph and the degree of each vertex. Whilst `quasi-random' can have many different definitions, ours will require that some `nice' properties additional to those listed above hold, not just for $H$ as a whole, but also in particular subgraphs of $H$. We describe more precisely this notion in due course. We first give details of the sequence of subgraphs of $H$ which get us to $L^*$. 

\begin{set}[Vortex parameters]\label{vortex}

We describe a {\it vortex}, a series of nested subgraphs of $\mathcal{T}$, through which we get from $H$ to $L^*$. In particular, let $t_0:=\frac{n-1}{2}$. (When $n$ is even instead set $t_0:=\frac{n}{2}$.) We define $c_{\vor}$ as close to $4/5$ as possible so that $-\frac{\log\log(n)}{\log(c_{\vor})}$ is an integer. (It will become clear why this is important later.)
Then we let $c_h:=\lceil{\frac{-0.99999\log(t_0)}{\log(c_{\vor})}}\rceil$ and for each $i \in [c_h]$, define $t_i=c_{\vor}t_{i-1}=c_{\vor}^it_0$.
\begin{fact} \label{fact_ch}
$c_{\vor}^{c_h} \approx t_0^{-0.99999}$ and $t_{c_h}\approx t_0^{0.00001}$.
\end{fact}
Hence these values are all chosen precisely  so that $t_{c_h} \leq n^{10^-5}$ which will be important for our requirements of $L^*$. For more motivation of these choices, we also discuss the choice of $c_{\vor}$ at the beginning of Chapter \ref{ch_it_abs}.

\begin{defn}[Box and square intervals, ($I_s$, $I'_s$).] \label{def_ints}
 We denote by $I_s \subseteq V(\mathcal{T})$ the set $I_s:=[-\frac{2s}{3}, \frac{2s}{3}]^{X \cup Y} \cup [-s, s]^{X+Y \cup X-Y}$ and say that $I_s$ is a {\it box interval with parameter $s$}.  
We also say that $I'_s \subseteq V(\mathcal{T})$ is a {\it square interval with parameter $s$} if $I'_s:=[-s, s]^{X \cup Y \cup X+Y \cup X-Y}$. \footnote{Our notation $[a,b]^{J}$ means the set of vertices in $V^J$ indexed by coordinates from $a$ to $b$.}
\end{defn}

We shall frequently be considering $I_{t_i}$, the box interval with parameter $t_i$ defined above for some $i \in [0, c_h]$. By abuse of notation, we also allow $I_{s}$ and $I'_s$ to refer to the subgraph of $\mathcal{T}$ induced on the set of vertices but this will always be clear from the context.

\begin{fact}\label{fact_box}
$|I_{t_i}| \approx \frac{20t_i}{3}$ and $|I_{t_i}\setminus I_{t_{i+1}}| \approx \frac{4t_i}{3}$.
\end{fact}

Our plan to reach $L^*$ is via the \emph{vortex}
$$H \supseteq H_0 \supseteq H_1 \supseteq \ldots \supseteq H_{c_h} \supseteq L^*,$$
such that $H_i \subseteq I_{t_i}$ for every $i \in [c_h]$. In particular, to go from $H_i$ to $H_{i+1}$ we may view the strategy as a covering problem, where we are required to cover all of the vertices in $V(H_i[I_{t_i} \setminus I_{t_{i+1}}])$ by disjoint edges from $H_i$, and we take $H_{i+1}$ to be the graph induced on the remaining vertices which were not used in the covering process. This process involves $c_h=O(\log(n))$ iterations and we require that any cover we choose in order to go from $H_i$ to $H_{i+1}$ will leave us with some `nice' properties in $H_{i+1}$. (Note, here, that our `cover' is a cover by disjoint edges so it contributes directly to the perfect matching we are trying to build for $H^{+A^*}$.) 
\end{set}

To ensure that our vortex successfully reaches $L^*$, one of our key tools is a weighted generalisation of a result of Ehard, Glock and Joos \cite{egj}, a strengthened form of a result of Kahn~\cite{kahn}, which says that in a given hypergraph, we may take a matching in such a way as to leave a subgraph that looks random-like with respect to many properties.  

In order to push through our vortex in a `nice' way, each graph $H_i$ will have various weight functions associated to it. We define a weight function $w:E(G) \rightarrow \mathbb{R}_{\geq 0}$ to be a function on the edges of a (hyper)graph $G$ that maps each edge to a non-negative real number, or {\it weight}. We say that a weight function $w$ is a {\it fractional matching} for $G$ if $\sum_{e \ni v} w(e) \leq 1$ for every $v \in V(G)$, and we say that $w$ is an {\it almost-perfect fractional matching} for $G$ if $w$ is a fractional matching for $G$ and additionally $\sum_{e \ni v} w(e) \geq 1-o(1)$ for every $v \in V(G)$. (The $o(1)$ term will be important for us throughout the process, and we shall require this to be of the form $n^{-\alpha}$ for some $\alpha>0$, which will become clear from the context as details are discussed further in Chapter \ref{ch_it_abs}.) We shall obtain an almost-perfect fractional matching $w_i$ for $H_i$ for every $i \in [c_h]$, and $w_{i+1}$ will be defined in terms of $w_i$ for all $i>1$. We shall define $w_H$ as a uniform weight function for $H$ and obtain $w_0$ from $w_H$ and then $w_1$ from $w_0$ in a similar way to how the weight functions are related later on in the process. However, in order to go from $H$ to $H_1$ our strategy has some additional considerations relating to parity constraints and {\it wrap-around edges} (see Section \ref{sec_lattice_notation} for the definition), so we discuss the details of these steps separately.
Furthermore, on reaching $H_1$, and no longer having to consider additional parity constraints, we run a procedure that we shall refer to as a {\it weight shuffle}, informally a process that shifts weight between edges preserving the weighted degree at each vertex. The aim of the weight shuffle is to modify the almost-perfect fractional matching $w_1$ for $H_1$ to another almost-perfect fractional matching $w_{1^*}$ which has some essential properties to allow us to then continue the process over the $O(\log(n))$ steps of the vortex. We defer remaining details of the weight shuffle and the iterative matching process to Chapter \ref{ch_it_abs} and proceed, here, to introduce some fundamental properties of $\mathcal{T}$ as well as some further definitions and notation which shall be key for filling in the details of the proof of Theorem \ref{thm_main}.

\subsection{Key constants and parameters}\label{sec_pars}

In order to be easily referred back to by the reader, we finish this section by listing the key constants and variables whose values and hierarchical order are key to the success of the proof and are not mentioned in the vortex parameters. In particular, we set 
\begin{defn}[Key constants and parameters]\label{def_const}
~\\
$p_L:=\log^{-1}(n)$, $p_A:=n^{-10^{-7}}$, $p_{\gr}:=n^{-10^{-25}}$, $\alpha_G:=n^{-10^{-8}}$, \\
$\eta:=10^{-8}$, $\eta_1:=10^{-9}$, $\epsilon=\frac{\eta}{204800}$ and $\epsilon_1=\frac{\eta_1}{204800}$.
\end{defn}

\section{The lattice of $\mathcal{T}$} \label{understanding_the_lattice}

In this section we take a detour to understanding the structure of $\mathcal{T}$ in terms of the integer span of the edges of $\mathcal{T}$. (We shall make this precise via linear maps.) Whilst, from what follows in this section, the only result we shall require for the proof of Theorem \ref{thm_main} is the statement of Proposition \ref{zer_sum_Q} (which shall be used in Chapter \ref{ch_absorber}), we take our time to motivate and formally derive the statement. Furthermore, we shall take the time to prove Proposition \ref{converse_Q}, the converse of Proposition \ref{zer_sum_Q} which, but defer this to Section \ref{ch_bidl}. Though not required to prove Theorem \ref{thm_main}, Proposition \ref{converse_Q} is important for our proof of Theorem \ref{thm_class}.  

For a labelled (hyper)graph, $H$, and for $i \geq j \geq 0$, the {\it $\mathbb{Z}$-linear boundary / shadow maps} $\partial_j:\mathbb{Z}^{K_i(H)} \rightarrow \mathbb{Z}^{K_j(H)}$ are defined by $\partial_j(J)_t=\sum_{t \subset u \in K_i(H)} J_u$. Informally, for $J \in \mathbb{Z}^{K_i(H)}$, a vector indexed by all copies of $K_i$ in $H$ assigning a weight to each copy, $\partial_j (J)$ is the vector indexed by all copies of $K_j$ in $H$, where the weight corresponding to a copy $K$ of $K_j$ is given by the sum of the weights indexed by the copies of $K_i$ in $J$ which contain $K$. For our purposes we shall be considering a vector indexed by edges $\Phi \in \mathbb{Z}^{E(\mathcal{T})}$, and its vertex shadow $\partial \Phi:= \partial_1 \Phi$, the vector indexed by the vertices of $\mathcal{T}$, such that $v^J_i$ is the weight given to vertex $i^J$ by the weights on the edges in $\Phi$ that contain $i^J$.
Here, we note that we associate $i^J$ both to the vertex in part $J$ which has numerical label $i$, and to the $4n$-dimensional unit vector $e_{i^J}$ which has a $1$ at the place indexed by vertex $i^J$, and $0$s elsewhere.

We also extend this definition to identify multisets containing elements from a set $S$, where each element has a sign attached ({\it intsets}), with vectors ${\bf v} \in \mathbb{Z}^S$. Here, the value $v_s$ in the coordinate indexed by $s \in S$ corresponds to $s$ appearing in $S$, $|v_s|$ times, and the sign attached to $s$ is the same as the sign of $v_s$. 

Identifying the edges $e$ with their corresponding vector indexed by edges, we can describe $E(\mathcal{T})$ as a set of indicator vectors $\{\mathbbm{1}_e\}_{e \in E(\mathcal{T})}$, where $(\mathbbm{1}_e)_f=1$ if $e=f$ and $0$ otherwise. Then $\partial \mathbbm{1}_e$ is a vector indexed by vertices with four $1$ entries, and all other entries $0$. Letting ${\bf v_e}:= \partial \mathbbm{1}_e$, we write ${\bf v_e}=(v^X, v^Y, v^{X+Y}, v^{X-Y})$, so that every part has exactly one $1$, and wherever $v^X_i=1$ and $v^Y_j=1$, we know that $v^{X+Y}_{i+j \mymod n}=1$ and $v^{X-Y}_{i-j \mymod n}=1$. Then we define the {\it lattice}, $\mathcal{L}(\mathcal{T})$, to be the integer span of $\{{\bf v_e}: e \in E(\mathcal{T})\}$.

\subsection{Modular notation} \label{sec_lattice_notation}

Since we are working in the toroidal problem, wherever we refer to sums and differences, these will generally be considered $\mymod n$, unless stated otherwise. We say that $a\pm b$ or $ab$ {\it wrap around} if $a \pm b \neq a\pm b \mymod n$, or $ab \neq ab \mymod n$, respectively. This is both in the context of an edge, and for a vector of weights ${\bf v} \in \mathcal{L}(\mathcal{T})$. If an integer should be read without modular arithmetic, it should be clear from the context, or will be marked with a superscript $\circ$. (For example, when considering coordinates $a$, $b$ and $c$, $a+b=c$ will mean that $a+b=c \mymod n$, and $x=c^{\circ}$ will mean that $x$ is the representative of the congruence class containing $c$.) This will be important in Section \ref{ch_bidl}, where we are interested in representatives which are powers of two.

\subsection{Understanding the lattice} 

We use the remainder of this section to explore particular subsets of $\mathcal{L}(\mathcal{T})$ and their properties. We refer to coordinates in $V$ with non-zero weight in some ${\bf v} \in \mathcal{L}(\mathcal{T})$ as the {\it support} of ${\bf v}$ denoted by $\supp({\bf v})$, the absolute value of the support added {\it to} a coordinate/part, or {\it from} an edge/coordinate/part $a$ as the {\it contribution} to/from $a$, and refer to $|{\bf v}|:=\sum_{v \in {\bf v}} |v|$ as the {\it size of the support} of ${\bf v}$. Our main motivation in this section is to understand the structure of the sub-lattices where all support is contained in only one or two parts. From now on we shall refer to the lattice with support contained in only part $J$ as $\mathcal{L}^J_1(\mathcal{T})$. We shall also use $\mathcal{L}^{J,K}_2(\mathcal{T})$ for the lattice with support contained in parts $J$ and $K$, and in both we may drop the superscripts if they are arbitrary or clear from context. 

Recalling that $\mathcal{T}'$ refers to the semi-queens graph, to understand these subsets of $\mathcal{L}(\mathcal{T})$, it will be helpful to consider subsets of $\mathcal{L}(\mathcal{T}')$. Thus we start by finding a simple generating set for $\mathcal{L}^{X+Y}_1(\mathcal{T}')$. That is, a set of vectors $\mathcal{G}(\mathcal{T}')$ which we can describe concisely, such that every vector in $\mathcal{L}^{X+Y}_1(\mathcal{T}')$ is described by an integer combination of vectors in $\mathcal{G}(\mathcal{T}')$. 
This will, in turn, yield a generating set for $\mathcal{L}^{X+Y, X-Y}_2(\mathcal{T})$. 

Consider the $n\times n$ matrix with rows corresponding to vertices in $X$, columns corresponding to vertices in $Y$ and the $(x,y)$ entry corresponding to the coefficient given to the edge $(x,y,x+y)$. Filling the $(x,y)$ entries in with any combination of integers yields all possible integer collections of edges, and it is easy to draw out the vertex shadow vectors by summing rows for vertices in the $X$ part, columns for vertices in the $Y$ part, and the relevant toroidal diagonals for vertices in the $X+Y$ part. (This easily extends to the queens case in the natural way, by summing the opposite toroidal diagonals to obtain the weights on the vertices in the $X-Y$ part.) Now, to describe $\mathcal{L}^{X+Y}_1(\mathcal{T}')$, we require the vertex shadow on all coordinates in $X \cup Y$ to be $0$. That is, we require the sums in each row and column of the matrix to be $0$. A natural way to obtain such matrices is to fix rows $a$ and $b$ and columns $c$ and $d$ and put $1$s in entries $(a,c)$ and $(b,d)$, and $-1$s in entries $(a,d)$ and $(b,c)$, with all other entries $0$.  
We refer to these $2 \times 2$ matrices indexed by their inherited column and row indices as {\it simple} matrices, and associate them with their $n \times n$ matrices where entries in all other rows and columns are $0$. We claim that the collection of all simple matrices generates $\mathcal{L}^{X+Y}_1(\mathcal{T}')$. Throughout this section we implicitly assume that our claims hold for all $n$ unless explicitly stated otherwise.

\begin{prop} \label{matrix}
Let ${\bf A}$ be any $n \times n$ matrix with integer entries such that the sum along each column and each row is identically zero. Then ${\bf A}$ can be expressed as the sum of simple matrices.
\end{prop}

\begin{proof}
Let $S:=\sum_{a_{ij} \in {\bf A}} |a_{ij}|$. We show that we can always reduce $S$ by subtracting a simple matrix from $A$. In this way, we are able to reduce $S$ to $0$ through the process of subtracting simple matrices from ${\bf A}$, and thus we have a decomposition of ${\bf A}$ into the sum of simple matrices. We start by noting that, since every row and column sums to zero, each row and column either has all entries equal to zero, or at least two non-zero entries. Furthermore, unless $S=0$, at least two rows and two columns must have non-zero entries. Assume $S \neq 0$ and choose any two rows $a$ and $b$ such that there exists a column $c$ in which $(a,c)>0$ and $(b,c)<0$. (This must exist, since we know that any non-zero column must have both a negative and positive value in it.) Then there must also be another column $d$ such that $(a,d)<0$. We subtract the simple matrix with $1$s at $(a,c)$ and $(b,d)$, and $-1$s at $(a,d)$ and $(b,c)$ from ${\bf A}$ to obtain ${\bf A'}$, and in this way obtain $S':=\sum_{a'_{ij} \in {\bf A'}} |a'_{ij}| \leq S-2$. This holds since, in ${\bf A'}$, $(b,c)$ and $(a,d)$ which were both negative have gained weight $1$ and so are closer to zero, and $(a,c)$, which was positive, has lost weight $1$, thus in total reducing the total sum by $3$. In the worst case $a_{bd} \leq 0$ and so $|a'_{bd}| = |a_{bd}|+1$, but then the total sum has still been reduced by $2$.
\end{proof}

So we see that the integer span of the collection of simple matrices describes all $n \times n$ integer matrices with zero-sum rows and columns.

\begin{prop} \label{zero_sum_SQ}
Let ${\bf v} \in \mathcal{L}^{X+Y}_1(\mathcal{T}')$. Then 
\begin{enumerate}[(i)]
\item $\sum_i v_i=0$, and 
\item $\sum_i i v_i=0 (\mymod n)$, 
\end{enumerate}
where $i$ sums over the coordinates in $V(\mathcal{T}')$,  
and $v_i$ is the entry in ${\bf v}$ associated with coordinate $i$.
\end{prop}

\begin{proof}
Recall that any ${\bf v} \in \mathcal{L}^{X+Y}_1(\mathcal{T}')$ is a vector obtained from an integer collection of edges such that the weights in the $X \cup Y$ parts all cancel each other, leaving only non-zero weights in $V^{X+Y}$. Thus ${\bf v}$ must correspond to some $n \times n$ integer matrix as in Proposition \ref{matrix}.   
Thus, by Proposition \ref{matrix}, this matrix can be decomposed as the sum of simple matrices. Note that these simple matrices yield $+1$ weights at coordinates $a+c$ and $b+d$, and $-1$ weights at $a+d$ and $b+c$ in $V^{X+Y}$, (where it could be that $a+c=b+d$, or $a+d=b+c$, but not both), thus yielding weight vectors ${\bf v} \in \mathcal{L}_1^{X+Y}(\mathcal{T}')$ such that all entries are $0$ except either two $+1$ entries and two $-1$ entries, or two $\pm 1$ entries and a $\mp 2$ entry (with all support in $X+Y$). 
That is, each of these simple matrices adds a total of $1+(-1)+(-1)+1=0$ to $\sum_i v^i$, which yields the first property. To see the second, for each of the simple matrices in the decomposition, let $a':=a+c$, $b':=a+d$, $c'=b+c$ and $d':=b+d$. It follows that $a'+d'=b'+c'(\mymod n)$. Hence, each contributes $a'+(-b')+(-c')+d'=(a'+d')-(b'+c')=0 (\mymod n)$ to $\sum_i iv_i$, and the result follows.
\end{proof}

The converse is also true:

\begin{prop} \label{converse}
Suppose ${\bf v}$ is a vector with non-zero weights only on vertices in $V^{X+Y}$ such that 
\begin{enumerate}[(i)]
\item $\sum_i v_i=0$, and 
\item $\sum_i i v_i=0 (\mymod n)$. 
\end{enumerate}
Then ${\bf v} \in \mathcal{L}^{X+Y}_1(\mathcal{T}')$.
\end{prop}

To prove this, we first prove the following proposition (which follows from a similar argument to that in the proof of Proposition \ref{matrix}).

\begin{prop} \label{efficient_SQ-gen_decomp}
Suppose ${\bf v}$ is a vector with non-zero weights only on vertices in $V^{X+Y}$ such that 
\begin{enumerate}[(i)]
\item $\sum_i v_i=0$, and 
\item $\sum_i i v_i=0 (\mymod n)$. 
\end{enumerate}
Then ${\bf v}$ can be efficiently written as the sum of $(1,-1,-1,1)$ vectors with indices of the form $(a,b,c,b+c-a)$, where $a$, $b$, $c$, and $b+c-a$ are all distinct.
\end{prop}

We remark that, by {\it efficiently}, we mean that ${\bf v}$ can be decomposed into $O(|{\bf v}|)$ such vectors.

\begin{proof}
We first note that for all ${\bf v}$ a vector with non-zero weights only on vertices in $V^{X+Y}$ such that (i) and (ii) hold, if ${\bf v} \neq {\bf 0}$ then there are at least three non-zero elements among $\{v^{X+Y}_i\}_{i \in \{0,...,n-1\}}$. (Indeed, if we have only one non-zero element, then $\sum_{i} v_i \neq 0$, and if we have two such elements, then $\sum_{i} v_i=0$ and $\sum_{i} iv_i=0$ cannot both hold.) We show that we can write every ${\bf v}$ satisfying the hypotheses of the claim as the sum of $O(|{\bf v}|)$ vectors of the required type by iteratively subtracting such vectors until we have reduced ${\bf v}$ to ${\bf 0}$. It shall be clear from the way we do this that no more than $|{\bf v}|$ (not necessarily unique) vectors are subtracted, which will prove the lemma. 

Fix ${\bf v}$. By the given conditions, over all coordinates in $V^{X+Y}$ with non-zero weight we must either have at least two negative elements and one positive, or vice versa. Then the following possibilities exist:
\begin{enumerate}[(i)]
\item For every pair $a \neq b$ with positive weights and $a'$ with negative weight, $a+b-a'=a'$, and for every possible choice of $a \neq b$ with negative weight and $a'$ with positive weight, $a+b-a'=a'$.
\item There exist distinct $a, b, a'$ such that $a+b-a' \neq a'$ and either
	\begin{enumerate}[(i)]
	\item the weights on $a, b$ are negative, and the weight on $a'$ is positive; or
	\item the weights on $a, b$ are positive, and the weight on $a'$ is negative.
	\end{enumerate}
\end{enumerate}
In the first case, we must have exactly three non-zero elements if $n$ is odd or either three or four elements if $n$ is even. (Suppose $a+b-a'=a'$ but there exists $c \notin \{a,b\}$ such that $c$ has the same sign as $a$ and $b$, then $a, c, a'$ are a triplet as in the second case. Or, if there exists $b' \neq a'$ such that $b'$ has the same sign as $a'$, then $a, b, b'$ are a triplet as in the second case when $n$ is odd. If $n$ is even then it may be the case that $2b'=2a'$ so that $a+b-a'=a'$ and $a+b-b'=b'$, but since the maximum pair degree is $2$ when $n$ is even this is the only case that yields four non-zero elements in place of three. Furthermore note that, in the first case, the absolute value of the weight on coordinate(s) $a'$ (and possibly $b'$ of the same sign when $n$ is even) must satisfy $|v_{a'}|+|v_{b'}|>1$ (since $\sum_i v_i=0$). Without loss of generality, assume that $a'$ has positive weight. Then define ${\bf e^1}$ with weights $(1,-1,-1,1)$ on coordinates $(a',a,c,a+c-a')$ such that coordinates $a, a', c, a+c-a'$ are distinct, and $c$ and $a+c-a'$ have weight $0$ in ${\bf v}$ (so are not $a, b, a'$ or $b'$). Then ${\bf e^1}$ is a vector of the required form, and setting ${\bf v^1}={\bf v}-{\bf e^1}$
yields that $\sum_{i}|v^1_i| = \sum_{i}|v_i|$ and $v^1$ now has at least two negative elements, and two positive elements, so we must be in the second case.

Assume we are in the second case. Then, without loss of generality assume coordinates $a$ and $b$ have negative weight, and $a'$ has positive weight. Set $b':=a+b-a'$. Define ${\bf e^1}$ the vector with weights $(1,-1,-1,1)$ on coordinates $(a',a,b,b')$. Then ${\bf e^1}$ is of the required form, and setting ${\bf v^1}={\bf v}- {\bf e^1}$ yields that $\sum_{i}|v^1_i| \leq \left(\sum_{i}|v_i|\right) -2$. (As two negative elements have been pushed closer to $0$, each by $1$, one positive element has been pushed closer to $0$ by $1$, and the other element has, at worst, been pushed further from $0$ by $1$.)
 
If ${\bf v^1}={ \bf 0 }$ we are done, and else we can repeat the process and at the $j^{th}$ iteration, $\sum_{i}|v^j_i|$ either remains the same or gets strictly smaller by at least $2$. But every time the sum remains the same from ${ \bf v^{j-1}}$ to ${\bf v^j}$, we have ${\bf v^j}$ a vector in the second case, and so moving from ${\bf v^j}$ to ${\bf v^{j+1}}$, the sum will decrease by at least $2$. Thus we subtract at most $2$ vectors of the required type from ${\bf v}$ in order to get a decrease of at least $2$. Hence, after at most $\sum_{i} |v_i|$ steps, we must have reached ${\bf 0}$, where the process terminates. 
\end{proof}

\begin{proof}[Proof of Proposition \ref{converse}]
Since $\sum_i v_i=0$ and $\sum_i i v_i=0 (\mymod n)$, by Proposition \ref{efficient_SQ-gen_decomp}, we may efficiently write ${\bf v}$ as the sum of $(1,-1,-1,1)$ vectors with indices of the form $(a,b,c,b+c-a)$. Each of these vectors can be obtained from a simple matrix, for example that indexed by $a$ and $b$ in the rows, and $0$ and $c-a$ in the columns. Thus ${\bf v}$ has a decomposition as simple matrices and so ${\bf v} \in \mathcal{L}^{X+Y}_1(\mathcal{T}')$.
\end{proof}

It is worth noting here that the simple matrix yielding ${\bf v} \in \mathcal{L}_1^{X+Y}(\mathcal{T}')$ with $\supp({\bf v})=\{a, b, c, b+c-a\}$ is not unique. This will be relevant when we consider sub-lattices of $\mathcal{T}$. We finish the discussion of $\mathcal{L}^{X+Y}_1(\mathcal{T}')$ with a final proposition.

\begin{prop} \label{four_ones}
The collection of vectors with weights $(1,-1,-1,1)$ and indices of the form $(a,b,c,b+c-a)$ is generating for $\mathcal{L}^{X+Y}_1(\mathcal{T}')$. 
\end{prop}

\begin{proof}
The collection of such vectors is generating if and only if any ${\bf v} \in \mathcal{L}^{X+Y}_1(\mathcal{T}')$ can be written as the sum of such vectors. Fix ${\bf v} \in \mathcal{L}^{X+Y}_1(\mathcal{T}')$. By Proposition \ref{zero_sum_SQ}, $\sum_i v_i=0$ and $\sum_i i v_i=0 (\mymod n)$, and thus by Proposition \ref{efficient_SQ-gen_decomp} we are done.
\end{proof}

\begin{rem} \label{rem_any_comp}
It is worth noting that a similar approach works to yield a generating set for $\mathcal{L}^{J}_1(\mathcal{T}')$ for any $J$. Indeed, if we start with an $n \times n$ matrix where rows correspond to vertices from $V^{X+Y}$ and columns correspond to vertices from $V^Y$, then integer combinations of simple matrices will yield all possible combinations where the total weight on each coordinate in $V^{X+Y}$ and $V^Y$ is identically zero, and the weights given to coordinates in $V^X$ can be obtained from taking the differences down the forward diagonals (i.e. if $(a,b)$ has entry $c$ in the $n \times n$ matrix, then this entry adds weight $c$ to $a-b \in V^{X}$, and clearly shifting up and down the forward diagonal containing $(a,b)$ we have entries $(a + \delta, b + \delta)$, where $\delta \in [0,n-1]$, and these entries contribute their weight to $(a + \delta)-(b + \delta)=a-b \in V^{X}$ too).
\end{rem}

We may also deduce the following lemma about $\mathcal{L}(\mathcal{T'})$.

\begin{lemma} \label{lem_lattice_t'}
${\bf v} \in \mathcal{L}(\mathcal{T}')$ if and only if the following both hold:
\begin{enumerate}[(i)]
\item $\sum_i v_i^X = \sum_i v_i^Y = \sum_i v_i^{X+Y}$,
\item $\sum_i iv_i^X + \sum_i iv_i^Y = \sum_i iv_i^{X+Y} \mymod(n)$.
\end{enumerate}
\end{lemma}
\begin{proof}
Considering an edge $e \in \mathcal{T}'$ we have that the vector $\bf{v} \in \mathcal{L}(\mathcal{T}')$ associated to the edge $e$ satisfies (i) and (ii). Thus the associated vector of any linear combination of edges will also satisfy (i) and (ii), which by definition covers all ${\bf v} \in \mathcal{L}(\mathcal{T}')$. Considering now a vector ${\bf v}$ such that (i) and (ii) hold we may modify ${\bf v}$ to ${\bf u}$ by adding the shadow vectors of edges in such a way as to make ${\bf u^X}, {\bf u^Y} \equiv {\bf 0}$. But adding these edges will force $\sum_i u_i^{X+Y}=0$ and $\sum_i iu_i^{X+Y} = 0 \mymod(n)$. Thus by Proposition \ref{converse} we have that ${\bf u} \in \mathcal{L}_1^{X+Y}(\mathcal{T}')$, so ${\bf u}$ can be seen as the vertex shadow of a linear combination of edges of $\mathcal{T}'$. But since ${\bf v}$ can be obtained from ${\bf u}$ by removing edges, we must then also have that ${\bf v} \in \mathcal{L}(\mathcal{T}')$. 
\end{proof}

We previously mentioned that $\mathcal{T}'(n)$ has a perfect matching only if $n$ is odd. We show that this can be deduced from the above understanding of the lattice. 

\begin{cor} \label{cor_lattice_t'}
Suppose that $\mathcal{T}'(n)$ has a perfect matching. Then $n$ is odd.
\end{cor}
\begin{proof}
Suppose that $\mathcal{T}'(n)$ has a perfect matching. Then ${\bf 1} \in \mathcal{L}(\mathcal{T}')$. By Lemma \ref{lem_lattice_t'} we have that $\sum_i v_i^X = \sum_i v_i^Y = \sum_i v_i^{X+Y}$ and $\sum_i iv_i^X + \sum_i iv_i^Y = \sum_i iv_i^{X+Y} \mymod(n)$. It is clear that the former holds for all values of $n$. For the latter, we deduce that $\frac{n(n+1)}{2} \equiv 0 \mymod(n)$ which is true if and only if $n$ is odd.
\end{proof}

Running the same arguments for $\mathcal{T}$, we have a generating set for $\mathcal{L}^{X+Y, X-Y}_2(\mathcal{T})$, just as above, but where, as well as having weights $(1,-1,-1,1)$ on vertices in $V^{X+Y}$ corresponding to $a+c$, $a+d$, $b+c$ and $b+d$, we have weights $(1,-1,-1,1)$ in $V^{X-Y}$ on vertices $a-c$, $a-d$, $b-c$ and $b-d$. Now note that we can say more than this: suppose we have any fixed $(1,-1,-1,1)$ vector in $V^{X+Y}$ on $a+c$, $a+d$, $b+c$ and $b+d$, then the indices of the simple matrix which yield this are not unique. In fact, writing, as before, $a'=a+c$, $b'=a+d$, $c'=b+c$ and $d'=b+d$, we have a free choice for one of the rows or columns (i.e. for one of the coordinates in $V^X \cup V^Y$). Choosing, say, row index $x_1$ and taking the column indices to be $a'-x_1$ and $b'-x_1$, and the final row index to be $c'-a'+x_1$ yields a simple matrix that generates the vector described above. Furthermore, this tells us that $a+b=c'-a'+2x_1$, where $c'$ and $a'$ are fixed, and $x_1$ was a free choice. When $n$ is odd this means that there is no restriction to what $a+b$ can be to obtain these vectors, and for even $n$ the restriction is only to those of odd or even parity, depending on $c'-a'$. Setting $s:=a+b$, we see that $a-c=-c'+s$, $a-d=-d'+s$, $b-c=-a'+s$, and $b-d=-b'+s$. That is, the weights on $V^{X-Y}$ relate to a coordinate $i \in V^{X+Y}$ via adding weight of the opposite sign to $-i+s \in V^{X-Y}$, where $s$ is a fixed shift which can take any value when $n$ is odd and either all even or all odd values when $n$ is even. Let vectors with weights $(1,-1,-1,1)$ on coordinates $(a,b,c,b+c-a)$ in $V^{X+Y}$ and weights $(-1,1,1,-1)$ on coordinates $(s-a,s-b,s-c,s-(b+c-a))$ in $V^{X-Y}$ be referred to as {\it $2$-part generators}.

\begin{prop} \label{gen_for_2_comps}
The collection of all $2$-part generators forms a generating set for $\mathcal{L}^{X+Y, X-Y}_2(\mathcal{T})$. 
\end{prop}

\begin{proof}
This follows from the fact that any matrix as in Proposition \ref{matrix} can be decomposed into the sum of simple matrices. If ${\bf v} \in \mathcal{L}^{X+Y, X-Y}_2(\mathcal{T})$, then it must correspond to some such matrix ${\bf A}$. Thus we can decompose ${\bf v}$ as the sum of vectors obtained from the simple matrices which form a decomposition of ${\bf A}$. These vectors are precisely $2$-part generators, so we are done.
\end{proof}

Furthermore, this interpretation indicates a nice way to describe $\mathcal{L}^{X+Y}_1(\mathcal{T})$. In particular, imagining this in matrix form, if we have one simple matrix on $(a,b)\times (c,d)$, then adding another, this time with weights flipped to $(-1,1,1,-1)$ on $(s+a, s+b) \times (s+c, s+d)$ for any shift $s$, yields a ${\bf 0}$ vector in $V^{X-Y}$ and a vector supported on eight coordinates in $V^{X+Y}$ (though some of these may be repeated or may cancel, for example setting $s=\frac{b+d-a-c}{2}$). That is, a vector with weights $(1,-1,-1,1,-1,1,1,-1)$ on coordinates $(a',a'+b',a'+c', a'+b'+c', s+a',s+a'+b',s+a'+c',s+a'+b'+c')$ in $V^{X+Y}$ for any choice of $a', b', c'$ and $s$. From now on we refer to such vectors as {\it queens generators}, (or {\it $Q$-gens}), and those in Proposition \ref {four_ones} as {\it $SQ$-gens}. Note that a $Q$-gen is in fact the difference of two $SQ$-gens (i.e.~a $Q$-gen with weights $(1,-1,-1,1,-1,1,1,-1)$ on coordinates $(a,a+b,a+c,a+b+c,s+a,s+a+b,s+a+c,s+a+b+c)$ is precisely the sum of the $SQ$-gen with weights $(1,-1,-1,1)$ on coordinates $(a, a+b, a+c, a+b+c)$ and the $SQ$-gen with weights flipped to $(-1,1,1,-1)$ on coordinates $(s+a,s+a+b,s+a+c,s+a+b+c)$).

\begin{prop} \label{eight_ones}
The collection of all queens generators ($Q$-gens) forms a generating set for $\mathcal{L}^{X+Y}_1(\mathcal{T})$.
\end{prop}

\begin{proof}
Fix ${\bf v} \in \mathcal{L}^{X+Y}_1(\mathcal{T})$. Note that $\mathcal{L}^{X+Y}_1(\mathcal{T}) \subseteq \mathcal{L}^{X+Y, X-Y}_2(\mathcal{T})$, and thus by Proposition \ref{gen_for_2_comps} we may express it as the sum of $2$-part generators. We may then replace each $2$-part generator with weights $(1,-1,-1,1)$ on coordinates $(a,b,c,b+c-a)$ in the $V^{X+Y}$ part and weights $(-1,1,1,-1)$ on coordinates $(s-a,s-b,s-c,s-(b+c-a))$ in the $V^{X-Y}$ part, by a vector with weights $(1,-1,-1,1,-1,1,1,-1)$ on coordinates $(a,b,c,b+c-a,a-s,b-s,c-s,(b+c-a)-s)$ in the $V^{X+Y}$ part. Since ${\bf v} \in \mathcal{L}^{X+Y}_1(\mathcal{T})$, the total weight to each vertex $s-a \in V^{X-Y}$ is identically $0$, and thus when we replace the generators and move all weight on $s-a \in V^{X-Y}$ to $a-s \in V^{X+Y}$, in total we are adding $0$ weight to $a-s$. Thus ${\bf v}$ has remained unchanged, and is now described as the sum of $Q$-gens. Hence the collection of $Q$-gens is generating for $\mathcal{L}_1^{X+Y}(\mathcal{T})$.
\end{proof}

It is clear that $\mathcal{L}^{X+Y, X-Y}_2(\mathcal{T})$ inherits the `zero-sum' properties of Proposition \ref{zero_sum_SQ}. As well as describing a nice generating family for $\mathcal{L}^{X+Y}_1(\mathcal{T})$, we can also obtain analogues to Propositions \ref{zero_sum_SQ} and \ref{converse}. Indeed, there is a natural recursive structure that gives motivation for the following proposition: let ${\bf e_i}$ be the unit vector with a $1$ at $i$, and $0$ for every other entry. Then for a vector ${\bf v}$ and a shift $s$, define $f_s({\bf v})$ by $f_s(v_{s+a})=-v_a$ for each $a$. Letting ${\bf v^0}={\bf e_a}$ for an arbitrary choice $a$, we define ${\bf v^1}:={\bf v^0}+f_{s_1}({\bf v^0})$ and see that $v^1_{s_1+a}=-1$, $v^1_{a}=1$ and $v^1$ is $0$ at all other entries. Any vector formed from an integer combination of these satisfies the $\sum_i v_i=0$ condition. Defining ${\bf v^2}:={\bf v^1}+f_{s_2}({\bf v^1})$, we see that ${\bf v^2}$ has $v^2_{s_1+a}=-1, v^2_{a}=1,  v^2_{s_2+a}=-1$, and $v^2_{s_1+s_2+a}=1$. Note that this yields all of the $SQ$-gens, since each is of the form with weights $(1,-1,-1,1)$ on coordinates $(a,a+s_1,a+s_2, a+s_1+s_2)$ for some free choice of $a$, $s_1$ and $s_2$, (and, in particular, any linear combination of these vectors satisfies both $\sum_i v_i=0$ and $\sum_i iv_i=0$). Repeating the recursion once more, i.e. so that ${\bf v^3}:={\bf v^2}+f_{s_3}({\bf v^2})$, yields all of the $Q$-gens, and naturally suggests the additional constraint that $\sum_i i^2v_i=0$ for all vectors in $\mathcal{L}^{X+Y}_1(\mathcal{T})$.  

\begin{prop} \label{zer_sum_Q}
Let ${\bf v} \in \mathcal{L}^{X+Y}_1(\mathcal{T})$. Then
\begin{enumerate}[(i)]
\item $\sum_i v_i=0$; 
\item $\sum_i i v_i=0 (\mymod n)$; and
\item $\sum_{i \in V^{X+Y}} i^2 v_i =0 (\mymod n)$.
\end{enumerate}
Furthermore, if $n$ is even then $2n~|~\sum_{i \in V^{X+Y}} i^2 v_i$.
\end{prop}

\begin{proof}
Suppose ${\bf v} \in \mathcal{L}^{X+Y}_1(\mathcal{T})$. The first two zero-sum properties follow easily from the discussion above. To see the third property note that, by Proposition \ref{eight_ones}, ${\bf v}$ is the sum of an integer collection of vectors with weights $(1,-1,-1,1,-1,1,1,-1)$ on coordinates $(a_1, a_2, a_3, a_2+a_3-a_1, s+a_1, s+a_2, s+a_3, s+a_2+a_3-a_1)$. Manual calculation of $\sum_i i^2v_i$ for each of these $Q$-gens individually yields the result for each generator (and in particular when $n$ is even satisfy $2n | \sum_i i^2v_i$), thus the integer sum of any collection of such generators will satisfy the equation. In particular, writing each $Q$-gen as the difference of two $SQ$-gens it can be seen that the sum of the squares of these are equal and so cancel as a $Q$-gen (modulo $n$).
\end{proof}

We also give the statement of the converse.

\begin{prop} \label{converse_Q}
First suppose that $n$ is odd. Any vector ${\bf v}$ which only has non-zero support in $V^{X+Y}$ and satisfies
\begin{enumerate}[(i)]
\item $\sum v_i = 0,$
\item $\sum i v_i = 0 (\mymod n),$ and
\item $\sum_{i \in V^{X+Y}} i^2 v_i =0 (\mymod n)$
\end{enumerate}
is in $\mathcal{L}^{X+Y}_1(\mathcal{T})$. Supposing that $n$ is even, we have that any vector ${\bf v}$ which only has non-zero support in $V^{X+Y}$ and satisfies
\begin{enumerate}[(i)]
\item $\sum v_i = 0,$
\item $\sum i v_i = 0 (\mymod n),$
\item $2n~|~\sum_{i \in V^{X+Y}} i^2 v_i$
\end{enumerate}
is in $\mathcal{L}^{X+Y}_1(\mathcal{T})$.
\end{prop}

To prove the converse is considerably more complicated. We defer the proof of Proposition \ref{converse_Q} to after the proof of Lemma \ref{BIDC}, from which it follows, in Section \ref{ch_bidl} where we discuss the `bounded integral decomposition lemma'. 

As done for $\mathcal{T}'$, we may also deduce the following two results regarding any vector in the lattice $\mathcal{L}(\mathcal{T})$. The proofs use the same strategy as for Lemma \ref{lem_lattice_t'} and Corollary \ref{cor_lattice_t'} respectively.

\begin{lemma} \label{lem_lattice_t}
First suppose that $n$ is odd. We have that ${\bf v} \in \mathcal{L}(\mathcal{T})$ if and only if the following all hold:
\begin{enumerate}[(i)]
\item $\sum_i v_i^X = \sum_i v_i^Y = \sum_i v_i^{X+Y}= \sum_i v_i^{X-Y}$,
\item $\sum_i iv_i^X + \sum_i iv_i^Y = \sum_i iv_i^{X+Y} \mymod n$,
\item $\sum_i iv_i^X - \sum_i iv_i^Y = \sum_i iv_i^{X-Y} \mymod n$,
\item $2\sum_i i^2v_i^X + 2 \sum_i i^2v_i^Y = \sum_i i^2v_i^{X+Y} + \sum_i i^2v_i^{X-Y} \mymod n$.
\end{enumerate}
Supposing now that $n$ is even, we have that ${\bf v} \in \mathcal{L}(\mathcal{T})$ if and only if the following all hold:
\begin{enumerate}[(a)]
\item $\sum_i v_i^X = \sum_i v_i^Y = \sum_i v_i^{X+Y}= \sum_i v_i^{X-Y}$,
\item $\sum_i iv_i^X + \sum_i iv_i^Y = \sum_i iv_i^{X+Y} \mymod n$,
\item $\sum_i iv_i^X - \sum_i iv_i^Y = \sum_i iv_i^{X-Y} \mymod n$,
\item $2n~|~ (\sum_i i^2v_i^{X+Y} + \sum_i i^2v_i^{X-Y} - 2\sum_i i^2v_i^X - 2 \sum_i i^2v_i^Y)$.
\end{enumerate}
\end{lemma}
\begin{proof}
Considering an edge $e \in \mathcal{T}$ we have that the vector ${\bf v} \in \mathcal{L}(\mathcal{T})$ associated to the edge $e$ satisfies (i)-(iv) when $n$ is odd, and (a)-(d) when $n$ is even. Thus the associated vector of any linear combination of edges will also satisfy (i)-(iv) or (a)-(d) respectively, which by definition covers all ${\bf v} \in \mathcal{L}(\mathcal{T})$. Considering now a vector ${\bf v}$ such that (i)-(iv) or (a)-(d) hold given $n$ is odd or even respectively, we may modify ${\bf v}$ to ${\bf u}$ by adding the shadow vectors of edges in such a way as to make ${\bf u^X}, {\bf u^Y} \equiv {\bf 0}$. We may then modify ${\bf u}$ to ${\bf w}$ where ${\bf w^{X-Y}} \equiv {\bf 0}$ and we retain ${\bf w^X}, {\bf w^Y} \equiv {\bf 0}$ by adding $2$-part generators. (To see this, first note that we have $\sum_i iu_i^{X-Y}=0 \mymod n$. Then considering $\mathcal{T}'$ with parts $X, Y, X-Y$ we know that ${\bf u^X \cup u^Y \cup u^{X-Y}} \in \mathcal{L}(\mathcal{T}')$ by Proposition \ref{converse} thus we can use $SQ$-gens, which we know are the vertex shadow of linear combinatorics of edges of $\mathcal{T}'$, to reduce ${\bf u^{X-Y}}$ to ${\bf 0}$. But lifting back from $\mathcal{T}'$ to $\mathcal{T}$, these $SQ$-gens dictate appropriate $2$-part generators which have precisely the same effect on the $X-Y$ part in $\mathcal{T}$.) It is clear that this forces $\sum_i w_i^{X+Y}=0$, $\sum_i iw_i^{X+Y} = 0 \mymod n$ and $\sum_i i^2w_i^{X+Y} = 0 \mymod n$. Furthermore, when $n$ is even we must have $2n~|~\sum_i i^2w_i^{X+Y}$. Thus by Proposition \ref{converse_Q} we have that ${\bf w} \in \mathcal{L}_1^{X+Y}(\mathcal{T})$, so ${\bf w}$ can be seen as the vertex shadow of a linear combination of edges of $\mathcal{T}$. But since ${\bf v}$ can be obtained from ${\bf w}$ by removing edges, we must then also have that ${\bf v} \in \mathcal{L}(\mathcal{T})$. 
\end{proof}

We are now able to deduce P\'{o}lya's~\cite{polya} result that $\mathcal{T}(n)$ has a perfect matching only if $n$ is not divisible by $2$ or $3$.

\begin{cor} \label{cor_lattice_t}
Suppose that $\mathcal{T}(n)$ has a perfect matching. Then $n \equiv 1,5 \mymod 6$.
\end{cor}
\begin{proof}
Suppose that $\mathcal{T}(n)$ has a perfect matching. Then ${\bf 1} \in \mathcal{L}(\mathcal{T})$. By Lemma \ref{lem_lattice_t} we have that $\sum_i v_i^X = \sum_i v_i^Y = \sum_i v_i^{X+Y}$, $\sum_i iv_i^X + \sum_i iv_i^Y = \sum_i iv_i^{X+Y} \mymod(n)$, $\sum_i iv_i^X - \sum_i iv_i^Y = \sum_i iv_i^{X-Y} \mymod(n)$ and furthermore that $2\sum_i i^2v_i^X + 2 \sum_i i^2v_i^Y = \sum_i i^2v_i^{X+Y} + \sum_i i^2v_i^{X-Y}$. It is clear that the first of these holds for all values of $n$. Both the second and the third are true if and only if $\frac{n(n+1)}{2} \equiv 0 \mymod(n)$ which is true if and only if $n$ is odd. The final statement implies that $2 \sum_{i \in [n]} i^2 = 0 \mymod(n)$ and since we know that $n$ must be odd this reduces to $\sum_{i \in [n]} i^2 = 0 \mymod(n)$. Using that $\sum_{i \in [n]} i^2 = \frac{n(n+1)(2n+1)}{6}$ yields that $n \equiv 1,2 \mymod 3$, completing the proof.
\end{proof}

\section{Zero-sum configurations} \label{sec_zero_sum}

\subsection{Overview}

We define a {\it zero-sum configuration} to be any collection of vertices, such that these vertices induce two distinct perfect matchings, $M^+$ and $M^-$. We refer to such a configuration in this way, since, giving each edge in $M^+$ a weight $+1$ and each edge in $M^-$ a weight $-1$, we have that every vertex covered by the configuration sees exactly one positive edge and one negative edge in the configuration, thus giving a `zero-sum' weight at each vertex. In the language of linear maps above, given $\Phi \in \{-1,0,1\}^{E(\mathcal{T})}$ we define $\Phi^{+} \in \{0,1\}^{E(\mathcal{T})}$ to be the vector such that 
$$\Phi^{+}_e = 
\begin{cases}
1 & \mbox{~if $\Phi_e=1$,} \\
0 & \mbox{~otherwise.} \\
\end{cases}$$ 
Similarly, we define $\Phi^{-} \in \{0,-1\}^{E(\mathcal{T})}$ to be the vector such that 
$$\Phi^{-}_e = 
\begin{cases}
-1 & \mbox{~if $\Phi_e=-1$,} \\
0 & \mbox{~otherwise.} \\
\end{cases}$$ 
Then we have that for $\Phi \in \{-1,0,1\}^{E(\mathcal{T})}$, the set of vertices which have non-zero support on either $\partial\Phi^+$ or $\partial \Phi^-$ is a zero-sum configuration if and only if $\partial\Phi={\bf 0}$. Whilst there can be many ways to form such configurations, we shall restrict our consideration to a specific family of zero-sum configurations, and henceforth, when referring to a zero-sum configuration, we shall be considering exclusively those which can be described in the following way. Our zero-sum configurations consist of $16$ vertices, so that each of the matchings $M^+$ and $M^-$ consist of four edges. Every configuration we are concerned with has free variables $a,b,c,s$, and we set $d:=b+c-a$. Then a zero-sum configuration $Z$ has {\it positive matching} $M^+(Z)$ consisting of
\\
$(a,b+s, a+b+s, a-b-s)$, \\
$(b,d+s, b+d+s, b-d-s)$, \\
$(c,a+s, a+c+s, c-a-s)$, \\
$(d,c+s, c+d+s, d-c-s)$, \\
and {\it negative matching} $M^-(Z)$ consisting of
\\
$(a,c+s, a+c+s, a-c-s)$, \\
$(b,a+s, a+b+s, b-a-s)$, \\
$(c,d+s, c+d+s, c-d-s)$, \\
$(d,b+s, b+d+s, d-b-s)$, \\
noting that $a-b=c-d$, $b-d=a-c$, $c-a=d-b$ and $d-c=b-a$, so that indeed the $16$ vertices are each covered exactly once in each of the matchings. In fact, concerning the free variables $a, b, c$, and $s$ we are able to make a linear change of variables so that when fixing some variables (not necessarily those named as $a, b, c, s$), the other new variables given by the linear change remain unconstrained and so can be thought of as `degrees of freedom'. (For example, fixing an edge $e$ to be in the positive matching of a zero-sum configuration uses two degrees of freedom, and leaves two degrees of freedom remaining to fix a specific configuration containing $e$.)  We use collections of such zero-sum configurations in two key ways for the proof. Firstly, for building the absorber, which comprises of {\it cascades}, gadgets built from collections of zero-sum configurations, which we describe in detail in Section \ref{sec_cascades}. Secondly, during the weight shuffle in the iterative matching process, where we have weights assigned to the edges of $\mathcal{T}$, and in order for the process to continue as long as we require, we wish to modify a given weight assignment in such a way that every vertex maintains its weighted degree. The weight shuffle uses certain collections of zero-sum configurations to shift weight between edges in a controlled fashion which will maintain the weighted degree of each vertex.

Whilst we shall take the absorber out at the beginning of the process, so that the cascades only need to be present in $\mathcal{T}$, we need to be able to find zero-sum configurations for the weight shuffle when we reach the iterative matching process, and as such, we need to track such zero-sum configurations during the random greedy count.

\subsection{Zero-sum configurations for the weight shuffle} \label{sec_zs_shuffle}

In order to describe the zero-sum configurations required for the weight shuffle, we require the following definitions. 
\begin{defn}[Weight shuffle parameters]
We define $j_1:=\frac{t_1}{\log(n)}$ and $k_1:=\frac{t_1}{\log^{2}(n)}$. Then we define $j_{i+1}=k_i=\frac{t_1}{\log^{i+1}(n)}$, and $k_{i+1}=\frac{k_i}{\log(n)}=\frac{t_1}{\log^{i+2}(n)}$, for all $i \in [c_g]$, where $c_g:=\lceil \frac{0.99999\log(t_1)}{\log\log(n)} \rceil$, so that $j_{c_g} \leq n^{10^{-5}}$. Then we let $J_i=I_{j_i}$ and $K_i=I_{k_i}$ be the box intervals with parameters $j_i$ and $k_i$ respectively, for every $i \in [c_g]$. In addition, we define $K_0=J_1$ and $K_{-1}=I_{t_1}$ and $K_{-2}=V(\mathcal{T})$. (This will be convenient for Definition \ref{def_depth}.)
\end{defn}

Note that these values are chosen precisely in view of the following fact.

\begin{fact} \label{fact_cg}
$(\log(n))^{c_g} \approx t_1^{0.99999}$ and $j_{c_g} \approx t_1^{0.00001}$. 
\end{fact}

Comparing Fact \ref{fact_cg} to Fact \ref{fact_ch}, we see that the values $c_{\vor}$, $\log(n)$, $c_h$ and $c_g$ are chosen to ensure both sequences $\{t_i\}_{i \in [c_h]}$ and $\{j_i\}_{i \in [c_g]}$ are such that the last elements $t_{c_h}$ and $j_{c_g}$ are of the same order of magnitude and in particular at most $n^{10^{-5}}$. Furthermore $c_{\vor}$ is defined such that $-\frac{\log\log(n)}{\log(c_{\vor})}$ is an integer precisely so that the partition of $I_{t_1}$ given by $\{I_{t_{i}}\setminus I_{t_{i+1}}\}_{i \in [1, c_h]}$ is a refinement of the partition given by $\{K_j\setminus K_{j+1}\}_{j \in [-1, c_g']}\cup K_{c_g'+1}$ where $c_g' \leq c_g$ is such that $k_{c_g'} \geq t_{c_h}$. (This last point concerning $c_g$ and $c_g'$ is just a technicality resulting from starting with $t_1$ rather than $t_0$ for defining $c_g$ so that $j_{c_g} \leq t_{c_h}$.) The key relationship here is that for every $j \in [c_g]$ such that $k_j \geq t_{c_h}$, there exists $i \in [c_h]$ such that $I_{t_i}=K_j$. 

\begin{defn}[edge types and $i$-legal zero-sum configurations]
We say that an edge $e$ is of type $(\alpha, \beta, \gamma)_i$ if $e$ contains $\alpha$ vertices in $K_i \setminus K_{i+1}$, $\beta$ vertices in $J_i \setminus K_i$ and $\gamma$ vertices in $I_{t_1} \setminus J_i$. We refer to the $\alpha$ vertices as {\it $i$-small} vertices, $\beta$ vertices as {\it $i$-medium} vertices, and $\gamma$ vertices as {\it $i$-large} vertices, or {\it small}, {\it medium} and {\it large} when $i$ is clear from the context. We say that an edge $e$ is {\it $i$-bad} if $e$ is of type $(\alpha, \beta, \gamma)_i$ with $\alpha \neq 0$ and $\gamma \neq 0$, and that an edge $e$ is {\it bad} if there exists $i \in [c_g]$ such that $e$ is $i$-bad.
We define an {\it $i$-legal zero-sum configuration} to be a zero-sum configuration such that $M^+$ consists of an $i$-bad edge $e$, and three other edges which are all of types $(1,0,3)_i$ and $(0,1,3)_i$, and $M^-$ consists of one edge of type $(\alpha, \beta, 0)_i$, with $\alpha \neq 0$, and three edges of types $(0,1,3)_i$ and $(0,0,4)_i$. We denote the collection of all $i$-legal zero-sum configurations in $\mathcal{T}$ by $\mathcal{Z}_{i, \mathcal{T}}$.
\end{defn}

Now for an edge $e$ that is in an $i$-legal zero-sum configuration either as $e \in M^+$ or $e \in M^-$ for any $i \in [c_g]$, we need to know how many $i$-legal zero-sum configurations that edge is in as an edge in $M^+$ and as an edge in $M^-$. We let 
$$\mathcal{Z}^{\pm}_{i,e,G}(\alpha, \beta, \gamma)$$ 
denote the family of $i$-legal zero-sum configurations in the graph $G \subseteq \mathcal{T}$, which contain the edge $e$ in $M^{\pm}$, which is an edge of type $(\alpha, \beta, \gamma)_i$. We consider $\mathcal{Z}^{\pm}_{i,e,\mathcal{T}}(\alpha, \beta, \gamma)$ for all $i \in [c_g]$ and every edge and possible edge type $(\alpha, \beta, \gamma)_i$.

Note first that, by construction, $\mathcal{Z}^{\pm}_{i,e,\mathcal{T}}(\alpha, \beta, \gamma)=\emptyset$ for all but the following:
$$\mathcal{Z}^{+}_{i,e,\mathcal{T}}(1,0,3),\mathcal{Z}^{+}_{i,e,\mathcal{T}}(1,1,2), \mathcal{Z}^{+}_{i,e,\mathcal{T}}(1,2,1), \mathcal{Z}^{+}_{i,e,\mathcal{T}}(0,1,3),$$
and
\begin{multline*}
\mathcal{Z}^{-}_{i,e,\mathcal{T}}(4,0,0),\mathcal{Z}^{-}_{i,e,\mathcal{T}}(3,1,0),\mathcal{Z}^{-}_{i,e,\mathcal{T}}(2,2,0),\\
\mathcal{Z}^{-}_{i,e,\mathcal{T}}(1,3,0),\mathcal{Z}^{-}_{i,e,\mathcal{T}}(0,1,3),\mathcal{Z}^{-}_{i,e,\mathcal{T}}(0,0,4).
\end{multline*}

In particular, recalling that no wrap-around edges appear in the positive or negative matching of an $i$-legal zero-sum configuration for any $i \in [c_g]$ since all vertices are in $I_{t_1}$ and $\mathcal{T}[I_{t_1}]$ contains no wrap-around edges, we cannot have edges of type $(\alpha, \beta, \gamma)_i$ with $\alpha \geq 2$ and $\gamma \neq 0$, for any $i \in [c_g]$. That is, $i$-bad edges can only be those of types $(1,0,3)_i$, $(1,1,2)_i$ and $(1,2,1)_i$ for every $i \in [c_g]$, with most $i$-bad edges being of type $(1,0,3)_i$. Thus it should be clear that all those $(\alpha, \beta, \gamma)$ not listed explicitly above satisfy $\mathcal{Z}^{\pm}_{i,e,\mathcal{T}}(\alpha, \beta, \gamma)=\emptyset$, since an $i$-legal zero-sum configuration is a zero-sum configuration with positive matching $M^+$ consisting of an $i$-bad edge $e$, and three other edges which are all of types $(1,0,3)_i$ and $(0,1,3)_i$, and negative matching $M^-$ comprising of one edge of type $(\alpha, \beta, 0)_i$, with $\alpha \neq 0$, and three edges of types $(0,1,3)_i$ and $(0,0,4)_i$. That gives, for example, that an edge $e$ of type $(0,4,0)_i$ cannot appear as a positive or negative matching edge in an $i$-legal zero-sum configuration. Thus the number of $i$-legal zero-sum configuration containing $e$ in $M^{\pm}$ is exactly $0$. Note also that type $(0,1,3)_i$ edges appear in a negative matching if and only if there are edges of type $(1,2,1)_i$ or $(1,1,2)_i$ in the positive matching of a configuration. As well as the types of edge and types of zero-sum configuration to which they belong listed above, we shall also be interested in the number of zero-sum configurations containing at least two bad edges (all with positive sign). We leave the motivation for this to Section \ref{sec_1}, but introduce the relevant notation here. 
\begin{defn}[$\mathcal{Z}^2_{i,G}$, $\mathcal{Z}^2_{i,e, G}$, $\mathcal{Z}^+_{i,e,G}(\bad)$]
For each $i \in [c_g]$, define $\mathcal{Z}^2_{i,G}$ to be the set of $i$-legal zero-sum configurations present in $G$ which contain at least two $i$-bad edges (all with positive sign attached). For each bad edge $e$, define $\mathcal{Z}^2_{i,e, G}:=\{z \in \mathcal{Z}^2_{i, G}: e \in z\}$. Furthermore, we define $\mathcal{Z}^+_{i,e,G}(\bad)$ to be the collection of $i$-legal zero-sum configurations in $G$ which contain the edge $e$ with positive sign, where $e$ is an $i$-bad edge.
\end{defn}
This last piece of notation, $\mathcal{Z}^+_{i,e,G}(\bad)$, is useful for when we are only interested in the fact that $e$ is $i$-bad, rather than the fact that $e$ is $i$-bad of a certain type.

\begin{fact} \label{fact_Z}
For every $i \in [c_g]$ the following hold:
$$|\mathcal{Z}^+_{i,e,\mathcal{T}}(\alpha, \beta, \gamma)|:=
\begin{cases}
\Theta\left(j_it_1\right) & \mbox{~if $e$ is a bad edge}, \\
O\left(k_it_1 \right) & \mbox{~if $\alpha=0$, $\beta=1$, and $\gamma=3$},\\
0 & \mbox{~otherwise}.\\
\end{cases}
$$
$$|\mathcal{Z}^-_{i,e,\mathcal{T}}(\alpha, \beta, \gamma)|:=
\begin{cases}
\Theta\left(t_1^2\right) & \mbox{~if $\alpha \neq 0$ and $\gamma=0$}, \\
O\left(j_ik_i \right) & \mbox{~if $\alpha=0$, $\beta=0$, and $\gamma=4$},\\
O\left(j_ik_i \right) & \mbox{~if $\alpha=0$, $\beta=1$, and $\gamma=3$},\\
0 & \mbox{~otherwise}.\\
\end{cases}
$$
Finally, for every bad edge $e$,
$$|\mathcal{Z}^2_{i,e,\mathcal{T}}|=O\left(k_it_1\right).$$
\end{fact}

We note that in the fact above it would be equivalent to have `$n$' in place of `$t_1$' for each clause, however we write $t_1$ here for convenient use later.

\begin{proof}
As previously noted, it is clear by construction for all but the non-zero cases. We start by considering the upper bounds for $|\mathcal{Z}^+_{i,e,\mathcal{T}}(\alpha, \beta, \gamma)|$. In particular, every edge $e$ for which this is not clearly $0$ has at least one vertex $v \in J_i$. Given it is a bad edge we have $v \in K_i$. We know that $|\mathcal{Z}^+_{i,e,\mathcal{T}}(\bad)|$ has two degrees of freedom and that one of these must, along with $v$, dictate an edge with negative sign that is contained in $J_i$. There are $O(j_i)$ choices that could dictate such an edge. Once such an edge is fixed, in order to extend the pair of edges intersecting at vertex $v$ to an $i$-legal zero-sum configuration, there is one remaining degree of freedom, and since all $i$-legal zero-sum configurations are contained in $I_{t_1}$ it is clear that there are $O(t_1)$ such choices. That gives, in total, $O(j_it_1)$ configurations when $e$ is an $i$-bad edge. In the case that $e$ is an edge of type $(0,1,3)_i$, and so $v \in J_i \setminus K_i$, we know that an edge of negative sign containing $v$ in an $i$-legal zero-sum configuration for $e$ is of type $(\alpha, \beta, 0)_i$ with $\alpha \neq 0$, so one degree of freedom, along with $v$ must dictate an edge containing a vertex in $K_i$. There can be at most $O(k_i)$ such choices. Then once this is fixed, as before, there are $O(t_1)$ choices for the final degree of freedom given $O(k_it_1)$ $i$-legal zero-sum configurations containing $e$ of type $(0,1,3)_i$ as a positive edge, as claimed. The argument is the same for $|\mathcal{Z}^2_{i,e,\mathcal{T}}|$, since in this case, though $e$ is bad as in the first case, we know that the edge containing $v \in e$ of negative sign must contain at least two vertices in $K_i$. Thus the degree of freedom to dictate such an edge through $v$ has $O(k_i)$ possibilities.

Considering upper bounds on $|\mathcal{Z}^-_{i,e,\mathcal{T}}(\alpha, \beta, \gamma)|$, it is clear when $\alpha \neq 0$ and $\gamma =0$, since there are two degrees of freedom and to ensure all vertices are in $I_{t_1}$ means at most $O(t_1)$ available choices for each. It remains to consider when $e$ is of type $(0,1,3)_i$ or $(0,0,4)_i$ for some $i$. Given that we are considering configurations containing $e$ with negative sign, we have that if $e$ is of type $(0,1,3)_i$, the vertex $v \in e$ such that $v \in J_i \setminus K_i$ must also lie precisely in an $i$-bad edge of positive sign (as opposed to an edge of type $(0,1,3)_i$ or any other type). So in either case exactly one vertex $v \in e$ will also be in a bad edge of positive sign in any $i$-legal zero-sum configuration containing $e$ with negative sign. Since $v \in I_{t_1} \setminus K_i$, in order to ensure the bad edge $e'$ contains a vertex in $K_i$ one degree of freedom can take at most $O(k_i)$ values. Once such a pair $e, e'$ is fixed, in order that the vertex $v' \in e' \cap K_i$ lies in a (positive) edge of type $(\alpha, \beta,0)_i$, as required to dictate an $i$-legal zero-sum configuration for $e$, there are at most $O(j_i)$ possibilities. This gives a total of $O(j_ik_i)$ $i$-legal zero-sum configurations containing such an edge $e$ of negative sign.

It remains to consider the lower bound in the cases where $e$ is either a bad edge of positive sign or an edge of type $(\alpha, \beta, 0)$ of negative sign, where $\alpha \neq 0$. We show the lower bounds are satisfied by giving specific possibilities for the degrees of freedom in each case. In particular, given $e$ is fixed, we list sufficiently many pairs which each dictate a distinct $i$-legal zero-sum configuration for $e$ as follows:

\begin{enumerate}
\item if $e=(a, b+s, a+b+s, a-b-s)$ is bad and $a$ is small and $b+s$ is positive with $b+s \leq \frac{t_1}{2}$ or negative with $|b+s| \geq \frac{t_1}{2}$, then pairs $(c,s)$ with $c \in [\frac{t_1}{12}, \frac{t_1}{10}]$ and $s \in [-c-\frac{j_i}{2}, -c+\frac{j_i}{2}]$ `work' (i.e. dictate an $i$-legal zero-sum configuration containing $e$),
\item if $e=(a, b+s, a+b+s, a-b-s)$ is bad and $a$ is small and $b+s$ is positive with $b+s \geq \frac{t_1}{4}$ or negative with $|b+s| \leq \frac{t_1}{4}$, then pairs $(c,s)$ with $c \in [-\frac{t_1}{10}, -\frac{t_1}{12}]$ and $s \in [-c-\frac{j_i}{2}, -c+\frac{j_i}{2}]$ work.
\item Similarly, if $e=(a, b+s, a+b+s, a-b-s)$ is bad, $a+b+s$ is small and $a$ is positive with $a \leq \frac{t_1}{4}$ or negative with $|a| \geq \frac{t_1}{4}$, then pairs $(c,s)$ with $c \in [2a+\frac{t_1}{12}, 2a+\frac{t_1}{10}]$ and $s \in [-a-\frac{j_i}{10}, -a+\frac{j_i}{10}]$ work, and
\item if $e=(a, b+s, a+b+s, a-b-s)$ is bad, $a+b+s$ is small and $a$ is positive with $a \geq \frac{t_1}{4}$ or negative with $|a| \leq \frac{t_1}{4}$, then pairs $(c,s)$ with $c \in [-2a+\frac{t_1}{12}, -2a+\frac{t_1}{10}]$ and $s \in [-a-\frac{j_i}{10}, -a+\frac{j_i}{10}]$ work.
\item Finally, if $e=(a, c+s, a+c+s, a-c-s)$ is of type $(\alpha, \beta, 0)$ with $\alpha \neq 0$, then pairs $(b,s)$ with $b \in [\frac{t_1}{8}, \frac{t_1}{6}]$ and $s \in [\frac{t_1}{12}, \frac{t_1}{10}]$ work.
\end{enumerate}
By symmetry between $X$ and $Y$, and between $X+Y$ and $X-Y$, this covers all cases.
\end{proof}

We shall wish to check that the properties of zero-sum configurations as listed in Fact \ref{fact_Z} remain closely related (in a quasi-random sense) to their values in $\mathcal{T}$ as we run the random greedy edge removal process. More details of this will be covered in Chapter \ref{ch_count}.

\section{Degree-type conditions} \label{sec_H_details}

As usual we shall write $d_G(v)$ for the number of edges containing $v$ in $G$, also known as the {\it degree of $v$ in $G$}. For a weight function $w$ defined on $E(G)$, we shall also write $d_{w,G}(v):=\sum_{e \ni v, e \in G} w(e)$ and refer to this as the {\it weighted degree of $v$ in $G$}.

Our proof strategy relies on the fact that throughout the random greedy edge removal process and the iterative matching process, the subgraphs we obtain continue to have quasi-random properties. The precise nature of the structure that needs to be maintained will vary depending on where in the process we are, however each notion of quasi-randomness that we shall require involves understanding the number of edges containing a vertex $v$ and other vertices in specific subsets of $V(\mathcal{T})$. Broadly speaking, we refer to these as `degree-type conditions'.

Throughout the process, for a given graph $G \subseteq \mathcal{T}$, we will be interested in pairs $(v, S)$ such that $v \in V(G)$ and $S \subseteq V(\mathcal{T})$, and the number of edges which contain $v$ and relate to $S$ in $G$ in some way. More specifically, we shall often have pairs $(v,S)$ with $S$ of the form $I_{t_i}$ or $I_{t_i}\setminus I_{t_j}$, and it may be that we are interested both in the number of edges that contain $v$ and {\it only} other vertices in $S$, and also the number of edges that contain $v$ and {\it at least one} vertex from $S$. To ensure that we have a notation that encompasses both of these scenarios (without having to change $S$), we shall first describe a more general notation, though most of the cases given by this setting will then also have a more concise shorthand notation. 

Let $G \subseteq \mathcal{T}$ and $w$ be a weight function defined on $E(G)$. For sets $S_1, S_2, S_3 \subseteq V(\mathcal{T})$ not necessarily distinct, we define 
$$E_G(v,S_1,S_2,S_3):=\{e \ni v: e=(v, e_1, e_2, e_3)\in G, e_i \in S_i\},$$ 
so that $E_G(v,S_1,S_2,S_3)$ is the set of edges $e$ which contain $v$, and there exist distinct vertices $v_1, v_2, v_3 \in e \setminus \{v\}$ such that $v_i \in V(G[S_i])$. When $S_1=S_2=S_3=S$, we also write $E_G(v,S)$ as a shorthand notation. Furthermore, if $S_i = V(\mathcal{T})$ for any $i \in [3]$, we may write $*$ in place of $V(\mathcal{T})$ for short. In this way we have that $E_G(v,S,*,*):=E_G(v,S,V(\mathcal{T}),V(\mathcal{T}))$ is the set of edges in $G$ which contain $v$ and have at least one vertex other than $v$ in $V(G[S])$. We refer to $E_G(v,S_1,S_2,S_3)$ as the {\it set of edges for $v$ in $(G,S_1,S_2,S_3)$} (or similarly for $(G,S)$ if $S_1=S_2=S_3=S$, and for $(G,S,*,*)$ if $S=S_1$ and $S_2=S_3=V(\mathcal{T})$). We say that $|E_G(v,S_1,S_2,S_3)|$ is the {\it degree of $v$ in $(G,S_1,S_2,S_3)$}, and define $w(E_G(v,S_1,S_2,S_3)):=\sum_{e \in E_G(v,S_1,S_2,S_3)} w(e)$ to be the {\it weighted degree of $v$ in $(G,S_1,S_2,S_3)$ with respect to $w$}.

As previously mentioned, we shall mostly be interested in using the definition of $E_G(v,S_1,S_2,S_3)$ when $S_1=S_2=S_3$ or $S_2=S_3=V(\mathcal{T})$, when we may instead write $E_G(v,S_1,S_2,S_3)$ as $E_G(v,S)$ or $E_G(v,S,*,*)$ for some $S$. 

\begin{defn}[Valid] \label{def_valid}
We say that $S \subseteq V(\mathcal{T})$ is {\it $j$-valid} if $S=I_{t_i}$ or $S=I_{t_i}\setminus I_{t_{l}}$ for some $i \in [0,h]$, $j \leq i$ and $l>i$. We say that $(S_1, S_2, S_3) \subseteq V(\mathcal{T})^3$ is {\it $j$-valid} if either $S_1=S_2=S_3=S$ or $S_1=S$ and $S_2=S_3=V(\mathcal{T})$ for some $S$ which is $j$-valid. 
Let $G \subseteq \mathcal{T}$. We say that $(v,S_1, S_2, S_3)$ is a {\it closed $j$-valid tuple for $G$} if $S_1=S_2=S_3=S$ where $S$ is $j$-valid, and $v \in V(G[I_{t_j}])$ such that $|E_{\mathcal{T}}(v,S)|=\Theta(|S|)$ and $|E_{\mathcal{T}}(v,S)| \geq 0.1n^{10^{-5}}$. In this case we also say that $(v,S)$ is a {\it closed $j$-valid pair} for $G$.\footnote{the reason for the two notations is that when considering a specific closed valid tuple $(v,S_1,S_2,S_3)$ we shall mostly only be interested in those with $S_1=S_2=S_3$, however when describing general properties, we want a notation that describes the properties for both the case where $S_1=S_2=S_3$ and where $S_2=S_3=V(\mathcal{T})$ at the same time.} We say that $(v,S_1, S_2, S_3)$ is an {\it open $j$-valid tuple for $G$} if $S_1=S$ where $S$ is $j$-valid, $S_2=S_3=V(\mathcal{T})$ and for $v \in V(G[I_{t_j}])$, $|E_{\mathcal{T}}(v,S, *,*)|=\Theta(|S|)$ and $|E_{\mathcal{T}}(v,S,*,*)| \geq 0.1n^{10^{-5}}$. In this case we also say that $(v,S)$ is an {\it open $j$-valid pair for $G$}. 
\end{defn}
Note that for every $G \subseteq \mathcal{T}$, $v \in V(G[I_{t_j}])$ and every $j$-valid $S$ we have that $(v,S)$ is an open $j$-valid pair for $G$. We extend all of the above definitions from $j \in [0,c_h]$ to say that $(v, S_1, S_2, S_3)$ is a $\mathcal{T}$-valid tuple for $G$ if $v \in V(G)$ such that $|E_{\mathcal{T}}(v, S_1, S_2, S_3)|=\Theta(|S_1|)$ and $|E_{\mathcal{T}}(v, S_1, S_2, S_3)| \geq 0.1n^{10^{-5}}$. Note that if a set $S$ is $j_1$-valid, then $S$ is $j_2$-valid for every $j_2 \leq j_1$. (The motivation for the definition of $j$-valid sets is that once we reach the graph $H_j \subseteq \mathcal{T}[I_{t_j}]$ as per the vortex we are interested in degree-type properties of subgraphs of $H_j$ which relate to the subsets $\{I_{t_i}\}_{i \in [c_h]}$ of $V(\mathcal{T})$. Since $V(H_j) \subseteq I_{t_j}$ for each $j \in [c_h]$ it follows that for $I_{t_i}$ or any subset of $I_{t_i}$ to be of concern to us with regards to properties of $H_j$, we want $i \geq j$.) Note also that for $G \subseteq \mathcal{T}[I_{t_j}]$ and every $v \in V(G)$, we have that $|E_{G}(v, I_{t_j})|=d_{G}(v)$ and $w(E_{G}(v, I_{t_j}))=d_{w,G}(v)$.

The motivation for the definition of open and closed $j$-valid tuples, comes from the fact that throughout the process, we want to ensure that how large the set of edges for a vertex $v \in V(H_j)$ is in $(H_j, S)$ and $(H_j, S,*,*)$ for various $j$-valid sets $S$ is controlled carefully, so that the degree of vertices remaining will be controlled enough to continue through the vortex. Once we reach $H_1$ and have done the weight shuffle (details of which are covered in Sections \ref{sec_overview} and \ref{sec_intro_reweight}), our interest will turn from open and closed $j$-valid pairs to a restricted sub-family of such pairs, but the process to take us from $\mathcal{T}$ to $H_1$ relies on nice degree-type properties for all open and closed $j$-valid pairs.

\section{Wrap-around edges} \label{sec_parity}

Recall that by a wrap-around edge, with the vertex indexing from $[-t_0, t_0]$ in each part, we mean any edge with indices $a \in V^X$, $b \in V^Y$ such that either $|a+b|>t_0$ or $|a-b|>t_0$. We start by noting some straightforward but crucial observations regarding wrap-around edges:
\begin{enumerate}
\item each wrap-around edge contains at least one coordinate with index whose modulus is at least $\frac{t_0}{2}$, i.e. that is far from the centre of $\mathcal{T}$; and 
\item when $n$ is odd a wrap-around edge will contain coordinates of different parity in $V^{X+Y}$ and $V^{X-Y}$. In particular, when $n$ is odd, choosing a pair of vertices $(a,b) \in V^{X+Y} \times V^{X-Y}$ where $a \not\equiv b \mymod 2$ dictates a wrap-around edge. When $n$ is even $\mathcal{T}$ only consists of edges with both vertices in the $X+Y$ and $X-Y$ parts of the same parity. In this case choosing a pair of vertices $(a,b) \in V^{X+Y} \times V^{X-Y}$ where $a \equiv b \mymod 2$ we have two edges $e_1, e_2$ both containing $\{a,b\}$ such that one is a wrap-around edge and the other does not wrap-around.
\end{enumerate} 
That is, whilst for `standard' arithmetic we have that for $a$ and $b$ of the same parity we get $a+b$ and $a-b$ both even, and for $a$ and $b$ of different parity we have $a+b$ and $a-b$ both odd, if an edge wraps around, when $n$ is odd, the coordinate which has gone around the torus will have a different parity to that which would be given by the natural arithmetic. In particular, we know then whether an edge is a wrap-around edge or not, purely by considering its coordinates in $V^{X+Y}$ and $V^{X-Y}$. As our iterative matching process relies on a vortex of nested subgraphs $\{H_i\}_i$ such that $H_i \subseteq I_{t_i}$, we have that all subgraphs from $H_1$ onwards do not contain wrap-around edges.
\begin{defn}[Vertex subsets of $G$]
From now on, we write $V^J(G):=V_J \cap V(G)$ for each $J \in \{X,Y,X+Y,X-Y\}$ and $G \subseteq \mathcal{T}$. Similarly, we write $V^J(S):=V_J \cap S$ for each $J \in \{X,Y,X+Y,X-Y\}$ and $S \subseteq V(\mathcal{T})$, and denote by $V^{X \pm Y}_{O/E}(G)$ the set of vertices in $V(G) \cap V^{X \pm Y}$ that have odd/even parity. We extend this definition naturally also to $V^{X / Y}_{O/E}(G)$.
\end{defn} 
For reasons that will become apparent in Chapter \ref{ch_absorber}, we'll want to be able to pair up vertices of the same parity in $V^{X+Y}$ and $V^{X-Y}$ (to dictate edges but avoid inducing wrap-around edges via these pairings), and as such will want to ensure that the number of vertices of odd parity remaining in $V^{X+Y}(H_1)$ is the same as the number of vertices of odd parity remaining in $V^{X-Y}(H_1)$. (This in turn implies also that the number of vertices of even parity remaining in $V^{X+Y}(H_1)$ is the same as the number of vertices of even parity remaining in $V^{X-Y}(H_1)$, since we reach $H_1$ only by removing disjoint edges from $\mathcal{T}$ so that the total size of $V^{X+Y}(H_1)$ is equal to that of $V^{X-Y}(H_1)$.) Thus in the process to reach $H_1$ we must keep track of the parities of vertices remaining in parts $V^{X+Y}$ and $V^{X-Y}$. Note that this all only applies to the case where $n$ is odd. When $n$ is even, reaching $H_1$ only by removing disjoint edges from $\mathcal{T}$ we can be certain that $|V^{X+Y}(H_1)|=|V^{X-Y}(H_1)|$ as every edge removes two vertices of the same parity in parts $X+Y$ and $X-Y$. Then when it comes to pairing up vertices in $V^{X+Y}$ and $V^{X-Y}$ to dictate edges but avoid wrap-around edges, whilst we know that such a pair with the same parity has pair degree $2$, we know that we can therefore \emph{choose} whether we take the wrap around edge or non-wrap around edge dictated by the pair, depending on what we wish to achieve by this pairing.

\begin{defn}[$J$-layer intervals] \label{def_layer}
We say that a set $I^J$ is a {\it valid $J$-layer interval} if $I^J \subseteq V(\mathcal{T}) \cap V^J$ where $J \in \{X,Y,X+Y,X-Y\}$, with $I^J=[a,b]$ (i.e. a subinterval of $[-t_0, t_0]$), and $|I^J|\geq t_0^{1-2\epsilon}$.\footnote{Recall that $\epsilon=\frac{10^{-8}}{204800}$ as in Section \ref{sec_pars}.}   
For $G \subseteq \mathcal{T}$, a valid $J$-layer interval $I^J$ and $v \notin J$ we let $E_G(v, I^J, O/E)$ denote the set of edges in $G$ which contain $v$ and a vertex $u \in V^J_{O/E}(G) \cap I^J$. 
\end{defn}
We shall keep an eye on these parity related quantities throughout the random greedy edge removal process as well as the initial steps of the iterative matching process. Additionally, for a graph $G$, we let 
$$E_{AB}(G)$$ 
be the set of edges in $G$ such that the $X+Y$ coordinate is of parity type $A \in \{O, E\}$ and the $X-Y$ coordinate is of parity type $B \in \{O,E\}$. Then when $n$ is odd $E_{OE}(G) \cup E_{EO}(G)$ gives the complete set of edges in $G$ which wrap-around. (When $n$ is even we have that $E_{AB}(G)= \emptyset$ for $AB \in \{OE, EO\}$.)

The following theorem, proved in Section \ref{sec_greedy}, gives the initial quasi-randomness properties which we shall show whp for $H$, the graph obtained from $\mathcal{T}$ after removing $A^*$, running the random greedy edge removal process and making small parity modifications, and that are required for the iterative matching process described in Chapter \ref{ch_it_abs}. Recall $\alpha_G$ and $p_{\gr}$ as defined in Definition \ref{def_const}.

\begin{theo} \label{thm_H}
After removing the absorber $A^*$ from $\mathcal{T}$, running the random greedy counting process and making some parity modifications, with high probability we obtain a (hyper)graph $H \subseteq \mathcal{T}$ satisfying the following:
\begin{enumerate}[(i)]
\item every $\mathcal{T}$-valid subset $S \subseteq V(\mathcal{T})$ satisfies 
$$|V(H[S])|=(1 \pm \alpha_G)|S|p_{\gr},$$
\item for every $v \in V(H)$ and every open or closed $\mathcal{T}$-valid tuple $(v,S_1,S_2,S_3)$, we have
$$|E_H(v,S_1, S_2, S_3)|=(1 \pm \alpha_G)|E_{\mathcal{T}}(v,S_1, S_2, S_3)|p_{\gr}^3,$$
\item for every $i \in [c_g]$, 
$$|\mathcal{Z}^+_{i,e,H}(\alpha, \beta, \gamma)|:=
\begin{cases}
(1 \pm \alpha_G)|\mathcal{Z}_{i,e,\mathcal{T}}^{+}(\alpha, \beta, \gamma)|p_{\gr}^{12} & \mbox{~if $e$ is a bad edge}, \\
O\left(k_it_1p_{\gr}^{12} \right) & \mbox{~if $\alpha=0$, $\beta=1$, $\gamma=3$},\\
0 & \mbox{~otherwise}.\\
\end{cases}
$$
$$|\mathcal{Z}^-_{i,e,H}(\alpha, \beta, \gamma)|:=
\begin{cases}
(1 \pm \alpha_G)|\mathcal{Z}_{i,e,\mathcal{T}}^{-}(\alpha, \beta, \gamma)|p_{\gr}^{12} & \mbox{~if $\alpha \neq 0$ and $\gamma=0$}, \\
O\left(j_ik_ip_{\gr}^{12} \right) & \mbox{~if $\alpha=0$, $\beta=0$, $\gamma=4$},\\
O\left(j_ik_ip_{\gr}^{12} \right) & \mbox{~if $\alpha=0$, $\beta=1$, $\gamma=3$},\\
0 & \mbox{~otherwise}.\\
\end{cases}
$$
Finally, for every bad edge $e$,
$$|\mathcal{Z}^2_{i,e,H}|=O\left(k_it_1p_{\gr}^{12}\right).$$
\item We have that $|V_{O}^{X+Y}(H)|=|V_{O}^{X-Y}(H)|$, $|V_{E}^{X+Y}(H)|=|V_{E}^{X-Y}(H)|$ and, furthermore $|V_O^{J}(H)|=(1 \pm 2\alpha_G)|V_E^{J}(H)|$ for every $J \in \{X,Y,X+Y,X-Y\}$. Additionally, $|V_{O/E}^{J_1}(H[S])|=(1 \pm 2\alpha_G)|V_{O/E}^{J_2}(H[S])|$ for every valid layer interval $S$, and $J_1,J_2 \in \{X,Y,X+Y,X-Y\}$.
\item For every $J \in \{X,Y,X+Y,X-Y\}$ and every valid $J$-layer interval $I^J$ and $v \notin J$,
$$|E_H(v, I^J, O/E)|=(1 \pm \alpha_G)|E_\mathcal{T}(v, I^J, O/E)|p_{\gr}^3.$$
\end{enumerate}
\end{theo}

Whilst we could define more general properties, or subsets for which such properties hold, these are the key subsets and properties required which allow us to run our process, and it is more intuitive to have these properties in mind throughout the process than something more general.

\section{Facts about $\mathcal{T}$}\label{sec_facts}

Our strategy will at times compare a subgraph $H_i \subseteq \mathcal{T}$ to $\mathcal{T}[I_{t_i}] \subseteq \mathcal{T}$, and in particular the following details are useful for calculations in Section \ref{sec_H}. More generally, the following facts should help develop a better picture of what $\mathcal{T}$ `looks like' in terms of various degree type and wrap-around edge type properties.

\begin{fact} \label{fact_basic}
Given that $\mathcal{T}[I_{t_k}]$ contains no wrap around edges and $v \in I_{t_1} \cap I_{t_{k-1}}$, we have that $|E_{\mathcal{T}}(v,I_{t_k})|$ satisfies the following:
$$|E_{\mathcal{T}}(v,I_{t_k})| \pm 1 =
\begin{cases}
(\frac{4t_k}{3}+ 1) & \mbox{~for~} v \in [-\frac{t_k}{3}, \frac{t_k}{3}]\cap (X \cup Y), \\
(2t_k+1-2|v|) & \mbox{~for~} v \in \pm[\frac{t_k}{3}, \frac{2t_{k-1}}{3}]\cap (X \cup Y)\\
(t_k+1_{|v| \mbox{~is even}}) & \mbox{~for~} v \in [-\frac{t_k}{3}, \frac{t_k}{3}]\cap (X+Y \cup X-Y)\\
(\frac{4t_k}{3}+1-|v|) & \mbox{~for~} v \in \pm[\frac{t_k}{3}, t_{k-1}]\cap (X+Y \cup X-Y).\\
\end{cases}
$$
\end{fact}

Note that $\mathcal{T}[I_{t_1}]$ contains no wrap-around edges, and hence Fact \ref{fact_basic} is true for all $k \geq 1$. When $v \in V(\mathcal{T}) \setminus I_{t_1}$, or we are considering $I_{t_0}$ there are still wrap-around edges to consider. It is helpful for the initial steps of the iterative matching process to observe some facts about $|E_{\mathcal{T}}(v,I_{t_0})|$ for $v \in \mathcal{T}$, as well as $|E_{\mathcal{T}}(v,I_{t_1})|$ for $v \in I_{t_0}\setminus I_{t_1}$.

\begin{fact} \label{fact_basic2}
The following all hold:
\begin{enumerate}[(i)]
\item For every $v \in V(\mathcal{T})$ we have that
$$\frac{n}{3}-2 \leq E_{\mathcal{T}}(v, I_{t_0}) \leq \frac{2n}{3}+2,$$
\item for every $v \in I_{t_0}$ we have that
$$ E_{\mathcal{T}}(v, I_{t_1}) \geq \frac{n}{30}-2,$$
\item when $n$ is odd, for each vertex $v \in (V(\mathcal{T}) \setminus I_{t_1})\cap (X \cup Y)$,
$$|E_{\mathcal{T}}(v, I_{t_0}\setminus I_{t_{20}}) \cap E_{AB}(\mathcal{T})| \geq \frac{n}{40}$$
for $AB \in \{OE, EO\}$.
\item for each vertex $v \in I_{t_0} \setminus I_{t_1}$
$$|E_{\mathcal{T}}(v, I_{t_1}\setminus I_{t_{20}}) \cap E_{AB}(\mathcal{T})| \geq \frac{n}{300}$$
for $AB \in \{EE, OO\}$.
\end{enumerate}
\end{fact}
\begin{proof}
Suppose first that $v \in X$. For the lower bound of (i) we have that $E_{\mathcal{T}}(v, I_{t_0})$ consists of all edges that contain $v$ and a vertex from $I_{t_0} \cap Y$, of which there are $\frac{4t_0}{3}+1$ such vertices. Each pair dictates a distinct edge and since $I_{t_0} \cap (X \pm Y)=X \pm Y$, each edge is in $E_{\mathcal{T}}(v, I_{t_0})$. By symmetry we have the same argument for $v \in Y$. For $v \in X+Y$, $E_{\mathcal{T}}(v, I_{t_0})$ contains all edges such that $v=u_1+u_2 \mod n$ where $u_1, u_2 \subseteq [-\frac{2t_0}{3}, \frac{2t_0}{3}]$. Now since every vertex $v$ has degree $n = 2t_0\pm 1$ in $\mathcal{T}$, and the maximum pair degree for a pair $u,v$ with $v \in X+Y$ and $u \in X \cup Y$ is $1$, this excludes at most $\frac{4t_0}{3}$ of these edges from $E_{\mathcal{T}}(v, I_{t_0})$ (in the case where each vertex in $V(\mathcal{T}) \setminus I_{t_0}$ is in a different edge with $v$). By symmetry we have the same for $v \in X-Y$. For the upper bound, every vertex $v$ in $Y \cup (X+Y) \cup (X-Y)$ shares an edge with every vertex in $X$ in $\mathcal{T}$. Since at least $\frac{2t_0}{3}$ vertices from $X$ are outside $I_{t_0}$, at least this many edges (of the $n$ containing $v$ in $\mathcal{T}$) are not in $E_{\mathcal{T}}(v, I_{t_0})$. This is similarly true for vertices in $X$, considering the vertices in $Y \setminus I_{t_0}$. To see (ii), note that $t_1 \approx \frac{4t_0}{5}$. For a vertex $v \in I_{t_0} \cap X$ every pairing to a vertex in $[-\frac{2t_0}{15}, \frac{2t_0}{15}]$ in $Y$ yields an edge with all other vertices in $I_{t_1}$. By symmetry the same holds for $v \in I_{t_0} \cap Y$. For $v \in I_{t_0} \cap (X+Y)$, every pair of vertices $(x,y) \in I_{t_1} \times I_{t_1}$ such that the sum is equal to $v$ and the difference is less than $\frac{4t_0}{5}$ creates a relevant edge for $v$. There are at least $\frac{t_0}{15}$ such pairings, and by symmetry the same is true for $v \in I_{t_0} \cap (X-Y)$. To see (iii) and (iv), note that $t_{20} \leq 0.012t_0$. So for a vertex $v \in X$ it suffices to ensure the choice of vertex in $Y$ has distance at least $0.012t_0$ from $v$. Furthermore, any choice of vertex in $I_{t_0}\setminus I_{t_1}$ with the same sign will result in a wrap-around edge. Restricting to when $n$ is odd, this gives at least $\frac{2t_0}{15}-0.024t_0 \geq \frac{t_0}{10}$ such edges. Half of these pairings will wrap around giving odd parity in $X+Y$ and even parity in $X-Y$, and half will dictate an edge with even parity in $X+Y$ and odd parity in $X-Y$. By symmetry we get the same when $v \in Y$. Finally, the count for (iv) follows similarly, considering now all $n$. In particular, for a vertex in $X \cap (I_{t_0}\setminus I_{t_1})$ pairing with vertices in $[\frac{t_0}{15}, \frac{2t_0}{15}]$ in $Y$ gives a non wrap-around edge with all other vertices in $I_{t_1}\setminus I_{t_{20}}$. The same is true swapping $X$ and $Y$. For $v \in (X+Y) \cap (I_{t_0}\setminus I_{t_1})$, if $v$ has positive index, pairing with a vertex in $[\frac{7t_0}{15}, \frac{8t_0}{15}]$ gives sufficiently many edges which do not wrap around and have all other vertices contained in $I_{t_1}$. This gives at least $\frac{t_0}{15}-2t_{20} \geq 0.04t_0 \geq n/150$ edges also avoiding $I_{t_{20}}$. Half of the pairings will dictate an edge with vertices in $X+Y$ and $X-Y$ both of odd parity, and half with both of even parity. Similar pairings with signs flipped work in the cases where $v \in X+Y$ has negative index, or $v \in X-Y$.
\end{proof}

%% file: ch_absorber.tex
\chapter{The absorber} \label{ch_absorber}

We say that $L^* \subseteq V(\mathcal{T})$ is a {\it qualifying leave} if $L^*$ satisfies the following properties:
\begin{enumerate}
\item the support vector ${\bf v_{L^*}}$ of $L^*$ satisfies ${\bf v_{L^*}} \in \mathcal{L}(\mathcal{T})$,
\item $L^* \subseteq I_{n^{10^{-5}}}$,
\item $|L^*| \leq p_Ln^{10^{-5}}$, recalling $p_L$ defined in Definition \ref{def_const},
\item $|V_O^{X+Y}(L^*)|=|V_O^{X-Y}(L^*)|$,
\end{enumerate}

The aim of this chapter is to build an absorber $A^*$ that can absorb any qualifying leave $L^*$. Formally, in our context, we say that $A^* \subseteq V(\mathcal{T})$ is an {\it absorber} for $\mathcal{T}$ if $\mathcal{T}[A^*]$ contains a perfect matching, and for every qualifying leave $L^*$ we have that $\mathcal{T}[A^* \cup L^*]$ contains a perfect matching. By slight abuse of notation we may think of an absorber $A^*$ both as a vertex subset of $\mathcal{T}$ (as described in the definition) and also as the sub(hyper)graph of $\mathcal{T}$ induced by this set $A^*$ of vertices. Wherever this distinction is important, it will be clear from the context. 
Now, supposing that we could find a set of disjoint edges $\Phi^+$ such that $L^* \subseteq \partial \Phi^+$, and a set of disjoint edges $\Phi^-$ (not disjoint from $\Phi^+$) such that for the signed multi-set $\Phi=\Phi^+ \cup \Phi^-$, we have that $\partial \Phi={\bf v_{L^*}}$. If $\partial \Phi^- \subseteq A^*$ and $A^* \setminus \partial \Phi^-$ has a perfect matching, then taking the perfect matching in $A^* \setminus \partial \Phi^-$ along with $\Phi^+$ yields a perfect matching for $\mathcal{T}[A^* \cup L^*]$. Before finding an absorber $A^*$ and setting $\Phi^-$ with these properties, we first relax to finding an integral solution on $\mathcal{T}$. That is, sets $\Phi^-$ and $\Phi^+$ may each consist of a multi-set of edges from $\mathcal{T}$ such that $\partial \Phi = {\bf v_L}$. In this way we have an {\it integral decomposition} of $L^*$ (we are able to describe $L^*$ as the difference of two collections of edges in $\mathcal{T}$). Key to our approach is that the number of edges used in $\Phi$ is not too large (where what constitutes `not too large' will become clear as we describe our process to obtain $\Phi$). On finding $\Phi$ which is sufficiently small, we can show that we may modify $\Phi$ so that $\Phi^+$ and $\Phi^-$ are both matchings in $\mathcal{T}$, and still sufficiently small, such that $\Phi \in \{-1,0,1\}^{E(\mathcal{T})}$. Details of this part of the strategy are covered in Section \ref{sec_int_dec}. From this step we wish to complete the process by modifying $\Phi^+$ and $\Phi^-$ once again so that $\partial \Phi^- \subseteq A^*$ and $A^* \setminus \partial \Phi^-$ has a perfect matching, all the time maintaining that $\partial \Phi={\bf v_{L^*}}$. In this way we have that $A^*$ is an absorber for $L^*$, as required. This is much like the strategy of {\it hole} in \cite{countingdesigns}. Our strategy for building $A^*$, detailed in Section \ref{sec_using_abs}, will start by considering a random subset of vertices taken from within a fixed subset of $V(\mathcal{T})$ such that each vertex is included independently with some probability $p_A$. Then we'll show that with high probability such a set satisfies many properties we require, and we will end the process by fixing such a set and modifying it slightly to ensure that it maintains all the key properties we will have shown, but additionally has a perfect matching.

\section{Finding an integral decomposition for $L^*$} \label{sec_int_dec}

We aim to find a vector of edges $\Phi \in \mathbb{Z}^{E(\mathcal{T})}$ such that $\partial \Phi - L^*={\bf 0}$. 

\begin{lemma} \label{cover_end}
For any qualifying leave $L^*$ there exists $\Phi \in \mathbb{Z}^{E(\mathcal{T})}$ such that every edge in $\Phi$ is contained in $I'_{n^{10^{-5}}}$, $|\Phi|=O\left(p_Ln^{2.6 \times 10^{-5}}\right)$, and ${\bf v}=\partial \Phi - L^*={\bf 0}$.
\end{lemma}

Whilst we shall eventually want $\Phi \in \{-1,0,1\}^{E(\mathcal{T})}$, the first key step to doing this is to find a bounded integral decomposition. That is, we shall do this by first finding $\Phi \in \mathbb{Z}^{E(\mathcal{T})}$, such that $\partial \Phi - L^*={\bf 0}$ and $|\Phi|:=\sum_{e \in E(\mathcal{T})} |\Phi_e| = o(n)$. We call $|\Phi|$ the {\it size of $\Phi$}. For our setting of building the absorber we'll need to be more precise about $\Phi$, both in the types of edges we are allowed to add to $\Phi$ and the size of $\Phi$ - we'll show that $|\Phi|=O(n^{2.6\times 10^{-5}})$. Let ${\bf v}:=\partial \Phi - L^*$. We view $\Phi$ dynamically - we are continually adding signed edges to an initially empty set $\Phi$ with the goal to stop when we find $\Phi$ such that ${\bf v}:=\partial \Phi - L^*={\bf 0}$, thus finding the required integral decomposition. Our strategy relies on first covering $L^*$ by edges (and adding these to $\Phi$) such that the support of ${\bf v}$ is pushed closer and closer to the centre of $\mathcal{T}$. Indeed, we show that we may push the support so that it is contained only on coordinates in $\{-1,0,1\}$. We'll then show that we can reduce such a support to ${\bf 0}$. Throughout this section we do not repeat the running assumptions when stating the following propositions.

From now on for a coordinate $i$ in $V^J$ with weight $v_i^J \neq 0$ in ${\bf v}$, we call $i^J$ with weight $\pm 1$ (according to whether $v^J_i$ is positive or negative) a {\it unit} of ${\bf v}$. In this way, for each $i$ and $J$, $v^J_i$ contributes $|v^J_i|$ units to ${\bf v}$. For any multi-subset of units $U$, write $|U|$ to be the size of the multi-set, and $\sum_{i \in U} v_i$ when $i \in U$ refers to the set of vertices $U \subseteq V$, so that we may recognise $U$ both as a multi-set of individual units, and as a set of the vertices with non-zero weight in ${\bf v}$. For our strategy to work, we shall wish to add edges to $\Phi$ to ensure ${\bf v} \in \mathcal{L}^{X,Y}_2(\mathcal{T})$, and we refer to this as {\it zero-summing} the support of ${\bf v}$. In addition, we shall wish to avoid any wrap-around edges in building $\Phi$, as we want to retain some parity properties that we inherit from $L^*$. In particular, if the support is contained in a bounded interval $[-n_0, n_0]$, for some $n_0 < n/4$, then in order to ensure that we stay within this interval when we zero-sum the support, we must add edges dictated by pairing up odd vertices in $V^{X+Y}$ with odd vertices in $V^{X-Y}$ and similarly even with even. (When $n$ is even the parity issue is not a concern, and we choose the non-wrap around edge dictated by such a pairing in order to stay within this interval when zero-summing the support.) 
Let $U^J_{O,\pm}({\bf v})$ be the multi-set of units of odd parity with positive or negative weight in $V^J$ from ${\bf v}$ respectively, where $J \in \{X,Y,X+Y,X-Y\}$. Similarly, define $U^J_{E,\pm}({\bf v})$ to be the units of even parity with positive or negative weight in $V^J$. When ${\bf v}$ is clear from the context, we just write $U^J_{O/E,\pm}$. We claim the following:

\begin{prop} \label{dummyunits}
Let ${\bf u} \in \mathcal{L}(\mathcal{T})$ be such that 
$$|U^{X+Y}_{O,+}({\bf u})|-|U^{X-Y}_{O,+}({\bf u})|=|U^{X+Y}_{O,-}({\bf u})|-|U^{X-Y}_{O,-}({\bf u})|,$$ 
and 
$$|U^{X+Y}_{E,+}({\bf u})|-|U^{X-Y}_{E,+}({\bf u})|=|U^{X+Y}_{E,-}({\bf u})|-|U^{X-Y}_{E,-}({\bf u})|.$$
Then there exists a signed multi-set of edges $\Phi'$, all of which are not wrap around edges, such that ${\bf u'}:=\partial\Phi'+{\bf u} \in \mathcal{L}^{X,Y}_2(\mathcal{T})$.
\end{prop}

\begin{proof}
Let ${\bf u} \in \mathcal{L}(\mathcal{T})$ and suppose that $|U^{X+Y}_{O,+}|-|U^{X-Y}_{O,+}|=|U^{X+Y}_{O,-}|-|U^{X-Y}_{O,-}|$, and $|U^{X+Y}_{E,+}|-|U^{X-Y}_{E,+}|=|U^{X+Y}_{E,-}|-|U^{X-Y}_{E,-}|$. Without loss of generality suppose that $|U^{X+Y}_{O,+}| \geq |U^{X-Y}_{O,+}|$. Then let $k:=|U^{X-Y}_{O,+}|$, $l:=|U^{X-Y}_{O,-}|$, and $c:=|U^{X+Y}_{O,+}|-k$. It follows that $|U^{X+Y}_{O,-}|=l+c$. Then we may pair up $k$ of the odd positive units in $V^{X+Y}$ with the $k$ in $V^{X-Y}$, and $l$ of the odd negative units in $V^{X+Y}$ with the $l$ in $V^{X-Y}$. For each of these, we add the oppositely signed edge dictated by the unit pairings. We are left with $c$ odd positive units in $V^{X+Y}$ and $c$ odd negative units in $V^{X+Y}$. Pair these units up, so that there are $c$ pairs each with one negative unit and one positive unit. Choosing any $c$ odd vertices in $V^{X-Y}$, we may then assign each of the pairs in $V^{X+Y}$ to such a vertex in $V^{X-Y}$. In this way, adding the oppositely signed edges dictated by these pairings, we have zero-summed the odd units, since we cancelled the weights that were on any odd units, and the additional vertices used in $V^{X-Y}$ had both positive and negative weight added, so that in total no weight was added to these vertices. Doing the same for the even units shows that we can zero-sum, and by pairing according to parity, we have ensured that there are no wrap around edges.
\end{proof}

Note that in the proof of Proposition \ref{dummyunits}, even when we can zero-sum, there are times at which we cannot pair up units in $V^{X+Y}$ and $V^{X-Y}$ directly, and instead have to choose some vertices to pair up with both a positive and a negative unit. From now on we refer to any such vertex as a {\it dummy vertex}. In particular, if such a vertex is used $s$ times in this role, then it can also be seen as $s$ negative {\it dummy units} and $s$ positive dummy units. Before doing any zero-summing of the support, we first show that we can iteratively push it down in a way that does not cause the size of the support to blow up too much, and uses only edges that are close to the centre of $\mathcal{T}$. Recall the notation $I'_s$ from Definition \ref{def_ints}. 

\begin{prop}\label{support_push_down}
Suppose that ${\bf u} \in \mathcal{L}(\mathcal{T})$ is such that $\supp({\bf u})\subseteq I'_{t}$, for some $t \leq t_0$, with $t$ even. Then there exists $\Phi' \in \mathbb{Z}^{E(\mathcal{T})}$ such that for ${\bf u'}:=\partial\Phi'+{\bf u}$, we have $\supp({\bf u'}) \subseteq I'_{t/2}$. Furthermore, $|{\bf u'}| \leq 6|{\bf u}|$, $|\Phi'|\leq 3|{\bf u}|$, and $\Phi' \subseteq \mathcal{T}[I'_{t}]$.
\end{prop}
\begin{proof} 
Enumerate the units of support which are not in $I'_{t/2}$. For each element in the enumeration we add an oppositely signed edge through the unit to cancel its support on that vertex, and only add support to elements which are smaller. For $J \in \{X,Y\}$, additional edges may be required to ensure that overall no weight is added to vertices outside $I'_{t/2}$ in the process. We give explicit constructions for the push down: for $2a \in V^{X+Y}$ add $(a,a,2a,0)$ with the opposite sign, and for $2a-1 \in V^{X+Y}$ add $(a,a-1,2a-1,1)$. Similarly for $2a \in V^{X-Y}$ add $(a,-a,0,2a)$ and for $2a-1 \in V^{X-Y}$ add $(a,-a+1,1,2a-1)$. For $2a-i \in V^X$, where $i \in \{0,1\}$ add $(2a-i,0,2a-i,2a-i)$, and then adding $(a,a-i,2a-i,i)$ and $(a,-a+i,i,2a-i)$ with opposite sign ensures that overall weight is only added to at most $0$, $\pm a$ and $\pm (a-1)$ in each part, which are all contained in $I'_{t/2}$. Finally, for $2a-i \in V^{Y}$, add $(0,2a-i,2a-i,-2a+i)$, and then cancel the additional weight in $V^{X+Y} \cup V^{X-Y}$ by adding $(a,a-i,2a-i,i)$ and $(-a+i,a,i,-2a+i)$ with opposite signs.

It is clear that this pushes all weight in, as required. Furthermore, at most $3$ edges are added for each unit of support outside $I'_{t/2}$, thus at most $3|{\bf u}|$ edges are added to $\Phi'$. Finally, each of these edges cancels the weight on at least one vertex, and adds weight to three others, thus ${\bf u'}$ now has at most six times as much support as ${\bf u}$. 
\end{proof}

By definition of the qualifying leave $L^*$ starting with $\Phi=\emptyset$ we have that $\supp({\bf v}) \subseteq I'_{n^{10^{-5}}}$, and $|{\bf v}|=O\left(p_Ln^{10^{-5}}\right)$. Repeatedly using Proposition \ref{support_push_down} we are able to add signed edges to $\Phi$ in such a way that the update to ${\bf v}$ ensures that $\supp({\bf v}) \subseteq I'_1$, i.e. $\supp({\bf v})$ is contained only on coordinates indexed by elements in $\{-1,0,1\}$. This process never uses any edges outside of $I'_{n^{10^{-5}}}$. Furthermore, this requires $\log_2(n^{10^{-5}})$ iterations of Proposition \ref{support_push_down}. Since each push down may increase the size of the support by a factor of $6$, pushing down $\log_2 (n^{10^{-5}})$ times may increase the size of the support to $O\left(p_Ln^{10^{-5}\log_2(6)}\right)<O\left(p_Ln^{2.6 \times 10^{-5}}\right)$. Additionally, since at most three edges are added to $\Phi$ for each unit of support, we have that $|\Phi| \leq \sum_{i=0}^{\log_2 (n^{10^{-5}})-1}3\cdot 6^i|L^*| = O\left(p_Ln^{2.6 \times 10^{-5}}\right)$. 

\begin{prop} \label{zero-summing}
We can modify $\Phi$ adding only $O\left(p_Ln^{2.6 \times 10^{-5}}\right)$ edges to $\Phi$ in such a way that we obtain ${\bf v} \in \mathcal{L}^{X,Y}_2(\mathcal{T})$, and $\supp({\bf v})$ is contained on coordinates in $[-1,1]$. Furthermore, the size of the support remains at most $O\left(p_Ln^{2.6 \times 10^{-5}}\right)$.
\end{prop}
\begin{proof}
We have that $\Phi$ consists of $O\left(p_Ln^{2.6 \times 10^{-5}}\right)$ edges, none of which are wrap-around edges, and so the support of ${\bf v}$ will still satisfy parity properties required to use Proposition \ref{dummyunits}. Note that any pairing of two vertices of the same parity in $V^{X \pm Y}$ contained in $[-1,1]$, dictates an edge whose values in $V^X \cup V^Y$ are also contained in $[-1,1]$. Thus the zero-summing process ensures that the support remains in $I'_{1}$, and the ${\bf v}$ obtained after zero-summing satisfies ${\bf v} \in \mathcal{L}^{X,Y}_2(\mathcal{T})$, as required. Furthermore, the process of zero-summing requires pairing up $O\left(p_Ln^{10^{-5}\log_2(6)}\right)<O\left(p_Ln^{2.6 \times 10^{-5}}\right)$ units in $V^{X+Y}$ and $V^{X-Y}$ (including any possibly dummy units), and so this process adds at most $O\left(p_Ln^{2.6 \times 10^{-5}}\right)$ edges to $\Phi$. Similarly, zero-summing the support can at worst double the size of the support and hence the size of the support also remains $O\left(p_Ln^{2.6 \times 10^{-5}}\right)$, as required.
\end{proof}

We are now in a position to prove Lemma \ref{cover_end}.

\begin{proof}[Proof of Lemma \ref{cover_end}]
By Propositions \ref{support_push_down} and \ref{zero-summing} we obtain $\Phi$ such that every edge in $\Phi$ is contained in $I'_{n^{10^{-5}}}$ and $|\Phi|=O\left(p_Ln^{2.6 \times 10^{-5}}\right)$. Furthermore, we have that ${\bf v} \in \mathcal{L}^{X,Y}_2(\mathcal{T})$, all support is contained on coordinates in $[-1,1]$ and the size of the support is at most $O\left(p_Ln^{2.6 \times 10^{-5}}\right)$. We reduce ${\bf v}$ to ${\bf 0}$ as follows.

First note that, after zero-summing, we have that $\sum_{i~\in~V^{X/Y}~\cap~[-1,1]}~v_i=0$, and that $\sum_{i~\in~V^{X/Y}~\cap [-1,1]}~iv_i=0~(\mymod~n)$. (This follows from Proposition \ref{zero_sum_SQ} and the remarks after Proposition \ref{eight_ones}.) Since the total support on $V^X$ and $V^Y$ is $O\left(p_Ln^{2.6 \times 10^{-5}}\right) \ll n$, it follows that $\sum_{i~\in~V^{X/Y}~\cap~[-1,1]}~iv_i~=~0$ and thus that $v^X_{1}-v^X_{-1}=0$ and $v^X_{1}+v^X_0+v^X_{-1}=0$. Then, writing $\sum_{i \in V^X} |v_i|=t_X$ and $\sum_{i \in V^Y} |v_i|=t_Y$, we see that $v^X_1, v^X_{-1}=\pm t_X/4$, and $v^X_0=\mp t_X/2$, and that $v^Y_1, v^Y_{-1}= \pm t_Y/4$, and $v^Y_0= \mp t_Y/2$. Without loss of generality assume that $v^X_1, v^X_{-1}=-t_X/4$ and $v^X_0= t_X/2$. Then we can push all of the support onto $V^Y \cap [-1,1]$ by adding the following construction of four edges $t_X/4$ times (adding $O\left(p_Ln^{2.6 \times 10^{-5}}\right)$ edges to $\Phi$): $(0,-1,-1,1)$ and $(0,1,1,-1)$ with negative sign, and $(-1,0,-1,-1)$ and $(1,0,1,1)$ with positive sign (see Figure \ref{finish_arg}). In this way, no weight is added to $V^{X+Y} \cup V^{X-Y}$ and, at worst, rather than cancelling weight in $V^Y$, we add at most $t_X$ to the support in $V^Y \cap [-1,1]$ and have no support anywhere else in $\mathcal{T}$.

Suppose that the total size of the support on $V^Y$ is now $t'_Y$. By the nature of the edges added, it still follows that $v^Y_1, v^Y_{-1}= \pm t'_Y/4$, and $v^X_0= \mp t'_Y/2$. However, we know that, since ${\bf v} \in \mathcal{L}^Y_1(\mathcal{T})$, by Proposition \ref{zer_sum_Q}, $\sum_{i \in V^Y \cap [-1,1]} i^2v_i=0\mymod n$, and thus that $\sum_{i \in V^Y \cap [-1,1]} i^2v_i=0$. That is, $(-1)^2t'_Y/4+(0)^2t'_Y/2+(1)^2t'_Y/4=0$. It follows that $t'_Y=0$, completing the proof.
\end{proof}

\begin{figure}
\centering
\includegraphics[scale=0.7]{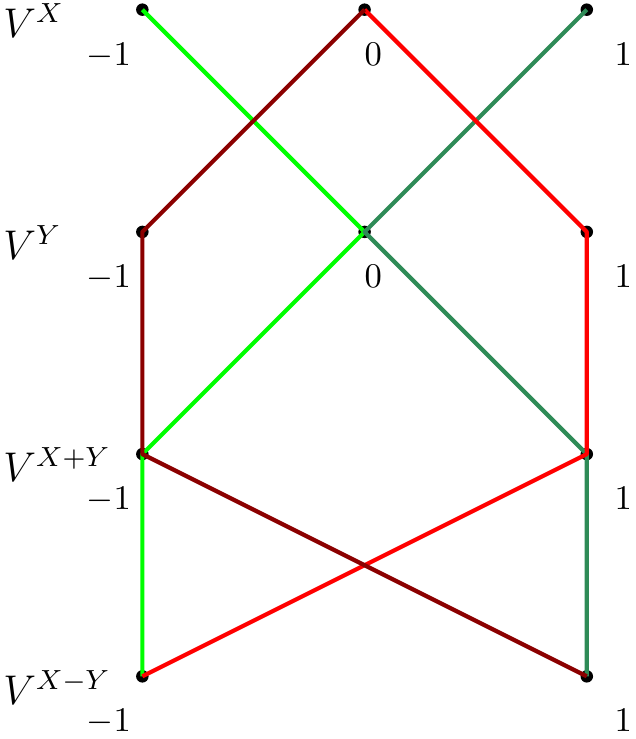}
\caption{Edges added to push support onto $V^Y$. Green edges are added with weight $+1$ and red edges are added with weight $-1$.} \label{finish_arg}
\end{figure}

Lemma \ref{cover_end} tells us that we have an integral decomposition $\Phi$ for $L^*$ with an upper bound on $|\Phi|$. More specifically it tells us that we can describe $L^*$ as the difference of the shadows of two (multi)-sets of edges, $\Phi^+$ and $\Phi^-$, each of size $O\left(p_Ln^{2.6 \times 10^{-5}}\right)$. We now wish to modify $\Phi$ so that $\Phi \in \{-1,0,1\}^{E(\mathcal{T})}$, $\partial \Phi - L^* ={\bf 0}$ (i.e. this property is not affected), but $\Phi^+$ and $\Phi^-$ are both matchings. We shall use zero-sum configurations to achieve this. By adding a zero-sum configuration to $\Phi$, we are adding four edges with positive weight, and four edges with negative weight, in total changing the support at any vertex by $-1+1=0$, therefore not affecting $\partial \Phi - L^*$. 

\begin{lemma} \label{t-match}
Let $\Phi$ be an integral decomposition for $L^*$ such that $\Phi$ consists of 
$O\left(p_Ln^{2.6 \times 10^{-5}}\right)$ edges and all edges are contained in the interval $I'_{n^{10^{-5}}}$. Then we can modify this to a decomposition $\Phi'$ which is the difference of two matchings, $M^+$ and $M^-$ using only $O\left(p_Ln^{2.6 \times 10^{-5}}\right)$ additional edges, such that $\Phi' \subseteq \mathcal{T}[I'_{n^{2.61 \times 10^{-5}}}]$.  
\end{lemma} 

\begin{proof}
We plan to use zero-sum configurations as follows: suppose a vertex $v$ is covered by more than one positive edge or more than one negative edge. Then either $\Phi$ contains the same number of positive and negative edges covering $v$, or $\Phi$ contains one more positive edge covering $v$ than the number of negative edges, and $v \in L^*$. Choose some vertex with more than two edges through it. Arbitrarily pair positive and negative edges together (leaving one positive edge unpaired if necessary). Then for each pair $(e^+, e^-)$, we can modify $\Phi$ by replacing $(e^+, e^-)$ instead by a set of six edges which have the same vertex shadow as $(e^+, e^-)$, and do not use any vertices which have already been used in $\partial \Phi$. In particular, we shall do this via zero-sum configurations as described in Section \ref{sec_zero_sum}, with the additional requirement that all free variables are within $I'_{n^{2.61 \times 10^{-5}}}$. We use such zero-sum configurations as follows: suppose, without loss of generality (we can argue similarly for vertices in any other part) that $e^+ \cap e^-=\{v_1\} \subseteq V^X$. Then set $a:=v_1$, $b+s = (e^-)_Y$ and $c+s = (e^+)_Y$. Since a zero-sum configuration has four degrees of freedom and we have fixed three choices, there are $\Theta(n^{2.61 \times 10^{-5}})$ different configurations $z$ we could complete this to, and any particular vertex $v \in z \setminus \{e^+ \cup e^-\}$ can only appear in at most three of these. The idea is that we choose the fourth degree of freedom so that all vertices other than those in $(e^+, e^-)$ are not covered by the edges of $\Phi$. Then, adding this set of edges to $\Phi$ cancels out $(e^+, e^-)$ and adds six new edges. Note that the vertex contained in both $e^+$ and $e^-$ now has two fewer edges through it (one positive and one negative), any other vertex in $e^+ \cup e^-$ has the same number of edges through it, and any other vertices get precisely one positive and one negative edge through them where they previously were not contained in any edges. Thus we have made progress towards expressing $L^*$ as the difference of two matchings. If after each step there is still always a choice of new zero-sum configuration as above for any pair $(e^+, e^-)$ in the updates $\Phi$, then indeed we can continue until we have modified $\Phi$ to be the difference of two matchings. 

More specifically, given a pair of edges $(a, e_2^-, e_3^-, e_4^-)$ and $(a, e_2^+, e_3^+, e_4^+)$, we first consider any $s \in [-\frac{n^{2.61 \times 10^{-5}}}{4}, \frac{n^{2.61 \times 10^{-5}}}{4}]$ such that 
\begin{enumerate}
\item $e_2^--s \neq a$,
\item $e_2^+-s \neq a$, and
\item $e_2^-+ e_2^+-2s \neq a$.
\end{enumerate}
That is, there are at most $3$ choices for $s$ forbidden by the edges already chosen for the configuration. Furthermore, we need the choice of $s$ so that we are not hitting any other vertex already covered by $\Phi$. For every such vertex, there are at most four choices of $s$ that could result in it being contained in the above configuration. Initially we have that all edges are contained in the interval $I'_{n^{10^{-5}}}$ and that there are $O\left(p_Ln^{2.6 \times 10^{-5}}\right)$ edges. Hence there are $O\left(p_Ln^{2.6 \times 10^{-5}}\right)$ edge pairs that need cancelling in the way described. In the first edge covering there is a constant $c$ such that we need to avoid at most $cn^{10^{-5}}$ vertices with the choice of $s$. Hence we have at least $\Theta(n^{2.61 \times 10^{-5}})-4cn^{10^{-5}}$ choices for $s$. When we have eliminated $i$ such edge pairs, there are at most $cn^{10^{-5}} + 16i$ vertices to avoid, so we have at least $\Theta(n^{2.61 \times 10^{-5}})-4(cn^{10^{-5}} + 16i)$ choices for $s$. In particular, since there are $O\left(p_Ln^{2.6 \times 10^{-5}}\right)$ pairs to deal with, there is always an available choice for $s$, and more specifically, for the last edge pair there are at least $\Theta(n^{2.61 \times 10^{-5}})-O\left(p_Ln^{2.6 \times 10^{-5}}\right)$ choices for $s$.
\end{proof}

Note that this allows us to describe $L^*$ as the difference of two matchings $M^+$ and $M^-$ such that $|M^{\pm}|=O\left(p_Ln^{2.6 \times 10^{-5}}\right)$, and all vertices covered by the matchings are within the interval $I'_{n^{2.61 \times 10^{-5}}}$. 

\section{Building and using the absorber}\label{sec_using_abs}

In this section we build an absorber $A^* \supseteq A$ that is an absorber for any qualifying leave $L^*$.
Let $A \subseteq I'_{n^{10^{-4}}}$ so that every vertex in the interval is included independently with probability $p_A$ as defined in Definition \ref{def_const}. 
Let $A' := A \cap I'_{n^{2.7 \times 10^{-5}}}$. The initial step to showing this is to prove the following:

\begin{lemma} \label{lemma_A_cover}
Given $A$ as above and qualifying leave $L^*$, whp there exists $\Phi_l \in \{-1,0,1\}^{E(\mathcal{T})}$ such that $\partial\Phi_l-L^*:={\bf 0}$ and $\partial \Phi_l^- \subseteq A'$.
\end{lemma}

That is, we show that given $A$ as above, whp we can describe $L^*$ as the vertex shadow of two matchings $M_A^+$ and $M_A^-$ such that $\bigcup M^-_A \subseteq A'$. We'll do this as follows: let $\Phi$ be a signed multi-set of edges obtained from covering the leave $L^*$ such that every edge is contained in the interval $I'_{n^{10^{-5}}}$, and $|\Phi|=O\left(p_Ln^{2.6 \times 10^{-5}}\right)$ as per Lemma \ref{cover_end}. By Lemma \ref{t-match}, we obtain $M^+$ and $M^-$, two matchings in the interval $I'_{n^{2.61 \times 10^{-5}}}$, such that $\partial(M^+-M^-)=L^*$, and $|M^-|=O\left(p_Ln^{2.6 \times 10^{-5}}\right)$. Let $B=\{v_1, \ldots, v_l\}$ be an enumeration of the vertices in $\bigcup M^-$, such that vertices are enumerated according to part in the order $V^{X-Y}, V^{X+Y}, V^Y, V^X$. We know that $B \subseteq I'_{n^{2.61 \times 10^{-5}}}$ and that $|B|=O\left(p_Ln^{2.6 \times 10^{-5}}\right)$. For each $v_i \in B$ there is a pair of oppositely signed edges $(e^+_i, e^-_i) \in M^+ \times M^-$ such that $e_i^+ \cap e_i^-=\{v_i\}$. Given such a pair $(e^+_i, e^-_i)$, let $Z_{(e^+_i, e^-_i)}$ be the collection of zero-sum configurations which contain $e^+_i$ with negative sign, $e^-_i$ with positive sign, and additionally satisfy the following: if $v_i \in V^{X-Y}$ then all vertices are contained in $I'_{n^{2.7 \times 10^{-5}}/64}$, if $v_i \in V^{X+Y}$ then all vertices are contained in $I'_{n^{2.7 \times 10^{-5}}/16}$, if $v_i \in V^{Y}$ then all vertices are contained in $I'_{n^{2.7 \times 10^{-5}}/4}$, and if $v_i \in V^{X}$ then all vertices are contained in $I'_{n^{2.7 \times 10^{-5}}}$.\footnote{We do this so that we don't have to worry about zero-sum configurations `pushing out' rather than in, which reduces case analysis.}

Let $\Phi_1 \in \{-1,0,1\}^{E(\mathcal{T})}$ be the vector of edges such that $\Phi_1^+=M^+$ and $\Phi_1^-=M^-$ We update $M^-$ to $M^-_A$ via the following algorithm:
\begin{alg} \label{alg_A_cover}
~

$i=1$

{\bf Input:} $(e_i^+, e_i^-)$, $\Phi_i$.

{\bf Step 1:} Enumerate the number of zero-sum configurations $Z^*_{(e_i^+, e_i^-)}$ which are in $Z_{(e^+_i, e^-_i)}$ and additionally use only vertices in $A'$, other than those in $e^+_i \cup e^-_i$, such that all vertices used are distinct from $\partial \Phi_i^+$.

{\bf Step 2:} If $|Z^*_{(e_i^+, e_i^-)}| = 0$, abort. Else, uniformly at random assign one of the $|Z^*_{(e_i^+, e_i^-)}|$ such zero-sum configurations $z_{(e_i^+, e_i^-)}$ to $(e_i^+, e_i^-)$.

{\bf Step 3:} If $i = l$ stop. Else, let $\Phi_{i+1}=\Phi_i \cup z_{(e_i^+, e_i^-)}$ and for each $j>i$ update $(e_j^+, e_j^-)$ according to $\Phi_{i+1}$. Increase $i$ by $1$ and go to Step 1.

\end{alg}

Note that in running Algorithm \ref{alg_A_cover}, the pairs $(e_i^+, e_i^-)$ for which we want to find a suitable zero-sum configuration are constantly updating, depending on previous choices. More specifically, when $z_{(e_i^+, e_i^-)}$ is added to $\Phi_i$, we are effectively deleting $e_i^+$ and $e_i^-$ from $\Phi_i$ and replacing them with six new edges which do not affect the vertex shadow of $\Phi_i$. Now, given that $e_i^-$ contains a vertex $v_j$ that occurs later in the enumeration $\{v_1, \ldots, v_l\}$, by adding $z_{(e_i^+, e_i^-)}$ to $\Phi_i$ to obtain $\Phi_{i+1}$ we have that $e_j^-$ (which was previous equal to $e_i^-$), is now given by the edge in $z^-_{(e_i^+, e_i^-)}$ which contains $v_j$.

Furthermore, note that proving that Algorithm \ref{alg_A_cover} does not abort prematurely suffices to prove Lemma \ref{lemma_A_cover}. Indeed, assuming that Algorithm \ref{alg_A_cover} does not abort prematurely, let $M^{\pm}_A:=\Phi^{\pm}_l$. The algorithm ensures that we replace every vertex in $B$ by nine vertices in $A'$, and in this process we don't add any vertices which are outside $A'$. Furthermore, we add these in such a way that $\partial \Phi_l - L^*={\bf 0}$ and $\Phi_l \in \{-1,0,1\}^{E(\mathcal{T})}$, so that we are indeed describing $L^*$ as the difference of two matchings $M^+_A$ and $M^-_A$ such that $\bigcup M^-_A \subseteq A'$. 

\begin{lemma} \label{lemma_A_cover_alg}
With high probability Algorithm \ref{alg_A_cover} does not abort prematurely.
\end{lemma}

\begin{proof}
Rather than having to concern ourselves directly with the dynamic change of the collection of pairs $(e_i^+, e_i^-)$ at every iteration of the algorithm, we show that with high probability every possible pair $(e_i^+, e_i^-)$ has sufficiently many zero-sum configurations in $Z^*_{(e_i^+, e_i^-)}$ such that Algorithm \ref{alg_A_cover} does not abort.

We have that $e_i^+$ is fixed for every vertex $v_i \in B$ (and does not get updated by Algorithm \ref{alg_A_cover} until step $i$, after which we are no longer concerned with the pair $(e_i^+, e_i^-)$). Furthermore, by the enumeration order, $e_i^-$ will not be updated by the algorithm until after step $i$ for every $v_i \in V^{X-Y}$. For every pair $(v_i, e_i^+)$ to be considered such that $v_i \notin V^{X-Y}$, let $\{e_{i,1}^-, \ldots, e_{i,m_i}^-\}$ be an enumeration of the edges in $I'_{n^{2.7 \times 10^{-5}}/s}$ (excluding $e_i^+$) which contain $v_i$, where $s \in \{1,4,16\}$ is given by whether $v_i$ is in $V^X$, $V^Y$ or $V^{X+Y}$, respectively. We have that $m_i=\Theta(n^{2.7 \times 10^{-5}})$ for every $i \in [l]$, and that by construction, as Algorithm \ref{alg_A_cover} runs, the pair $(e_i^+, e_i^-)$ for which we wish to add a zero-sum configuration to $\Phi_i$ must be one of $(e_i^+, e_{i,j}^-)$ for some $j \in [m_i]$. We shall estimate $|Z_{(e_i^+, e_{i,j}^-)}|$ for every $i \in [l]$ and every $j \in [m_i]$. Note that there are $O(p_Ln^{5.3 \times 10^{-5}})$ such pairs $(e_i^+, e_{i,j}^-)$ to consider.

Let $Z^A_{(e^+, e^-)} \subseteq Z_{(e^+, e^-)}$ be the collection of zero-sum configurations in $Z_{(e^+, e^-)}$ that additionally use only vertices in $A'$, other than those in $e^+ \cup e^-$. We have that $|Z_{(e^+, e^-)}|=\Theta(n^{2.7 \times 10^{-5}})$ (there are this many choices for the remaining degree of freedom), and $\mathbb{E}(|Z^A_{(e^+, e^-)}|)=p_A^9|Z_{(e^+, e^-)}|$.

Since every vertex in $I'_{n^{2.7 \times 10^{-5}}}$ is in $A'$ independently with probability $p_A$, we may view $|Z^A_{(e^+, e^-)}|$ as a function of independent Bernoulli random variables. Furthermore, whether a vertex in $I'_{n^{2.7 \times 10^{-5}}}$ is in $A'$ or not, affects $|Z^A_{(e^+, e^-)}|$ by at most three (- given $(e^+, e^-)$ fixed, there at most three positions another vertex $v$ could possibly take in a zero-sum configuration, and fixing the vertex in position either dictates exactly one zero-sum configuration, or none, if it causes inconsistency in the equations). Thus we may use McDiarmid's Inequality (Corollary \ref{colin}) with $2\sum c_i^2=\Theta(n^{2.7 \times 10^{-5}})$. Hence, taking $t=n^{1.4 \times 10^{-5}}$ we have that 
$$\mathbb{P}(||Z^A_{(e^+, e^-)}|-\mathbb{E}(|Z^A_{(e^+, e^-)}|)|\geq n^{1.4 \times 10^{-5}})\leq 2\exp\left(-\Omega(n^{10^{-6}})\right).$$
Taking union bounds we have that with high probability 
$$|Z^A_{(e^+, e^-)}|=\mathbb{E}(|Z^A_{(e^+, e^-)}|) \pm n^{1.4 \times 10^{-5}}=\Theta(p_A^9n^{2.7 \times 10^{-5}}) \pm n^{1.4 \times 10^{-5}}$$ 
for every pair $(e^+, e^-)$.

Now, suppose that $j>i$ and consider the affect to $|Z^*_{(e^+_j, e^-_j)}|$ as a result of the choice for $z_{(e^+_i, e^-_i)}$. Such a choice forbids only $O(1)$ subsequent choices for any feasible pair $(e^+_j, e^-_j)$. There are $O(p_Ln^{2.6 \times 10^{-5}})$ pairs to consider in the running of Algorithm \ref{alg_A_cover} and we initially have $\Theta(p_A^9n^{2.7 \times 10^{-5}}) \gg p_Ln^{2.6 \times 10^{-5}}$ possible configurations for each pair $(e^+, e^-)$. Hence, since each choice forbids only $O(1)$ subsequent choices, the algorithm will not abort.
\end{proof}

Now let $T'$ be a fixed matching covering the vertices of $A'$ so that $\bigcup T' \setminus A' \subseteq A$. 

\begin{prop} \label{prop_t'}
With high probability such a matching $T'$ exists.
\end{prop}

\begin{proof}
First note that (by Chernoff bounds), with high probability we have that $|A'|=\Theta(p_An^{2.7 \times 10^{-5}})$. Now for each $v \in A'$, let $d_A(v)$ be the degree of $v$ in $A$. Then $\mathbb{E}(d_A(v))=p_A^3d_{I'_{n^{10^{-4}}}}(v)=\Theta(p_A^3n^{10^{-4}})$. Thus by Chernoff and union bounds, with high probability $d_A(v)=\Theta(p_A^3n^{10^{-4}})$ for every $v \in A'$. Since $p_A^3n^{10^{-4}} \gg p_An^{2.7 \times 10^{-5}}$ with high probability we may greedily find such a matching $T'$ for $A'$, as required.
\end{proof}

We call $T'$ the {\it $A'$-template}. Before describing how we use $T'$, we introduce {\it cascades}, the notion of which comes from~\cite{designs1}.

\subsection{Cascades} \label{sec_cascades}

Here we describe a general {\it cascade} in terms of zero-sum configurations. A cascade is a gadget comprising of zero-sum configurations for a fixed tuple of five edges $e_T:=(e, T_1, T_2, T_3, T_4)$ such that $e \cap T_i = \{e_i\}$ for every $i \in [4]$, where $e=(e_1, e_2, e_3, e_4)$, and $|e \cup \bigcup_{i \in [4]} T_i|=16$, so there are no other intersections between edges. Informally, we can think of $T_1, T_2, T_3, T_4$ as edges chosen to cover the vertices of $e$, chosen so that they are disjoint from one another. Then a {\it cascade} for $e_T$, is a collection $C_{e_T}$ of zero-sum configurations as follows. Let $Z'$ be a zero-sum configuration containing $e$ and no other vertices from $\bigcup_{i \in [4]}T_i$. Let $S_i$ be the edge of $Z' \setminus \{e\}$ which intersects $e_i$ (i.e. the edge in the matching of opposite sign to $e$ which contains $e_i$). Then for each pair $(S_i, T_i)$ we also build a zero-sum configuration, $Z_i$, ensuring that all additional vertices for $Z_i$ have not already been used. Then our cascade $C_{e_T}$ is the graph induced on this collection of zero-sum configurations. As well as considering $C_{e_T}$ as a subgraph, we also associate it with the quintuple of zero-sum configurations $C_{e_T}:=(Z', Z_1, Z_2, Z_3, Z_4)$ which make it. We write $\mathcal{C}_{e_T}$ for the collection of cascades for $e_T$. 
Note that a cascade $C_{e_T}$ consists of $64$ vertices and induces two distinct perfect matchings each consisting of $16$ edges; one which uses edge $e$, and one which uses edges $T_1, T_2, T_3, T_4$.

\subsection{Using $T'$} 

We now show that whp we can obtain a large family of cascades in $A$ for every quintuple $e_T:=(e, T_1, T_2, T_3, T_4)$, such that $e \in E(\mathcal{T}[A'])$, $e \notin T'$, $T_i \in T'$ for every $i \in [4]$, and $T_i \cap e = \{e_i\}$, where $e=(e_1, e_2, e_3, e_4)$. Since $M^-_A \subseteq E(\mathcal{T}[A'])$, this builds cascades for every $e \in M^-_A$. Given $e \neq e'$ such that $e_T:=(e, T_1, T_2, T_3, T_4)$ and $e_T':=(e', T_1', T_2', T_3', T_4')$ as above, we define the two cascades $C_{e_T}$ and $C_{e'_T}$ as {\it almost-disjoint} if the following all hold: $V(C_{e_T}) \cap V(C_{e'_T}) \subseteq A'$, and given a vertex $v \in V(C_{e_T}) \cap V(C_{e'_T})$ there exists $f \in T'$ such that $v \in f$ and $f=T_i=T_j'$ for some $i, j \in [4]$. That is, $C_{e_T}$ and $C_{e'_T}$ only share vertices which are part of edges in the $A'$-template and, given a vertex $v$ in the $A'$-template is used, the edge in $T'$ containing $v$ is in both $C_{e_T}$ and $C_{e'_T}$, precisely acting as one of the edges in the quintuple $e_T$ for $e$ and $e_T'$ for $e'$ respectively.
Note that this implies that for $e, e'$ such that $e \cap e' = \emptyset$ the cascades $C_{e_T}$ and $C_{e'_T}$ are almost-disjoint if and only if they are disjoint in the usual sense, and if $e \cap e' = u$ then they intersect in precisely the edge in $T'$ which contains $u$. By the pair degree condition on $\mathcal{T}$ this covers all cases. (Even when $n$ is even, we are only considering cascades which cannot contain wrap-around edges and the maximum pair degree when ignoring wrap-around edges is $1$ for every $n$.) 

\begin{lemma} \label{lemma_cascade}
Given $A$, $A'$ and $T'$ as above, for every edge $e \in \mathcal{T}[A']\setminus T'$, with high probability there exists a cascade $C_{e_T}$ such that $V(C_{e_T}) \subseteq A$, and for any two edges $e \neq e'$, cascades $C_{e_T}$ and $C_{e'_T}$ are almost-disjoint. 
\end{lemma}

\begin{proof}
The proof follows a similar strategy as the proof of Lemma \ref{lemma_A_cover}. For every quintuple $e_T$ defined from an edge $e \in  \mathcal{T}[A']\setminus T'$, we wish to find a cascade on the vertices of $A$ such that every other vertex in the cascade is in $A\setminus \bigcup T'$ and is almost-disjoint from all other cascades chosen for any $e' \in E(\mathcal{T}[A'])$. Since $A' \subseteq I'_{n^{2.7 \times 10^{-5}}}$, certainly there are at most $O(n^{5.4 \times 10^{-5}})$ edges for which we need to build a cascade.

Recalling the definition of a cascade in Section \ref{sec_cascades}, given $e=(e_1, e_2, e_3, e_4)$, we have two free choices to define a zero-sum configuration $Z'$ covering $e$. We shall wish to make these choices so that any vertices used are in $A \setminus \bigcup T'$, and not yet used in the collection of cascades, $\mathcal{C}$, which have been created in this process. For each pair $(S_i, T_i)$ a cascade contains a zero-sum configuration, $Z_i$, such that all vertices in $Z_i$ apart from those in $(S_i, T_i)$ have not previously been used in the process. For each $i \in [4]$ there is one degree of freedom to choose $Z_i$ for $(S_i, T_i)$. So in total we have six free choices we can make to build our cascade $C_{e_T}=(Z', Z_1, Z_2, Z_3, Z_4)$ for $e$. Each of these choices must dictate edges only on vertices in $A \subseteq I'_{n^{10^{-4}}}$. In particular, this means that the number of choices for each degree of freedom is at most $n_c:=2n^{10^{-4}}$. Furthermore, for each of these free choices and the edges they subsequently dictate, we wish to avoid introducing vertices in $\bigcup T'$, and any other cascade choices already made. This gives $O(n^{2.7 \times 10^{-5}}+i)$ vertices to avoid, and hence $O((n^{2.7 \times 10^{-5}}+i)n_c^5)$ cascades to avoid, when building the $(i+1)^{th}$ cascade.  Let $C_{e_T}$ be the family of cascades available for $e_T$ in $I'_{n^{10^{-4}}}$, and let $C^A_{e_T}$ be the family of cascades available in $A$ for $e_T$. Then we have that $|C_{e_T}|=\Theta(n^{6 \times 10^{-4}})$ and 
$\mathbb{E}(|C^A_{e_T}|)=p_A^{48}|C_{e_T}|=\Theta(p_A^{48}n^{6 \times 10^{-4}})$.
By McDiarmid's bounded differences inequality, whp, $|C^A_{e_T}|= \Theta(p_A^{48}n^{6 \times 10^{-4}})$ for every $e_T$. Indeed, this follows by noting, as for $|Z_{(e^+, e^-)}|$ in Lemma \ref{lemma_A_cover}, that we may consider $|C^A_{e_T}|$ as a function of independent Bernoulli random variables. In this case, whether a vertex is in $A$ or not may affect $|C^A_{e_T}|$ by $O(n^{5 \times 10^{-4}})$ (- since fixing it, there are five free variables remaining to dictate the cascade).

Choosing a cascade greedily, one by one for each of the $O(n^{5.4 \times 10^{-5}})$ edges for which we wish to build a cascade, every edge $e$ (and related quintuple $e_T$) has a choice of at least $\Theta(p_A^{48}n^{6 \times 10^{-4}}) - O(n^{5.4 \times 10^{-5}}n_c^5)= \Theta(p_A^{48}n^{6 \times 10^{-4}}) - O(n^{5.54 \times 10^{-4}})$ cascades which are almost-disjoint from any previous choices.
\end{proof}

In what follows we'll show that with high probability, as well as $A$ being such that Lemmas \ref{lemma_A_cover}, \ref{prop_t'} and \ref{lemma_cascade} are all satisfied simultaneously, we can find $A^* \supseteq A$ such that $A^*$ is an absorber for any possible leave $L$ satisfying the conditions noted at the beginning of the chapter, and $\mathcal{T}^{-A^*}:=\mathcal{T}[V(\mathcal{T})\setminus A^*]$ has `nice' properties that leave us in a good position to continue with the random greedy count that follows. In particular, by union bounding we'll be able to show that Lemmas \ref{lemma_A_cover}, \ref{prop_t'} and \ref{lemma_cascade} are all satisfied simultaneously, and then fixing a collection of almost-disjoint cascades for each of the possible edges $e \in \mathcal{T}[A']\setminus T'$ (now that $A$ is fixed), we can extend $A$ to a collection of vertices $A^*$ that has a perfect matching by considering all vertices remaining in $A$ which have not been assigned to any of the almost-disjoint cascades, and covering these vertices by a matching avoiding all other vertices in $A$. We'll show that $|A|$ is small enough that we can do this greedily, without having too much of an adverse effect on any property we wish to maintain, and by nature of $A$ being picked in a uniformly random way, $\mathcal{T}^{-A^*}$ is well structured.

\begin{theo}\label{thm_absorber}
There exists a set $A^* \supseteq A$ such that $A^*$ is an absorber for any qualifying leave $L^*$, and $\mathcal{T}^{-A^*}$ satisfies the following:
\begin{enumerate}[(i)]
\item every $\mathcal{T}$-valid subset $S \subseteq V(\mathcal{T})$ satisfies 
$$|V(\mathcal{T}^{-A^*}[S])|=(1 \pm O(p_A))|S|,$$
\item for every $v \in V(\mathcal{T}^{-A^*})$ and every open or closed $\mathcal{T}$-valid tuple $(v,S_1,S_2,S_3)$, we have
$$|E_{\mathcal{T}^{-A^*}}(v,S_1, S_2, S_3)|=(1 \pm O(p_A))|E_{\mathcal{T}}(v,S_1, S_2, S_3)|,$$
\item for every $i \in [c_g]$, 
$$|\mathcal{Z}^+_{i,e,\mathcal{T}^{-A^*}}(\alpha, \beta, \gamma)|:=
\begin{cases}
(1 \pm O(p_A))|\mathcal{Z}_{i,e,\mathcal{T}}^{+}(\alpha, \beta, \gamma)| & \mbox{~if $e$ is a bad edge}, \\
O\left(k_it_1 \right) & \mbox{~if $\alpha=0$, $\beta=1$, $\gamma=3$},\\
0 & \mbox{~otherwise}.\\
\end{cases}
$$
$$|\mathcal{Z}^-_{i,e,\mathcal{T}^{-A^*}}(\alpha, \beta, \gamma)|:=
\begin{cases}
(1 \pm O(p_A))|\mathcal{Z}_{i,e,\mathcal{T}}^{-}(\alpha, \beta, \gamma)| & \mbox{~if $\alpha \neq 0$ and $\gamma=0$}, \\
O\left(j_ik_i \right) & \mbox{~if $\alpha=0$, $\beta=0$, $\gamma=4$},\\
O\left(j_ik_i \right) & \mbox{~if $\alpha=0$, $\beta=1$, $\gamma=3$},\\
0 & \mbox{~otherwise}.\\
\end{cases}
$$
Finally, for every bad edge $e$,
$$|\mathcal{Z}^2_{i,e,\mathcal{T}^{-A^*}}|=O\left(k_it_1\right).$$
\item $|V_{O}^{X+Y}(\mathcal{T}^{-A^*})|=|V_{O}^{X-Y}(\mathcal{T}^{-A^*})|$, $|V_{E}^{X+Y}(\mathcal{T}^{-A^*})|=|V_{E}^{X-Y}(\mathcal{T}^{-A^*})|$ and, furthermore $|V_O^{J}(\mathcal{T}^{-A^*})|=(1 \pm O(p_A))|V_E^{J}(\mathcal{T}^{-A^*})|$ for every $J \in \{X,Y,X+Y,X-Y\}$. 

Additionally, $|V_{O/E}^{J_1}(\mathcal{T}^{-A^*}[S])|=(1 \pm O(p_A))|V_{O/E}^{J_2}(\mathcal{T}^{-A^*}[S])|$ for every valid layer interval $S$, and $J_1,J_2 \in \{X,Y,X+Y,X-Y\}$.
\item For every $J \in \{X,Y,X+Y,X-Y\}$ and every valid $J$-layer interval $I^J$ and $v \notin J$,
$$|E_{\mathcal{T}^{-A^*}}(v, I^J, O/E)|=(1 \pm O(p_A))|E_\mathcal{T}(v, I^J, O/E)|.$$
\item $|\mathcal{T}^{-A^*}|=(1 \pm O(p_A))n^2.$
\end{enumerate}
\end{theo}
\begin{proof}
We start by showing that whp $A$ is such that $\mathcal{T}^{-A}$ satisfies all of the statements claimed for $\mathcal{T}^{-A^*}$, with the exception that we shall not show that $|V_{O}^{X+Y}(\mathcal{T}^{-A})|=|V_{O}^{X-Y}(\mathcal{T}^{-A})|$, nor the final claim on the number of edges in $\mathcal{T}^{-A^*}$ (since this will be easy to show directly in $\mathcal{T}^{-A^*}$). Additionally we'll upper bound $|A|$ whp. We'll then fix $A$ so that all of these hold, alongside the statements of Lemma \ref{lemma_A_cover}, Proposition \ref{prop_t'} and Lemma \ref{lemma_cascade}, which will be possible by a union bound. 

When looking at the change from $\mathcal{T}$ to $\mathcal{T}^{-A}$ for any property considered above, first note that the property is only affected if it involves consideration of vertices inside $I'_{n^{10^{-4}}}$. It follows from Chernoff bounds that whp $|A|=(1\pm o(1))p_A|I'_{n^{10^{-4}}}|$. Furthermore, for any subset $S \subseteq I'_{n^{10^{-4}}}$ with $|S|=\Omega(n^{10^{-5}})$ whp we have by Chernoff bounds that $|S \cap A|=(1 \pm o(1))p_A|S|$, and so then for any set $S \subseteq V(\mathcal{T})$ (still with $|S|=\Omega(n^{10^{-5}})$) we have that $|V(\mathcal{T}[S])|\geq |V(\mathcal{T}^{-A}[S])|=|S|-|S \cap A|\geq(1 \pm 2p_A)|S|$. Furthermore, $|V^J_O(\mathcal{T})|=t_0 \pm 1$, and $|V^J_E(\mathcal{T})|=t_0 \pm 1$ for every $J \in \{X,Y,X+Y,X-Y\}$. Thus in (i) and (iv) where we are considering polynomially many subsets $S \subseteq V(\mathcal{T})$, by union bounds with high probability all the statements of (i) and (iv) (apart from that $|V_{O}^{X+Y}(\mathcal{T}^{-A})|=|V_{O}^{X-Y}(\mathcal{T}^{-A})|$) hold with $A$ in place of $A^*$.

Considering degree-type properties, fix a vertex $v$ and a tuple $(S_1, S_2, S_3)$ such that we are interested in $E_{\mathcal{T}^{-A}}(v, S_1, S_2, S_3)$. (By abuse of notation, we let this include $J$-layer intervals as per (v).) Recall from Definitions \ref{def_valid} and \ref{def_layer} that we have $|E_{\mathcal{T}}(v, S_1, S_2, S_3)|=\Theta(|S_1|)$. Let $v_1, v_2, \ldots, v_{\chi}$ be an enumeration of all vertices in $(I'_{n^{10^{-4}}}\setminus \{v\}) \cap (S_1 \cup S_2 \cup S_3)$. Using the vertex exposure martingale, let $X_j$ be the expected number of edges from $E_{\mathcal{T}}(v, S_1, S_2, S_3)$ which are known to not be in $E_{\mathcal{T}^{-A}}(v, S_1, S_2, S_3)$ as a result of revealing vertices $v_1, \ldots, v_j$. We have that $|X_j-X_{j-1}|\leq 2$ for every $j$, since the pair degree of $v$ and $v_i$ is at most $2$ for every $i \in [\chi]$, and $|X_j-X_{j-1}|=0$ if there is no edge $e \supseteq \{v, v_j\}$ such that $e \in E_{\mathcal{T}}(v, S_1, S_2, S_3)$. Noting that $X_0=O(p_A|E_{\mathcal{T}}(v, S_1, S_2, S_3)|)$, by Azuma-Hoeffding we have that
$$\mathbb{P}(|X_{\chi}-X_0|\geq p_A|S_1|) \leq 2 \exp\left(-\frac{p_A^2|S_1|^2}{\Theta(|S_1|)}\right).$$
Since $p_A^2|S_1| \gg 1$ for every $S_1$ of interest, we have that whp 
$$|E_{\mathcal{T}^{-A}}(v, S_1, S_2, S_3)|=|E_{\mathcal{T}}(v, S_1, S_2, S_3)|\pm O(p_A|S_1|),$$
where since $|E_{\mathcal{T}}(v, S_1, S_2, S_3)|=\Theta(|S_1|)$ it follows that whp
$$|E_{\mathcal{T}^{-A}}(v, S_1, S_2, S_3)|=(1 \pm O(p_A))|E_{\mathcal{T}}(v, S_1, S_2, S_3)|$$
for any tuple $(v, S_1, S_2, S_3)$ as in (ii) or (v). Hence, by union bounds, we have that (ii) and (v) hold for $A$ in place of $A^*$. It remains to consider the case for zero-sum configurations. 

First note that the only cases that are non-trivial are when $e$ is a bad edge, or $\alpha \neq 0$ and $\gamma=0$. Otherwise, we only reduced numbers of configurations and so by Fact \ref{fact_Z} the statements hold.
Considering the two remaining cases, we start by proving that for every bad edge $e$, $|\mathcal{Z}^+_{i,e,\mathcal{T}^{-A}}(\bad)|=(1 \pm O(p_A))|\mathcal{Z}_{i,e,\mathcal{T}}^{+}(\bad)|$. Fix an $i$-bad edge $e$ and let $v_1, v_2, \ldots, v_{\chi'}$ be an enumeration of all vertices in $I'_{n^{10^{-4}}}\setminus e$, ordered by the modulus of the coordinate they are indexed by, splitting ties arbitrarily. Let $Y_j$ be the expected number of $i$-legal zero-sum configurations containing $e$ which are known to not be in $\mathcal{T}^{-A}$ as a result of revealing vertices $v_1, \ldots, v_j$. Then note that for a vertex $v_j \in J_i$ we have that $|Y_j-Y_{j-1}|=O(t_1)$, since the number of $i$-legal zero-sum configurations containing $e$ and vertex $v_j$ leave one degree of freedom which can take $O(t_1)$ values. However, for a vertex $v_j \in I'_{n^{10^{-4}}} \setminus J_i$ we find that the remaining degree of freedom to dictate an $i$-legal zero-sum configuration containing $e$ and $v_j$ is of $O(j_i)$ (and there are $O(n^{10^{-4}})$ such vertices to consider). Furthermore, we have that $Y_0=O(p_A|\mathcal{Z}_{i,e,\mathcal{T}}^{+}(\bad)|)$. Thus by Azuma-Hoeffding we have that
$$\mathbb{P}(|Y_{\chi'}-Y_0|\geq p_Aj_it_1) \leq 2 \exp\left(-\frac{p_A^2j_i^2t_1^2}{\Theta(j_it_1^2)}\right).$$
Since $p_A^2j_i \gg 1$ for every $i \in [c_g]$, we have, using Fact \ref{fact_Z}, that whp 
$$|\mathcal{Z}^+_{i,e,\mathcal{T}^{-A}}(\bad)|=|\mathcal{Z}^+_{i,e,\mathcal{T}}(\bad)|\pm O(p_Aj_it_1),$$
or equivalently that whp
$$|\mathcal{Z}^+_{i,e,\mathcal{T}^{-A}}(\bad)|=(1 \pm O(p_A))|\mathcal{Z}^+_{i,e,\mathcal{T}}(\bad)|$$
for any $i$-bad edge $e$ and every $i \in [c_g]$. 

Considering now $|\mathcal{Z}^-_{i,e,\mathcal{T}^{-A}}(\alpha, \beta, 0)|$ with $\alpha \neq 0$, we have that the statement holds via a greedy argument. In particular, letting $e$ be a fixed edge of type $(\alpha, \beta, 0)_i$ with $\alpha \neq 0$ and consider the $i$-legal zero-sum configurations containing $e$, note that all other vertices in the configuration have index $\Omega(j_i)$. 
If $j_i \gg n^{10^{-4}}$ this implies that all $i$-legal zero-sum configurations containing $e$ survive the removal of $A$ from $\mathcal{T}$. If $j_i =O( n^{10^{-4}})$ this is not the case, however the number of configurations lost is small. In particular, we have two degrees of freedom to define a zero-sum configuration containing $e$, each of which can take values in a $\Theta(t_1)$ sized range, and vertices outside $e$ will only fall in $I'_{n^{10^{-4}}}$ when (at least) one of these degrees of freedom is chosen in a particular $O(n^{10^{-4}})$ range. Thus, we could only lose at most $O(n^{10^{-4}}t_1)$ of the $i$-legal configurations containing $e$ moving from $\mathcal{T}$ to $\mathcal{T}^{-A}$. So we find that $|\mathcal{Z}^-_{i,e,\mathcal{T}^{-A}}(\alpha, \beta, 0)|=|\mathcal{Z}^-_{i,e,\mathcal{T}}(\alpha, \beta, 0)| \pm O(n^{10^{-4}}t_1)$, and since $|\mathcal{Z}^-_{i,e,\mathcal{T}}(\alpha, \beta, 0)|=\Theta(t_1^2)$, this gives
$$|\mathcal{Z}^-_{i,e,\mathcal{T}^{-A}}(\alpha, \beta, 0)|=(1 \pm O(n^{-0.999}))|\mathcal{Z}^-_{i,e,\mathcal{T}}(\alpha, \beta, 0)|,$$
which gives the desired result since $p_A \gg n^{-0.999}$.

Taking union bounds, we find that with high probability all of the above events hold for $A$. We fix $A$ accordingly. This yields $A \subseteq I'_{n^{10^{-4}}}$ and a template matching $T'$ covering all vertices in $A'=A \cap I'_{n^{2.7 \times 10^{-5}}}$, such that every possible edge $e \in \mathcal{T}[A'] \setminus T'$ has a cascade such that these cascades are almost-disjoint. For each such cascade, $C'_{e_T}$, taking the edges of $C'_{e_T}$ which are in $T'$, we obtain a perfect matching for $C'_{e_T}$ by additionally taking the other edges of the same sign as $T_i$ in each of $Z_i$. Denote this matching by $M_{e_T}$. Then $\bigcup_{e_T} M_{e_T}$ is a perfect matching for $\bigcup_{e_T} C'_{e_T}$ since, by construction, every pair of matchings $M_{e_T}$ and $M_{e'_T}$ are either on entirely disjoint vertex sets, or their vertex sets intersect only in vertices of $T'$, in which case these vertices are only used in the matchings precisely as the edges in $T'$. Let $T:=\bigcup_{e_T} M_{e_T}$. Now note that each of these cascades $C'_{e_T}$ also has a perfect matching containing $e_T$. Indeed, this matching uses the other three edges in $Z'$ with the same sign as $e_T$, and in each $Z_i$ uses the three edges which have the same sign as $S_i$. 

Now, we have that $| \bigcup C'_{e_T}|=O(n^{5.4 \times 10^{-5}})$ and $|A|=\Theta(p_A n^{10^{-4}})$, and hence $|A \setminus ( \bigcup C'_{e_T})|=\Theta(p_A n^{10^{-4}})$. We obtain $A^*$ from $A$ and a perfect matching $T^*$ for $A^*$ by covering the vertices of $A \setminus ( \bigcup C'_{e_T})$ (and possibly a small number of additional vertices to balance parity requirements) by a collection of disjoint edges $E_{A^*}$ using vertices in $V(\mathcal{T})\setminus I_{t_{20}}$. We claim that $A^*$ is now an absorber for every qualifying leave $L^*$. Firstly note that $\mathcal{T}[A^*]$ has a perfect matching by construction. Then by Lemma \ref{lemma_A_cover} we may describe $L^*$ as the difference of two matchings $M_A^+$ and $M_A^-$, where $M_A^- \subseteq \mathcal{T}[A']$. Subsequently, Lemma \ref{lemma_cascade} will allow us to find a perfect matching $M_{A^*}$ in $A^*$ that uses the edges of $M_A^-$ so that $(M_{A^*}\setminus M_A^-) \cup M_A^+$ is a perfect matching for $\mathcal{T}[A^* \cup L^*]$ as desired. Indeed $M_{A^*}$ consists of the following collections of edges. Firstly we take $E_{A^*} \subseteq M_{A^*}$. Then for each edge $e \in M_A^-$, we find the cascade $C'_{e}$ and take in $M_{A^*}$ all the edges with the same sign as $e$. Now, by construction this union is itself a matching. Furthermore, consider all cascades not yet considered by this approach. Then, the only places in which these cascades may overlap with other cascades are in the vertices of $\bigcup T'$, and if this is the case, they overlap precisely in an integer number of edges in $T'$. Thus, taking the perfect matching for such a cascade which uses all edges of the same sign as the edges in $T'$, we have covered all vertices of $A^*$ by a perfect matching such that $M_A^- \subseteq M_{A^*}$ as required.

It remains to show how to make the modifications to obtain $A^*$ from $A$ and show that these modifications have an insignificant effect on the properties we showed hold for $A$ at the beginning of the proof, and that we can fix the parity requirements as per (iv). Regarding parity, we need to ensure that $|V_{O}^{X+Y}(\mathcal{T}^{-A^*})|=|V_{O}^{X-Y}(\mathcal{T}^{-A^*})|$ and $|V_{E}^{X+Y}(\mathcal{T}^{-A^*})|=|V_{E}^{X-Y}(\mathcal{T}^{-A^*})|$. We were able to show using Chernoff bounds that, before fixing $A$, with high probability $|V_{O}^{X+Y}(\mathcal{T}[A])|=p_An^{10^{-4}} \pm n^{0.9 \times {10^{-4}}}$, where $n^{0.9 \times {10^{-4}}} \ll p_An^{10^{-4}}$. The same holds for $|V_{O}^{X-Y}(\mathcal{T}[A])|$, $|V_{E}^{X+Y}(\mathcal{T}[A])|$ and $|V_{E}^{X-Y}(\mathcal{T}[A])|$. Without loss of generality assume that we have $|V_{O}^{X+Y}(\mathcal{T}^{-A})|\geq|V_{O}^{X-Y}(\mathcal{T}^{-A})|$. Then when $n$ is odd we need to add at most $2n^{0.9 \times {10^{-4}}}$ edges that have an even coordinate in $V^{X+Y}$ and an odd coordinate in $V^{X-Y}$ and ensure that all other edges added to obtain $A^*$ from $A$ are not wrap-around edges. By Fact \ref{fact_basic2}(iii), the wrap around edges with required parity pairings can be added greedily so that we add at most $8n^{0.9 \times {10^{-4}}}$ vertices to $A^*$ from this process and all such additional vertices are in $V(\mathcal{T})\setminus I_{t_{20}}$. When $n$ is even, we instead add a set of vertices $P$ containing the right number and parity of vertices to ensure that $|V_{O}^{X+Y}(\mathcal{T}^{-A^*})|=|V_{O}^{X-Y}(\mathcal{T}^{-A^*})|$ and $|V_{E}^{X+Y}(\mathcal{T}^{-A^*})|=|V_{E}^{X-Y}(\mathcal{T}^{-A^*})|$, so that $|P| \leq 4n^{0.9 \times {10^{-4}}}$ and $P \subseteq I_{t_0} \setminus I_{t_{1}}$. Then setting $P = \emptyset$ when $n$ is odd, by Fact \ref{fact_basic2}(iv) we can greedily add disjoint edges to cover the vertices remaining uncovered in $(A \cup P) \setminus \bigcup C'_{e_T}$ so that every edge uses only additional vertices in $V(\mathcal{T})\setminus I_{t_{20}}$, where the number of such vertices is $O(p_An^{10^{-4}})$. (Note that whilst Fact \ref{fact_basic2}(iv) only covers the case of vertices $v \in I_{t_0}\setminus I_{t_1}$, it is also clear that for vertices $v \in I'_{n^{10^{-4}}}$ there are $\Theta(n)$ edges containing $v$ which do not wrap-around and use only other vertices in $V(\mathcal{T}\setminus I_{t_{20}})$.) By the pair degree condition on vertices in $V(\mathcal{T})$, since there are only $O(p_An^{10^{-4}})$ to cover, we can therefore greedily choose the edges, as claimed. These greedy choices complete $A$ to $A^*$. 

Since only vertices from $V(\mathcal{T})\setminus I_{t_{20}}$ are taken going from $A$ to $A^*$, we know that properties concerning subsets contained within $I_{t_{20}}$ are unaffected. Since we are removing at most $O(p_An^{10^{-4}})$ additional vertices within this subset of size $\Theta(n)$, where degrees of vertices into such valid intervals are also $\Theta(t_0^{1-2\epsilon}) \gg O(p_An^{10^{-4}})$, it is clear that the effect on the relevant vertex set and degree-type properties is sufficiently small. Considering the condition on the number of edges in $\mathcal{T}^{-A^*}$ clearly we have $|\mathcal{T}^{-A^*}| \leq n^2$. For the lower bound, we know that we removed $O(p_An^{10^{-4}})$ vertices from $\mathcal{T}$ to obtain $\mathcal{T}^{-A^*}$. Since each of these vertices is in exactly $n$ edges in $\mathcal{T}$ we have lost at most $O(p_An^{1+10^{-4}})$ edges from $\mathcal{T}$ in the process to reach $\mathcal{T}^{-A^*}$. Thus $|\mathcal{T}^{-A^*}| \geq n^2-O(p_An^{1+10^{-4}}) \geq (1-O(p_A))n^2$, as required. Finally, concerning zero-sum configurations, once again a greedy argument works. In particular, removing $O(p_An^{10^{-4}})$ vertices from $V(\mathcal{T})\setminus I_{t_{20}}$ can remove at most another $O(p_An^{10^{-4}}j_i)$ configurations from $\mathcal{Z}^+_{i,e,\mathcal{T}^{-A}}(\alpha, \beta, \gamma)$ when $e$ is a bad edge, and at most another $O(p_An^{10^{-4}}t_1)$ configurations from $\mathcal{Z}^-_{i,e,\mathcal{T}^{-A}}(\alpha, \beta, 0)$ when $e$ is an edge of type $(\alpha, \beta, 0)_i$ with $\alpha \neq 0$ for any $i \in [c_g]$. The same greedy argument then follows through as was used to show that $|\mathcal{Z}^-_{i,e,\mathcal{T}^{-A}}(\alpha, \beta, 0)|=(1 \pm O(n^{-0.999}))|\mathcal{Z}^-_{i,e,\mathcal{T}}(\alpha, \beta, 0)|$. This completes the proof.
\end{proof}

\section{The bounded integral decomposition lemma} \label{ch_bidl}

We complete this chapter by proving Proposition \ref{converse_Q}, the converse of Proposition \ref{zer_sum_Q} which was used in building the absorber $A^*$. This result is not required for the proof of Theorem \ref{thm_main}, however is included as it has independent interest. In fact, this section includes a proof of something stronger that yields Proposition \ref{converse_Q} as a (sort of) corollary. We'll first prove the `bounded integral decomposition lemma' which states that, given any subset $S \subseteq V(\mathcal{T})$ such that $S \in \mathcal{L}(\mathcal{T})$ and $|S^X|=|S^Y|=|S^{X+Y}|=|S^{X-Y}|$ and $|S|= O(n^{1-\alpha})$ for some $\alpha >0$, we are able to describe $S$ as the difference of two (multi)-sets of edges in $\mathcal{T}$, such that both sets of edges have size of $O(n^{1-\beta})$ for some $\beta>0$. This in turn, by the same methods as those used in Lemma \ref{t-match}, implies that for such $S$ we can also describe $S$ as the difference of two matchings in $\mathcal{T}$. This is interesting in itself, and also marks the initial strategy that one would use for `hole' following the methods of Keevash in \cite{countingdesigns} if one were able to additionally find a suitable `template' to make the same method work. In what follows, we have that $\alpha>0$ is fixed.

\begin{lemma}[Bounded integral decomposition lemma] \label{bidl}
Given $S \subseteq V$, $S \in \mathcal{L}(\mathcal{T})$ such that $|S^X|=|S^Y|=|S^{X+Y}|=|S^{X-Y}|$ and $|S|= O(n^{1-\alpha})$, as above, there exists $\Phi \in \mathbb{Z}^{E(\mathcal{T})}$ such that $\partial \Phi - S={\bf 0}$, and $|\Phi|=O(n^{1-\beta})$ for some $\beta>0$. 
\end{lemma}
  
In the proof of Lemma \ref{bidl} we add signed edges to an initially empty set $\Phi$, keeping track of how this affects the vector of weights on the vertices, ${\bf v}:=\partial \Phi - S$, and $|\Phi|$. At certain stages the proof relies on adding edges to $\Phi$ such that, in ${\bf v}$, those entries indexed by vertices in three of the four parts are all identically $0$, relating to the information of the lattice set out in Section \ref{understanding_the_lattice}. Recall, also the notation in Section \ref{sec_lattice_notation}.

To prove Lemma \ref{bidl} using the information about the sub-lattice $\mathcal{L}_1^{X+Y}(\mathcal{T})$ derived in Section \ref{understanding_the_lattice}, we first need to prove that we can add a sublinear number of edges to $\Phi$ in a way that ensures ${\bf v} \in \mathcal{L}_1^{X+Y}(\mathcal{T})$. An important point to note for this proof is that we shall consider the representatives $\{0,...,n-1\}$ for the vertices in each part of $\mathcal{T}$ (as opposed to $\{\frac{n-1}{2}, \ldots, 0, \ldots \frac{n-1}{2}\}$ as used for proving the main result).

\begin{prop} \label{put_in_L_1}
Suppose that ${\bf v} \in \mathcal{L}(\mathcal{T})$, such that $|{\bf v}|=:t$. Then, adding only $O(t)$ edges to $\Phi$, we can modify ${\bf v}$ so that $|{\bf v}|=O(t)$ and ${\bf v} \in \mathcal{L}_1^{X+Y}(\mathcal{T})$.
\end{prop}

\begin{proof}
By the zero-summing process discussed in the proof of Proposition \ref{dummyunits}, (though in this case there are no parity requirements), we know that by adding $O(t)$ edges to $\Phi$ we can obtain ${\bf v} \in \mathcal{L}_2^{X+Y, X-Y}(\mathcal{T})$. 
Furthermore, by applying Proposition \ref{efficient_SQ-gen_decomp} (for $V^{X-Y}$ rather than $V^{X+Y}$) to ${\bf v^{X-Y}}$, the vector taking only the support in $V^{X-Y}$, we can write ${\bf v^{X-Y}}$ as the sum of $O(t)$ $SQ$-gens. For each of these we can efficiently generate a vector with weight added to the same coordinates in $V^{X-Y}$, and then also to some coordinates in $V^{X+Y}$. Thus, subtracting these vectors for each $SQ$-gen, we reduce ${\bf v^{X-Y}}$ to ${\bf 0}$, and add at most $O(t)$ weight to $V^{X+Y}$, and no weight to $V^X \cup V^Y$.
\end{proof}

Combining Proposition \ref{put_in_L_1} with the next lemma is enough to prove Lemma \ref{bidl}.

\begin{lemma} \label{BIDC}
Suppose ${\bf v} \in \mathcal{L}_1^{X+Y}(\mathcal{T})$ satisfies $|{\bf v}|<O(n^{1-\alpha_1})$. Then, adding only $O(n^{1-\alpha_2})$ edges to $\Phi$ for some $0<\alpha_2<\alpha_1$, we can reduce ${\bf v}$ to ${\bf 0}$. 
\end{lemma}

\begin{proof}[Proof of Lemma \ref{bidl}]
Starting with $\Phi = \emptyset$, for each vertex $s \in S$ add an edge to $\Phi$ through $s$. Let ${\bf v}:=\partial \Phi - S$. Then ${\bf v} \in \mathcal{L}(\mathcal{T})$, $|{\bf v}|=O(n^{1-\alpha})$ and $|\Phi|=O(n^{1-\alpha})$. By Proposition \ref{put_in_L_1}, we can add $O(n^{1-\alpha})$ edges to $\Phi$ in such a way that ${\bf v} \in \mathcal{L}_1^{X+Y}(\mathcal{T})$, and $|{\bf v}|=O(n^{1-\alpha})$. Thus, by Lemma \ref{BIDC}, adding $O(n^{1-\beta})$ edges to $\Phi$ for some $0<\beta<\alpha$ we can reduce ${\bf v}$ to ${\bf 0}$. Clearly $\Phi \in \mathbb{Z}^{E(\mathcal{T})}$ and satisfies $\partial \Phi - S={\bf 0}$, and $|\Phi|=O(n^{1-\alpha})+O(n^{1-\beta})=O(n^{1-\beta})$, for some $\beta>0$.     
\end{proof}

In order to prove Lemma \ref{BIDC} we shall first introduce some auxiliary propositions.  
Before this we also remark that, from now on, when we refer to an $SQ$-gen $(a,a+b,a+c,a+b+c)$, we mean an $SQ$-gen with weights {\it either} $(1,-1,-1,1)$ {\it or} $(-1,1,1,-1)$ on coordinates $(a,a+b,a+c,a+b+c)$. Clearly writing the coordinates in a different order would ensure that we always mean one with weights $(1,-1,-1,1)$, however it will often be useful to write the coordinates of the $SQ$-gens so that only the last coordinate may wrap around. 
Writing $-(a,a+b,a+c,a+b+c)$ simply means that we flip the signs of the weights on the coordinates. Similar remarks apply for $Q$-gens and the vectors of the form $(a,a+b,a+c,a+b+c,a',a'+b',a'+c',a'+b'+c')$ where $bc=b'c'$, where these are seen with weights {\it either} $(1,-1,-1,1,-1,1,1,-1)$ {\it or} $(-1,1,1,-1,1,-1,-1,1)$ unless explicitly stated otherwise. When multiple $SQ$-gens or queens vectors are used in the same equation, we assume that they all have the same weight pattern attached to them, using the `$-$' to reflect a vector which needs to be considered with the opposite weight pattern. 

\begin{prop} \label{shift_SQ}
Suppose we have a decomposition of ${\bf v} \in \mathcal{L}_1^{X+Y}(\mathcal{T})$ into the sum of $SQ$-gens. Let $\mathcal{A}$ be the multiset of $SQ$-gens. Suppose for some $(a,a+b,a+c,a+b+c) \in \mathcal{A}$ we wish instead to have a vector $(s,s+b,s+c,s+b+c) \in \mathcal{A}$. Then adding a single $Q$-gen to ${\bf v}$ (and thus adding a constant number of edges to $\Phi$), we have ${\bf v} \in \mathcal{L}_1^{X+Y}(\mathcal{T})$ a vector with $SQ$-gens decomposition consisting of $(\mathcal{A}\setminus\{(a,a+b,a+c,a+b+c)\}) \cup \{(s,s+b,s+c,s+b+c)\}$. 
\end{prop}

\begin{proof}
This follows directly from the structure of $Q$-gens. Let $a':=s-a$. Then, if $(a,a+b,a+c,a+b+c)$ is in the $SQ$-gens decomposition of ${\bf v}$ with weights $(1,-1,-1,1)$, adding $Q$-gen $(a,a+b,a+c,a+b+c, a'+a, a'+a+b, a'+a+c, a'+a+b+c)$ with weights $(-1,1,1,-1,1,-1,-1,1)$ to ${\bf v}$ gives what is required. 
\end{proof}

We refer to changing ${\bf v}$ in this way as {\it shifting} an $SQ$-gen, and say that we {\it shift} a copy of an $SQ$-gen when we interchange $SQ$-gens in the decomposition of ${\bf v}$ (potentially modifying {\bf v}) by adding a $Q$-gen to ${\bf v}$ which cancels one $SQ$-gen out and leaves another $SQ$-gen in the decomposition.

\begin{prop} \label{shift_to_1}
Queens vectors of the form $(a, a+2^x, a+2^{y}, a+2^x+2^{y}, a, a+2^{x-1}, a+2^{y+1}, a+2^{x-1}+2^{y+1})$ with weights $(1,-1,-1,1,-1,1,1,-1)$ are efficiently generated.
\end{prop}

\begin{proof}
Note that $(a, a+2^x, a+2^{y}, a+2^x+2^{y}, a, a+2^{x-1}, a+2^{y+1}, a+2^{x-1}+2^{y-1})= (a, a+2^{x-1}, a+2^y, a+2^{x-1}+2^y, a+2^y, a+2^{x-1}+2^y, a+2^{y+1}, a+2^{x-1}+2^{y+1})-(a, a+2^{x-1}, a+2^{y},a+2^{x-1}+2^y, a+2^{x-1},a+2^x, a+2^{x-1}+2^y, a+2^x+2^y)$,  
and those on the RHS are $Q$-gens which we know can be generated efficiently.
\end{proof}

Recall from Section \ref{sec_lattice_notation} that an integer marked with a superscript $\circ$ means that it should be read without modular arithmetic.

\begin{prop} \label{1st_power}
Let ${\bf v}=(a,a+b,a+c,a+b+c)$ be an $SQ$-gen written so that $a=a^\circ  \leq a+b = (a+ b)^{\circ} \leq (a+c)^{\circ}=a+c$. Then we can write ${\bf v}$ as the sum of $O(\log(n))$ $SQ$-gens, where each is of the form $(a',a'+b,a'+2^j, a'+b+2^j)$ for some $a'<n$, and $2^j < \min\{n-a', n/2\}$.
\end{prop}

\begin{proof}
Write $c=\sum_{i=1}^k 2^{c_i}$ such that $2^{c_i} \leq n/2$ for every $i$, and $c_i \geq c_{i+1}$ with equality if and only if $c>n/2$ and $c_1=c_2$ is necessary to get a sum of the above form. 
Since $c < n$ we have that $k =O(\log (n))$. We can then decompose ${\bf v}$ into $k=O(\log(n))$ vectors of the required form via an iterative shifting process which ensures that weight added to any coordinates in a way that would modify ${\bf v}$ is cancelled out by adding other vectors with carefully chosen first coordinates. 

In particular, writing $c_i':=2^{c_i}$ for every $i$, it is easy to see that
\begin{eqnarray*}
(a,a+b,a+c,a+b+c) &=& ~~~~(a,a+b,a+c_1',a+b+c_1') \\ &~& +~(a+c_1', a+c_1'+b, a+c_1'+c_2', a+c_1'+c_2'+b) \\ &~& +~(a+c_1'+c_2', a+c_1'+c_2'+b, \\&~& ~~~~~~~~~~~~~~ a+c_1'+c_2'+c_3', a+c_1'+c_2'+c_3'+b) \\ &~& +~\dots \\ &~& +~(a+\sum_{i=1}^{k-2} c_i', a+\sum_{i=1}^{k-2} c_i' +b, a+\sum_{i=1}^{k-1} c_i', a+\sum_{i=1}^{k-1} c_i'+b) \\ &~& +~(a+\sum_{i=1}^{k-1} c_i', a+\sum_{i=1}^{k-1} c_i' +b, a+c, a+b+c),
\end{eqnarray*}
where each of the vectors on the RHS is an $SQ$-gen of the form $(s,s+b,s+2^j, s+b+2^j)$ for some $s<a+c<n$ and $2^j< \min\{n-s, n/2\}$.
\end{proof}

Note, in particular, that if $a=0$ and $b=1$, then we are able to write any vector $(0,1,x,1+x)$, where $x<n$ as the sum of $O(\log(n))$ $SQ$-gens each of the form $(0,1,2^i,1+2^i)$, where $2^i<n/2$. Using Proposition \ref{1st_power}, we can also prove the following.

\begin{prop} \label{power_of_2}
Let ${\bf v}=(a,a+b,a+c,a+b+c)$ be an $SQ$-gen written so that $a=a^\circ \leq a+b = (a+ b)^{\circ} \leq (a+c)^{\circ}=a+c$. Then we can write ${\bf v}$ as the sum of $O(\log^2(n))$ $SQ$-gens, where each is of the form $(a',a'+2^i,a'+2^j, a'+2^i+2^j)$ for some $a'<n$, and $2^i, 2^j < \min\{n-a', n/2\}$.
\end{prop}

\begin{proof}
By Proposition \ref{1st_power}, we can write ${\bf v}$ as the sum of $O(\log(n))$ $SQ$-gens, where each is of the form $(s,s+b,s+2^j, s+b+2^j)$, where $s<n$ and $2^j<\min\{n-s, n/2\}$. By applying Proposition \ref{1st_power} to each of these $O(\log(n))$ $SQ$-gens, it is clear that we can write $(s,s+b,s+2^j, s+b+2^j)$ as the sum of $O(\log(n))$ $SQ$-gens where each is of the form $(s',s'+2^i, s'+2^j, s'+2^i+s^j)$, where $s'<n$, and $2^i <\min\{n-s', n/2\}$. That is, we have expressed ${\bf v}$ as the sum of $O(\log^2(n))$ $SQ$-gens of the form $(a',a'+2^i,a'+2^j,a'+2^i+2^j)$ for some $a'<n$, and $2^i, 2^j < \min\{n-a', n/2\}$, as required.  
\end{proof}

\begin{prop} \label{zero_sum}
Suppose that ${\bf v} \in \mathcal{L}_1^{X+Y}(\mathcal{T})$, that ${\bf v}$ can be written as the sum of $z=O(n^{1-\alpha_3})$ $SQ$-gens of the form $(0,1,2^i,1+2^i)$ for some $2^i<n/2$ and $0<\alpha_3<1/2$, and that $\sum i^2 v_i \neq 0$. Then we can add $Q$-gens to ${\bf v}$ in such a way that $\sum i^2v_i=0$, and we add at most $O(n^{1-\alpha_3})$ vectors to the $SQ$-gens decomposition of ${\bf v}$, each of the form $(0,1,2,3)$ or $(0,1,n-2,n-1)$.
Furthermore, we can do this in such a way that we have added $O(n^{1-\alpha_3})$ edges to $\Phi$. 
\end{prop}

\begin{proof}
We start by observing that any $SQ$-gen of the form $(0,1,2^i,1+2^i)$, where $2^i<n/2$, adds $\pm 2^{i+1}$ to 
$\sum i^2v_i$, where $2^{i+1}<n$. Thus $|\sum i^2v_i| \leq nz=O(n^{2-\alpha_3})$ and, furthermore, $\sum i^2v_i$ is even. Without loss of generality assume $\sum_i i^2v_i>0$ and write $\sum i^2v_i=:a'=an$, where $a \in \mathbb{Z}$, $a=O(n^{1-\alpha_3})$, and $a$ is even since either $n$ is odd, or when $n$ is even we have that $2n | \sum i^2v_i$. Then we may add $a/2$ $Q$-gens of the form $(n-2,n-1,0,1,0,1,2,3)$ to ${\bf v}$. Since each of these adds $-2n$ to $\sum i^2 v_i$, this reduces $\sum i^2 v_i$ to $0$, and adds $a/2$ vectors of the form $(0,1,2,3)$ and $a/2$ vectors of the form $(0,1,n-2,n-1)$ to the $SQ$-gens decomposition of ${\bf v}$. That is, in total we have added $O(n^{1-\alpha_3})$ vectors to the $SQ$-gens decomposition of ${\bf v}$, and $O(n^{1-\alpha_3})$ edges to $\Phi$ and $\sum i^2 v_i =0$, as required.
\end{proof}

We now turn to the proof of Lemma \ref{BIDC}:

\begin{proof}[Proof of Lemma \ref{BIDC}]
Rather than considering ${\bf v} \in \mathcal{L}_1^{X+Y}(\mathcal{T})$, we consider ${\bf v}$ such that
\begin{enumerate}[(i)]
\item ${\bf v \setminus v^{X+Y}}={\bf 0}$,
\item $\sum v_i = 0,$
\item $\sum i v_i = 0 (\mymod n),$
\item $\sum i^2 v_i =0 (\mymod n)$, or
\item $2n~|~\sum_{i \in V^{X+Y}} i^2 v_i$ when $n$ is even.
\end{enumerate}  
Since every ${\bf v} \in \mathcal{L}_1^{X+Y}(\mathcal{T})$ satisfies the above properties (by Proposition \ref{zer_sum_Q}), it follows that if we can show that any such vector also satisfying $|{\bf v}|=:t=O(n^{1-\alpha_1})$ can be reduced to ${\bf 0}$ adding only $O(n^{1-\alpha_2})$ edges to $\Phi$ for some $0<\alpha_2<\alpha_1$, then every ${\bf v} \in \mathcal{L}_1^{X+Y}(\mathcal{T})$ satisfying $|{\bf v}|=:t=O(n^{1-\alpha_1})$ can be reduced to ${\bf 0}$ in this way, satisfying the claims of the lemma. Suppose $|{\bf v}|=:t=O(n^{1-\alpha_1})$. By Proposition \ref{efficient_SQ-gen_decomp} we can write ${\bf v}$ as the sum of $O(t)$ $SQ$-gens. 
From now on, unless otherwise stated, we always write a semi-queens vector as $(a,a+b,a+c,a+b+c)$ so that $a=a^\circ \leq a+b = (a+ b)^{\circ} \leq (a+c)^{\circ} =a+c$. That is, if the vector wraps around, we write it in the order that ensures only the last coordinate need be considered $\mymod n$. By Proposition \ref{power_of_2}, we can rewrite each of these using $O(\log^2(n))$ $SQ$-gens of the form $(a',a'+2^i,a'+2^j,a'+2^i+2^j)$, so that in total we have written ${\bf v}$ as the sum of $O(t \log^2(n))$ of these power of 2 $SQ$-gens.

Now, by Proposition \ref{shift_to_1}, each of these power of 2 $SQ$-gens of the form $(a',a'+2^i,a'+2^j,a'+2^i+2^j)$ can be replaced by one of the form $(a', a'+2^{i-1}, a'+2^{j+1}, a'+2^{i-1}+2^{j+1})$, adding only a constant number of edges to $\Phi$. Thus, repeating this move $i=O(\log(n))$ times for each of these generators, we are able to write the modified ${\bf v}$ as the sum of $O(t \log^2(n))$ $SQ$-gens of the form $(a', a'+1, a'+2^{i+j}, a'+1+2^{i+j})$ (where $a'+2^{i+j}$ may not be equal to $(a'+2^{i+j})^{\circ}$), and in order to do this, we used $O(\log(n))$ edges for each generator, thus adding $O(t \log^3(n))$ edges to $\Phi$. Furthermore, we may shift each of these $SQ$-gens to $SQ$-gens of the form $(0,1,2^x,1+2^x)$ using a single queens generator for each of these, (and thus only affecting $\Phi$ by a constant factor), by Proposition \ref{shift_SQ}, and so ${\bf v}$ now has a decomposition as the sum of $O(t \log^2(n))$ $SQ$-gens of the form $(0,1,2^x,1+2^x)$. Note that each of these contributes $\pm 2((2^{x})^{\circ})$ to $\sum i^2 v_i$. Writing $a=(2^x)^{\circ}$, by Proposition \ref{1st_power}, we may write each of these $SQ$-gens as the sum of $O(\log(n))$ $SQ$-gens of the form $(s, s+1, s+2^i, s+1+2^i)$, where $s<n$ and $2^i<\min\{n-s, n/2\}$. For each of these where $s \neq 0$, we may add the $Q$-gen $(0,1,2^i,1+2^i,s, s+1, s+2^i, s+1+2^i)$ to ${\bf v}$ to shift the weight on $(s, s+1, s+2^i, s+1+2^i)$ to $(0,1,2^i,1+2^i)$. In this way ${\bf v}$ now has a decomposition into $O(t\log^3(n))$ $SQ$-gens of the form $(0,1,2^i,1+2^i)$ where $2^i<n/2$. Furthermore, as we have added $O(\log(n))$ $SQ$-gens for each originally of the form $(0,1,a,1+a)$ and then shifted each of these using only one $Q$-gen, we have that $|\Phi|=O(t \log^4(n))$.

Suppose that $\sum i^2 v_i \neq 0$. Then, by Proposition \ref{zero_sum}, we can reduce it to zero. This process adds $O(n^{1-\alpha_3})$ vectors to the $SQ$-gens decomposition for some $0<\alpha_3<1/2$. Without loss of generality, assume that $\alpha_3<\alpha_1$. Then this yields that $|\Phi|=O(n^{1-\alpha_3})$, and we have an $SQ$-gens decomposition of ${\bf v}$ into $O(n^{1-\alpha_3})$ vectors.
Furthermore, these vectors are all either of the form $(0,1,2^i,1+2^i)$ for some $2^i<n/2$, or $(0,1,n-2,n-1)$. 
As was done before arranging that $\sum i^2 v_i=0$, we modify ${\bf v}$ so that all $O(n^{1-\alpha_3})$ $SQ$-gens of the form $(0,1,n-2,n-1)$ are replaced by vectors of the form $(0,1,2^i,1+2^i)$ for some $2^i<n/2$. Using Proposition \ref{1st_power} and single $Q$-gen shifts we can thus translate each of these to $O(\log(n))$ $SQ$-gens of the form $(0,1,2^i,1+2^i)$ for some $2^i<n/2$. Since $(0,1,n-2,n-1)$ does not wrap around, the additional vectors and shifts do not affect $\sum i^2 v_i$. So we have ${\bf v}$ with a decomposition into $O(n^{1-\alpha_3}\log(n))$ $SQ$-gens where each is of the form $(0,1,2^i, 1+2^i)$ for some $2^i<n/2$.
In particular, we may write ${\bf v}=\sum_{i=0}^{i=\lfloor \log_2(n/2) \rfloor} c_i (0,1,2^i,1+2^i)$ where $\sum_i |c_i|=O(n^{1-\alpha_3} \log(n))=o(n)$. 
Write $t^*:=\max\{i:c_i \neq 0\}$, so that $2^{t^*} \leq n/2$ and no coordinate larger than $1+2^{t^*}$ has non-zero support. We wish now to modify ${\bf v}$ so that $c_i \in \{0,1\}$ for every $i<t^*$. We do this greedily as follows: find the first $i$ such that $c_i \notin \{0,1\}$. Then subtract $\lfloor c_i/2 \rfloor$ copies of $Q$-gen $(0,1,2^i,1+2^i,2^i, 1+2^i, 2^{i+1}, 1+2^{i+1})$ with weights $(1,-1,-1,1,-1,1,1,-1)$ from ${\bf v}$ and update ${\bf c}$ to ${\bf c^i}$ so that $c^i_i=0$ if $c_i$ is even, or $c^i_i=1$ if $c_i$ is odd, and $c^i_{i+1}=c_{i+1} + \lfloor c_i/2 \rfloor$.\footnote{Note that $c_i$ may be negative. In this case, by `subtract $\lfloor c_i/2 \rfloor$ copies of $Q$-gen $(0,1,2^i,1+2^i,2^i, 1+2^i, 2^{i+1}, 1+2^{i+1})$', we mean `add $|\lfloor c_i/2 \rfloor|$ copies of $Q$-gen $(0,1,2^i,1+2^i,2^i, 1+2^i, 2^{i+1}, 1+2^{i+1})$'.}  Repeat for the updated ${\bf v}$. In this way, eventually we reach ${\bf v}$ as desired, and in the process we have added at most a $\log(n)$ factor to the number of edges in $\Phi$ (so $|\Phi|=O(n^{1-\alpha_3}\log^2(n))$), and have maintained that $\sum i^2 v_i=0$. 
Let $I=\{i:c_i=1\}$ in the updated summation. Then $\sum i^2 v_i=a2^{t^*+1}+\sum_{i \in I} 2^{i+1}$ for some integer $a$. 
But since $0 \leq \sum_{i \in I} 2^{i+1} \leq \sum_{i=0}^{t^*-1} 2^{i+1}<2^{t^*+1}$, it follows that $a=0$, and $I=\emptyset$. That is, ${\bf v}={\bf 0}$. Furthermore, taking $\alpha_2=\alpha_3/2$ we have that $|\Phi|=O(n^{1-\alpha_2})$ where $0<\alpha_2<\alpha_1$, and we are done.
\end{proof}

We finish by giving the proof of Proposition \ref{converse_Q}.

\begin{proof}[Proof of Proposition \ref{converse_Q}]
The proof follows from the reduction argument in the proof of Lemma \ref{BIDC} above. In particular, though the lemma assumes that 
${\bf v} \in \mathcal{L}_1^{X+Y}(\mathcal{T})$, the only assumptions we use about this vector to decompose it to the zero vector are the zero-summing items in the hypotheses of this proposition. The proof above also assumes that $|{\bf v}|=o(n)$ to ensure that we can decompose it using a sublinear number of edges, but loosening this restriction shows that we can reduce any such vector to ${\bf 0}$ using the same argument and an arbitrary number of edges, which shows that any such vector is indeed in the lattice.
\end{proof}

\begin{rem}
Proposition \ref{converse_Q} tells us that, as well as vectors of the form already described in Section \ref{understanding_the_lattice} (see e.g. the discussion preceeding Proposition \ref{eight_ones}), vectors of the form $(s_1,s_1+b,s_1+c,s_1+b+c, s_2, s_2+b', s_2+c', s_2+b'+c')$ are in $\mathcal{L}^{X+Y}_1(\mathcal{T})$, provided that $bc=b'c'$. It is not clear that all such vectors can be generated efficiently, and as such we have no analogue to Proposition \ref{efficient_SQ-gen_decomp}. If indeed an analogue did exist,
the bounded integral decomposition lemma would be an immediate consequence. However, we suspect that there are vectors ${\bf v} \in \mathcal{L}^{X+Y}_1(\mathcal{T})$ which cannot be efficiently generated. If this is the case, as alluded to in the first paragraph of Section \ref{understanding_the_lattice}, this might explain a structural difference in $\mathcal{T}$ which sets it apart from the graphs covered by Keevash's result \cite[Theorem 1.7]{designs2}.
\end{rem}

%% file: ch_count_new.tex
\chapter{The random greedy count} \label{ch_count}

In this chapter we establish how to obtain $H \subseteq \mathcal{T}^{-A^*}$, where $H$ is the graph satisfying Theorem \ref{thm_H} and $\mathcal{T}^{-A^*}$ the graph resulting from Theorem \ref{thm_absorber}. We obtain $H$ from $\mathcal{T}^{-A^*}$ by a random greedy matching process, analysed by Bennett and Bohman in \cite{randomgreedy}. In our context, this process is as follows. We start with $T_0:=\mathcal{T}^{-A^*}$ and choose an edge $e_0$ uniformly at random ({\it uar}) from $E(\mathcal{T}^{-A^*})$ and add it to a set $M$. We then delete the vertices in $e_0$ from $\mathcal{T}^{-A^*}$ to obtain a subgraph $T_1$ from which we choose an edge $e_1$ uar, which we add to $M$. We continue the process, so that after the $i^{th}$ step we have a subgraph $T_i \subseteq \mathcal{T}^{-A^*}$, a set $M$ containing $i$ disjoint edges, and we proceed by adding an edge $e_i$ uar from $T_i$ to $M$ and removing the vertices of $e_i$ from $T_i$ to obtain $T_{i+1}$. Clearly the process terminates when $E(T_j)=\emptyset$ for some $j \in \mathbb{N}$. Furthermore, $M$ is a matching in $\mathcal{T}$. This and related processes have been well studied in the general setting of regular uniform hypergraphs with small pair degrees (see \cite{alonnibble, randomgreedy, grable}). The first and third of these are both
extensions of the semi-random technique introduced by R\"{o}dl \cite{rodlnibble}, which has come to be known as the {\it R\"{o}dl nibble}, and is a cornerstone of probabilistic combinatorics.  
Here, we focus on the process outlined above, since this allows more easily for counting matchings, as will be seen below. 

The method used by Bennett and Bohman \cite{randomgreedy} to analyse the random greedy matching process is known as the {\it differential equations method}. For several reasons which shall become clear later, we are unable to use their result as a `black box', but we follow their strategy, applying the differential equations method in precisely the same ways, but to different objects and with slightly different parameters. In the context of probabilistic combinatorics, the differential equations method is a strategy that can be used to analyse random processes that evolve one step at a time. The method in this setting was popularised in the 1990s by Wormald~\cite{wormald}.
We use the method to establish {\it dynamic concentration}, so called because at every step in the process the random variables that we are tracking are shown to be concentrated around their expectation, but their expectation is changing with every step. For a very nice and more general introduction to the method in this context see \cite{gentle}. We refer to the dynamic expectation of each random variable we track as its {\it trajectory}. Given a random variable whose one step change we wish to follow through the process, it is typically the case that to follow its trajectory we also need to understand the one step change in other random variables too. The random variables we need to track can be divided into {\it primary} and {\it secondary} random variables. The primary variables are those we need to track to ensure that the random greedy matching process can be understood, specifically the number of edges remaining at each step and any other variables that tracking this variable depends on. The secondary variables are those relating to any other properties we want to guarantee hold in $H$ as per Theorem \ref{thm_H}. To calculate the one step change in the number of edges depends on the degrees of the vertices, and in turn the one step change in the degree of each vertex depends on the number of edges and the degrees of other vertices. Together these form a closed collection under which the random greedy matching process can be understood and so this is a complete list of the primary random variables. Additionally, for Theorem \ref{thm_H}, we need to track other degree-type properties as well as subsets of vertices and zero-sum configurations. Fortunately, as we'll see in the following section, each of these random variables depends only on the random variables relating to the number of edges and the degree of each vertex, so together with the primary random variables we retain a closed form system without the addition of any extraneous variables to track. Then these random variables (specifically relating to $\mathcal{T}$-valid subsets and tuples and $i$-legal zero-sum configurations as per Theorem \ref{thm_H}) form our collection of secondary random variables. 

The differential equations method is named as such because, due to the one step changes being very small relative to the whole process, we can essentially treat the discrete process as continuous and subsequently approximate the one step change in each variable's trajectory by a derivative of a function of the variables expected value.    
Regarding the use of the method to establish dynamic concentration, we then apply martingale concentration inequalities and a union bound to prove that the collection of all our primary and secondary random variables are indeed concentrated around their trajectories. We use a method known as the {\it critical interval method}, used by Bohman, Frieze and Lubetzky \cite{bfl}, which exploits the `self-correcting' nature of the random variable we track. In particular, supposing a random variable deviates from its expected trajectory far enough to enter some pre-defined `critical interval', we can exploit some terms of the expected one-step change to show that the variable drifts back towards its expected trajectory.

In this chapter we first cover the details of the differential equations method in relation to our application of it to reach $H$. Then we establish how this process enables us to count matchings in $\mathcal{T}$, leaving us with the job of showing that we can find a matching in $H$ that covers all vertices but that of a qualifying leave $L^*$.

\section{Details of the process}\label{sec_greedy_process}

The intuition regarding the evolution of the random greedy matching process is that the subgraph $T_i$ of $T_0$ remaining after $i$ steps of the process resembles a random subgraph of $T_0$ where each vertex survived independently with probability $p=p(i)=1-\frac{4i}{|V(T_0)|}$. From now on write $V(i):=V(T_i)$. Note that by Theorem \ref{thm_absorber} we have that $|V(0)|=(1 \pm b)4n$ where $b=O(p_A)$, and so $\frac{i}{|V(0)|}=(1 \pm 2b)\frac{i}{4n}$. We take $b=O(p_A)$ to be sufficiently large that every $(1 \pm O(p_A))$ statement in Theorem \ref{thm_absorber} holds with $b$ in place of $O(p_A)$. We also introduce a continuous time variable $t$ which we relate to the process by setting $t=t(i)=\frac{i}{|V(0)|}$, so that $p$ can be seen both as a (continuous) function of $t$ with $p=p(t)=1-4t$ and a (discrete) function of $i$. We shift between the interpretations as a function of $i$ and $t$ throughout the process. If not mentioned explicitly the meaning should be clear from the context. Let $Q(i):=E(T_i)$ be the random variable tracking the number of edges remaining at each step of the process, and let $d_v(i)$ denote the number of edges containing a vertex $v \in V(i)$ at step $i$ of the process. Note that by Theorem \ref{thm_absorber} and our condition on $b=O(p_A)$, $Q(0)=(1 \pm b)n^2$, and $d_v(0)=(1 \pm b)n$ for every vertex $v \in V(0)$. Thus we would guess that $Q(i) \approx (1 \pm b)n^2p^4$, and $d_v(i) \approx (1 \pm b)np^3$ for every $v \in V(i)$.

We'll show that, in fact, for every $0 \leq i \leq (1-n^{-\alpha_{\gr}})n$, where $\alpha_{\gr}:= 10^{-25}$ we have that 
$$Q(i)=n^2p^4 \pm e_q \mbox{~and~} d_v(i)=np^3 \pm e_d$$ 
for every $v \in V(i)$, where 
$$e_q=2(1-4\log p)bn^2  \mbox{~and~} e_d=2(1-4\log p)b^{2/3}n.$$ 
Note that we chose such errors with no attempt to optimise the process, and instead choose them to be sufficient for the process to complete to reach $H$ as per Theorem \ref{thm_H}.

Additionally we consider our secondary random variables which are of three types: the number of vertices remaining in a particular subset $S \subseteq V(\mathcal{T})$, the number of edges containing a fixed vertex $v$, with subset-style conditions on the `types' of edge to be counted (`degree-type conditions'), and the number of zero-sum configurations of specific `types' for a fixed edge $e$. (Note that our degree-type conditions would include the actual degree $d_v(i)$ already considered in the primary random variables but since $d_v(i)$ plays a more crucial role in tracking this process than the other degree-type properties we use the different notation for clarity and since we require a tighter error bound.) We write $V_S(i)$ to denote the random variable tracking $|V(T_i[S])|$, $E_{v,S}(i)$ to denote the random variable tracking $|E_{T_i}(v,S)|$ where $S$ represents any of the subsets of $V(\mathcal{T})$ which are considered in degree-type conditions in Theorem \ref{thm_H}, and $Z_{\mathcal{A}}(i)$ to denote the random variables tracking zero-sum configurations, where $\mathcal{A}$ describes both the fixed edge for which we are tracking configurations as well as all details of the type of configuration being tracked. By assumption, we have that $V_S(0)=(1 \pm b)|S|$, $E_{v,S}(0)=(1 \pm b)|E_{\mathcal{T}}(v,S)|$, and $Z_{\mathcal{A}}(0)=Z_{\mathcal{A}}(\mathcal{T}^{-A^*})$, where $Z_{\mathcal{A}}(\mathcal{T}^{-A^*})=(1\pm b)Z_{\mathcal{A}}(\mathcal{T})$ if $\mathcal{A}$ considers configurations containing a fixed bad edge $e$ (with positive sign), or a fixed edge of type $(\alpha, \beta, 0)$ with $\alpha \neq 0$ (with negative sign). If $\mathcal{A}$ considers configurations containing a fixed edge of type $(0,1,3)$ with positive sign, or containing a fixed bad edge and at least two bad edges in total with positive sign, we are only interested in tracking $Z_{\mathcal{A}}$ if $Z_{\mathcal{A}}(\mathcal{T}^{-A^*})=\Omega(k_it_1n^{-12\alpha_{\gr}})$. Similarly, if $\mathcal{A}$ considers configurations containing a fixed edge of type $(0,1,3)$ or $(0,0,4)$ with negative sign, we are only interested in tracking $Z_{\mathcal{A}}$ if $Z_{\mathcal{A}}(\mathcal{T}^{-A^*})=\Omega(k_ij_in^{-12\alpha_{\gr}})$. As in the case for the primary variables, we define 
$$e_{E_{v,S}}=2b^{1/3}|E_{\mathcal{T}}(v,S)|,~e_{V_S}=2b^{1/3}|S|,\mbox{~and~} e_{Z_{\mathcal{A}}}=2b^{1/3}|Z_{\mathcal{A}}(\mathcal{T}^{-A^*})|,$$ 
and we'll show that whp for every $0 \leq i \leq (1-n^{-\alpha_{\gr}})n$ we have that 
$$E_{v,S}(i)=|E_{\mathcal{T}}(v,S)|p^3 \pm e_{E_{v,S}},$$ 
$$V_S(i)=|S|p \pm e_{V_S},$$ 
and 
$$Z_{\mathcal{A}}(i) = |Z_{\mathcal{A}}(\mathcal{T}^{-A^*})|p^{12} \pm e_{Z_{\mathcal{A}}},$$ 
where the last holds only for $\mathcal{A}$ discussed above. Note that for all other $\mathcal{A}$ of concern in Theorem \ref{thm_H}, the statements regarding them are already true, since we can only lose configurations through the random greedy edge removal process. Note additionally that whilst the error terms for the primary variables change with $i$ (since they are dependent on $p$), the error terms for the secondary variables remain constant throughout the process (and in particular the one-step change is $0$). By abuse of notation, we may sometimes denote by $Z_{\mathcal{A}}(i)$ the collection of zero-sum configurations counted by the random variable of the same notation. It will be clear from context each time whether we are considering the family of relevant zero-sum configurations or the cardinality of that family.

For a random variable $X(i)$ we say that $X(i)$ {\it becomes bad at step $i$} if it deviates outside of the error bounds we defined above. We define the stopping time $T$ to be the earliest time $i$ such that either any of the primary and secondary variables we are tracking become bad, or $i=(1-n^{-\alpha_{\gr}})n$, whichever occurs first. In order to show that with high probability $T=(1-n^{-\alpha_{\gr}})n$, which would prove Theorem \ref{thm_H}, it is convenient to consider shifted variables. In particular we let
$$Q^{\pm}(i)=Q(i)-n^2p^4 \mp e_q,$$
$$d_v^{\pm}(i)=d_v(i)-np^3 \mp e_d,$$
$$E_{v,S}^{\pm}(i)=E_{v,S}(i)-|E_{\mathcal{T}}(v,S)|p^3 \mp e_{E_{v,S}},$$
$$V_S^{\pm}(i)=V_S(i)-|S|p \mp e_{V_S},$$
$$Z_{\mathcal{A}}^{\pm}(i)=Z_{\mathcal{A}}(i) - |Z_{\mathcal{A}}(\mathcal{T}^{-A^*})|p^{12} \mp e_{Z_{\mathcal{A}}}.$$ 

We illustrate the idea of the critical interval method by discussing it with regards to the vertex degrees. The idea is that we only need to worry about a variable when it is close to either end of the interval listed above. For $d_v^-(i)$ we refer to $[np^3 - e_d, np^3 - e_d + f_d]$ as its (lower) critical interval, where $f_d=b^{2/3}n$ (and for $d_v^+(i)$ we refer to $[np^3 + e_d - f_d, np^3 + e_d]$ as its (upper) critical interval). Now suppose that $d_v^-(i)$ first enters this critical interval at step $s_{1,v} < T$. That is, $d_v^-(i)> np^3 - e_d + f_d$ for all $i \leq s_{1,v}-1$ and $d_v^-(s_{1,v}) \leq np^3 - e_d + f_d$. Then we define $T^-_{s_{1,v}}$ to be the first time $t>s_{1,v}$ such that $d_v^-(t)$ leaves the critical interval again. There are two possibilities for how it leaves - we could have $d_v^-(T^-_{s_{1,v}})<np^3 - e_d$ or we could have $d_v^-(T^-_{s_{1,v}})> np^3 - e_d + f_d$. In the first case we get that $T^-_{s_{1,v}}=T$ is the final stopping time since $d_v$ has become bad. We will, however, show that with high probability we are always in the second case whenever $s_{1,v}<t \leq (1-n^{-\alpha_{\gr}})n$, exploiting the self-correcting nature of the process. In fact, for $j \geq 2$ and each $v$ we also define $s_{j,v}$ be the first time $t> T^-_{s_{j-1}, v}$ such that $d_v^-(t)$ is in the lower critical interval and $T^-_{s_{j,v}}$ to be the first time $t>s_{j,v}$ such that $d_v^-(t)$ leaves the critical interval again. This is defined for all $j$ such that $s_{j,v} < T$. Then we show that for all $j$ such that $s_{j,v}$ is defined, with high probability we are always in the second case. Since the one step change is always sufficiently small, we have that the second case always takes us into $[np^3-e_d+f_d, np^3 + e_d - f_d]$, i.e. we can not jump from one critical interval to the other (nor beyond it) in a single step of the algorithm. 

In order to show that this deviation back towards the trajectory occurs when we hit a critical interval we use martingale concentration inequalities. To do this, we first show that the random variables $X^+$ are supermartingales and $X^-$ are submartingales. In fact, the nature of the critical interval method means that we actually show that subintervals of the process, starting from each time we enter the upper or lower critical interval to the first time they leave again, are super- or submartingales respectively. Formally, for a random variable $X$, writing $\mathbb{E}'$ to denote conditional expectation with respect to the natural filtration, $X$ is a {\it supermartingale} if $\mathbb{E}'(X(i+1)) \leq X(i)$ and a {\it submartingale} if $\mathbb{E}'(X(i+1)) \geq X(i)$ for all $i$ over which we are tracking the variable. Equivalently, writing $\Delta X(i):=X(i+1)-X(i)$, we have that $X$ is a supermartingale if $\mathbb{E}'(\Delta X(i)) \leq 0$ and a submartingale if $\mathbb{E}'(\Delta X(i)) \geq X(i)$. Thus referring back to our vertex degree setting and letting $S_{j,v}=[s_{j,v}, T^-_{s_{j,v}}]$ for all $j$ for which $s_{j,v}$ is defined, it suffices to show that $\mathbb{E}'(\Delta d_v^-(i)) \geq 0$ for all $i \in S_{j,v}$ and $j$ for which $s_{j,v}$ is defined. Once this is done, we may use the following concentration inequalities to prove that with high probability the random variables do not jump outside of their error bounds before time $t=(1-n^{-\alpha_{\gr}})n$. It will then be clear that still with high probability this holds for all relevant random variables simultaneously via a union bound. 

All inequalities are variations of the Hoeffding-Azuma inequality. The first two results are used for the variables tracking the number of edges in the process.

\begin{lemma}
Let $X(i)$ be a submartingale such that $|\Delta X(i)| \leq c_i$ for all $i$. Then
$$\mathbb{P}(X(m)-X(0) \leq -a) \leq \exp\left(-\frac{a^2}{2\sum_{i \in [m]} c_i^2}\right).$$
\end{lemma}

\begin{lemma}
Let $X(i)$ be a supermartingale such that $|\Delta X(i)| \leq c_i$ for all $i$. Then
$$\mathbb{P}(X(m)-X(0) \geq a) \leq \exp\left(-\frac{a^2}{2\sum_{i \in [m]} c_i^2}\right).$$
\end{lemma}

The next two are used for all remaining variables we wish to track in the process. 

\begin{lemma}\cite{bohman}\label{lemma_bohman_sub}
Let $X(i)$ be a submartingale such that $-\Theta \leq \Delta X(i) \leq \theta$ for all $i$ and $\theta \leq \Theta/2$. Then for any $a < \theta m$ we have
$$\mathbb{P}(X(m)-X(0) \leq -a) \leq \exp\left(-\frac{a^2}{3\theta\Theta m}\right).$$
\end{lemma}

\begin{lemma}\cite{bohman}\label{lemma_bohman_sup}
Let $X(i)$ be a supermartingale such that $-\Theta \leq \Delta X(i) \leq \theta$ for all $i$ and $\theta \leq \Theta/10$. Then for any $a < \theta m$ we have
$$\mathbb{P}(X(m)-X(0) \geq a) \leq \exp\left(-\frac{a^2}{3\theta\Theta m}\right).$$
\end{lemma}

Turning again to the setting of vertex degrees, using the latter two results with $f_d(1-o(1))$ in place of $a$ if the hypotheses are satisfied and $(f_d(1-o(1)))^2$ is sufficiently large, we get that with high probability $d^-_v(T^-_{s_{j,v}})-d^-_v(s_{j,v})> -f_d(1-o(1))$. Thus, since, $d^-_v(s_{j,v})=np^3-e_d+f_d(1-o(1))$ we get that $d^-_v(T^-_{s_{j,v}})>np^3-e_d$ and since $d^-_v(T^-_{s_{j,v}})$ is not in the lower critical interval by definition, and $T^-_{s_{j,v}}$ is the first time after $s_{j,v}$ for which this is true, it follows that with high probability $d^-_v(T^-_{s_{j,v}}) \in [np^3-e_d+f_d, np^3-e_d+f_d(1+o(1))]$. The required $o(f_d)$ term is determined by the maximum of $|\Delta d_v^{-}(i)|$ which we'll see below is indeed $o(f_d)$.

Given the above discussion and results, we can reduce the problem to considering the following {\it trend hypotheses}, which yield the necessary supermartingale and submartingale properties of $X^{\pm}$, and {\it boundedness hypotheses} which are the necessary constraints to successfully apply the relevant martingale concentration inequalities as required. In particular, 

{\bf Trend hypotheses:}

{\it Supermartingale conditions:}

If $Q^+(i) \geq -bn^2$ then $\mathbb{E}'(\Delta Q^+(i)) \leq 0$. 

If $d_v^{+}(i) \geq -b^{2/3}n$ then $\mathbb{E}'(\Delta d_v^{+}(i)) \leq 0$. 

If $E_{v,S}^{+}(i) \geq -b^{1/3}|E_{\mathcal{T}}(v,S)|$ then $\mathbb{E}'(\Delta E_{v,S}^{+}(i)) \leq 0$.

If $V_S^{+}(i) \geq -b^{1/3}|S|$ then $\mathbb{E}'(\Delta V_S^{+}(i)) \leq 0$.

If $Z_{\mathcal{A}}^+(i) \geq -b^{1/3}|Z_{\mathcal{A}}(\mathcal{T}^{-A^*})|$ then $\mathbb{E}'(\Delta Z_{\mathcal{A}}^+(i)) \leq 0$.

{\it Submartingale conditions:}

If $Q^-(i) \leq bn^2$ then $\mathbb{E}'(\Delta Q^-(i)) \geq 0$. 

If $d_v^{-}(i) \leq b^{2/3}n$ then $\mathbb{E}'(\Delta d_v^{-}(i)) \geq 0$.

If $E_{v,S}^{-}(i) \leq b^{1/3}|E_{\mathcal{T}}(v,S)|$ then $\mathbb{E}'(\Delta E_{v,S}^{-}(i)) \geq 0$.

If $V_S^{-}(i) \leq b^{1/3}|S|$ then $\mathbb{E}'(\Delta V_S^{-}(i)) \geq 0$.

If $Z_{\mathcal{A}}^-(i) \leq b^{1/3}|Z_{\mathcal{A}}(\mathcal{T}^{-A^*})|$ then $\mathbb{E}'(\Delta Z_{\mathcal{A}}^-(i)) \geq 0$. ~\\

{\bf Boundedness hypotheses:}

{\it Supermartingale conditions:}

$\Delta Q^+(i)^2np\log^2(n) < (bn^2)^2$. 

$-\Theta_d<\Delta d_v^{+}(i)<\theta_d$ with $\theta_d < \Theta_d/10$ 

and $\theta_d \Theta_d np\log^2(n)<(b^{2/3}n)^2$.

$-\Theta_E<\Delta E_{v,S}^{+}(i)<\theta_E$ with $\theta_E < \Theta_E/10$ 

and $\theta_E \Theta_E np\log^2(n)<(b^{1/3}|E_{\mathcal{T}}(v,S)|)^2.$

$-\Theta_V<\Delta V_S^{+}(i)<\theta_V$ with $\theta_V < \Theta_V/10$ 

and $\theta_V \Theta_V np\log^2(n)<(b^{1/3}|S|)^2$.

$-\Theta_Z<\Delta Z_{\mathcal{A}}^+(i)<\theta_Z$ with $\theta_Z < \Theta_Z/10$ 

and $\theta_Z \Theta_Z np\log^2(n)<(b^{1/3}|Z_{\mathcal{A}}(\mathcal{T}^{-A^*})|)^2$.

{\it Submartingale conditions:}

$\Delta Q^-(i)^2np\log^2(n) < (bn^2)^2$. 

$-\Theta_d<\Delta d_v^{-}(i)<\theta_d$ with $\theta_d < \Theta_d/2$ 

and $\theta_d \Theta_d np\log^2(n)<(b^{2/3}n)^2$.

$-\Theta_E<\Delta E_{v,S}^{-}(i)<\theta_E$ with $\theta_E < \Theta_E/2$ 

and $\theta_E \Theta_E np\log^2(n)<(b^{1/3}|E_{\mathcal{T}}(v,S)|)^2$.

$-\Theta_V<\Delta V_S^{-}(i)<\theta_V$ with $\theta_V < \Theta/2_V$ 

and $\theta_V \Theta_V np\log^2(n)<(b^{1/3}|S|)^2$.

$-\Theta_Z<\Delta Z_{\mathcal{A}}^-(i)<\theta_Z$ with $\theta_Z < \Theta_Z/2$ 

and $\theta_Z \Theta_Z np\log^2(n)<(b^{1/3}|Z_{\mathcal{A}}(\mathcal{T}^{-A^*})|)^2$.

~\\
For the boundedness hypotheses, informally we want, for example, that 
$$\Delta Q^-(i)^2np=o(f_q^2),$$ 
where we require the `little-o' term sufficiently small for union bounding the polynomially many variables to be considered. The $\log^2(n)$ factor suffices for this.

In order to verify the trend hypotheses we make use of Taylor's Theorem. 

\begin{theo}[Taylor's Theorem]
Let $f: \mathbb{R} \rightarrow \mathbb{R}$ be twice differentiable on $[a,b]$. Then there exists $\tau \in [a,b]$ such that
$$f(b)-f(a)=f'(a)(b-a) + \frac{f''(\tau)}{2}(b-a)^2.$$
\end{theo}

We use this with $a=t(i)$ and $b=t(i+1)$ so that $b-a=\frac{1}{|V(0)|}$ and $(b-a)^2=\frac{1}{|V(0)|^2}$. In particular, considering $\mathbb{E}'(\Delta d_v^{-}(i))$, by linearity of expectation we get that $\mathbb{E}'(\Delta d_v^{-}(i))=\mathbb{E}'(\Delta d_v(i)) - \mathbb{E}'(\Delta np(i)^3) + \mathbb{E}'(\Delta e_d(i)) = \mathbb{E}'(\Delta d_v(i)) - n(p(i+1)^3-p(i)^3) + e_d(i+1)-e_d(i)$, so that applying Taylor's theorem, we have that $\mathbb{E}'(\Delta d_v^{-}(i)) \approx \mathbb{E}'(\Delta d_v(i))-\frac{12np^2}{|V(0)|}+\frac{e_d'}{|V(0)|}$. We give the full details below.

Dealing first with our primary variables, we could take the calculations for the supermartingales directly from Bennett and Bohman \cite{randomgreedy}. The submartingale details are not given since they are very similar. For completeness we show the submartingale conditions from scratch but emphasise that this is essentially repeating the details from \cite{randomgreedy} with different values for $e_d$ and $e_q$. Before proceeding, recall that $b=O(p_A)=O(n^{-10^{-7}})$ and
\begin{equation}\label{eq_p}
1 \geq p = \Omega(n^{-10^{-25}})
\end{equation} 
in the range for which we are interested in $p$.

Starting with $d_v^-(i)$, we have that 
\begin{multline*}
\mathbb{E}'(\Delta d_v(i))=-\frac{1}{Q(i)}\sum_{e \in E_{T_i}(v, \mathcal{T})} \sum_{u \in e \setminus \{v\}} d_u(i) \pm O\left(\frac{d_v}{Q(i)}\right) \\ \geq -\frac{3(np^3-e_d+b^{2/3}n)(np^3+e_d)}{Q(i)}.
\end{multline*}
Note that this does not take into account the contribution to the expected change that comes from the selection of an edge that itself contains $v$. Since we are not interested in the random variable once $v$ has left $V(i)$ we may instead use the convention that whenever $v \notin V(i+1)$ we take $d_{v}(i+1)=d_v(i)$. This convention will follow through to the calculations for all random variables we are tracking. Then
\begin{multline*}
\mathbb{E}'(\Delta d_v^-(i))\geq -\frac{3(np^3-e_d+b^{2/3}n)(np^3+e_d)}{Q(i)} + \frac{12np^2}{|V(0)|}+\frac{e'_d}{|V(0)|} \\
+O\left(\frac{d_v}{Q}+\frac{np}{|V(0)|^2}+\frac{e''_d}{|V(0)|^2}\right)
\end{multline*}
Expanding, cancelling and regrouping terms we get that
\begin{multline*}
\mathbb{E}'(\Delta d_v^-(i)) \geq -\frac{12b^{2/3}n}{|V(0)|p} + \frac{e'_d}{|V(0)|} \\ + O\left(\frac{(e_d- b^{2/3}n)e_d}{|V(0)|np^4}+\frac{e_q}{|V(0)|^2p^2} + \frac{d_v(i)}{Q(i)} + \frac{np}{|V(0)|^2} + \frac{e''_d}{|V(0)|^2}\right).
\end{multline*}
Noting that $|V(0)|=\Theta(n)$ and $Q(i)=\Theta(n^2p^4)$, we have that the terms carried in the `big-O' are given by
$O\left(\frac{b^{4/3}(1-\log^2(p))}{p^4} + \frac{b(1-\log(p))}{p^2} + \frac{1}{np} + \frac{p}{n} + \frac{b^{2/3}}{np^2}\right)$. Then since $\frac{e'_d}{|V(0)|}=\Theta\left(\frac{b^{2/3}}{p}\right)$ and $p^3 \gg b^{2/3}\log^2(n)$ by (\ref{eq_p}), we have that the `big-O' terms are $o\left(\frac{e'_d}{|V(0)|}\right)$. Furthermore, since $e'_d=\frac{32b^{2/3}n}{p}$, we have that
$$-\frac{12b^{2/3}n}{|V(0)|p} + \frac{e'_d}{|V(0)|}=\frac{20b^{2/3}n}{|V(0)|p},$$
so that $-\frac{12b^{2/3}n}{|V(0)|p} + \frac{e'_d}{|V(0)|}=\Theta\left(\frac{e'_d}{|V(0)|}\right)$, and $-\frac{12b^{2/3}n}{|V(0)|p} + \frac{e'_d}{|V(0)|} \geq 0$, which proves the trend hypothesis for $d_v^-(i)$. 

Now, checking the boundedness hypotheses, note that $np^3$ is decreasing, $e_d$ is increasing and $d_v$ is non-increasing. Thus we have that $d_v(i+1)-d_v(i)\leq \Delta d_v^-(i)\leq n(p(i)^3-p(i+1)^3)+(e_d(i+1)-e_d(i))$. 
Now 
\begin{multline*}
e_d(i+1)-e_d(i)=8b^{2/3}n(\log(p(i+1))-\log(p(i))) \\ =8b^{2/3}n\left(\log\Big(1-\frac{4i}{|V(0)|}\Big)-\log\Big(1-\frac{4(i+1)}{|V(0)|}\Big)\right) =\Theta(b^{2/3})=o(1),
\end{multline*}
$n(p(i)^3-p(i+1)^3)=O(1)$ and $\Delta d_v(i) \geq -4$. Thus there is some absolute constant $c$ such that we may take $\theta=c$ and $\Theta=10c$, and we have $\theta\Theta np\log^2(n)=O(np\log^2(n)) \ll b^{4/3}n^2$, as required.

Similarly, the supermartingale details for the trend and boundedness hypotheses for $Q$ are given in \cite{randomgreedy}. We get that
$$\mathbb{E}'(\Delta Q(i)) =-\frac{16Q(i)}{|V(0)|p} \pm \frac{2|V(0)|pe_d^2}{Q(i)}+O(1).$$

Thus letting $f_q=bn^2$, and assuming at step $i$ that we are in the lower critical interval, we find that
\begin{eqnarray*}
\mathbb{E}'(\Delta Q^-(i)) &\geq& -\frac{16(n^2p^4-e_q+f_q)}{|V(0)|p} - \frac{2|V(0)|pe_d^2}{Q(i)} + \frac{16n^2p^3}{|V(0)|} + \frac{e'_q}{|V(0)|} \\
&+& O\left(1+\frac{n^2p^2}{|V(0)|^2}+\frac{e''_q}{|V(0)|^2}\right).
\end{eqnarray*}

Then we get that 
\begin{multline*}
\mathbb{E}'(\Delta Q^-(i)) \geq \frac{16bn^2}{|V(0)|p}-\frac{8bn^2\log(p)}{|V(0)|p}-\frac{8|V(0)|b^{4/3}n^2p(1-4\log(p))^2}{Q(i)}+\frac{32bn^2}{|V(0)|p} \\ +O\left(1+\frac{n^2p^2}{|V(0)|^2}+\frac{e''_q}{|V(0)|^2}\right)
\end{multline*}

Noting that $\log(p) \leq 0$ we have that
$$\mathbb{E}'(\Delta Q^-(i)) \geq \frac{48bn^2}{|V(0)|p}-\frac{8|V(0)|b^{4/3}n^2p(1-4\log(p))^2}{Q(i)}+O\left(1+p^2+\frac{b}{p^2}\right)$$

Since $p^2 \gg b^{1/3}$ and $\frac{48bn^2}{|V(0)|p}=\Theta\left(\frac{bn}{p}\right)$, we have that $\frac{8|V(0)|b^{4/3}n^2p(1-4\log(p))^2}{Q(i)}=O\left(\frac{b^{4/3}n\log(n)}{p^3}\right)=o\left(\frac{bn}{p}\right)$, and the `big-O' terms are all also $o\left(\frac{bn}{p}\right)$. Thus we have that $\mathbb{E}'(\Delta Q^-(i)) \geq 0$ as required.

For the supermartingale we have
\begin{eqnarray*}
\mathbb{E}'(\Delta Q^+(i)) &\leq& -\frac{16(n^2p^4+e_q-f_q)}{|V(0)|p} + \frac{2|V(0)|pe_d^2}{Q(i)} + \frac{16n^2p^3}{|V(0)|} - \frac{e'_q}{|V(0)|} \\
&+& O\left(1+p^2+\frac{b}{p^2}\right),
\end{eqnarray*}
and it is clear the same arguments for the submartingale follow through with the switched signs, so that $\mathbb{E}'(\Delta Q^+(i)) \leq 0$, as required.

Then verifying boundedness hypotheses, by assumption, for each $i$ we are considering we have $i < T$ and may use our bounds on degrees to consider $\Delta Q^{\pm}(i)^2$. In particular, we have that $|\Delta Q^{\pm}(i)|\leq (1+o(1))4e_d$ and thus $\Delta Q^{\pm}(i)^2np\log^2(n)=O\left(b^{4/3}n^3p\log^4(n)\right) \ll b^2n^4$, satisfying the necessary requirements.

The calculations for the secondary variables are very similar to those for the vertex degrees but we provide the details here for completeness. First note that
$$\mathbb{E}'(\Delta E_{v,S}(i))=-\frac{1}{Q(i)}\sum_{e \in E_{T_i}(v, S)} \sum_{u \in e \setminus \{v\}} d_u(i) \pm O\left(\frac{E_{v,S}(i)}{Q(i)}\right),$$
$$\mathbb{E}'(\Delta V_{S}(i))=-\frac{1}{Q(i)}\sum_{u \in V_i \cap S} d_u(i),$$
$$\mathbb{E}'(\Delta Z_{\mathcal{A}}(i))=-\frac{1}{Q(i)}\sum_{z \in Z_{\mathcal{A}}(i)} \sum_{u \in z \setminus \{e\}} d_u(i) \pm O\left(\frac{Z_{\mathcal{A}}(i)}{Q(i)}\right),$$

The first equality follows identically to the case for vertex degrees. In particular, to see the expected number of edges lost to $E_{T_i}(v, S)$, since we are assuming that $v$ itself is not contained in the edge that is taken, when we sum over $u \in e\setminus \{v\}$ we want to take into account all edges that contain $u$ except any that also contain $v$, but since pair degrees are at most two, and $u$ is only summed over since it is in an edge with $v$, that is at most two edge we want to exclude from the sum (explaining the $\pm O\left(\frac{E_{v,S}(i)}{Q(i)}\right)$). 

To see the case for vertices in a subset of $V(\mathcal{T})$, note that summing over the degree of each vertex in $V(i) \cap S$ will count a particular edge for every vertex in the edge that is also in $S$. Thus the edge gets counted precisely the same number of times as the number of vertices in $V(i) \cap S$ that would be lost if that edge were removed. For zero-sum configurations, firstly we remark that this calculation assumes that no edge is taken which contains a vertex from the fixed edge $e$ for which the number of zero-sum configurations is being considered. (As with degree-type properties, if this were the case, we set $Z_{\mathcal{A}}(i+1)=Z_{\mathcal{A}}(i)$.) Now for a zero-sum configuration $z$, consider an edge $f \in z$ such that $e \cap f = \emptyset$. Then this edge is counted four times (in the degree of each of the four vertices of $f$) when summing over the vertices in $z \setminus \{e\}$, which is an over count since this contribution should count the number of edges whose removal would result in $z$ no longer being present in $T_{i+1}$. However, due to the pair degrees being at most two, since every zero-sum configuration consists of a constant number of vertices, we cannot over count the number of edges interacting with the zero-sum configurations and the effect of their removal on the number of configurations in $Z_{\mathcal{A}}(i+1)$ by more than $O(Z_{\mathcal{A}}(i))$.

Additionally note that since $e_{E_{v,S}}$, $e_{V_S}$ and $e_{Z_{\mathcal{A}}}$ are constant with respect to $p$, their derivatives disappear, and the second derivative of $|S|p$ also disappears. We see that
\begin{multline*}
\mathbb{E}'(\Delta E_{v,S}^+(i)) \leq - \frac{3\big(|E_{\mathcal{T}}(v,S)|p^3+e_{E_{v,S}}-b^{1/3}|E_{\mathcal{T}}(v,S)|\big)\big(np^3-e_d-1\big)}{n^2p^4+e_q} \\ + \frac{12|E_{\mathcal{T}}(v,S)|p^2}{|V(0)|} + O\left(\frac{|E_{\mathcal{T}(v,S)}|p}{|V(0)|^2}\right),
\end{multline*}
\begin{multline*}
\mathbb{E}'(\Delta E_{v,S}^-(i)) \geq - \frac{3(|E_{\mathcal{T}}(v,S)|p^3-e_{E_{v,S}}+b^{1/3}|E_{\mathcal{T}}(v,S)|)(np^3+e_d-1)}{n^2p^4+e_q} \\ + \frac{12|E_{\mathcal{T}}(v,S)|p^2}{|V(0)|} + O\left(\frac{|E_{\mathcal{T}(v,S)}|p}{|V(0)|^2} \right).
\end{multline*}
$$
\mathbb{E}'(\Delta V_S^+(i)) \leq - \frac{(|S|p+e_{V_S}-b^{1/3}|S|)(np^3-e_d)}{n^2p^4+e_q}  + \frac{4|S|}{|V(0)|},
$$
$$
\mathbb{E}'(\Delta V_S^-(i)) \geq - \frac{(|S|p-e_{V_S}+b^{1/3}|S|)(np^3+e_d)}{n^2p^4+e_q}  + \frac{4|S|}{|V(0)|},
$$
\begin{multline*}
\mathbb{E}'(\Delta Z_{\mathcal{A}}^+(i)) \leq - \frac{12(|Z_{\mathcal{A}}(\mathcal{T}^{-A^*})|p^{12}+e_{Z_{\mathcal{A}}}-b^{1/3}|Z_{\mathcal{A}}(\mathcal{T}^{-A^*})|)(np^3-e_d)}{n^2p^4+e_q} \\ + \frac{48|Z_{\mathcal{A}}(\mathcal{T}^{-A^*})|p^{11}}{|V(0)|} + O\left(\frac{|Z_{\mathcal{A}}(\mathcal{T}^{-A^*})|p^12}{n^2p^4} + \frac{|Z_{\mathcal{A}}(\mathcal{T}^{-A^*})|p^{10}}{|V(0)|^2}\right),
\end{multline*}
\begin{multline*}
\mathbb{E}'(\Delta Z_{\mathcal{A}}^-(i)) \leq - \frac{12(|Z_{\mathcal{A}}(\mathcal{T}^{-A^*})|p^{12}-e_{Z_{\mathcal{A}}}+b^{1/3}|Z_{\mathcal{A}}(\mathcal{T}^{-A^*})|)(np^3+e_d)}{n^2p^4+e_q} \\ + \frac{48|Z_{\mathcal{A}}(\mathcal{T}^{-A^*})|p^{11}}{|V(0)|} + O\left(\frac{|Z_{\mathcal{A}}(\mathcal{T}^{-A^*})|p^12}{n^2p^4} + \frac{|Z_{\mathcal{A}}(\mathcal{T}^{-A^*})|p^{10}}{|V(0)|^2}\right),
\end{multline*}

Now, starting by chasing the details for $V_S^+(i)$ we see that
\begin{multline*}
\mathbb{E}'(\Delta V_S^+(i)) \leq - \frac{(|S|p+b^{1/3}|S|)(np^3-e_d)}{n^2p^4}  + \frac{|S|}{n}+O\left(\frac{b|S|}{n}+\frac{e_q|S|np^4}{Q(i)^2}\right) \\
\leq - \frac{b^{1/3}|S|}{np} + \frac{e_d|S|}{n^2p^3}+O\left(\frac{b|S|}{n}+\frac{e_q|S|np^4}{Q(i)^2}+\frac{b^{1/3}|S|e_d}{n^2p^4}\right)
\end{multline*}
Then since we have $p \gg b^{1/3}\log(n)$ and $- \frac{b^{1/3}|S|}{np} + \frac{e_d|S|}{n^2p^3}=\Theta\left(\frac{b^{1/3}|S|}{np}\right) \leq 0$, and the `big-O' terms are all $o\left(\frac{b^{1/3}|S|}{np}\right)$, we have that $\mathbb{E}'(\Delta V_S^+(i)) \leq 0$, as required.

Via very similar calculations we end up deducing that
$$
\mathbb{E}'(\Delta V_S^-(i)) \geq \frac{b^{1/3}|S|}{np} - \frac{e_d|S|}{n^2p^3}+O\left(\frac{b|S|}{n}+\frac{e_q|S|np^4}{Q(i)^2}+\frac{b^{1/3}|S|e_d}{n^2p^4}\right)
$$
and so obtain immediately also that $\mathbb{E}'(\Delta V_S^-(i)) \geq 0$.

For $E_{v,S}^+(i)$ we find that
\begin{multline*}
\mathbb{E}'(\Delta E_{v,S}^+(i)) \leq - \frac{3\big(|E_{\mathcal{T}}(v,S)|p^3+b^{1/3}|E_{\mathcal{T}}(v,S)|\big)\big(np^3-e_d\big)}{n^2p^4} + \frac{3|E_{\mathcal{T}}(v,S)|p^2}{n} \\ + O\left(\frac{|E_{\mathcal{T}(v,S)}|p}{|V(0)|^2} + \frac{e_q|E_{\mathcal{T}}(v,S)|np^6}{Q(i)^2} + \frac{b|E_{\mathcal{T}}(v,S)|p^2}{n} + \frac{|E_{\mathcal{T}}(v,S)|}{Q(i)^2}\right) \\
\leq -\frac{3b^{1/3}|E_{\mathcal{T}}(v,S)|}{np} + \frac{3e_d|E_{\mathcal{T}}(v,S)|}{n^2p} \\
+ O\left(\frac{|E_{\mathcal{T}(v,S)}|p}{n^2} + \frac{e_q|E_{\mathcal{T}}(v,S)|}{n^3p^2} + \frac{b|E_{\mathcal{T}}(v,S)|p^2}{n} + \frac{|E_{\mathcal{T}}(v,S)|}{n^2p^4}+\frac{e_db^{1/3}|E_{\mathcal{T}}(v,S)|}{n^2p^4}\right)
\end{multline*}
As was the case for $V_S$ we have that 
$$-\frac{3b^{1/3}|E_{\mathcal{T}}(v,S)|}{np} + \frac{3e_d|E_{\mathcal{T}}(v,S)|}{n^2p}=\Theta\left(\frac{b^{1/3}|E_{\mathcal{T}}(v,S)|}{np}\right) \leq 0$$ 
and the `big-O' terms are all $o\left(\frac{b^{1/3}|E_{\mathcal{T}}(v,S)|}{np}\right)$ and we also yield that $\mathbb{E}'(\Delta E_{v,S}^-(i))\geq 0$ by the same arguments.

Finally, concerning the trend hypotheses, by similar arguments we get that
\begin{multline*}
\mathbb{E}'(\Delta Z_{\mathcal{A}}^+(i)) \leq \\ - \frac{12(|Z_{\mathcal{A}}(\mathcal{T}^{-A^*})|p^{12}+b^{1/3}|Z_{\mathcal{A}}(\mathcal{T}^{-A^*})|)(np^3-e_d)}{n^2p^4} + \frac{12|Z_{\mathcal{A}}(\mathcal{T}^{-A^*})|p^{11}}{n} \\ + O\left(\frac{|Z_{\mathcal{A}}(\mathcal{T}^{-A^*})|p^{12}}{n^2p^4} + \frac{|Z_{\mathcal{A}}(\mathcal{T}^{-A^*})|p^{10}}{|V(0)|^2} + \frac{e_q|Z_{\mathcal{A}}(\mathcal{T}^{-A^*})|np^{15}}{Q(i)^2} + \frac{b|Z_{\mathcal{A}}(\mathcal{T}^{-A^*})|p^{11}}{n}\right) \\
\leq -\frac{12b^{1/3}|Z_{\mathcal{A}}(\mathcal{T}^{-A^*})|}{np} + \frac{12e_d|Z_{\mathcal{A}}(\mathcal{T}^{-A^*})|p^{8}}{n^2} \\+ 
O\Bigg(\frac{|Z_{\mathcal{A}}(\mathcal{T}^{-A^*})|p^{8}}{n^2} + \frac{|Z_{\mathcal{A}}(\mathcal{T}^{-A^*})|p^{10}}{n^2} + \frac{e_q|Z_{\mathcal{A}}(\mathcal{T}^{-A^*})|p^{7}}{n^3}  \\ + \frac{b|Z_{\mathcal{A}}(\mathcal{T}^{-A^*})|p^{11}}{n}  + \frac{e_db^{1/3}|Z_{\mathcal{A}}(\mathcal{T}^{-A^*})|}{n^2p^4}\Bigg),
\end{multline*}
and as before see both that 
$$-\frac{12b^{1/3}|Z_{\mathcal{A}}(\mathcal{T}^{-A^*})|}{np} + \frac{12e_d|Z_{\mathcal{A}}(\mathcal{T}^{-A^*})|p^{8}}{n^2}=\Theta\left(\frac{b^{1/3}|Z_{\mathcal{A}}(\mathcal{T}^{-A^*})|}{np}\right) \leq 0$$ 
and the `big-O' terms are $o\left(\frac{b^{1/3}|Z_{\mathcal{A}}(\mathcal{T}^{-A^*})|}{np}\right)$. Again we have that the details showing $\mathbb{E}'(\Delta Z_{\mathcal{A}}^-(i)) \geq 0$ follow by precisely the same arguments.

It remains to confirm the boundedness hypotheses for each of our secondary variables. Once again the strategy is very similar to that for the vertex degrees. One difference is to note immediately (as already noted when verifying the trend hypotheses) that since $e_{V_S}$, $e_{E_{v,S}}$ and $e_{Z_{\mathcal{A}}}$ are all constant so when considering the maximum change in variable from step $i$ to step $i+1$ this term does not need any consideration. It then also follows that the same bounds that apply for $\Delta X^+(i)$ will also immediately apply for $\Delta X^-(i)$. For degree-type properties we again have that $-|E_{\mathcal{T}}(v,S)|p^3$ is increasing and in particular $0 \leq |E_{\mathcal{T}}(v,S)|(p(i+1)^3-p(i)^3)=O\left(\frac{|E_{\mathcal{T}}(v,S)|}{n}\right)$, and $0 \geq \Delta E_{v,S}(i) \geq -4$ (due to the maximum pair degree of $1$ unless $(v, u)$ is such that $v \in V^{X+Y}$ and $u \in V^{X-Y}$ or vice versa, in which case pair degree is at most $2$). Then we may set $\theta_E = O\left(\frac{|E_{\mathcal{T}}(v,S)|}{n}\right)$ and $\Theta_E=C_{E_{v,S}}$ for some constant $C_{E_{v,S}}$ satisfying $C=\max\{4, 11\theta_E\}$. Then $\theta_E\Theta_Enp\log^2(n) =O(|E_{\mathcal{T}}(v,S)|p\log^2(n)) \ll b^{2/3}|E_{\mathcal{T}}(v,S)|^2$ and this cover the cases for both $\Delta E_{v,S}^\pm(i)$. The argument is similar for $\Delta V_S^{\pm}$. In particular we have $0 \leq |S|(p(i+1)^3-p(i)^3)=O\left(\frac{|S|}{n}\right)$ and also that $0 \geq \Delta V_S(i) \geq -4$ since removing an edge from $T_i$ removes at most $4$ vertices from $V(i) \supseteq V(i) \cap S$. Then we may set $\theta_V = O\left(\frac{|S|}{n}\right)$ and $\Theta_V=C_{V_S}$ for some constant $C_{V_S}$ satisfying $C_{V_S}=\max\{4, 11\theta_V\}$ and we see that $\theta_V\Theta_Vnp\log^2(n) =O(|S|p\log^2(n)) \ll b^{2/3}|S|^2$ and this cover the cases for both $\Delta V_S^\pm(i)$. Finally for $\Delta Z_{\mathcal{A}}^{\pm}(i)$, whilst very similar, the maximum change in $\Delta Z_{\mathcal{A}}(i)$ is slightly more complex than the previous cases. In particular, we know that every zero-sum configuration we consider (of a particular type for a fixed edge $e$) has two degrees of freedom. We have upper bounds on the order of the range of each degree of freedom. Then given that an edge is removed from $T_i$ (not containing the fixed edge $e$) we could consider the four vertices in this edge as each separately taking one of the two degrees of freedom in different configurations containing $e$. Then the total number of zero-sum configurations that could be lost from $Z_{\mathcal{A}}(i)$ to $Z_{\mathcal{A}}(i+1)$ is four times the largest possible range over which one of the two degrees of freedom is chosen from. Writing $r_1 \geq r_2$ for the largest possible ranges for each degree of freedom, we have then that $0 \geq \Delta Z_{\mathcal{A}}(i) = O(r_1)$ and also, over the $\mathcal{A}$ for which we are interested, we have that $|Z_{\mathcal{A}}(\mathcal{T}^{-A^*})|=\Omega(r_1r_2n^{-12\alpha_{\gr}})$. Again we also have that $0 \leq |Z_{\mathcal{A}}(\mathcal{T}^{-A^*})|(p(i+1)^3-p(i)^3) =O\left(\frac{|Z_{\mathcal{A}}(\mathcal{T}^{-A^*})|}{n}\right)$. Thus setting $\theta_Z=O\left(\frac{|Z_{\mathcal{A}}(\mathcal{T}^{-A^*})|}{n}\right)$ and $\Theta_Z=Cr_1$ where $C$ is a constant sufficiently large that $\Theta_Z>10\theta_z$ (which exists since $|Z_{\mathcal{A}}(\mathcal{T}^{-A^*})|=O(r_1r_2)$ and $\frac{r_2}{n} \leq 1$ since both degrees of freedom could be over a range of size at most $2t_1+1$), we have that $\theta_Z\Theta_Znp\log^2(n)=O\left(r_1|Z_{\mathcal{A}}(\mathcal{T}^{-A^*})|p\log^2(n)\right)$. Recalling that $|Z_{\mathcal{A}}(\mathcal{T}^{-A^*})|=\Omega(r_1r_2n^{-12\alpha_{\gr}})$, we want that $O(r_1p\log^2(n))<b^{2/3}r_1r_2n^{-12\alpha_{\gr}}$ and in particular it suffices to have that $\log^2(n)<r_2b^{2/3}n^{-12\alpha_{\gr}}$ which holds since $r_2=\Omega(n^{10^{-5}})$ and $b^{2/3}n^{-12\alpha_{\gr}}=O(n^{-10^{-6}})$.

In particular we have shown that all the trend and boundedness hypotheses are satisfied and thus that the stopping time $T$ for our random greedy matching process is indeed $T=(1-n^{-\alpha_{\gr}})n$. Denote the graph remaining at this point by $H_{\gr}$. In order to reach $H$ we additionally need to do some small modifications to $H_{\gr}$ so that $|V_O^{X+Y}(H)|=|V_O^{X-Y}(H)|$ but we first complete our count over perfect matchings in $\mathcal{T}$ which is done on the proviso that with high probability $H_{\gr} \cup A^*$ has a perfect matching.

\section{Proof of the main result} \label{sec_mainres}

\begin{lemma} \label{finish_PM} Let $H_{\gr}$ be the graph remaining after running the random greedy matching process until $4n^{1-\alpha_{\gr}}$ vertices remain. Then for $n \equiv 1,5 \mymod 6$, with high probability $H_{\gr} \cup A^*$ has a perfect matching.
\end{lemma}

The proof of Lemma \ref{finish_PM} is mostly contained in Chapters \ref{ch_absorber} and \ref{ch_it_abs}. Firstly, in Chapter \ref{ch_absorber} we build an absorber $A^*$ that has the capacity to absorb {\it any} qualifying leave $L^*$ with corresponding support vector ${\bf v_{L^*}} \in \mathcal{L}(\mathcal{T})$. Theorem \ref{thm_absorber} removes the required absorber $A^*$ from $\mathcal{T}$ leaving us with a subgraph $\mathcal{T}^{-A^*}$ on which we run the random greedy matching process to obtain $H_{\gr}$ as detailed in Section \ref{sec_greedy_process}. We have, by definition of $A^*$ that $\mathcal{T}[A^*]$ has a perfect matching, and moving from $\mathcal{T}^{-A^*}$ to $H_{\gr}$ we removing a matching $M_{\gr}$ that is disjoint from $A^*$. From here we then take a small matching from $H_{\gr}$ to obtain $H$ as in Theorem \ref{thm_H}, details of which are in Section \ref{sec_greedy}. Following this we run the iterative matching process that takes us through the \emph{vortex} from $H$ to $H_{c_h}$ by removing disjoint matchings at every step. Breaking this step up more, we have that Section \ref{sec_H} describes the process that takes us from $H$ to a subgraph $H_1$ which has properties given by Theorem \ref{thm_H_1}, where we have to be careful to maintain certain parity requirements (to obtain a qualifying leave at the end of the iterative matching process). Once we reach $H_1$ such properties are maintained by the nature of the process (the fact that we are removing disjoint edges from $H_1$ and $H_1$ contains no wrap-around edges). This process takes us to $H_{c_h}$. Let $L^* := V(H_{c_h})$. We are able to show that $L^*$ is indeed a qualifying leave. Firstly, that the support vector ${\bf v_{L^*}}$ of $L^*$ is in $\mathcal{L}(\mathcal{T})$ follows from the process used to obtain $L^*$ - we know that $\mathcal{T}$ contains a perfect matching (so ${\bf 1} \in \mathcal{L}(\mathcal{T})$)  
and letting ${\bf v_M}$ represent the support vector of the matching consisting of vertices in $A^*$, the random greedy edge removal process and the iterative matching process, we have that ${\bf 1} - {\bf v_M}={\bf v_{L^*}} \in \mathcal{L}(\mathcal{T})$. That $L^* \subseteq I_{n^{10^{-5}}}$ and $|L^*| \leq p_Ln^{10^{-5}}$ follow from completion of the iterative matching process. Finally, that $|V_O^{X+Y}(L^*)|=|V_O^{X-Y}(L^*)|$ will follow from taking care throughout the random greedy edge removal process and the iterative matching process that we ensure the number of odd and even indexed vertices remaining in the $X+Y$ and $X-Y$ parts at each step of the process are balanced in the required way. In particular, given that $|V_O^{X+Y}(H_1)|=|V_O^{X-Y}(H_1)|$ as per Theorem \ref{thm_H_1}, since $H_1$ contains no wrap-around edges, any process that removes a matching $M$ from $H_1$ maintains this property in the subgraphs with the vertices of $M$ removed, so we only need to be concerned with this property until we reach $H_1$ as in Theorem \ref{thm_H_1}. Thus whp we obtain $L^*$ by removing a matching from $H_{\gr}$ and since $L^*$ is a qualifying leave, we have that whp $\mathcal{T}[A^* \cup L^*]$ has a perfect matching, thus $H_{\gr} \cup A^*$ has a perfect matching, as required.

\begin{theo} \label{edge_removal}
Let $n \equiv 1,5 \mymod 6$. Then $T(n)\geq ((1+o(1))\frac{n}{e^3})^n$.
\end{theo}

\begin{proof}
We run the random greedy matching process on $\mathcal{T}^{-A^*}$ until $4n^{1-\alpha_{\gr}}$ vertices remain, and we have a matching $M$ in $\mathcal{T}$ of size $n(1-n^{-\alpha_{\gr}})$. Then, by Lemma \ref{finish_PM}, with high probability we can complete $M$ to a perfect matching in $\mathcal{T}$. Recalling that $p(i):=1-\frac{4i}{|V(0)|}=1-\frac{i}{n}(1 \pm 2b)$ is the proportion of vertices remaining after the $i^{th}$ edge has been added to $M$ and that, by the analysis of the process above, the number of edges remaining when $4np(i)$ vertices remain is $(1\pm 2bn^{4\alpha_{\gr}}\log(n))n^2p(i)^4$,
 the number of choices in this process is
$$ N_1:=\prod_{i=0}^{(1 \pm b)n-n^{1-\alpha_{\gr}}-1} (1\pm 2bn^{4\alpha_{\gr}}\log(n))n^2p(i)^4.$$
Taking logs, and using Proposition \ref{log(1+x)},
\begin{eqnarray*}
\log(N_1) &=& \log \left(\prod_{i=0}^{(1\pm b)n-n^{1-\alpha_{\gr}}-1} (1\pm 2bn^{4\alpha_{\gr}}\log(n))n^2p(i)^4 \right) \\ &=& \sum_{i=0}^{(1 \pm b)n-n^{1-\alpha_{\gr}}-1} \log \left((1\pm 2bn^{4\alpha_{\gr}}\log(n))n^2p(i)^4\right) \\ &=& \sum_{i=0}^{(1 \pm b)n-n^{1-\alpha_{\gr}}-1} \left(\log(n^2p(i)^4) +\log(1 \pm 2bn^{4\alpha_{\gr}}\log(n))\right) \\ &=& \sum_{i=0}^{(1 \pm b)n-n^{1-\alpha_{\gr}}-1} \left(\log(n^2p(i)^4) \pm 2bn^{4\alpha_{\gr}}\log(n)\right).
\end{eqnarray*}
Furthermore, we have from Proposition \ref{log(p)} that 
$$\sum_{i=0}^{(1\pm b)n-n^{1-\alpha_{\gr}}-1} \log(p(i))=(1 \pm b)n(-1+O(n^{-\alpha_{\gr}}\log(n))),$$ and it follows that 
\begin{eqnarray*}
\log(N_1) &=& \sum_{i=0}^{(1 \pm b)n-n^{1-\alpha_{\gr}}-1} \log(n^2) + 4 \sum_{i=0}^{(1 \pm b)n-n^{1-\alpha_{\gr}}-1}\log(p(i)) \\
&~& ~~~~~~~~~~~~~~~~~~~~~~~~~~~~~~~~~~~~~~~~~~\pm \sum_{i=0}^{(1 \pm b)n-n^{1-\alpha_{\gr}}-1} 2bn^{4\alpha_{\gr}}\log(n) \\ &=& n\log(n^2)+O(n^{1-\alpha_{\gr}}\log(n^2)) -4n+O(n^{1-\alpha_{\gr}}\log(n)) \\
&~& ~~~~~~~~~~~~~~~~~~~~~~~~~~~~~~~~~~~~~~~~~~ \pm O(bn^{1+4\alpha_{\gr}}\log(n)) \\ &=& n(2\log(n)-4 \pm n^{-\alpha_{\gr}/2}).
\end{eqnarray*}
Now, fixing a perfect matching $M$ in $\mathcal{T}$, the number of times $M$ could be counted in this process is at most the number of ways to pick from $\mathcal{T}^{-A^*}$ (with order) the first $(1 \pm b)n-n^{1-\alpha_{\gr}}$ edges of $M$, that is 
$N_2:=\prod_{i=0}^{(1 \pm b)n-n^{1-\alpha_{\gr}}-1} ((1 \pm b)n-i)$. Again taking logs we have
\begin{eqnarray*}
\log(N_2) &=& \sum_{i=0}^{(1 \pm b)n-n^{1-\alpha_{\gr}}-1} \log((1 \pm b)n-i) = \sum_{i=0}^{(1 \pm b)n-n^{1-\alpha_{\gr}}-1} \log((1 \pm b)p(i)n) \\ &=& \sum_{i=0}^{(1 \pm b)n-n^{1-\alpha_{\gr}}-1} \log(p(i)) + \sum_{i=0}^{(1 \pm b)n-n^{1-\alpha_{\gr}}-1} \log(n) \\
&~& ~~~~~~~~~~~~~~~~~~~~~~~~~~~~~~~~~~~~~~~~~~ + \sum_{i=0}^{(1 \pm b)n-n^{1-\alpha_{\gr}}-1} \log(1 \pm b) \\ &=& n(-1+O(n^{-\alpha_{\gr}}\log(n))+\log(n) \pm O(b\log(n))) \\ &=& n(\log(n)-1 \pm n^{-\alpha_{\gr}/2}).
\end{eqnarray*}
It follows that 
$$ \log(T(n)) \geq \log\left(\frac{N_1}{N_2}\right)=n(\log(n)-3\pm 2n^{-\alpha_{\gr}/2}),$$
and so (using a Taylor expansion),
\begin{eqnarray*}
T(n) &\geq & \left(e^{\log(n)}\cdot e^{-3} \cdot e^{\pm 2n^{-\alpha_{\gr}/2}}\right)^n \\ &=& \left(\frac{n}{e^3}(1 \pm n^{-\alpha_{\gr}/3})\right) ^n \\ &=& \left((1+o(1))\frac{n}{e^3}\right)^n,
\end{eqnarray*}
as claimed.
\end{proof}

Thus the problem of lower bounding $T(n)$ becomes the problem of proving Lemma~\ref{finish_PM}.

\section{Reaching $H$} \label{sec_greedy}

At time $T=(1-n^{-\alpha_{\gr}})n$ we have reached a graph $H_{\gr}$ on $4n_{\gr}:=|V(0)|-4(n-n^{1-\alpha_{\gr}})=n^{1-\alpha_{\gr}}(1 \pm bn^{\alpha_{\gr}})$ vertices, where $bn^{\alpha_{\gr}}=o(1)$ (with $n_{\gr}$ vertices in each part). From now on we write $p_{\gr}:=n^{-\alpha_{\gr}}$. Then in addition, we have that $V_{O}^{X+Y}(H_{\gr})=(1 \pm 2b^{1/3}p_{\gr}^{-1})np_{\gr}/2$ and $V_{O}^{X-Y}(H_{\gr})=(1 \pm 2b^{1/3}p_{\gr}^{-1})np_{\gr}/2$. With out loss of generality, we may assume that $|V_{O}^{X+Y}(H_{\gr})| \leq V_{O}^{X-Y}(H_{\gr})$. Then we obtain $H$ from $H_{\gr}$ by removing a matching $M_{\gr}$ so that $|V_{O}^{X+Y}(H)| = |V_{O}^{X-Y}(H)|$ and $|V_{E}^{X+Y}(H)| = |V_{E}^{X-Y}(H)|$. (Note that this step is only necessary when $n$ is odd. Indeed, since $\mathcal{T}^{-A^*}$ satisfies these parity requirements whp by Theorem \ref{thm_absorber}, and going from $\mathcal{T}^{-A^*}$ to $H_{\gr}$ only removes a matching, when $n$ is even the parity requirements are not modified by this process.) In particular, we want to take wrap-around edges with an even vertex in $X+Y$ and an odd vertex in $X-Y$, and we need to take at most $2b^{1/3}n$ edges of this type to achieve equality. Note that in $\mathcal{T}$ choosing any vertex $v \in X-Y$ with coordinate of modulus at least $t_0/4$ there are at least $t_0/4$ vertices in $Y$ whose coordinate dictates a wrap-around edge with $v$ such that the $X+Y$ and $X-Y$ coordinates both have modulus at least $t_0/4$. Half of these will use an even vertex in $X+Y$ and an odd vertex in $X-Y$. So in $\mathcal{T}$ every vertex $v \in [-t_0, -t_0/4]^X \cup [t_0/4, t_0]^X$ is in at least $t_0/8$ edges of the required type. It follows, since this is a degree-type property, that in $H_{\gr}$ every such vertex is in at least $(1-2b^{1/3})\frac{t_0p_{\gr}^3}{8}$ such edges. Since this process is only required for odd $n$ we have that pair degrees are at most one and so the choice of one such edge for a particular vertex in $X$ destroys at most three choices for a different vertex in $X$. Thus since $2b^{1/3}n \ll (1-2b^{1/3})\frac{3t_0p_{\gr}^3}{8}$ we may greedily choose edges to fix the parity disparity. By restricting to edges that only contain vertices with modulus $t_0/4$ or larger, this does not affect any of the properties we maintained during the random greedy matching process too much. In particular, following the discussion above we are now in the position to prove Theorem \ref{thm_H}.

\begin{proof}[Proof of Theorem \ref{thm_H}]
We start by recalling (as per Section \ref{sec_pars}) that $\alpha_G=n^{-10^{-8}}$, $p_{\gr}=n^{-10^{-25}}$ and $b=O(n^{-10^{-7}})$. From the random greedy matching process, with high probability we obtain $H_{\gr}$ such that the following all hold:
\begin{enumerate}[(i)]
\item every $\mathcal{T}$-valid subset $S \subseteq V(\mathcal{T})$ satisfies 
$$|V(H_{\gr}[S])|=(1 \pm 2b^{1/3}p_{\gr}^{-1})|S|p_{\gr},$$
\item for every $v \in V(H_{\gr})$ and every open or closed $\mathcal{T}$-valid tuple $(v,S_1,S_2,S_3)$, we have
$$|E_{H_{\gr}}(v,S_1, S_2, S_3)|=(1 \pm 2b^{1/3}p_{\gr}^{-3})|E_{\mathcal{T}}(v,S_1, S_2, S_3)|p_{\gr}^3,$$
\item for every $i \in [c_g]$, 
$$|\mathcal{Z}^+_{i,e,H_{\gr}}(\alpha, \beta, \gamma)|:=
\begin{cases}
(1 \pm 2b^{1/3}p_{\gr}^{-12})|\mathcal{Z}_{i,e,\mathcal{T}}^{+}(\alpha, \beta, \gamma)|p_{\gr}^{12} & \mbox{~if $e$ is a bad edge}, \\
O\left(k_it_1p_{\gr}^{12} \right) & \mbox{~if $\alpha=0$, $\beta=1$, $\gamma=3$},\\
0 & \mbox{~otherwise}.\\
\end{cases}
$$
$$|\mathcal{Z}^-_{i,e,H_{\gr}}(\alpha, \beta, \gamma)|:=
\begin{cases}
(1 \pm 2b^{1/3}p_{\gr}^{-12})|\mathcal{Z}_{i,e,\mathcal{T}}^{-}(\alpha, \beta, \gamma)|p_{\gr}^{12} & \mbox{~if $\alpha \neq 0$ and $\gamma=0$}, \\
O\left(j_ik_ip_{\gr}^{12} \right) & \mbox{~if $\alpha=0$, $\beta=0$, $\gamma=4$},\\
O\left(j_ik_ip_{\gr}^{12} \right) & \mbox{~if $\alpha=0$, $\beta=1$, $\gamma=3$},\\
0 & \mbox{~otherwise}.\\
\end{cases}
$$
Finally, for every bad edge $e$,
$$|\mathcal{Z}^2_{i,e,H_{\gr}}|=O\left(k_it_1p_{\gr}^{12}\right).$$
\item $|V_O^{J}(H_{\gr})|=(1 \pm 4b^{1/3}p_{\gr}^{-1})|V_E^{J}(H_{\gr})|$ for every $J \in \{X, Y, X+Y, X-Y\}$. Additionally, $|V_{O/E}^{J_1}(H_{\gr}[S])|=(1 \pm 4b^{1/3}p_{\gr}^{-1})|V_{O/E}^{J_2}(H_{\gr}[S])|$ for every valid layer interval $S$, and $J_1,J_2 \in \{X, Y, X+Y, X-Y\}$.
\item For every $J \in \{X, Y, X+Y, X-Y\}$ and every valid $J$-layer interval $I^J$ and $v \notin J$,
$$|E_{H_{\gr}}(v, I^J, O/E)|=(1 \pm 2b^{1/3}p_{\gr}^{-3})|E_\mathcal{T}(v, I^J, O/E)|p_{\gr}^3.$$
\end{enumerate}
Then, having reached $H_{\gr}$, we greedily remove a matching $M_{\gr}$ containing at most $2b^{1/3}n$ wrap-around edges avoiding $I_{t_0/4}$ from $H_{\gr}$, as discussed above, to obtain $H:=H_{\gr}[V(H_{\gr})\setminus V(M_{\gr})]$. We show that this graph satisfies the conditions of Theorem \ref{thm_H}. 

First note that the bounds listed for each property in relation to $H_{\gr}$ are upper bounds for the same properties in $H$, since we only {\it removed} vertices and edges. Since $S$ is $\mathcal{T}$-valid if and only if $S=I_{t_i}$ or $S=I_{t_i}\setminus I_{t_j}$ for some $i \in [0, c_h]$ and $j>i$. In particular, then, removing at most $8b^{1/3}n$ vertices from outside $I_{t_0/4}$ is affecting only $\mathcal{T}$-valid subsets which have size $\Theta(t_0p_{\gr})$. Then for such $S$ we have that $|V(H_{\gr}[S])| \geq |V(H[S])| \geq (1 \pm 2b^{1/3}p_{\gr}^{-1})|S|p_{\gr}-8b^{1/3}n \geq (1 \pm b^{1/3}p_{\gr}^{-1}\log(n))|S|p_{\gr}$. Since $\alpha_G \gg b^{1/3}p_{\gr}^{-1}\log(n)$, (i) holds.

Similarly for (ii), removing $8b^{1/3}n$ vertices from outside $I_{t_0/4}$ could remove at most $24b^{1/3}n$ edges containing a fixed vertex $v$ and these edges are only containing in subsets such that $|E_{H_{\gr}}(v, S_1, S_2, S_3)|=\Theta(np_{\gr}^3)$. Thus we again find that $|E_{H_{\gr}}(v, S_1, S_2, S_3)| \geq |E_{H}(v, S_1, S_2, S_3)|\geq (1 \pm b^{1/3}p_{\gr}^{-3}\log(n))|E_{\mathcal{T}}(v, S_1, S_2, S_3)|p_{\gr}^3$, where again $\alpha_G \gg b^{1/3}p_{\gr}^{-3}\log(n)$, so (ii) holds.

For zero-sum configurations, note that we are only interested in the changes to configurations containing a fixed bad edge or edge of type $(\alpha, \beta,0)_i$ with $\alpha \neq 0$ and for any $i \in [c_g]$. In these cases, the removal of $8b^{1/3}n$ vertices from outside $I_{t_0/4}$ could remove at most $O(b^{1/3}nj_i)$ zero-sum configurations for a bad edge $e$ and at most $O(b^{1/3}nt_1)$ configurations for an edge of type $(\alpha, \beta, 0)_i$ with $\alpha \neq 0$ for any $i$. By Fact \ref{fact_Z} each $i$-bad edge is in $\Theta(j_it_1)$ configurations and each edge of type $(\alpha, \beta, 0)_i$ with $\alpha \neq 0$ is in $\Theta(t_1^2)$ configurations for any $i$. Thus, as with (i) and (ii), we get that for every bad edge $e$ and for every edge of type $(\alpha, \beta, 0)_i$ with $\alpha \neq 0$ for any $i$ we have that
$$|\mathcal{Z}^+_{i,e,H}(\alpha, \beta, \gamma)|= (1 \pm b^{1/3}p_{\gr}^{-12}\log(n))|\mathcal{Z}_{i,e,\mathcal{T}}^{+}(\alpha, \beta, \gamma)|p_{\gr}^{12},$$
as required.
For (iv) note that this is similar to (i) though now a valid layer interval could have size $\alpha^2_Gn$, where $2b^{1/3}n$ vertices have been removed. Then in the worst case we have that $|V_{O}^J(H_{\gr}[S])| \geq |V_{O}^J(H[S])|\geq |V_{O}^J(H_{\gr}[S])|-2b^{1/3}n \geq |V_{O}^J(H_{\gr}[S])|(1-b^{1/3}p_{\gr}^{-1}\alpha_G^{-2}\log(n))$, where $\alpha_G \gg b^{1/3}p_{\gr}^{-1}\alpha_G^{-2}\log(n)$, satisfying (iv). 

Finally (v) follows by combining the arguments for (ii) and (iv) and noting that $\alpha_G^3 \gg b^{1/3}p_{\gr}^{-3}$. 
\end{proof}

%% file: ch_vortex.tex
\chapter{The iterative matching process} \label{ch_it_abs}

\section{Overview}

Recall that at this stage we have $\alpha_G:=n^{-10^{-8}}$ and $p_{\gr}:=n^{-10^{-25}}$, and have shown that with high probability the random greedy count leaves us with a graph $H$ as per Theorem \ref{thm_H}. Until Section \ref{sec_H}, we are concerned with the process starting from $H_1$ (two steps into the \emph{vortex} from $H$). Thus we start this chapter with a theorem stating the key properties of $H_1$. In fact, for the iterative matching process we consider weighted subgraphs and so need to track properties concerning a weighted subgraph $(H_1, w_1)$.

Recall from Section \ref{sec_pars} that we defined $\eta_1:=10^{-9}$ and $\epsilon_1:=\frac{\eta_1}{204800}$ (so that $\epsilon_1=\frac{\eta_1}{50L^2r^2}$ as in Theorem \ref{thm_weighted_egj}, with $r=4$ and $L=16$). We prove the following theorem in Section \ref{sec_H}.

\begin{theo}\label{thm_H_1}
Given $n$ is sufficiently large, there is a graph $H_1 \subseteq H$ such that $V(H_1) \subseteq I_{t_1}$ and $H_1$ has an almost-perfect fractional matching $w_1$ with the following properties:
\begin{enumerate}[(i)]
\item $d_{w_1, H_1}(v) \geq 1-t_1^{-\epsilon_1}$ for every $v \in V(H_1)$,
\item there exist absolute constants $0 <  c_{1,1} < c_{1,2}$ such that for every edge $e \in H_1$, we have that
$$\frac{c_{1,1}}{p_{\gr}^3t_1} \leq w_1(e)\leq \frac{c_{1,2}}{p_{\gr}^3t_1},$$
\item there exist absolute constants $0< c_{1,3} < c_{1,4}<1$ such that for every $1$-valid subset $S \subseteq V(\mathcal{T})$ we have that
$$c_{1,3}|S|p_{\gr} \leq |V(H_1[S])|\leq c_{1,4}|S|p_{\gr},$$
\item there exist absolute constants $0< c_{1,5} < c_{1,6}<1$ such that for every open or closed $1$-valid tuple $(v,S_1,S_2,S_3)$, 
$$c_{1,5}|S_1|p_{\gr}^3 \leq |E_{H_1}(v,S_1,S_2,S_3)|\leq c_{1,6}|S_1|p_{\gr}^3,$$
\item there exist absolute constants $0< c_{1,7} < c_{1,8}$ such that for every $i \in [c_g]$ the number of $i$-legal zero-sum configurations $\mathcal{Z}_{i,e,H_1}^{\pm}(\alpha, \beta, \gamma)$ containing the edge $e \in H_1$ of type $(\alpha, \beta, \gamma)_i$ with positive or negative sign respectively, satisfies 
$$|\mathcal{Z}^+_{i,e,H_1}(\alpha, \beta, \gamma)|\leq
\begin{cases}
c_{1,8}|\mathcal{Z}_{i,e,\mathcal{T}}^{+}(\alpha, \beta, \gamma)|p_{\gr}^{12} & \mbox{~if $e$ is a bad edge}, \\
c_{1,8}k_it_1p_{\gr}^{12} & \mbox{~if $\alpha=0$, $\beta=1$, $\gamma=3$},\\
0 & \mbox{~otherwise}.\\
\end{cases}
$$
$$|\mathcal{Z}^-_{i,e,H_1}(\alpha, \beta, \gamma)| \leq
\begin{cases}
c_{1,8}|\mathcal{Z}_{i,e,\mathcal{T}}^{-}(\alpha, \beta, \gamma)|p_{\gr}^{12} & \mbox{~if $\alpha \neq 0$ and $\gamma=0$}, \\
c_{1,8}j_ik_ip_{\gr}^{12} & \mbox{~if $\alpha=0$, $\beta=0$, $\gamma=4$},\\
c_{1,8}j_ik_ip_{\gr}^{12} & \mbox{~if $\alpha=0$, $\beta=1$, $\gamma=3$},\\
0 & \mbox{~otherwise}.\\
\end{cases}
$$
Additionally, for every bad edge $e$,
$$|\mathcal{Z}^2_{i,e,H_1}| \leq c_{1,8}k_it_1p_{\gr}^{12}.$$
Furthermore,
$$|\mathcal{Z}^+_{i,e,H_1}(\alpha, \beta, \gamma)|\geq
c_{1,7}|\mathcal{Z}_{i,e,\mathcal{T}}^{+}(\alpha, \beta, \gamma)|p_{\gr}^{12} \mbox{~~~~~if $e$ is a bad edge,}
$$
and
$$|\mathcal{Z}^-_{i,e,H_1}(\alpha, \beta, \gamma)| \geq
c_{1,7}|\mathcal{Z}_{i,e,\mathcal{T}}^{-}(\alpha, \beta, \gamma)|p_{\gr}^{12} \mbox{~~~~~~if $\alpha \neq 0$ and $\gamma=0$}.
$$
\item $|V_{O}^{X+Y}(H_1)|=|V_{O}^{X-Y}(H_1)|$ and $|V_{E}^{X+Y}(H_1)|=|V_{E}^{X-Y}(H_1)|$.
\end{enumerate}
\end{theo}

 Our aim in this chapter is to cover most of the vertices of $H$ by a matching leaving only a small subset $L^* \subseteq I_{n^{10^{-5}}}$ uncovered. Recall that an additional property required of $L^*$ is that $|V_O^{X+Y}(L^*)|=|V_O^{X-Y}(L^*)|$ and $|V_E^{X+Y}(L^*)|=|V_E^{X-Y}(L^*)|$. As discussed in Section \ref{sec_parity}, this condition is affected by use of wrap-around edges, and only affected when $n$ is odd. Hence, an additional constraint on the matching we find in $H$ is that the number of wrap-around edges using odd-even parity matches the number using even-odd parity in parts $X+Y$ and $X-Y$. Having reached $H$ which does satisfy these parity requirements, as the process to reach $L^*$ only removes disjoint matchings we only need to keep an eye on the parity conditions when $n$ is odd.

As described in Section \ref{sec_overview} we plan to obtain the required matching for $H$ over a sequence of nested subgraphs $H \supseteq H_0 \supseteq H_1 \supseteq \ldots \supseteq H_{c_h} \supseteq L^*$, where $H_i \subseteq I_{t_i}$. This process involves $\Theta(\log(n))$ subgraphs to reach $L^*$. It is natural to think that one might try to have a larger distance between vertices of the largest index in consecutive nested intervals, however the factor $1/c_{\vor} \sim 0.8$ needs to be sufficiently large to ensure the process works. In particular, the process to go from $H_i$ to $H_{i+1}$ requires one to find a matching to cover all vertices in $V(H_i) \setminus I_{t_{i+1}}$. We do this in two steps: firstly we find an `almost-cover', that is, we find a matching that covers all but a $o(1)$ proportion of the vertices which are present in $V(H_i) \setminus I_{t_{i+1}}$. Then we run a random greedy algorithm to cover the vertices which still remain uncovered in $V(H_i) \setminus I_{t_{i+1}}$ after the almost-cover. To enable the random greedy algorithm to run, we need that every vertex remaining in $V(H_i) \setminus I_{t_{i+1}}$ is in many edges that otherwise contain only vertices in $H_i[I_{t_{i+1}}]$. Since every edge that contains a vertex of index $|t|$ must also contain a distinct vertex with index of order at least $|t/2|$, we certainly don't want $c_{\vor} \leq 1/2$. Additionally, the `shape' of the sets $I_{t_{i}}$ is important to ensure that each vertex in $V(H_i)$ is in enough edges to ensure the random greedy algorithm does not abort. There is some flexibility to be had in both the shape of $I_{t_{i}}$ and the value of $c_{\vor}$ which ensure that the following arguments still work, but of the other possibilities, there is nothing to be gained by using a different choice.

\subsection{The weight shuffle} \label{sec_intro_reweight}

Recall at the end of Section \ref{sec_overview} we discuss the required {\it weight shuffle}, a process that shifts weight between edges preserving the weighted degree at each vertex. As part of the process to obtain the vortex $H \supseteq H_0 \supseteq H_1 \supseteq \ldots \supseteq H_{c_h} \supseteq L^*$, we obtain the almost-cover for $V(H_i)\setminus I_{t_{i+1}}$ using a {\it random matching tool} (see Section \ref{sec_random_matching_tool} for details), which uses a weighting $w_i$ for $H_i$, such that $w_i$ is an almost-perfect fractional matching for $H_i$. As part of the strategy, we obtain the almost-perfect fractional matching $w_{i+1}$ from $w_i$ for every $i \neq 1$. The weight shuffle intervenes in this process once we have $(H_1, w_1)$ and turns the almost-perfect fractional matching $w_1$ for $H_1$ into another almost-perfect fractional matching $w_{1^*}$ by transferring weight between edges via zero-sum configurations of specific forms as described in detail in Section \ref{sec_zs_shuffle}. The intention of the weight shuffle is to shift weight around in such a way that the vertices closer to the centre of $\mathcal{T}$ do not get used too early on in the process. In particular, as we obtain disjoint edges to cover vertices remaining in $I_{t_i}\setminus I_{t_{i+1}}$ for each subgraph $H_i$, if we use too many vertices from $I_{t_j}$ where $j\gg i$, then by the time we reach $I_{t_j}$ we may have that $H_j$ is too sparse to cover vertices in $I_{t_j}\setminus I_{t_{j+1}}$ by disjoint edges, and in this case the process would fail. The weight shuffle helps to avoid this as follows. It ensures that in our new weighting, $w_{1^*}$, every bad edge $e$ satisfies $w_{1^*}(e)=0$. This means that when covering vertices in an outer interval, we can never remove vertices too close to the centre of $\mathcal{T}$ at the same time. We start by ensuring that all of the weight on $c_g$-bad edges is shifted to edges of type $(\alpha,\beta,0)_{c_g}$ and $(0,\beta,\gamma)_{c_g}$. We define $w_1=w_{(1,c_g)}$ and let $w_{(1,i)}$ denote the current weighting on $H_1$ once the reweighting of $j$-bad edges for all $j >i$ has taken place. We shall use $i$-legal zero-sum configurations to reduce the weight on the $i$-bad edges to $0$ by transferring the weight to the edges with the opposite sign in such configurations. By definition of an $i$-legal zero-sum configuration, this will subsequently shift weight from, in addition to $i$-bad edges, edges of type $(0,1,3)_i$, and shift the weight onto edges of type $(\alpha,\beta,0)_i$ where $\alpha \neq 0$, $(0,0,4)_i$ and $(0,1,3)_i$. Then with the updated weighting $w_{(1,i-1)}$, we repeat the process to reduce the weight on all $(i-1)$-bad edges to $0$, and continue until every edge $e$ in $H_1$ for which there exists $l$ such that $e$ contains one vertex in $K_l$ and another in $I_{t_1}\setminus K_{l-1}$ has $w_{1^*}(e)=0$, and additionally, any edge $e$ with a vertex inside $K_1$ and a vertex outside $J_1$ also has $w_{1^*}(e)=0$. The details of the running of this process follow in Section \ref{sec_1}.

\subsection{Organisation}

In Section \ref{sec_random_matching_tool}, we present the key tool that we use to do the `almost-cover' at each step of the iterative matching process, a tool developed from a result of Ehard, Glock, Joos \cite{egj} which enables us at each step $i$, when we have reached $H_i$, to remove a matching covering most of the vertices remaining in $I_{t_i}\setminus I_{t_{i+1}}$ in such a way as to ensure that the remaining graph has `nice' random-like properties. In Section \ref{sec_1} we describe the process to reach $w_{1^*}$ and properties of $w_{1^*}$ given $(H_1, w_1)$ as in Theorem \ref{thm_H_1}. In Section \ref{sec_functions} we describe the details of how we shall use the random matching tool in our setting and introduce new notions of reachability and graph permissibility. In Section \ref{sec_i} we give details that show a general step of the process to get from $(H_i, w_i)$ to $(H_{i+1}, w_{i+1})$ for some $i \in [c_h]$ and show that the process can continue to reach $L^*$. Finally, in Section \ref{sec_H}, we show that a process very similar to that described in Section \ref{sec_i}, but adjusted for additional parity constraints, allows us to go from $H$ to $(H_1, w_1)$ as in Theorem \ref{thm_H_1}.

%% file: sec_prob_tools.tex
\subsection{The random matching tool} \label{sec_random_matching_tool}

One of the key tools used in our iterative absorption strategy is a generalisation of a result by Ehard, Glock and Joos \cite{egj}
which enables us to find a matching with random-like properties given an $r$-uniform hypergraph with particular conditions. This generalises a result of Kahn \cite{kahn}
who proved a similar result, but with more restrictive conditions on properties that could be tracked in relation to the matching.

\subsubsection{Weighted version}

For our purposes, we wish to consider a weighted $r$-uniform hypergraph $(H, w)$, where $w: E(H) \rightarrow \mathbb{R}_{\geq 0}$, and $d_{w}(v):= \sum_{e \ni v} w(e) \leq 1$ for every $v \in V(H)$ (so $w$ is also a fractional matching). We write $\Delta_w(H):= \max_{v \in V(H)} d_w(v)$ and letting $d_{w}(v,u):=\sum_{e \supseteq \{v,u\}} w(e)$, we also write $\Delta_w^{\co}(H):= \max_{\{v,u\} \in \binom{V(H)}{2}} d_w(v,u)$. Let $\delta_w(H):= \min_{v \in V(H)} d_w(v)$. 

Given a function $p_l:\binom{V(H)}{l} \rightarrow \mathbb{R}_{\geq 0}$, and a set $\bar{v}_k \in \binom{V(H)}{k}$, define 
$$p_l(\bar{v}_k):=\sum_{\bar{v} \supseteq \bar{v}_k: \bar{v} \in \binom{V(H)}{l}} p_l(\bar{v}) = \sum_{\bar{v}_{l-k} \in \binom{V(H) \setminus \bar{v}_k}{l-k}} p_l(\bar{v_k} \cup \bar{v}_{l-k})$$ 
for every $k \leq l$. Furthermore, for $M \subseteq E(H)$, we define $p_l(M):=\sum_{\bar{v} \in \binom{V(M)}{l}} p_l(\bar{v})$. Also for a function $q: E(H) \rightarrow \mathbb{R}_{\geq 0}$ we define $q(M)=\sum_{e \in M} q(e)$. Our key tool is the following:

\begin{theo} \label{thm_weighted_egj}
Suppose $\eta \in (0,1)$ and $r, L \in \mathbb{N}$ with $r \geq 2$. Let $\epsilon:=\frac{\eta}{50L^2r^2}$ and $\psi:=\frac{\eta}{6L^2r}$. Then there exists $\Delta_0$ such that for all $\Delta \geq \Delta_0$, the following holds. Let $(H, w)$ be an $r$-uniform weighted hypergraph with 
$w(e) \geq \Delta^{-1}$ for every $e \in H$, $\delta_w(H) \geq \frac{\Delta^{-\psi}}{1-\Delta^{-1}}$ and $\Delta_w^{\co}(H) \leq \Delta^{-\eta}$ as well as $e(H) \leq \exp(\Delta^{\epsilon^2/4})$.

Suppose that for each $l \in [L]$, we are given a set $\mathcal{P}_l$ of $l$-tuple weight functions on $V(H)$ of size polynomial in $\Delta$ 
such that 
\begin{equation} \label{eq_wgt_egj_hyp}
\max_{v \in V(H)} p_l(\{v\}) \leq f\Delta^{-2\eta}\sum_{\bar{v} \in \binom{V(H)}{l}} p_l(\bar{v})
\end{equation}
for all $p_l \in \mathcal{P}_l$, where $f \leq \frac{1}{2r^l}$.

Suppose further that $\mathcal{Q}$ is a set of weight functions on $E(H)$ of size 
polynomial in $\Delta$ such that 
\begin{equation} \label{eq_edge_egj}
q(E(H))\geq \frac{\Delta^{1+\eta}}{1-\Delta^{-1}} \max_{e \in E(H)} q(e)
\end{equation} 
for all $q \in \mathcal{Q}$.

Then there exists a matching $M$ in $H$ such that 
$$p_l(M)=(1 \pm \Delta^{-\epsilon})\sum_{\bar{v} \in \binom{V(H)}{l}} p_l(\bar{v})\prod_{v \in \bar{v}} d_w(v)$$
for every $l \in [L]$, and every $p_l \in \mathcal{P}_l$, and 
$$q(M)=\left(1 \pm \Delta^{-\epsilon}\right)\sum_{e \in E(H)} q(e)w(e)$$ 
for all $q \in \mathcal{Q}$.
\end{theo}

We shall show in Section \ref{sec_functions} that all properties we wish to keep track of (those relating to vertex subsets, degrees and zero-sum configurations) for the general iterative absorption strategy can be written as linear combinations of functions of vertices which will satisfy (\ref{eq_wgt_egj_hyp}). Additionally, in Section \ref{sec_H}, we show that a particular edge related parity requirement satisfies (\ref{eq_edge_egj}), where it is required to proceed with one of the two initial steps of the iterative matching process.

\subsubsection{Deriving Theorem \ref{thm_weighted_egj}}

We state the theorem of Ehard, Glock and Joos \cite{egj}, which is for functions on collections of edges in an unweighted graph, and then give a corollary in terms of a weighted $r$-graph. We then derive a version that is stated in terms of functions on collections of vertices, and finally derive Theorem \ref{thm_weighted_egj}.

In what follows, we write $\Delta(H)$ for the maximum vertex degree of a graph $H$, and $\Delta^{\co}(H)$ for the maximum pair degree. For $\bar{e}_l \in \binom{E(H)}{l}$ consisting of $l$ distinct edges from $E(H)$ we say that $\bar{e}_l$ is a {\it clean} $l$-set if the edges in $\bar{e}_l$ are pairwise disjoint. We write $\binom{E(H)}{l}'$ for the set of clean $l$-sets of edges in $H$, and for $A \subset E(H)$ we denote $\binom{E(H) \setminus A}{k}'$ to be the set of clean $k$-sets $\bar{e}_k$ of edges in $E(H)$ such that for every $e \in \bar{e}_k$, we have $e \cap a = \emptyset$ for every $a \in A$. That is, $\binom{E(H) \setminus A}{k}'$ denotes the set of clean $k$-sets such that every edge in a $k$-set is disjoint from $\bigcup_{a \in A} a$. In the theorem that follows, the properties that we consider are functions of clean $l$-sets of edges of the form $q: \binom{E(H)}{l}' \rightarrow \mathbb{R}_{\geq 0}$.

\begin{theo}[\cite{egj}] \label{thm_egj}
Suppose $\delta \in (0,1)$ and $r, L \in \mathbb{N}$ with $r \geq 2$, and let $\epsilon := \delta/50L^2r^2$. Then there exists $\Delta_0$ such that for all $\Delta \geq \Delta_0$, the following holds. Let $H$ be an $r$-uniform hypergraph with $\Delta(H) \leq \Delta$ and $\Delta^{\co}(H) \leq \Delta^{1-\delta}$ as well as $e(H) \leq \exp(\Delta^{\epsilon^2})$. Suppose that for each $l \in [L]$ we are given a set $\mathcal{Q}_l$ of clean $l$-set weight functions on $E(H)$ of size at most $\exp(\Delta^{\epsilon^2})$ such that 
\begin{equation} \label{eq_function_egj}
q(E(H))\geq \Delta^{k+\delta} \max_{T \in \binom{E(H)}{k}'} \sum_{S \supseteq T} q(S)
\end{equation} 
for all $q \in \mathcal{Q}_l$ and $k \in [l]$.

Then there exists a matching $M$ in $H$ such that 
$$q(M)=\left(1 \pm \frac{\Delta^{-\epsilon}}{2}\right)q(E(H))/\Delta^l$$ 
for all $l \in [L]$ and $q \in \mathcal{Q}_l$.
\end{theo}

This theorem shows that we can find matchings in hypergraphs that act similarly to how we would expect things to look if we had picked a set of edges by choosing each edge independently with the same probability $1/\Delta$. Whilst this result and Theorem \ref{thm_weighted_egj} describe the existence of a single (deterministic) matching satisfying the conclusion for a suitably small collection of functions $p$ and $q$, this is obtained by proving the existence of a distribution on matchings for which each statement holds with high probability, and taking a union bound suffices to show that such a matching exists for which all hold simultaneously. We use this fact throughout the proof, sometimes implicitly moving between the distribution on matchings for which a statement holds with high probability and a fixed matching which satisfies many statements simultaneously. 
We now derive an equivalent version for weighted $r$-graphs.

\begin{cor}\label{cor_egj_edge_weight}
Suppose $\delta \in (0,1)$ and $r, L \in \mathbb{N}$ with $r \geq 2$, and let $\epsilon:=\frac{\delta}{50L^2r^2}$. Then there exists $\Delta_0$ such that for all $\Delta \geq \Delta_0$, the following holds: Let $(H, w)$ be an $r$-uniform weighted hypergraph with 
$w(e) \geq \Delta^{-1}$ and $\Delta_w^{\co}(H) \leq \Delta^{-\delta}$ as well as $e(H) \leq \exp(\Delta^{\epsilon^2/4})$.

Let $q$ be a weight function on $E(H)$ such that 
\begin{equation}
q(E(H))\geq \frac{\Delta^{1+\delta}}{1-\Delta^{-1}} \max_{e \in E(H)} q(e).
\end{equation} 
Then there is a distribution on matchings $M$ in $H$ such that with high probability
$$q(M)=\left(1 \pm \Delta^{-\epsilon}\right)\sum_{e \in E(H)} q(e)w(e).$$ 
\end{cor}
\begin{proof}
From $(H, w)$ we define an unweighted multigraph $H'$ where $V(H')=V(H)$ and every edge $e \in E(H)$ appears in $E(H')$ with multiplicity $\lfloor Dw(e)\rfloor$, where $D=\Delta^2$. First note that $D w(e) \geq \Delta$, and so $\lfloor D w(e)\rfloor =(1 \pm \Delta^{-1})Dw(e)$. In addition $\Delta(H') = \max_{v \in V(H)} \sum_{e \ni v} \lfloor D w(e)\rfloor \leq D$, and 
$$\Delta^{\co}(H') = \max_{\{u, v\} \in \binom{V(H)}{2}} \sum_{e \supseteq \{u,v\}} \lfloor D w(e)\rfloor \leq D\cdot \Delta^{-\delta} = D^{1-\frac{\delta}{2}}.$$ 
Finally $e(H')\leq De(H) \leq D\exp(\Delta^{\epsilon^2/4}) \leq \exp(D^{\epsilon^2/4})$.
Now we define weight function $q': H' \rightarrow \mathbb{R}_{\geq 0}$ on the multigraph $H'$ via $q'(e)=q(e)$ for every $e \in H'$. 
Note first that $\max_{e \in E(H)} q(e)= \max_{e \in E(H')} q'(e)$ and also that given $q(E(H)) \geq \frac{\Delta^{1+\delta}}{1-\Delta^{-1}}\max_{e \in E(H)} q(e)$, we have that
\begin{eqnarray*}
q'(E(H'))&=&\sum_{e \in E(H')}q'(e) = \sum_{e \in E(H)} q(e)\lfloor{Dw(e)}\rfloor \\
&=&(1 \pm \Delta^{-1})\sum_{e \in E(H)} q(e) Dw(e) \\
&\geq& (1 - \Delta^{-1})\sum_{e \in E(H)} q(e) D/\Delta \\
&=& (1 - \Delta^{-1})\Delta\sum_{e \in E(H)} q(e) \\
&\geq& \Delta^{2+\delta}\max_{e \in H} q(e)
= D^{1+\delta/2}\max_{e \in H'} q'(e)
\end{eqnarray*}
It follows from Theorem \ref{thm_egj} with parameters $\delta'=\delta/2, r, L$, and assuming $\Delta$ sufficiently large, that there exists a matching $M$ in $H'$ such that
$$q'(M)=(1 \pm \frac{D^{-\epsilon/2}}{2}) \sum_{e \in H'} \frac{q'(e)}{D}.$$
Furthermore, $M$ is a matching in $(H, w)$ and, by construction, 
$$\sum_{e \in H'} \frac{q'(e)}{D}= \sum_{e \in H} \frac{q(e)}{D}\lfloor{Dw(e)}\rfloor= (1 \pm \Delta^{-1})\sum_{e \in H} q(e)w(e).$$
Hence
$$q(M)=(1 \pm \Delta^{-\epsilon}) \sum_{e \in H} q(e)w(e),$$
as claimed.
\end{proof}

Of course, we could have derived a weighted version that applies to all clean $l$-set weight functions, however we only require the weighted edge version for one specific application which relates only to individual edges, rather than to sets of edges. (This application comes up in the initial steps of the iterative matching process where parity constraints are still an issue -- see Section \ref{sec_parity_H_1}.) For most of the iterative matching process the properties we are concerned with tracking are those of vertices rather than edges. As such we now derive a vertex version of Theorem \ref{thm_egj}. In what follows, by abuse of notation, for $\bar{e} \in \binom{E(H)}{l}'$, and $\bar{v} \in \binom{V(H)}{l}$, we write $\bar{v} \in \bar{e}$ to mean that $\bar{v}$ is an $l$-set containing exactly one vertex from each edge in $\bar{e}$.

\begin{prop} \label{prop_vertex_function_egj}
Suppose that $H$ is an $r$-graph and  
$p_l:\binom{V(H)}{l} \rightarrow \mathbb{R}_{\geq 0}$ is a weight function on $l$-sets of vertices such that
\begin{equation} \label{eq_vtx_hyp}
\max_{v \in V(H)} p_l(\{v\}) \leq f\Delta^{-2\eta}\sum_{\bar{v} \in \binom{V(H)}{l}} p_l(\bar{v})
\end{equation}
for some fixed $l \in \mathbb{N}$ and $f=f(r,l) \leq \frac{1}{2r^l}$. Then,
\begin{equation} \label{eq_vtx_egj} 
\max_{\bar{v}_k \in \binom{V(H)}{k}} p_l(\bar{v}_k) \leq f\Delta^{-2\eta}\sum_{\bar{v} \in \binom{V(H)}{l}} p_l(\bar{v})
\end{equation}
for every $k \in [l]$.
\end{prop}
 
This leads us to a vertex version of Theorem \ref{thm_egj}.

\begin{theo} \label{thm_vertex_egj}
Suppose $\eta \in (0,1)$ and $r, L \in \mathbb{N}$ with $r \geq 2$, and let $\epsilon:=\frac{\eta}{50L^2r^2}$ and $\psi:=\frac{\eta}{5L^2r}$. Then there exists $\Delta_0$ such that for all $\Delta \geq \Delta_0$, the following holds. Let $H$ be an $r$-uniform hypergraph with $\Delta(H) \leq \Delta$, $\delta(H) \geq \Delta^{1-\psi}$ and $\Delta^{\co}(H) \leq \Delta^{1-\eta}$ as well as $e(H) \leq \exp(\Delta^{\epsilon^2})$. 
Suppose that for some $l \in [L]$ we have $p_l$ an $l$-set weight function on $V(H)$ such that 
\begin{equation} \label{eq_vtx_egj_hyp}
\max_{v \in V(H)} p_l(\{v\}) \leq f\Delta^{-2\eta}\sum_{\bar{v} \in \binom{V(H)}{l}} p_l(\bar{v}),
\end{equation} 
where $f \leq \frac{1}{2r^l}$ is fixed.

Then there exists a distribution on matchings $M$ in $H$ such that with high probability 
$$p_l(M)=\left(1 \pm \frac{2\Delta^{-\epsilon}}{3}\right)\sum_{\bar{v} \in \binom{V(H)}{l}} p_l(\bar{v})\prod_{v \in \bar{v}} \frac{d_H(v)}{\Delta}.$$
\end{theo}

\begin{proof}
We start by fixing $l \in [L]$ and defining $q_l:\binom{E(H)}{l}' \rightarrow \mathbb{R}_{\geq 0}$ via 
$$q_l(\bar{e})=\sum_{\bar{v} \in \bar{e}} p_l(\bar{v}).$$ 
\begin{claim}
$q_l$ satisfies $q_l(E(H))\geq \Delta^{k+\eta}\max_{\bar{e}_k \in \binom{E(H)}{k}'} q_l(\bar{e}_k)$ for all $k \in [l]$.
\end{claim}
\begin{proof}[Proof of Claim]
Fix $\bar{e}_k \in \binom{E(H)}{k}'$ and let $\bar{v}_{k,1}, \ldots, \bar{v}_{k, r^k}$ be an enumeration of the $r^k$ different $\bar{v}_k \in \bar{e}_k$. Then
\begin{eqnarray*}
q_l(\bar{e}_k) &=& \sum_{\bar{e}_{l-k} \in \binom{E(H) \setminus \bar{e}_k}{l-k}'} q_l(\bar{e}_k \cup \bar{e}_{l-k}) \\
&=& \sum_{\bar{e}_{l-k} \in \binom{E(H) \setminus \bar{e}_k}{l-k}'} \sum_{\bar{v} \in \bar{e}_k \cup \bar{e}_{l-k}} p_l(\bar{v}) \\
&=& \sum_{i \in [r^k]} \left( \sum_{\bar{e}_{l-k} \in \binom{E(H) \setminus \bar{e}_k}{l-k}'} \sum_{\bar{v}_{l-k} \in \bar{e}_{l-k}} p_l(\bar{v}_{k,i} \cup \bar{v}_{l-k}) \right) \\
&\leq& r^k \max_{i \in [r^k]}  \left( \sum_{\bar{e}_{l-k} \in \binom{E(H) \setminus \bar{e}_k}{l-k}'} \sum_{\bar{v}_{l-k} \in \bar{e}_{l-k}} p_l(\bar{v}_{k,i} \cup \bar{v}_{l-k}) \right) \\
&\leq& \Delta^{l-k}r^k \max_{\bar{v}_k \in \binom{V(H)}{k}} \sum_{\bar{v}_{l-k} \in \binom{V(H)}{l-k}} p_l(\bar{v}_k \cup \bar{v}_{l-k}),
\end{eqnarray*}
where the last inequality holds since for a fixed $\bar{v} \in \binom{V(H)}{l}$ over which the summation holds, we have that $\bar{v}=\bar{v}_k \cup \bar{v}_{l-k}$, where $\bar{v}_{l-k}$ is obtained from any $\bar{v}_{l-k} \in \bar{e}_{l-k}$ such that $\bar{e}_{l-k} \in \binom{E(H) \setminus \bar{e}_k}{l-k}'$. There are at most $\Delta^{l-k}$ sets $\bar{e}_{l-k} \in \binom{E(H) \setminus \bar{e}_k}{l-k}'$ which contain such a $\bar{v}_{l-k}$ yielding the inequality.

Furthermore, observe that 
$$q_l(E(H))\geq \sum_{\bar{v} \in \binom{V(H)}{l}} p_l(\bar{v})\prod_{v \in \bar{v}} \left(d_H(v) - r(l-1)\Delta^{\co}\right),$$
since $\bar{v}=\{v_1, v_2, \ldots, v_l\}$ is counted every time there is a disjoint $l$-set of edges $\{e_1, e_2, \ldots, e_l\}$ such that $v_i \in e_i$. There are $d_H(v_i)$ edges containing $v_i$, and for the choice of $(l-1)$ other edges that could be in $\bar{e}$, each one $e_j$ must contain $v_j$ but not share a vertex with any other $e_j$. In particular, fixing an $l$-set $\{v_1, v_2, \ldots, v_l\}$ and enumerating how many $\bar{e}$ count $\bar{v}$, the choices for $e_1$ are those such that $e_1 \ni v_1$ and $e_1$ does not contain $v_2, \ldots, v_l$. There are at least $d_H(v_1)-(l-1)\Delta^{\co}$ such choices for $e_1$. Then for each of these we can take all edges containing $v_2$ for $e_2$ except those containing a vertex from $e_1$ or a vertex in $\{v_3, \ldots, v_l\}$. Then there are at least $d_H(v_2)-(r+(l-2))\Delta^{\co}$ choices for $e_2$. Continuing this way, there are at least $d_H(v_j)-((j-1)r + (l-j)) \Delta^{\co} \geq d_H(v_j)-(l-1)r \Delta^{\co}$ choices for $e_j$ for every $j \in [l]$, thus giving the above inequality.
Then, since $d_H(v) \geq \Delta^{1-\psi}$, we have 
\begin{eqnarray*}
q_l(E(H))&\geq & \left(\Delta^{1-\psi}- r(l-1)\Delta^{\co} \right)^l \sum_{\bar{v} \in \binom{V(H)}{l}} p_l(\bar{v})  \\
&\geq & \Delta ^{l(1-\psi)} \left(1-\frac{r l^2}{\Delta^{\eta-\psi}} \right)\sum_{\bar{v} \in \binom{V(H)}{l}} p_l(\bar{v}).
\end{eqnarray*}
It follows that
\begin{eqnarray*}
\Delta^{k+\eta} q_l(\bar{e}_k)  &\leq & r^k \Delta^{l+\eta} \max_{\bar{v}_k \in \binom{V(H)}{k}} \sum_{\bar{v}_{l-k} \in \binom{V(H)}{l-k}} p_l(\bar{v}_k \cup \bar{v}_{l-k})  \\
&\leq & r^k \Delta^{l+\eta} f\Delta^{-2\eta} \sum_{\bar{v} \in \binom{V(H)}{l}} p_l(\bar{v}) \\
&\leq &  \frac{r^k\Delta^{l-\eta}}{2r^l\Delta^{l-l\psi}(1-\frac{rl^2}{\Delta^{\eta-\psi}})}q_l(E(H)) \\
&\leq &  \frac{2r^k\Delta^{l-\eta}}{2r^{l}\Delta^{l-l\psi}}q_l(E(H)) \\
&\leq &  \frac{1}{r^{l-k}\Delta^{\eta(1-\frac{l}{5L^2r})}}q_l(E(H))
\leq   q_l(E(H)),
\end{eqnarray*} 
using (\ref{eq_vtx_egj_hyp}) and Proposition \ref{prop_vertex_function_egj}, for every $k \in [l]$. Thus $\Delta^{k+\eta} \max_{\bar{e}_k} q_l(\bar{e}_k) \leq q_l(E(H))$, as required.
\end{proof}
So $H$ is a graph satisfying the hypotheses of Theorem \ref{thm_egj}, and 
$q_l$ satisfies the requirements of the theorem.  
It follows that there exists a distribution on matchings $M$ in $H$ such that with high probability
$$q_l(M)=\left(1 \pm \frac{\Delta^{-\epsilon}}{2}\right)\frac{q_l(E(H))}{\Delta^l}.$$

We have, from above, that
$$q_l(E(H))\geq \sum_{\bar{v} \in \binom{V(H)}{l}} p_l(\bar{v})\prod_{v \in \bar{v}} \left(d_H(v) - r(l-1)\Delta^{\co}\right),$$
and by similar arguments it is clear that
$$q_l(E(H))\leq \sum_{\bar{v} \in \binom{V(H)}{l}} p_l(\bar{v})\prod_{v \in \bar{v}} d_H(v),$$
so, in particular, since $\delta(H) \geq \Delta^{1-\psi}$, 
\begin{eqnarray*}
q_l(E(H)) &= & \sum_{\bar{v} \in \binom{V(H)}{l}} p_l(\bar{v})\prod_{v \in \bar{v}} \left(d_H(v)\left(1 \pm \frac{r(l-1)\Delta^{\co}}{\delta(H)}\right)\right) \\
& = & \left(1 \pm \frac{l^2r}{\Delta^{\eta-\psi}}\right) \sum_{\bar{v} \in \binom{V(H)}{l}} p_l(\bar{v})\prod_{v \in \bar{v}} d_H(v).
\end{eqnarray*}

Now for any matching $M$ we have that 
$$\sum_{\bar{v} \in \binom{V(M)}{l}} p_l(\bar{v}) = \sum_{\bar{e} \in \binom{M}{l}} q_l(\bar{e}) + \Gamma,$$
where $\Gamma$ sums $p_l$ over all of the $l$-sets of vertices in $V(M)$ which contain at least two vertices from one edge $e \in M$. 
This gives 
$$\Gamma = \sum_{\bar{v} \in \binom{V(M)}{l}} p_l(\bar{v}) - \sum_{\bar{e} \in \binom{M}{l}} q_l(\bar{e}) \leq \sum_{e \in M} \sum_{\bar{v}_2 \subseteq e, \bar{v}_{l-2} \in \binom{V(M)\setminus \bar{v}_2}{l-2}} p_l(\bar{v}_2 \cup \bar{v}_{l-2}),$$
and
\begin{equation}\label{eq_gamma}
\sum_{\bar{v} \in \binom{V(M)}{l}} p_l(\bar{v}) - \Gamma = \left(1 \pm \frac{\Delta^{-\epsilon}}{2}\right)\left(1 \pm \frac{l^2r}{\Delta^{\eta-\psi}}\right)\sum_{\bar{v} \in \binom{V(H)}{l}} p_l(\bar{v}) \prod_{v \in \bar{v}} \frac{d_H(v)}{\Delta}.
\end{equation}
\begin{claim} 
There exists a distribution on matchings $M$ in $H$ such that with high probability
$$\Gamma \leq \Delta^{-20Lr\epsilon}\sum_{\bar{v} \in \binom{V(H)}{l}} p_l(\bar{v}).$$
\end{claim}
\begin{proof}[Proof of Claim]
We follow the proof of Theorem \ref{thm_egj} from \cite{egj} which uses a result of Alon and Yuster \cite{alon_yuster}. They find $M$ as above by taking a random partition of the vertex set of $H$ into $V_1, \cdots V_t$, where $t=\Delta^{20Lr\epsilon}$. Write $H_i:=H[V_i]$. Then each $H_i$ is randomly (edge) partitioned into $H_{i1}, \ldots, H_{is}$, where $s=\Delta^{1-20(r-1+\frac{1}{4L})Lr\epsilon}$. In particular, $t$ is a very small power of $\Delta$, and $s$ is such that $\Delta/s$ is also a very small power of $\Delta$ (but larger than $t$). $M$ is constructed as follows: each $H_{ij}$ is partitioned into $q_j$ matchings and this dictates a partition of $H_i$ into $\prod_{j \in [s]} q_j$ matchings. Then $M$ is formed by, for each $i \in [t]$, independently choosing one of the partition matchings of $H_i$ uniformly at random from all such matchings of $H_i$. We write $M=\bigcup_{i \in [t]} M_i$, where $M_i$ is the matching in part $V_i$. Then letting $\Gamma_i:=\sum_{e \in M_i} \sum_{\bar{v}_2 \subseteq e, \bar{v}_{l-2} \in \binom{V(M)\setminus \bar{v}_2}{l-2}} p(\bar{v}_2 \cup \bar{v}_{l-2})$ we have that $\Gamma \leq \sum_{i \in [t]} \Gamma_i$ and 
$$\Gamma_i \leq \max_j \sum_{e \in H_i} \mathbbm{1}_{e \in H_{ij}}\cdot \mathcal{P}_{e,2},$$
where $\mathcal{P}_{e,2}:= \sum_{\bar{v}_2 \subseteq e, \bar{v}_{l-2} \in \binom{V(M)\setminus \bar{v}_2}{l-2}} p_l(\bar{v}_2 \cup \bar{v}_{l-2})$.
Furthermore,
$$\mathbb{E}(\Gamma_i) \leq \frac{1}{s} \sum_{e \in H_i} \mathcal{P}_{e,2}.$$
Since the indicator variables $\mathbbm{1}_{e \in H_{ij}}$ are independent, we may use Bernstein's inequality (Lemma \ref{bernstein}) to bound each $\Gamma_i$ whp. Write 
$$Z_e:= \mathbbm{1}_{e \in H_{ij}} \mathcal{P}_{e,2}$$ 
for each $e \in H_i$, so $\{Z_e\}_{e \in H_i}$ are independent variables taking the value 
$$
Z_e=
\begin{cases}
\mathcal{P}_{e,2} & \mbox{~with probability $1/s$, and} \\
0 & \mbox{~otherwise}.\\
\end{cases}
$$
Furthermore, using (\ref{eq_vtx_egj_hyp}), note that for all $e$,
$$|Z_e| \leq \mathcal{P}_{e,2} \leq r^2\max_{v \in e} p(\{v\}) \leq \frac{1}{2r^{l-2}\Delta^{2\eta}}\sum_{\bar{v} \in \binom{V(H)}{l}} p(\bar{v}).$$ 
It follows that 
$$\sum_{e \in H_i} \mathbb{E}(Z_e^2) \leq \max_{e \in H_i} |Z_e| \sum_{e \in H_i} \mathbb{E}(Z_e) \leq \left( \frac{1}{2r^{l-2}\Delta^{2\eta}}\sum_{\bar{v} \in \binom{V(H)}{l}} p(\bar{v})\right)\mathbb{E}(\Gamma_i).$$ 
In addition we have that every $\bar{v} \in \binom{V(H)}{l}$ appears in $\sum_{e \in H_i} \mathcal{P}_{e,2}$ at most when any of the $\binom{l}{2}$ pairs from $\bar{v}$ are in an edge $e \in H_i$, and so it follows that 
$$\sum_{e \in H_i} \mathcal{P}_{e,2} \leq \binom{l}{2}\Delta^{1-\eta}\sum_{\bar{v} \in \binom{V(H)}{l}} p_l(\bar{v}).$$ 
Hence, using Bernstein's inequality with parameter $\frac{\Delta^{1-\eta}}{s}\sum_{\bar{v} \in \binom{V(H)}{l}} p_l(\bar{v})$, we get that whp 
\begin{multline*}
\Gamma_i < \frac{1}{s} \binom{l}{2}\Delta^{1-\eta}\sum_{\bar{v} \in \binom{V(H)}{l}} p_l(\bar{v}) + \frac{\Delta^{1-\eta}}{s}\sum_{\bar{v} \in \binom{V(H)}{l}} p_l(\bar{v})
\leq \frac{l^2\Delta^{1-\eta}}{s}\sum_{\bar{v} \in \binom{V(H)}{l}} p(\bar{v}),
\end{multline*}  
for every $i$, and so whp 
$$\Gamma<\frac{tl^2\Delta^{1-\eta}}{s}\sum_{\bar{v} \in \binom{V(H)}{l}} p(\bar{v})< \Delta^{-20Lr\epsilon}\sum_{\bar{v} \in \binom{V(H)}{l}} p_l(\bar{v}).$$
\end{proof}
Now, since $\prod_{v \in \bar{v}} \frac{d_H(v)}{\Delta} \geq \left(\frac{\delta(H)}{\Delta}\right)^l \geq \Delta^{-\frac{\eta}{5Lr}}$, we have that 
$$\Gamma \leq \Delta^{-20Lr\epsilon+\frac{\eta}{5Lr}} \sum_{\bar{v} \in \binom{V(H)}{l}} p_l(\bar{v}) \prod_{v \in \bar{v}} \frac{d_H(v)}{\Delta}=\Delta^{-10Lr\epsilon} \sum_{\bar{v} \in \binom{V(H)}{l}} p_l(\bar{v}) \prod_{v \in \bar{v}} \frac{d_H(v)}{\Delta}.$$ 
Thus using (\ref{eq_gamma}), 
\begin{eqnarray*}
\sum_{\bar{v} \in \binom{V(M)}{l}} p_l(\bar{v}) &=& \left(1 \pm \frac{\Delta^{-\epsilon}}{2}\right)\left(1 \pm \left( \frac{l^2r}{\Delta^{\eta-\psi}} + \Delta^{-10Lr\epsilon} \right)\right)\sum_{\bar{v} \in \binom{V(H)}{l}} p_l(\bar{v}) \prod_{v \in \bar{v}} \frac{d_H(v)}{\Delta} \\
&=& \left(1 \pm \frac{2\Delta^{-\epsilon}}{3}\right)\sum_{\bar{v} \in \binom{V(H)}{l}} p_l(\bar{v}) \prod_{v \in \bar{v}} \frac{d_H(v)}{\Delta},
\end{eqnarray*}
as required.
\end{proof}

Finally we prove Theorem \ref{thm_weighted_egj}.

\begin{proof}[Proof of Theorem \ref{thm_weighted_egj}] 
By Corollary \ref{cor_egj_edge_weight} we know that there exists a distribution on matchings $M$ in $H$ such that the theorem holds with high probability for each $q \in \mathcal{Q}$ individually. The same distribution on matchings $M$ in $H$ is used to obtain Theorem \ref{thm_vertex_egj}. Thus we prove the theorem by showing that whp the conclusion of the theorem holds for any individual $l \in [L]$ and $p_l \in \mathcal{P}_l$. We prove this via precisely the same strategy as that for proving Corollary \ref{cor_egj_edge_weight}, but now appealing to Theorem \ref{thm_vertex_egj}. In particular, from $(H, w)$ we define an unweighted multigraph $H'$ where $V(H')=V(H)$ and every edge $e \in E(H)$ appears in $E(H')$ with multiplicity $\lfloor Dw(e)\rfloor$, where $D=\Delta^2$. Recall, as in the proof of Corollary \ref{cor_egj_edge_weight}, that $D w(e) \geq \Delta$ yielding $\lfloor D w(e)\rfloor =(1 \pm \Delta^{-1})Dw(e)$, $\Delta(H') = \max_{v \in V(H)} \sum_{e \ni v} \lfloor D w(e)\rfloor \leq D$, and 
$$\Delta^{\co}(H') = \max_{\{u, v\} \in \binom{V(H)}{2}} \sum_{e \supseteq \{u,v\}} \lfloor D w(e)\rfloor \leq D\cdot \Delta^{-\eta} = D^{1-\frac{\eta}{2}}.$$ 
Also as in the proof of Corollary \ref{cor_egj_edge_weight}, $e(H')\leq De(H) \leq D\exp(\Delta^{\epsilon^2/4}) \leq \exp(D^{\epsilon^2/4})$. In addition, we have that 
$$\delta(H') \geq \min_{v \in V(H)} \sum_{e \ni v} \lfloor{Dw(e)}\rfloor \geq (1-\Delta^{-1})D\delta_w(H) \geq D^{1-\psi/2}.$$
Now for a fixed $l \in [L]$ and function $p_l$ we define weight function $p'_l: \binom{V(H')}{l} \rightarrow \mathbb{R}_{\geq 0}$ via
$$p'_l(\bar{v})=p_l(\bar{v})$$
for every $\bar{v} \in \binom{V(H')}{l}$. 
Then 
$$\sum_{\bar{v} \in \binom{V(H')}{l}} p'_l(\bar{v}) \geq 2r^lD^{\eta} \max_{v \in V(H')} p'_l(\{v\}).$$
It follows from Theorem \ref{thm_vertex_egj} with parameters $\eta/2, r, L$, and assuming $\Delta$ sufficiently large, that with high probability $M$ in $H'$ satisfies
$$p'_l(M)=(1 \pm \frac{2D^{-\epsilon/2}}{3}) \sum_{\bar{v} \in \binom{V(H')}{l}} p'_l(\bar{v}) \prod_{v \in \bar{v}} \frac{d_{H'}(v)}{D}.$$
Furthermore, $M$ is a distribution on matchings in $(H, w)$ and, by construction, $\frac{d_{H'}(v)}{D} = (1 \pm \Delta^{-1})d_w(v)$. Hence
$$p_l(M)=(1 \pm \Delta^{-\epsilon}) \sum_{\bar{v} \in \binom{V(H)}{l}} p_l(\bar{v}) \prod_{v \in \bar{v}} d_{w}(v).$$
Taking a union bound yields that there exists a matching $M$ in $H$ such that the conclusion holds simultaneously for polynomially many $q \in \mathcal{Q}$ and $p_l \in \mathcal{P}_l$ for each $l \in [L]$, proving the theorem.
\end{proof}

%% file: sec_1.tex
\section{The weight shuffle} \label{sec_1}

In this section we list some key properties of $(H_1,w_1)$ and describe the process to reach $(H_{1^*}, w_{1^*})$, as well as listing out the key properties of $(H_{1^*}, w_{1^*})$ resulting from this process. From $(H_1, w_1)$ as given in Theorem \ref{thm_H_1}, we perform the weight shuffle as described in Section \ref{sec_intro_reweight}. We give more details of this process in the following algorithm.

\begin{alg} \label{alg_reweighting}
~

$i=1$, $w_{(1,c_g)}:=w_1$

{\bf Initialise:} $i=1$, $w_{(1,c_g)}:=w_1$, where $w_{(1,c_g-i+1)}$ is an almost-perfect fractional matching for $H_1$ such that $w_{(1,c_g-i+1)}(e)=0$ for all $(c_g-j+1)$-bad edges and all $j < i$.

{\bf Step 1:} Find all $(c_g-i+1)$-bad edges with weight $w_{(1,c_g-i+1)} \neq 0$. For each $z \in \mathcal{Z}^2_{c_g-i+1, H_1}$, transfer the weight $w(z):=\frac{1}{p_{\gr}^{15}j_{c_g-i+1}t_1^2},$ 
from the edges all with the same sign as any $(c_g-i+1)$-bad edge to edges with the opposite sign.

{\bf Step 2:} For each $(c_g-i+1)$-bad edge $e$ and $z_e \in \mathcal{Z}^+_{c_g-i+1,e,H_1}(\bad) \setminus Z^2_{c_g-i+1, H_1}$, define
$$w(z_e):=\frac{w_{(1,c_g-i+1)}(e)-\sum_{z \in  \mathcal{Z}^2_{c_g-i+1,e, H_1}} w(z)}{|\mathcal{Z}^+_{c_g-i+1,e,H_1}(\bad) \setminus \mathcal{Z}^2_{c_g-i+1, H_1}|},$$ 
and transfer the weight $w(z_e)$ from the edges all with the same sign as $e$ to edges with the opposite sign.

{\bf Step 3:} Define $w_{(1,c_g-(i+1)+1)}$ to be the weighting on $H_1$ resulting from Steps 1 and 2.

{\bf Step 4:} If $i=c_g$, stop. Else, take $i:=i+1$ and go to Step 1.

\end{alg}

Note that if all edges retain non-negative weight throughout, then since our algorithm keeps the same total weight at each vertex after every iteration, and $w_1$ is an almost-perfect fractional matching for $H_1$, the new weighting must also be an almost-perfect fractional matching for $H_1$. 

An added complication of the process described above which motivates the definition for $\mathcal{Z}^2_{i,H_1}$ and which prevents us from dividing the weight for each $i$-bad edge $e$ evenly among all $i$-legal zero-sum configurations (and instead leads us to use a two step process to divide up the weight) is that some $i$-legal zero-sum configurations contain more than one $i$-bad edge, and so we cannot ensure that the weight on all $i$-bad edges is simultaneously reduced to precisely $0$ in a straightforward manner. One alternative would be to define $i$-legal zero-sum configurations only to include those with exactly one $i$-bad edge. The issue with this strategy is that our algorithm, as well as reducing the weight on $i$-bad edges to $0$, substantially increases the weight on edges of types $(\alpha,\beta,0)_i$ with $\alpha \neq 0$. If we restricted $i$-legal zero-sum configurations only to include those with exactly one $i$-bad edge, the only edges of type $(\alpha,\beta,0)_i$ which would gain a substantial increase in weight would be those of type $(1,3,0)_i$, and the weight on an edge $e$ of type $(\alpha,\beta,0)_i$ with $\alpha \geq 2$ would remain very close to $w_1(e)$ in the reweighting process. Whilst such a discrepancy would not necessarily be an issue, it is more straightforward to manage the subsequent arguments in the iterative matching process if we are able to ensure that all edges with vertices spread between two adjacent sections $K_i$ and $K_{i+1}$ for some $i$ all have weight of the same order, which is what Algorithm \ref{alg_reweighting} will achieve.

Given $w_{1^*}$, we define $H_{1^*}$ to be the graph with vertex set $V(H_{1^*})=V(H_1)$ and edge set $E(H_{1^*})=\{e\in E(H_1): w_{1^*}(e) \neq 0\}$. That is, $H_{1^*} \subseteq H_1$ on the same vertex set, with all edges removed from $H_1$ which have weight $0$ according to the weight function $w_{1^*}$. Before stating the properties that $(H_{1^*}, w_{1^*})$ will have resulting from running Algorithm \ref{alg_reweighting} (where, by slight abuse, $w_{1^*}$ is, in this context, considered restricted to the edges of $H_{1^*}$), we introduce some additional notation in terms of a general graph $H_i$, since this will be useful for describing properties that we wish to track in subsequent steps, as well as for describing key properties of $(H_{1^*}, w_{1^*})$. The motivation for these definitions is to give a sensible restriction to the definitions of {\it valid} subsets, pairs and tuples based on the weight shuffle. In particular, since the weight shuffle will reduce the weight on all $i$-bad edges to $0$, (or equivalently we shall think of any edge that has weight $0$ as a non-edge, and so we eliminate all $i$-bad edges), it follows that a vertex $v \in K_i \setminus K_{i+1}$ is only in edges completely contained in $K_{i-1} \cup K_i \cup K_{i+1}$. We have two stages of terminology to deal with this. Firstly, {\it permissibility} will allow us a sensible restriction of properties of valid sets, pairs and tuples now that many pairs and tuples which are valid in $H_1$ would no longer be valid in $H_{1^*}$. Secondly, {\it reachability} (see Definition \ref{def_reach}) will restrict the notion of permissibility to sets, pairs and tuples which can be affected at a certain step of the iterative matching process. We remind the reader that the definitions for `valid' sets can be found in Section \ref{sec_H_details}.

\begin{defn}[Depth]\label{def_depth}
We say that $H_i$ has {\it depth $j$} if $V(H_i) \subseteq K_{j-1}$ but $V(H_i)\not\subseteq K_{j}$. We say that $v$ has {\it depth $j$} if $v \in K_{j-1}\setminus K_j$, and $S \subseteq V(\mathcal{T})$ has {\it depth $j$} if $S \subseteq K_{j-2} \setminus K_j$ and $S \cap K_{j-1} \neq \emptyset$.
\end{defn} 
Note that $V(H_{1^*}) \subseteq K_{-1}$ and $V(H_{1^*}) \not\subseteq K_{0}$, so $H_{1^*}$ has depth $0$. Furthermore, recall that for every $j \in [c_g]$ there exists $i \in [c_h]$ such that $I_{t_i}=K_j$ (recalling Set-up \ref{vortex} and how $c_{\vor}$ was defined such that $-\frac{\log\log(n)}{\log(c_{\vor})}$ is an integer). As a result this will give that $H_i$ and $V(H_i)$ have the same depth.   

\begin{defn}[Permissible sets]
We say that $(v,S)$ is a {\it closed $i$-permissible pair} if $(v,S)$ is a closed $i$-valid pair and there exists $l \in [0,c_g]$ such that $v$ has depth $l$ and $V(E_{\mathcal{T}}(v, S,S,S)) \subseteq K_{l-2} \setminus K_{l+1}$.  
We also say that $(v,S)$ is an {\it open $i$-permissible pair} if $(v,S)$ is an open $i$-valid pair and there exists $l \in [0,c_g]$ such that $v$ has depth $l$ and $V(E_{\mathcal{T}}(v, S,*,*)) \subseteq K_{l-2} \setminus K_{l+1}$. We say that something is {\it permissible} if there exists $i \in [0,c_h]$ such that it is $i$-permissible.
\end{defn}

Note that by nature of $(v,S)$ being an $i$-valid pair, we implicitly have that $v \in I_{t_i}$ and $S \subseteq I_{t_i}$ and, since $v$ has depth $l$, $v \in K_{l-1}\setminus K_l$ and in $H_{1^*}$ all edges containing $v$ are in $K_{l-2}\setminus K_{l+1}$. Using the properties of $(H_1, w_1)$ as given in Theorem \ref{thm_H_1}, we show that Algorithm \ref{alg_reweighting} does not abort prematurely and leads to $(H_{1^*}, w_{1^*})$ with properties as listed in the following theorem:

\begin{theo} \label{thm_special_reweight}
From running Algorithm \ref{alg_reweighting} we obtain the weighting $w_{1^*}:=w_{(1,0)}$, which is an almost-perfect fractional matching for $H_1$ and $H_{1^*}$, (such that $E(H_{1^*})$ does not contain any bad edges and all edges which are not bad in $H_1$ are contained in $E(H_{1^*})$), with the following properties: 
\begin{enumerate}[(i)]
\item $d_{w_{1^*}, H_{1^*}}(v) \geq 1-t_1^{-\epsilon_1}$ for every $v \in V(H_1)$,
\item there exist absolute constants $0<c^*_{1,1} < c_{1,2}^*$ such that for every $i \in [c_g]$ and every edge $e_i$ of type $(\alpha,\beta,0)_i$ with $\alpha \neq 0$, 
$$\frac{c^*_{1,1}}{p_{\gr}^3k_i\log(n)}\leq w_{1^*}(e_i) \leq \frac{c^*_{1,2}}{p_{\gr}^3k_i\log(n)},$$
and edges $e$ in $H_1$ with all vertices outside $K_1$ satisfy 
$$\frac{c^*_{1,1}}{p_{\gr}^3t_1}\leq w_{1^*}(e) \leq \frac{c^*_{1,2}}{p_{\gr}^3t_1},$$
\item for every $1$-valid subset $S \subseteq \mathcal{T}$ we have that $|V(H_{1^*}[S])|=|V(H_1[S])|$,
\item for every open or closed $1$-permissible pair $(v,S)$ (given by a tuple $(v,S,S,S)$ or $(v,S,*,*)$), we have that $|E_{H_{1^*}}(v,S)|=|E_{H_{1}}(v,S)|$.
\end{enumerate} 
\end{theo}

Theorem \ref{thm_special_reweight} confirms that Algorithm \ref{alg_reweighting} not only reduces the weight on all bad edges to $0$, but also shows that no other edges in $H_1$ have their weight reduced to $0$, and more specifically gives a fairly precise window for the weight of each type of edge. (Note that an edge of type $(0,4,0)_i$ is an edge of type $(4,0,0)_{i+1}$, so every type of edge really is considered by Theorem \ref{thm_special_reweight}.) Furthermore, note that Theorem \ref{thm_special_reweight}(iii) and (iv) are both trivial consequences of the definition of $H_{1^*}$. We leave them in the statement of the theorem 
as we shall have results in subsequent sections that follow a similar shape to that of Theorem \ref{thm_special_reweight}, but where the conditions on the size of vertex subsets and degree-type properties are not a trivial consequence in the same way they are here. We describe the process of Algorithm \ref{alg_reweighting} inductively to deduce that it does not abort prematurely, and produces a weighting $w_{1^*}$ satisfying the claims of Theorem \ref{thm_special_reweight}.

In preparation for the proof of Theorem \ref{thm_special_reweight}, recalling Fact \ref{fact_Z} about zero-sum configurations in $\mathcal{T}$, we state the following corollary about these configurations in $H_1$. 

\begin{cor} \label{cor_Z}
For every $i \in [c_g]$ the following hold:
$$|\mathcal{Z}^+_{i,e,H_1}(\alpha, \beta, \gamma)|:=
\begin{cases}
\Theta\left(j_it_1p_{\gr}^{12}\right) & \mbox{~if $e$ is a bad edge}, \\
O\left(k_it_1p_{\gr}^{12} \right) & \mbox{~if $\alpha=0$, $\beta=1$, and $\gamma=3$},\\
0 & \mbox{~otherwise}.\\
\end{cases}
$$
$$|\mathcal{Z}^-_{i,e,H_1}(\alpha, \beta, \gamma)|:=
\begin{cases}
\Theta\left(t_1^2p_{\gr}^{12}\right) & \mbox{~if $\alpha \neq 0$ and $\gamma=0$}, \\
O\left(j_ik_ip_{\gr}^{12} \right) & \mbox{~if $\alpha=0$, $\beta=0$, and $c=4$},\\
O\left(j_ik_ip_{\gr}^{12} \right) & \mbox{~if $\alpha=0$, $\beta=1$, and $\gamma=3$},\\
0 & \mbox{~otherwise}.\\
\end{cases}
$$
Finally, for every bad edge $e$,
$$\mathcal{Z}^2_{i,e,H_1}=O\left(k_it_1p_{\gr}^{12}\right).$$
\end{cor}
\begin{proof}
The statement follows immediately from Fact \ref{fact_Z} and Theorem \ref{thm_H_1}~(v).
\end{proof}

Unpacking Algorithm \ref{alg_reweighting}, note that the $i^{th}$ iteration of the algorithm shifts weight {\it from}  only those edges which are $(c_g-i+1)$-bad or of type $(0, 1, 3)_{c_g-i+1}$. Observe that $(c_g-i+1)$-bad edges are then not touched in any further iterations of the algorithm, and edges of type $(0, 1, 3)_{c_g-i+1}$ are only considered in one more iteration. In particular, edges of type $(0, 1, 3)_{c_g-i+1}$ are of type $(1, \beta, \gamma)_{c_g-i}$, and by nature of having a vertex in $K_{c_g-i}$ are no longer in play after step $i+1$ of the algorithm (since by definition they are not of type $(\alpha, \beta, \gamma)_j$ for any combination of $\alpha, \beta, \gamma$ with $j < c_g-i$). Thus when considered this last time, they are either $(c_g-i)$-bad, in which case their weight is reduced to $0$ in the $i+1$ iteration of the algorithm, or they are of type $(1,3,0)_{c_g-i}$ in which case they {\it gain} weight. The key point here is that any edge can only {\it lose} weight in at most two steps of the process, and in a step where this is not reducing the weight on a bad edge to $0$ (which can be at most one of the two steps), we'll show that the edge $e$ loses weight at most $O\left(\frac{1}{p_{\gr}^3t_1\log(n)}\right)$. Thus, since until this point the edge only gains weight or remains at weight $w_1(e)$ as a result of previous iterations of the algorithm, we have that it retains weight at least $w_1(e)(1-O(\log^{-1}(n)))$ before becoming either a bad edge in the next step, or an edge of type $(1,3,0)$, and so in particular the algorithm never aborts due to the weight on an edge becoming negative.

As per Steps 1 and 2 of Algorithm \ref{alg_reweighting}, for each $z \in \mathcal{Z}^2_{i, H_1}$, we define 
$$w(z):=\frac{1}{p_{\gr}^{15}j_it_1^2},$$ 
and for each $i$-bad edge $e$ and $z_e \in \mathcal{Z}^+_{i,e, H_1}(\bad) \setminus \mathcal{Z}^2_{i, H_1}$, we define
$$w(z_e):=\frac{w_{(1,i)}(e)-\sum_{z \in  \mathcal{Z}^2_{i,e, H_1}} w(z)}{|\mathcal{Z}^+_{i,e,H_1}(\bad) \setminus \mathcal{Z}^2_{i, H_1}|}.$$
The following proposition about the weight transferred over legal zero-sum configurations will be useful in proving the induction step.

\begin{prop}\label{prop_w_z_gen}
Suppose we have reached iteration $i$ of Algorithm \ref{alg_reweighting} (where we transfer weight from all $(c_g-i+1)$-bad edges in $H_1$), and we have that for every $(c_g-i+1)$-bad edge $e$,
\begin{equation}\label{eq_w1g}
w_1(e)(1-O(\log^{-1}(n))) \leq w_{(1,c_g-i+1)}(e) \leq w_1(e) + \sum_{j=1}^{i-1} O\left(\frac{k_{c_g-j+1}}{p_{\gr}^3t_1^2}\right).
\end{equation}
Then every $(c_g-i+1)$-legal zero-sum configuration $z \in \mathcal{Z}_{c_g-i+1, H_1}$ carries a weight $w(z)$ satisfying 
$$w(z)=\Theta\left(\frac{1}{p_{\gr}^{15}j_{c_g-i+1}t_1^2}\right).$$
\end{prop}
\begin{proof}
It is clear by construction that this is the case for $z \in \mathcal{Z}^2_{c_g-i+1, H_1}$. Consider now an edge $e_{c_g-i+1}$ which is a $(c_g-i+1)$-bad edge, and a configuration $z \in \mathcal{Z}^+_{c_g-i+1,e_{c_g-i+1}, H_1}(\bad)$. First note that $i \leq c_g$ and hence for every iteration we have that $\sum_{j=1}^{i-1} O\left(\frac{k_{c_g-j+1}}{p_{\gr}^3t_1^2}\right) \leq O\left(\frac{c_gk_2}{p_{\gr}^3t_1^2}\right)$. Since $c_g = \lceil{\frac{0.99999\log(t_1)}{\log\log(n)}}\rceil$, and $k_2=\frac{t_1}{\log^3(n)}$, we get $\sum_{j=1}^{i-1} O\left(\frac{k_{c_g-j+1}}{p_{\gr}^3t_1^2}\right)= O\left(\frac{1}{p_{\gr}^3t_1\log^2(n)\log\log(n)}\right)$, and in particular that 
$$w_{(1,c_g-i+1)}(e)= \left(1 \pm O\left(\log^{-1}(n)\right)\right)w_1(e).$$

Note that by Corollary \ref{cor_Z} and Theorem \ref{thm_H_1} we have that 
\begin{eqnarray*}
\sum_{z \in  \mathcal{Z}^2_{c_g-i+1, e_{c_g-i+1}, H_1}} w(z) &=& O\left(k_{c_g-i+1}t_1p_{\gr}^{12}\cdot \frac{1}{p_{\gr}^{15}j_{c_g-i+1}t_1^2}\right)  \\
&=& O\left(\frac{1}{p_{\gr}^3t_1\log(n)}\right)=O(w_{1}(e_{c_g-i+1})\log^{-1}(n)).
\end{eqnarray*} 
It follows, again using Corollary \ref{cor_Z}, as well as the assumption on $w_{(1,c_g-i+1)}(e)$ for each $(c_g-i+1)$-bad edge above, that
\begin{eqnarray*}
w(z) &=& \frac{w_{1, c_g-i+1}(e_{c_g-i+1})-\sum_{z \in  \mathcal{Z}^2_{c_g-i+1, e_{c_g-i+1}, H_1}} w(z)}{|\mathcal{Z}^+_{c_g-i+1, e_{c_g-i+1},H_1}(\bad) \setminus \mathcal{Z}^2_{c_g-i+1, H_1}|} \\
&=& \Theta\left(\frac{1}{p_{\gr}^3t_1}\cdot \frac{1}{p_{\gr}^{12}j_{c_g-i+1}t_1}\right) \\
&=& \Theta\left(\frac{1}{p_{\gr}^{15}j_{c_g-i+1}t_1^2}\right),
\end{eqnarray*}
as required.
\end{proof}

We now consider $w_{(1,c_g-i)}$ for every $i \in [c_g]$.

\begin{lemma}\label{lem_general_it}
For every $i \in [c_g]$ the following holds:
\begin{enumerate}[(i)]
\item $w_{(1,c_g-i)}(e)=0$ for every $l$-bad edge, where $l>c_g-i$,
\item $w_{(1,c_g-i)}(e)=\Theta\left(\frac{1}{p_{\gr}^3k_l \log(n)}\right)$ for every edge of type $(\alpha,\beta,0)_l$ with $\alpha \neq 0$ and $l>c_g-i$,
\item $w_1(e)(1-O(\log^{-1}(n))) \leq w_{(1,c_g-i)}(e) \leq w_1(e) + \sum_{j=0}^{i-1} O\left(\frac{k_{c_g-j}}{p_{\gr}^3t_1^2}\right)$ for every edge of type $(\alpha, \beta, \gamma)_{c_g-i}$ with $\alpha \neq 0$,
\item $w_1(e) \leq w_{(1,c_g-i)}(e) \leq w_1(e)+ \sum_{j=0}^{i-1} O\left(\frac{k_{c_g-j}}{p_{\gr}^3t_1^2}\right)$ for every edge $e$ of type $(0, \beta, \gamma)_{c_g-i}$.
\end{enumerate}
\end{lemma}
Before proving the lemma, notice that every edge type is considered. In particular for any $i \in [c_g]$, and for any edge $e \in H_1$, either $e$ is of type $(\alpha, \beta, \gamma)_{c_g-i}$, for some combination of $\alpha, \beta, \gamma$ so that it is considered in statements (iii) and (iv), or it has at least one vertex in $K_{c_g-i+1}$, in which case it is of type $(\alpha, \beta, \gamma)_l$ for some $l>c_g-i$ and with $\alpha \neq 0$, and then it is considered in either statement (i) or (ii).
\begin{proof}
We prove this lemma by induction. Note that the base case, $i=1$ is obtained from running the first iteration of Algorithm \ref{alg_reweighting}. By construction, running Algorithm \ref{alg_reweighting} reduces the weight on all $c_g$-bad edges to $0$, so (i) holds. Additionally, every $c_g$-bad edge satisfies (\ref{eq_w1g}), so by Proposition \ref{prop_w_z_gen} every zero-sum configuration over which weight is transferred in the first iteration of the algorithm carries weight $\Theta\left(\frac{1}{p_{\gr}^{15}j_{c_g}t_1^2}\right)$. Apart from bad edges, whose weight is reduced to precisely $0$, the only edges to {\it lose} weight are those of type $(0,1,3)_{c_g}$. By Corollary \ref{cor_Z}, it follows that such an edge loses weight at most $O\left(\frac{k_{c_g}t_1p_{\gr}^{12}}{p_{\gr}^{15}j_{c_g}t_1^2}\right)=O\left(\frac{1}{p_{\gr}^{3}t_1\log(n)}\right)=O\left(w_1(e)\log^{-1}(n)\right)$. Edges that gain weight are those of types $(\alpha, \beta, 0)_{c_g}$ with $\alpha \neq 0$, $(0,1,3)_{c_g}$ and $(0,0,4)_{c_g}$, and all other edges retain weight $w_{(1, c_g)}=w_1$. By Corollary \ref{cor_Z}, we get that for $e$ of type $(\alpha, \beta, 0)_{c_g}$ with $\alpha \neq 0$ obtains weight $w_{c_g-1}(e)=w_1(e)+\Theta\left(\frac{1}{p_{\gr}^3k_{c_g}\log(n)}\right)=\Theta\left(\frac{1}{p_{\gr}^3k_{c_g}\log(n)}\right)$, so (ii) holds. Similarly, an edge $e$ of type $(0,0,4)_{c_g}$ satisfies $w_{c_g-1}(e)=w_1(e)+O\left(\frac{k_{c_g}}{p_{\gr}^3t_1^2}\right)$, and we have that an edge $e$ of type $(0,1,3)_{c_g}$ satisfies $w_1(e)(1-O(\log^{-1}(n)) \leq w_{c_g-1}(e) \leq w_1(e)+O\left(\frac{k_{c_g}}{p_{\gr}^3t_1^2}\right)$. Then note that every edge of type $(\alpha, \beta, \gamma)_{c_g-1}$ with $\alpha \neq 0$ is of type $(0, \beta, \gamma)_{c_g}$ with $\beta \neq 0$. We saw above that (iii) holds for all such edges. Finally, every edge of type $(0,\beta, \gamma)_{c_g-1}$ is an edge of type $(0,0,4)_{c_g}$ so (iv) holds. Having proved the lemma true for $i=1$ it is not difficult to extend to all $i \in [c_g]$. Indeed, suppose the lemma is true for all $i \leq i^*$ for some $i^* \leq c_g-1$. So
\begin{enumerate}[(i)]
\item $w_{(1,c_g-i^*)}(e)=0$ for every $l$-bad edge, where $l>c_g-i^*$,
\item $w_{(1,c_g-i^*)}(e)=\Theta\left(\frac{1}{p_{\gr}^3k_l \log(n)}\right)$ for $e$ of type $(\alpha,\beta,0)_l$ with $\alpha \neq 0$ and $l>c_g-i^*$,
\item $w_1(e)(1-O(\log^{-1}(n))) \leq w_{(1,c_g-i^*)}(e) \leq w_1(e) + \sum_{j=0}^{i^*-1} O\left(\frac{k_{c_g-j}}{p_{\gr}^3t_1^2}\right)$ for $e$ of type $(\alpha, \beta, \gamma)_{c_g-i^*}$ with $\alpha \neq 0$,
\item $w_1(e) \leq w_{(1,c_g-i^*)}(e) \leq w_1(e)+ \sum_{j=0}^{i^*-1} O\left(\frac{k_{c_g-j}}{p_{\gr}^3t_1^2}\right)$ for $e$ of type $(0, \beta, \gamma)_{c_g-i^*}$.
\end{enumerate}
Now to reach $w_{(1,c_g-(i^*+1))}$ we again run one iteration of Algorithm \ref{alg_reweighting}. By assumption, we have that all $(c_g-i^*)$-bad edges satisfy (\ref{eq_w1g}) and so by Proposition \ref{prop_w_z_gen} every zero-sum configuration used to transfer weight over the iteration has weight $\Theta\left(\frac{1}{p^{15}_{\gr}j_{c_g-i^*}t_1^2}\right)$. Note that running the iteration of Algorithm \ref{alg_reweighting} to get from $w_{(1,c_g-i^*)}$ to $w_{(1,c_g-(i^*+1))}$ leaves edges of type $(\alpha, \beta, \gamma)_i$ with $\alpha \neq 0$ and $i>c_g-i^*$ undisturbed. Furthermore, the iteration by construction attains $w_{(1, c_g-(i^*+1))}$ so that every $(c_g-i^*)$-bad edge $e$ has weight $w_{(1, c_g-(i^*+1))}(e)=0$, so (i) holds for $i^*+1$. Additionally, by Corollary \ref{cor_Z}, edges of type $(\alpha, \beta, 0)_{c_g-i^*}$ with $\alpha \neq 0$ gain weight $\Theta\left(\frac{1}{p_{\gr}^3k_{c_g-i^*}\log(n)}\right)$ in this step of the algorithm. Before running the iteration, by induction such an edge $e$ satisfies (iii), where $\sum_{j=0}^{i^*-1} O\left(\frac{k_{c_g-j}}{p_{\gr}^3t_1^2}\right)=o(w_1(e))$. So after running the iteration we get that $w_{(1, c_g-(i^*+1))}(e)=(1 \pm o(1))w_1(e)+\Theta\left(\frac{1}{p_{\gr}^3k_{c_g-i^*}\log(n)}\right)=\Theta\left(\frac{1}{p_{\gr}^3k_{c_g-i^*}\log(n)}\right)$for every edge of type $(\alpha, \beta, 0)_{c_g-i^*}$, so (ii) holds. To see (iii) and (iv), note that every edge $e$ of type $(\alpha, \beta, \gamma)_{c_g-(i^*+1)}$ is of type $(0, \beta, \gamma)_{c_g-i^*}$, and so by induction $w_1(e) \leq w_{(1, c_g-i^*)}(e) \leq w_1(e)+ \sum_{j=0}^{i^*-1} O\left(\frac{k_{c_g-j}}{p_{\gr}^3t_1^2}\right)$. In running the iteration of the algorithm, again using Corollary \ref{cor_Z}, such an edge gains weight at most $O\left(\frac{k_{c_g-i^*}}{p_{\gr}^3t_1^2}\right)$, so the upper bounds for (iii) and (iv) are satisfied. Finally to see the lower bounds, note that $e$ could only have lost weight if it is of type $(0,1,3)_{c_g-i^*}$. Such an edge is of type $(1, \beta, \gamma)_{c_g-(i^*+1)}$, so (iv) holds. Also, such an edge $e$ loses weight at most $O(w_1(e)\log^{-1}(n))$, and since it was of type $(0, \beta, \gamma)_{c_g-i^*}$, by induction (iii) also holds, completing the proof.
\end{proof}

In particular, Lemma \ref{lem_general_it} tells us that Algorithm \ref{alg_reweighting} completes. Indeed, by construction it could only fail if at some iteration we caused some edges to have negative weight, as then the new weighting would not be an almost-perfect fractional matching for $H_1$. Furthermore, taking $w_{1^*}:=w_{(1,0)}$, we have from Lemma \ref{lem_general_it} that
\begin{equation} \label{eq_w1*}
w_{(1,0)}(e)=
\begin{cases}
0 & \mbox{~for every $l$-bad edge, where $l>0$}, \\
\Theta\left(\frac{1}{p_{\gr}^3k_l \log(n)}\right) & \mbox{~for every edge of type $(\alpha, \beta,0)_l$} \\ 
& \mbox{~where $\alpha \neq 0$ and $l>0$},\\
w_1(e)\left(1 \pm o(1))\right) & \mbox{~for every edge of type $(0, \beta, \gamma)_{1}$}.\\
\end{cases}
\end{equation}
It remains to prove Theorem \ref{thm_special_reweight} which follows almost immediately.

\begin{proof}[Proof of Theorem \ref{thm_special_reweight}]
Assuming Theorem \ref{thm_H_1}, we have by construction of Algorithm \ref{alg_reweighting}, and observing from Lemma~\ref{lem_general_it} and the discussion following it, that the Algorithm does not abort, and it follows immediately that $(i)$, $(iii)$ and $(iv)$ all hold. By (\ref{eq_w1*}), which considers every edge in $H_1$, we have immediately that (ii) holds, as required. 
\end{proof}

Having proved Theorem \ref{thm_special_reweight}, we note the following corollary, that will be useful to note for the remainder of the iterative matching process.

\begin{cor}\label{cor_delta_const}
For every $j \in [c_h]$ and every $v \in H_{1^*}[I_{t_{j}}]$, 
$$\frac{w_{1^*}(E_{H_{1^*}}(v, I_{t_{j+1}}))}{w_{1^*}(E_{H_{1^*}}(v, I_{t_{j}}))}=\Theta(1) \leq 1.$$
\end{cor}
\begin{proof}
Firstly, using the fact that $E_{H_{1^*}}(v, I_{t_{j+1}}) \subseteq E_{H_{1^*}}(v, I_{t_{j}})$, it follows trivially that $\frac{w_{1^*}(E_{H_{1^*}}(v, I_{t_{j+1}}))}{w_{1^*}(E_{H_{1^*}}(v, I_{t_{j}}))} \leq 1$. Suppose that $I_{t_j}$ has depth $l$. If $v \in K_{l+1}$, then $\frac{w_{1^*}(E_{H_{1^*}}(v, I_{t_{j+1}}))}{w_{1^*}(E_{H_{1^*}}(v, I_{t_{j}}))}=1$. Suppose $v \in K_l \setminus K_{l+1}$. Then every edge in $H_{1^*}$ containing $v$ has weight either $\Theta\left(\frac{1}{p_{\gr}^3k_l\log(n)}\right)$ or $\Theta\left(\frac{1}{p_{\gr}^3k_{l+1}\log(n)}\right)$, and every edge containing $v$ in $H_{1^*}[I_{t_{j}}]$ which is {\it not} in $H_{1^*}[I_{t_{j+1}}]$ has weight $\Theta\left(\frac{1}{p_{\gr}^3k_l\log(n)}\right) \ll \Theta\left(\frac{1}{p_{\gr}^3k_{l+1}\log(n)}\right)$. We know from Theorems \ref{thm_H_1}  and \ref{thm_special_reweight} that there are $O\left(t_jp_{\gr}^3\right)$ such edges, and that there are $\Theta\left(t_{j+1}p_{\gr}^3\right)$ edges containing $v$ in $H_{1^*}[I_{t_{j+1}}]$. Thus since each edge in $H_{1^*}[I_{t_{j+1}}]$ has weight at least of the same order as those with a vertex in $I_{t_j} \setminus I_{t_{j+1}}$, the claim holds.

It remains to consider $v \in K_{l-1}\setminus K_l$. Then, in $H_{1^*}[I_{t_j}]$, every edge containing $v$ has weight either $\Theta\left(\frac{1}{p_{\gr}^3k_{l-1}\log(n)}\right)$ or $\Theta\left(\frac{1}{p_{\gr}^3k_{l}\log(n)}\right)$, and those with weight $\Theta\left(\frac{1}{p_{\gr}^3k_{l}\log(n)}\right)$ are all included in the numerator. The argument follows through in the same way as for the previous case. 
\end{proof}

%% file: sec_functions.tex
\section{Using the random matching tool} \label{sec_functions}

In this section we shall describe the graphs for which we wish to use Theorem \ref{thm_weighted_egj}, and show that all the properties we wish to track in such a graph can be described as linear combinations of functions satisfying (\ref{eq_wgt_egj_hyp}). We shall then see in the subsequent sections that each graph to which we wish to apply Theorem \ref{thm_weighted_egj} satisfies the necessary hypotheses. For $H$ and then for each $i \in [0,c_h]$ we wish to use Theorem \ref{thm_weighted_egj} on a weighted subgraph $(H_i^o, w_i^o)$ of $(H_i, w_i)$, in such a way as to obtain a matching $M_i^o$, and show that $H_i[V(H_i)\setminus V(M_i^o)]$ has `nice' properties with regard to valid subsets of $V(\mathcal{T})$, as well as degree-type and zero-sum configuration conditions. We are only required to keep track of zero-sum configurations until we reach $(H_1, w_1)$, as these are only required for the weight shuffle.

\subsection{Key properties as functions}

As seen in the list of properties given for $(H_1, w_1)$ in Theorem \ref{thm_H_1}, the properties we are keen to track are those relating to numbers of vertices in given subsets of $V(\mathcal{T})$, degree-type conditions, and initially also zero-sum configurations.

In this section we describe the functions and how they will be useful, assuming that they satisfy (\ref{eq_wgt_egj_hyp}). In the section that follows we show that, given that a graph $(H_i, w_i)$ satisfies various properties, the functions relating to the properties we wish to track in $(H_i, w_i)$ do indeed satisfy (\ref{eq_wgt_egj_hyp}).

\subsubsection{Number of vertices and weighted functions on vertices remaining in an interval} \label{sec_func_vert}

For a fixed (hyper)graph $G$, $S \subseteq V(\mathcal{T})$, and for some weight function $f:V(G) \rightarrow \mathbb{R}_{\geq 0}$, let $p_S: V(G) \rightarrow \mathbb{R}_{\geq 0}$ be given by $p_S(v):=f(v)\mathbbm{1}_{v \in V(G[S])}$. Note that we suppress $f$ in our notation, but wherever used, $f$ will be clear from the context. Then for a matching $M$,
$p_S(M)=\sum_{v \in V(M)} p_S(v)$ yields the sum of the weights $f(v)$ on vertices in $V(G[S])$ which are in $V(M)$. In particular, for $f=1$ (i.e. $f(v)=1$ for every $e \in V(G)$), $p_S(M)$ counts the number of vertices in $V(G[S])$ which are in $M$. Hence, the weight $f$ on vertices remaining in $S$ after removing $V(M)$ from $G$ is $\sum_{v \in V(G[S])} f(v)-p_S(M)$ and in particular the number of vertices remaining in $S$ after removing $V(M)$ from $G$ is $|V(G[S])|-p_S(M)$.

Given that $p_S$ satisfies (\ref{eq_wgt_egj_hyp}) we get from Theorem \ref{thm_weighted_egj} that there exists a matching $M$ in $G$ such that
\begin{equation} \label{eq_psm}
p_S(M)=(1 \pm \Delta^{-\epsilon})\sum_{v \in V(G)} p_S(v)d_{w,G}(v)=(1 \pm \Delta^{-\epsilon})\sum_{v \in V(G[S])} f(v)d_{w,G}(v).
\end{equation} 
It follows that the weight remaining in $V(G[S])$ once $M$ is removed is
\begin{equation} \label{eq_func_weighted_vert}
\sum_{v \in V(G[S])} f(v)-p_S(M) =\sum_{v \in V(G[S])} f(v)(1-d_{w,G}(v)) \pm \Delta^{-\epsilon}\sum_{v \in V(G[S])} f(v)d_{w,G}(v),
\end{equation}
and in particular that the number of vertices remaining in $V(G[S])$ is
\begin{equation} \label{eq_func_vert}
|V(G[S])|-p_S(M) = \sum_{v \in V(G[S])} (1-d_{w,G}(v)) \pm \Delta^{-\epsilon}\sum_{v \in V(G[S])} d_{w,G}(v).
\end{equation}

\subsubsection{Degree-type properties} \label{sec_func_deg}

Let $G \subseteq \mathcal{T}$, $S_1, S_2, S_3 \subseteq V(\mathcal{T})$ and $v \in V(G)$. For a function $f_v: E(G) \rightarrow \mathbb{R}_{\geq 0}$ such that $f_v(e)=0$ if $v \notin e$, we describe $f_v(E_{G}(v, S_1, S_2, S_3)):=\sum_{e \in E_{G}(v, S_1, S_2, S_3)} f_v(e)$ in terms of functions on vertices of $G$. In particular taking $f_v=1$, we'll describe $|E_{G}(v, S_1, S_2, S_3)|$ in terms of functions on vertices of $G$. Given $f_v: E(G) \rightarrow \mathbb{R}_{\geq 0}$, we let $f_v:\binom{V(G)}{\leq 3} \rightarrow \mathbb{R}_{\geq 0}$ be defined by 
$$f_v(\bar{u})=\begin{cases}
\sum_{e \supseteq \{v\} \cup \bar{u}} f_v(e) & \mbox{~for~} v \notin \bar{u}, \\
0 & \mbox{~otherwise.}\\
\end{cases}
$$
Note that when $n$ is odd, or $G \subseteq \mathcal{T}[I_{t_1}]$ for any $n$, $\sum_{e \supseteq \{v\} \cup \bar{u}} f_v(e)$ sums over only one edge, since these graphs have maximum pair-degree $1$. In the remaining case (when $n$ is even and $G \not\subseteq \mathcal{T}[I_{t_1}]$), the maximum pair-degree is $2$.

Now we define the following function on $V(G)$:
$$q^{(1)}_{(G,v,S_1,S_2,S_3)}(u) := 
\begin{cases}
f_v(\{u\})\mathbbm{1}_{\{\{v,u\} \subseteq e \in E_{G}(v, S_1, S_2, S_3)\}} & \mbox{~for~} v \neq u, \\
0 & \mbox{~for~} v=u. \\
\end{cases}
$$

Then for a matching $M$ in $G$, $q^{(1)}_{(G,v,S_1,S_2,S_3)}(V(M))$ certainly includes the weight for each edge $e \in E_{G}(v, S_1, S_2, S_3)$ such that $v \in e$ and $(e \setminus \{v\}) \cap V(M) \neq \emptyset$, but for such an edge $\{v,u,w,x\}$, if both $u$ and $w$ are in $M$, then the weight on the edge will be counted twice. Hence this function alone does not allow us to count precisely how the degree of a vertex into a particular subgraph of $G$ relates to $M$, and we need to modify for over-counting. Thus we define more generally
$$q^{(l)}_{(G,v,S_1,S_2,S_3)}(\bar{u}) = 
\begin{cases}
f_v(\bar{u})\mathbbm{1}_{\{\bar{u} \cup \{v\} \subseteq e \in E_{G}(v, S_1, S_2, S_3)\}} & \mbox{~for~} v \notin \bar{u}, \\
0 & \mbox{~for~} v \in \bar{u}, \\
\end{cases}
$$
for $l \in [3]$. Note that for every $G \subseteq \mathcal{T}$, even when $n$ is even, the degree of any set of $3$ vertices in $V(G)$ is at most one. Then considering
$$q^{(1)}_{(G,v,S_1,S_2,S_3)}(V(M))-q^{(2)}_{(G,v,S_1,S_2,S_3)}\left(\binom{V(M)}{2}\right)+q^{(3)}_{(G,v,S_1,S_2,S_3)}\left(\binom{V(M)}{3}\right),$$ 
for some $v \notin V(M)$, the weight of each edge $e \in E_{G}(v, S_1, S_2, S_3)$ which intersects $V(M)$ in more than one vertex is only counted once by the linear combination of functions. This yields that for $v \notin V(M)$
\begin{multline*}
f_v(E_{G[V(G)\setminus V(M)]}(v, S_1, S_2, S_3)) = f_v(E_{G}(v, S_1, S_2, S_3)) - \\
\Big(q^{(1)}_{(G,v,S_1,S_2,S_3)}(V(M))-q^{(2)}_{(G,v,S_1,S_2,S_3)}\left(\binom{V(M)}{2}\right) + q^{(3)}_{(G,v,S_1,S_2,S_3)}\left(\binom{V(M)}{3}\right)\Big).
\end{multline*}
From now on, when it is arbitrary or clear from the context, we write $q^{(l)}$ in place of $q^{(l)}_{(G,v,S_1,S_2,S_3)}$.

\begin{prop} \label{prop_degs}
Suppose that $G$ and $q^{(1)}, q^{(2)}, q^{(3)}$ all satisfy the hypotheses of Theorem \ref{thm_weighted_egj} whp. Then 
\begin{multline*}f_v(E_{G[V(G)\setminus V(M)]}(v, S_1, S_2, S_3))=\sum_{e \in E_{G}(v, S_1, S_2, S_3)} f_v(e) \prod_{u \in e \setminus\{v\}}(1-d_{w,G}(u)) \\
\pm O\left(\Delta^{-\epsilon}\max_{u \in \bigcup E_{G}(v, S_1, S_2, S_3) \setminus \{v\}} d_{w,G}(u) \sum_{e \in E_{G}(v, S_1, S_2, S_3)} f_v(e)\right).
\end{multline*}
\end{prop}
\begin{proof}
We have that 
\begin{eqnarray*}
f_v(E_{G[V(G)\setminus V(M)]}(v, S_1, S_2, S_3)) &=& f_v(E_{G}(v, S_1, S_2, S_3)) - (q^{(1)}-q^{(2)} + q^{(3)}).
\end{eqnarray*}
By Theorem \ref{thm_weighted_egj}, we have that
$$q^{(l)}\left(\binom{V(M)}{l}\right)=\left(1 \pm \Delta^{-\epsilon}\right)\sum_{\bar{v} \in \binom{V(G)}{l}} q^{(l)}(\bar{v})\prod_{u \in \bar{v}} d_{w,G}(u),$$
for every $l \in [3]$. We claim that
\begin{multline}\label{claim_deg_func}
\sum_{e \in E_{G}(v, S_1, S_2, S_3)} f_v(e)\prod_{u \in e \setminus\{v\}}(1-d_{w,G}(u)) = \\ f(E_{G}(v, S_1, S_2, S_3)) - \sum_{u \in V(G)} q^{(1)}(u)d_{w,G}(u) \\ +\sum_{\{\bar{v}\} \in \binom{V(G)}{2}} q^{(2)}(\bar{v})\prod_{u \in \bar{v}} d_{w,G}(u) - \sum_{\bar{v} \in \binom{V(G)}{3}} q^{(3)}(\bar{v})\prod_{u \in \bar{v}} d_{w,G}(u).
\end{multline}
Indeed, consider an edge $\{v,u,w,x\} \in E_{G}(v, S_1, S_2, S_3)$. It contributes 
$$f_v(e)(1-d_{w,G}(u))(1-d_{w,G}(w))(1-d_{w,G}(x))$$ 
to the LHS. To the RHS it contributes 
\begin{multline*}
f_v(e)\big(1-(d_{w,G}(u)+d_{w,G}(w)+d_{w,G}(x)) +(d_{w, G}(u)d_{w,G}(w) \\ +d_{w,G}(u)d_{w,G}(x) +d_{w,G}(w)d_{w,G}(x))-d_{w, G}(u)d_{w,G}(w)d_{w,G}(x)\big) \\
\equiv f_v(e)(1-d_{w,G}(u))(1-d_{w,G}(w))(1-d_{w,G}(x)).
\end{multline*} 
Furthermore, there are no contributions to either the LHS or the RHS other than these. It follows that
\begin{multline*}
f(E_{G[V(G)\setminus V(M)]}(v, S_1, S_2, S_3)) =
\Big(f(E_{G}(v, S_1, S_2, S_3)) - \\ \sum_{\bar{v} \in V(G)} q^{(1)}(\bar{v})\prod_{u \in \bar{v}} d_{w,G}(u)
+ \sum_{\bar{v} \in \binom{V(G)}{2}} q^{(2)}(\bar{v})\prod_{u \in \bar{v}} d_{w,G}(u) -\sum_{\bar{v} \in \binom{V(G)}{3}} q^{(3)}(\bar{v})\prod_{u \in \bar{v}} d_{w,G}(u)\Big) \\
\pm \Delta^{-\epsilon}\Big(\sum_{u \in V(G)} q^{(1)}(u)d_{w,G}(u) +\sum_{\bar{v} \in \binom{V(G)}{2}} q^{(2)}(\bar{v})\prod_{u \in \bar{v}}d_{w,G}(u) \\
+ \sum_{\bar{v} \in \binom{V(G)}{3}} q^{(3)}(\bar{v})\prod_{u \in \bar{v}}d_{w,G}(u)\Big),
\end{multline*}
which by (\ref{claim_deg_func}) gives
\begin{multline*}f_v(E_{G[V(G)\setminus V(M)]}(v, S_1, S_2, S_3))=\sum_{e \in E_{G}(v, S_1, S_2, S_3)} f_v(e) \prod_{u \in e \setminus\{v\}}(1-d_{w,G}(u)) \\
\pm \Delta^{-\epsilon}\Big(\sum_{u \in V(G)} q^{(1)}(u)d_{w,G}(u) +\sum_{\bar{v} \in \binom{V(G)}{2}} q^{(2)}(\bar{v})\prod_{u \in \bar{v}}d_{w,G}(u) \\ + \sum_{\bar{v} \in \binom{V(G)}{3}} q^{(3)}(\bar{v})\prod_{u \in \bar{v}}d_{w,G}(u)\Big).
\end{multline*}
Noting that $q^{(l)}(\bar{v})\prod_{u \in \bar{v}}d_{w,G}(u)$ is only non-zero for $O(|E_{G}(v, S_1, S_2, S_3)|)$ elements $\bar{v} \in \binom{V(G)}{l}$ for every $l \in [3]$ and that $d_{w,G}(u) \leq 1$ for every $u \in V(G)$ completes the proof.
\end{proof}

\subsubsection{Number of zero-sum configurations} \label{sec_zsfunc}

We wish to run counts on specific types of zero-sum configuration for the weight shuffle. Now, just as with degree-type properties, where we were interested in counting the number of edges relating in some way to specific subsets of the vertex set, we shall wish to do the same for zero-sum configurations containing some fixed edge $e$. Define
$$f^{(l)}_{Z^{\pm}_{i,e, G}(\alpha, \beta, \gamma)}(\bar{v}):=
\begin{cases}
\sum_{z \in Z^{\pm}_{i,e, G}(\alpha, \beta, \gamma)} \mathbbm{1}_{\{\bar{v} \subseteq z\}} & \mbox{~for~} \bar{v} \cap e = \emptyset, \\
0 & \mbox{~otherwise}, \\
\end{cases}
$$
for $l \in [12]$, so that $f^{(l)}_{Z^{\pm}_{i, e, G}(\alpha, \beta, \gamma)}(\bar{v})$ considers all $z \in Z^{\pm}_{i,e, G}(\alpha, \beta, \gamma)$ which contain $\bar{v} \in \binom{V(G)}{l}$ and $e$, where $\bar{v}$ is disjoint from $e$. Note that a zero-sum configuration contains $12$ vertices other than those in the given edge $e$ so, given a matching $M$ such that $e \cap V(M)=\emptyset$, counting the number of zero-sum configurations from $Z^{\pm}_{i,e, G}(\alpha, \beta, \gamma)$ remaining in $G\setminus V(M)$ is an inclusion-exclusion sum over $12$ different tuples. Following the strategy used in Section \ref{sec_func_deg} to get counts for degree-type properties, and assuming that (\ref{eq_wgt_egj_hyp}) is satisfied for all twelve functions, we obtain that
\begin{multline*}
|Z^{\pm}_{i, e, G[V(G)\setminus V(M)]}(\alpha, \beta, \gamma)|=\sum_{z \in Z^{\pm}_{i, e, G}(\alpha, \beta, \gamma)} \prod_{u \in z \setminus e}(1-d_{w,G}(u)) \\
\pm O\left(\Delta^{-\epsilon}|Z^{\pm}_{i, e, G}(\alpha, \beta, \gamma)|\max_{u \in \bigcup Z^{\pm}_{i, e, G}(\alpha, \beta, \gamma)} d_{w,G}(u)\right).
\end{multline*}

\subsection{Reachability}

Recall from Section \ref{sec_1}, just above Theorem \ref{thm_special_reweight}, the notion of open and closed permissible pairs and tuples, a restriction of open and closed valid pairs and tuples. A key aspect of our strategy and the vortex of nested subgraphs to reach $L^*$ relies on, at step $i$, being able to obtain a matching that ensures that every vertex outside $I_{t_{i+1}}$ has been matched. In the process of doing this we shall also inevitably match some vertices from within $I_{t_{i+1}}$. However, having reached $(H_{1^*}, w_{1^*})$, in any subgraph of $H_{1^*}[I_{t_i}]$ (for some $i$), in order to match all vertices outside $I_{t_{i+1}}$, given that $\mathcal{T}[I_{t_i}]$ has depth $j$, it follows that we do not need to use any vertices inside $K_{j+1}$. In particular, any edge that uses vertices inside $K_{j+1}$ will not be covering any vertices outside $I_{t_{i+1}}$. Since we are trying to preserve vertices closer to the centre of $\mathcal{T}$ for as long as possible, it would be a waste to remove an edge containing a vertex inside $K_{j+1}$ at a step in the process focused on matching vertices outside $I_{t_{i+1}}$. As such, our algorithm to progress through the vortex of nested subgraphs will ensure that such edges are not picked up in the matching that moves the process from $I_{t_i}$ to $I_{t_{i+1}}$. To deal with this more concisely, we introduce the notion of {\it reachability}. This notion is more complex than just covering whether a vertex is in an edge that might be useful at a particular step of the vortex, and extends to vertices that might have degree-type properties affected by edges that can be useful at a particular step. 

\begin{defn}[Reachable vertices and edges]\label{def_reach}
We say that a vertex $v$ is $i$-reachable if $v \in H_{1^*}[I_{t_i}]$ and there exist $e,f \in H_{1^*}[I_{t_i}]$ such that $e \cap I_{t_{i}}\setminus I_{t_{i+1}} \neq \emptyset$, $f \cap e \neq \emptyset$ and $f \ni v$. Similarly we say that an edge $f$ is $i$-reachable, if $f \in H_{1^*}[I_{t_i}]$ and there exists $e \in H_{1^*}[I_{t_i}]$ such that $e \cap I_{t_{i}}\setminus I_{t_{i+1}} \neq \emptyset$ and $f \cap e \neq \emptyset$. 
\end{defn}
That is, a vertex $v$ is $i$-reachable if and only if it is feasible that the process to go from $I_{t_i}$ to $I_{t_{i+1}}$ might affect the degree of $v$ and an edge $f$ is $i$-reachable if and only if matching edges assigned in the process to go from $I_{t_i}$ to $I_{t_{i+1}}$ might then mean that $f$ is not present in the subgraph induced on vertices remaining once the vertices of the matching are removed. In particular, a vertex $v$ is $i$-reachable if and only if $v$ is in an $i$-reachable edge. Note that if $H_{1^*}[I_{t_i}]$ has depth $j$, then every vertex in $K_{j+2}$ is {\it not} $i$-reachable. That is, all edges sharing a vertex with a vertex in $K_{j+2}$ are contained in $K_{j+1}$, and since $H_{1^*}[I_{t_i}]$ has depth $j$, every edge containing a vertex in $I_{t_i}\setminus I_{t_{i+1}}$ must have a vertex in $K_{j-1}\setminus K_j$, and so only reaches into $K_j\setminus K_{j+1}$. It follows that edges that are $i$-reachable have all vertices outside $K_{j+2}$. 

Extending the notion of reachability further, we define \emph{reachable sets}. We want to consider sets which might lose vertices in step $i$ of the iterative matching process. Additionally we define \emph{open and closed $i$-reachable pairs} where the motivation here is to distinguish pairs $(v,S)$ where step $i$ of the iterative matching process will affect the degree-type properties of $v$ relating to $H_i$ and $S$.
\begin{defn}[Reachable sets] 
Given that $H_{1^*}[I_{t_i}]$ has depth $j$, we say that a subset $S \subseteq V(\mathcal{T})$ is {\it $i$-reachable} if $S \subseteq I_{t_i}$ is $i$-valid and $S \subseteq K_{j-1} \setminus K_{j+1}$.   
We also say that $(v,S)$ is a {\it closed $i$-reachable pair} if $(v,S)$ is a closed $i$-valid pair, such that $v$ is $i$-reachable, $S \subseteq I_{t_i}$ and, if $v$ has depth $j$ or $j+1$, then $S \cap (I_{t_i}\setminus K_{j+1}) \neq \emptyset$, and if $v$ has depth $j+2$ then $S \cap (K_j\setminus K_{j+1}) \neq \emptyset$. Similarly, we say that $(v,S)$ is an {\it open $i$-reachable pair} if $(v,S)$ is an open $i$-valid pair, such that $v$ is $i$-reachable, $\bigcup E_{H_{1^*}}(v, S, *, *) \subseteq I_{t_i}$ and, if $v$ has depth $j$ or $j+1$, then $\left(\bigcup E_{H_{1^*}}(v, S, *, *)\right) \cap (I_{t_i}\setminus K_{j+1}) \neq \emptyset$, and if $v$ has depth $j+2$, then $\left(\bigcup E_{H_{1^*}}(v, S, *, *)\right) \cap (K_j\setminus K_{j+1}) \neq \emptyset$.
\end{defn}

Note that for most values of $i \in [c_h]$, what is $i$-reachable is a subset of what is $(i-1)$-reachable; this only fails when $\mathcal{T}[I_{t_{i-1}}]$ has depth $j$, and $\mathcal{T}[I_{t_i}]$ has depth $j+1$, in which case $i$-reachability extends to vertices at a depth one below that to which $(i-1)$-reachability extends.

We now extend the notion of permissible pairs and tuples to weighted subgraphs of $(H_{1^*}, w_{1^*})$. First let $\delta^-:=\min_{v \in V(H_{1^*})} \frac{w_{1^*}(E_{H_{1^*}}(v, I_{t_2}))}{w_{1^*}(E_{H_{1^*}}(v, I_{t_1}))}$, and $\delta^-_{1^*}=\delta^-/2$.
\begin{defn}[$\delta_*^-$]\label{def_delta_*-}
We define $\delta_*^-:=\min\{\min_{i \in [c_h]}\frac{w_{1^*}(E_{H_{1^*}}(v,I_{t_{i+1}}))}{w_{1^*}(E_{H_{1^*}}(v,I_{t_{i}}))}, \delta_{1^*}^-\}$. 
\end{defn}
Then we have by Corollary \ref{cor_delta_const} that $0<\delta_*^- \leq 1$ is $\Theta(1)$. We now fix some absolute constants.
\begin{defn}[$l_1, l_2, l_3, l_4, c_1, c_2, c_3$] \label{def_l} 
Let $l_1 := 10$, $l_2 :=3$, $l_3 := 7$ and $l_4 := 80\log\left(\frac{2}{\delta_*^-}\right)$. We use this notation for powers of logs to clarify the source of various terms in subsequent calculations. 
We also fix absolute constants $c_1 := \frac{c_{1,1}^*}{2}$, $c_2:= c_{1,4}$ and $c_3:=\max\{1, 2c_{1,6}\}$ (where $c_{1,1}^*$ is an absolute constant defined in Theorem \ref{thm_special_reweight} and $c_{1,4}$ and $c_{1,6}$ are absolute constants defined in Theorem \ref{thm_H_1}).
\end{defn} 

We now come to the key definition for the remainder of this chapter. 

\begin{defn}[Graph permissibility] \label{def_graph_permiss}
Suppose that $\mathcal{T}[I_{t_i}]$ has depth $j$. We say that a weighted subgraph $(G, \omega)$ is $i$-permissible if the following all hold:
\begin{enumerate}[(P1)]
\item $G \subseteq H_{1^*}[I_{t_i}]$,
\item for every $v \in V(G)$ which is $(i-1)$-reachable
$$1 \geq d_{\omega, G}(v) \geq 1-\log^{l_{4}}(n)t_i^{-\epsilon_1},$$
and if $v$ is not $(i-1)$-reachable, then 
$$d_{\omega, G}(v)=d_{w_{1^*}, H_{1^*}}(v).$$
\item for every edge $e \in G$ which is $(i-1)$-reachable,
$$\frac{c_1}{p_{\gr}^3t_i\log^{2}(n)} \leq \omega(e) \leq \frac{\log^{l_{1}}(n)}{p_{\gr}^3t_i},$$
and every edge $e \in G$ which is not $(i-1)$-reachable satisfies $\omega(e)=w_{1^*}(e)$.
\item for every $i$-valid subset $S$ which is $(i-1)$-reachable,
$$\frac{|S|p_{\gr}}{\log^{l_{2}}(n)} \leq |V(G[S])| \leq c_2|S|p_{\gr},$$
and for every $i$-valid subset $S \subseteq K_{j+1}$, $|V(G[S])|=|V(H_{1^*}[S])|$.
\item for every open or closed $(i-1)$-reachable tuple $(v,S_1, S_2, S_3)$ which is an $i$-permissible tuple, we have that 
$$\frac{|S_1|p_{\gr}^3}{\log^{l_{3}}(n)} \leq |E_{G}(v,S_1,S_2,S_3)| \leq c_3|S_1|p_{\gr}^3,$$
and for every open or closed $i$-permissible tuple $(v,S_1, S_2, S_3)$ such that $\bigcup \{ e\setminus \{v\}: e \in E_{\mathcal{T}}(v,S_1,S_2,S_3)\} \subseteq K_{j+1}$,  
we have that $|E_{G}(v,S_1,S_2,S_3)|=|E_{H_{1^*}}(v,S_1,S_2,S_3)|$. 
\end{enumerate}
\end{defn}

We now show that Theorem \ref{thm_special_reweight} implies that $(H_{1^*}, w_{1^*})$ is $1$-permissible.

\begin{lemma}\label{lemma_1_permiss}
$(H_{1^*}, w_{1^*})$ is $1$-permissible, and in particular we have that $c_{H_{1^*}} \leq t_1^{-\epsilon_1}$, and there exist constants $\delta^{\pm}_{1^*}>0$ such that 
$$\delta^-_{1^*}\leq w_{1^*}(E_{H_{1^*}}(v, I_{t_2})) \leq  \delta^+_{1^*}.$$
\end{lemma}
\begin{proof}
We consider the properties of $(H_{1^*}, w_{1^*})$ given in the statement of Theorem \ref{thm_special_reweight} to deduce the lemma. First note that \ref{thm_special_reweight}(i) tells us immediately that $c_{H_{1^*}} \leq t_1^{-\epsilon_1}$, and in particular implies that (P2) holds. It is also clear that (P1) holds. (P3) by Theorem \ref{thm_special_reweight}(ii) and (P4) holds by Theorem \ref{thm_special_reweight}(iii) and Theorem \ref{thm_H_1}(iii). Similarly, (P5) holds by Theorem \ref{thm_special_reweight}(iv) and Theorem \ref{thm_H_1}(iv).

Finally we consider $\delta_{H_{1^*}, v}:= w_{1^*}(E_{H_{1^*}}(v, I_{t_{2}}))$. By Corollary \ref{cor_delta_const} we have that there exist constants $\delta^{\pm}>0$ such that $\delta^-\leq \frac{w_{1^*}(E_{H_{1^*}}(v, I_{t_2}))}{w_{1^*}(E_{H_{1^*}}(v, I_{t_1}))} \leq \delta^+$. Furthermore, $1 \geq w_{1^*}(E_{H_{1^*}}(v, I_{t_1})) \geq 1-t_1^{-\epsilon_1}$. Thus
$$(1-t_1^{-\epsilon_1})\delta^-\leq w_{1^*}(E_{H_{1^*}}(v, I_{t_2})) \leq  \delta^+,$$
and taking $\delta^-_{1^*}=\delta^-/2$ and $\delta^+_{1^*}=\delta^+$ completes the proof.
\end{proof}

We want to show that removing a particular matching from some $i$-permissible $(G, \omega)$, we may obtain a subgraph with some `nice' properties. We shall end up (in Section \ref{sec_1} and in particular Subsection \ref{sec_i_to_1}) ensuring that our nested subgraphs $(H_i, w_i)$ are each $i$-permissible, but first we use the definition of $i$-permissibility here to show that {\it given} a particular weighted graph $(G, \omega)$ is $i$-permissible, we can always define a subgraph $G^o$ of $G$ with weight function $\omega^o:=\omega|_{G^o}$ such that $(G^o, \omega^o)$ satisfies the hypotheses of Theorem \ref{thm_weighted_egj}. This is key, since we use Theorem \ref{thm_weighted_egj} to remove a matching $M^o$ from $(G^o, \omega^o)$ which we'll show (due to the $i$-permissibility of $(G, \omega)$) covers most vertices in $I_{t_i} \setminus I_{t_{i+1}}$, and does so in a way that ensures that the subgraph $G' \subseteq G$ with $V(G'):=V(G)\setminus V(M^o)$ and $E(G')=E(G[V(G')])$ has many nice properties. 

\begin{prop} \label{prop_unique_i}
Suppose that $(G, \omega)$ is $i$-permissible for some $i$. Then $(G, \omega)$ is not $j$-permissible for all $j \neq i$.
\end{prop}
\begin{proof}
First note that by (P1), $G \subseteq H_{1^*}[I_{t_i}]$. Furthermore, by (P4) we have that $V(G) \cap (I_{t_i}\setminus I_{t_{i+1}}) \neq \emptyset$, since $I_{t_i}\setminus I_{t_{i+1}}$ is itself an $i$-valid subset. Thus, $G \not\subseteq H_{1^*}[I_{t_{i+1}}]$ and so $(G, \omega)$ is not $j$-permissible for all $j \geq i+1$. Furthermore, we see that, supposing $(G, \omega)$ is $j$-permissible for $j <i$, again from (iv) we must have $V(G) \cap I_{t_j}\setminus I_{t_{j+1}} \neq \emptyset$, but since $G \subseteq H_{1^*}[I_{t_i}]$ this is clearly not possible. Hence $(G, \omega)$ is not $j$-permissible for any $j \leq i-1$ and the proposition holds.
\end{proof}

Given Proposition \ref{prop_unique_i}, and an $i$-permissible pair $(G, \omega)$ we may subsequently define $G^o \subseteq G$ as follows (noting that $G^o$ is well defined due to Proposition \ref{prop_unique_i}). For fixed $i$ such that $(G, \omega)$ is $i$-permissible, we define 
$$E(G^o):=\{e \in G: e \cap (I_{t_i}\setminus I_{t_{i+1}}) \neq \emptyset\} \mbox{~and~} V(G^o):=\bigcup_{e \in E(G^o)} e.$$ 
We let 
$$\omega^o:=\omega|_{G^o}$$ 
be the restriction of $\omega$ to edges in $G^o$. We write 
$$d_{\omega^o}(v):=d_{\omega^o, G^o}(v)$$ for every $v \in V(G^o)$.

\begin{prop}\label{prop_io}
Given an $i$-permissible pair $(G, \omega)$ for some $i$, for every $v \in V(G[I_{t_{i+1}}])$ we have that
$$\omega(E_{G}(v, I_{t_{i+1}})) + d_{\omega^o}(v) = d_{\omega, G}(v),$$
and consequently, 
$$\frac{\omega(E_G(v, I_{t_{i+1}}))}{1-d_{\omega^o}(v)} \leq 1.$$
\end{prop}
\begin{proof}
By definition, for a vertex $v \in V(G[I_{t_{i+1}}])$, $G^o$ includes an edge $e \in G$ containing $v$ if and only if $e \cap (I_{t_i}\setminus I_{t_{i+1}}) \neq \emptyset$. All edges $e \in G$ containing $v$ which are not in $G^o$ must therefore satisfy $e \subseteq I_{t_{i+1}}$. Since $v$ is also in $I_{t_{i+1}}$ it follows that $e \in G[I_{t_{i+1}}]$. The first result follows. The second result is clear, noting that since $(G, \omega)$ is $i$-permissible we have that $d_{\omega, G}(v) \leq 1$ and $d_{\omega^o}(v)<1$, since $\omega(E_G(v,I_{t_{i+1}}))>0$. 
\end{proof}

We note two more facts about $(G, \omega)$ given that $(G, \omega)$ is $i$-permissible for some $i$. The first is a trivial consequence of the definition, but we note it explicitly since we shall use the consequence several times in subsequent sections.

\begin{prop}\label{prop_omega_permiss}
Suppose that $(G, \omega)$ is $i$-permissible for some $i$. Then for every open or closed $(i-1)$-reachable tuple $(v,S_1, S_2, S_3)$ which is an $i$-permissible tuple, we have that 
$$\frac{c_1|S_1|}{t_i\log^{l_3+2}(n)} \leq \omega(E_{G}(v,S_1,S_2,S_3)) \leq \frac{c_3|S_1|\log^{l_1}(n)}{t_i},$$
and for every open or closed $i$-permissible tuple $(v,S_1, S_2, S_3)$ satisfying $\bigcup \{ e\setminus \{v\}: e \in E_{\mathcal{T}}(v,S_1,S_2,S_3)\} \subseteq K_{j+1}$, we have that 
$$\omega(E_{G}(v,S_1,S_2,S_3))=w_{1^*}(E_{H_{1^*}}(v,S_1,S_2,S_3)).$$ 
\end{prop}
\begin{proof}
This follows by the bounds in (P3) and (P5).
\end{proof}

\begin{cor}\label{cor_bound}
Suppose that $(G, \omega)$ is $i$-permissible for some $i$. Let $v \in I_{t_{i+1}}$. Then
$$1-d_{w^o}(v) \geq \omega(E_{G}(v, I_{t_{i+1}})) \geq \frac{c_1}{\log^{l_3+2}(n)}.$$
\end{cor}
\begin{proof}
By Proposition \ref{prop_io} we have that
$$\omega(E_{G}(v, I_{t_{i+1}})) + d_{\omega^o}(v) = d_{\omega, G}(v) \leq 1.$$
So $1-d_{\omega^o}(v) \geq \omega(E_{G}(v, I_{t_{i+1}})) \geq \frac{c_1}{\log^{l_3+2}(n)}$ by Proposition \ref{prop_omega_permiss}, since $|I_{t_{i+1}}| \geq t_i$.
\end{proof}

\subsection{$i$-permissibility and the iterative matching process}

In this section we'll introduce an algorithm that takes us through the iterative matching process, and in particular, describes how to obtain $(H_{i+1}, w_{i+1})$ from $(H_i, w_i)$ for each $i \in [c_h-1]$. Our strategy relies on using Theorem \ref{thm_weighted_egj}, and we'll show that, for every $i$, given an $i$-permissible pair $(G, \omega)$ we can use Theorem \ref{thm_weighted_egj} on $(G^o, \omega^o)$ to obtain a matching $M^o$ and have control over various properties in the graph $G[V(G)\setminus V(M^o)]$. To do this we introduce the following definition. Recalling $p_S$ as defined in Section \ref{sec_func_vert}, we say that a function $f:V(G^o) \rightarrow \mathbb{R}_{\geq 0}$ is {\it vertex allowable for $(G, \omega, \Delta, \eta)$} if $p_S$ satisfies (\ref{eq_wgt_egj_hyp}) (with $r=4$ and $L=16$), and we say that a function $f_v:E(G^o) \rightarrow \mathbb{R}_{\geq 0}$ is {\it $v$-edge allowable for $(G, \omega, \Delta, \eta)$} if the functions $q^{(1)}, q^{(2)}, q^{(3)}$ from Proposition \ref{prop_degs} all satisfy (\ref{eq_wgt_egj_hyp}) with respect to parameters $(G, \omega, \Delta, \eta)$. We also say that $f:E(G^o) \rightarrow \mathbb{R}_{\geq 0}$ is {\it edge allowable} if we can define $f_v$ such that $f_v(e)=f(e)$ if $e \ni v$ and $f(e)=0$ otherwise for each $v \in V(G^o)$ and $f_v$ is {\it $v$-edge allowable for $(G, \omega, \Delta, \eta)$}. With these definitions in mind we have the following key theorem. 

\begin{theo}\label{thm_key}
Given that $(G, \omega)$ is $i$-permissible and $\mathcal{T}[I_{t_i}]$ has depth $j$, we may obtain a matching $M^o$ such that in the subgraph $G':= G[V(G) \setminus V(M^o)]$ the following all hold:
\begin{enumerate}[(K1)]
\item for every $i$-reachable subset $S\subseteq I_{t_{i+1}}$,
$$|V(G'[S])|=(1 \pm c_1^{-1}t_i^{-\epsilon_1}\log^{l_3+2}(n))\sum_{v \in V(G[S])} (1-d_{\omega^o}(v)),$$
and for a function $p_S(v):V(G) \rightarrow \mathbb{R}_{\geq 0}$ such that $p_S(v)=f_S(v)\mathbbm{1}_{v \in V(G[S])}$ and $\frac{\max_{v \in V(G^o[S])} p_S(v)}{\min_{v \in V(G^o[S])} p_S(v)} \leq \log^{500}(n)$,
$$\sum_{v \in V(G'[S])} p_S(v)=(1 \pm c_1^{-1}t_i^{-\epsilon_1}\log^{l_3+2}(n))\sum_{v \in V(G[S])} p_S(v)(1-d_{\omega^o}(v)).$$
Furthermore, for every $i$-valid subset $S \subseteq K_{j+1}$, $|V(G'[S])|=|V(H_{1^*}[S])|$. 
\item for every open or closed $i$-reachable tuple $(v,S_1, S_2,S_3)$ with $v \in V(G')$ and $S_1, S_2, S_3 \subseteq I_{t_{i+1}}$, we have both that
$$|E_{G'}(v,S_1,S_2,S_3)|=(1 \pm t_i^{-\epsilon_1}\log^{3l_3+7}(n))\sum_{e \in E_{G}(v,S_1,S_2,S_3)} \prod_{u \in e \setminus \{v\}} (1-d_{\omega^o}(v)),$$
and for a function $f_v:E(G) \rightarrow \mathbb{R}_{\geq 0}$ such that $f_v(e)=0$ wherever $v \notin e$ and $\frac{\max_{e \in E_{G^o}(v,S_1,S_2,S_3)} f_v(e)}{\min_{e \in E_{G^o}(v,S_1,S_2,S_3)} f_v(e)} \leq \log^{500}(n)$,
$$f_v(E_{G'}(v,S_1,S_2,S_3))=(1 \pm t_i^{-\epsilon_1}\log^{3l_3+7}(n))\sum_{e \in E_{G}(v,S_1,S_2,S_3)} f_v(e) \prod_{u \in e \setminus \{v\}} (1-d_{\omega^o}(v)).$$
In particular,
$$\omega(E_{G'}(v,S_1,S_2,S_3))=(1 \pm t_i^{-\epsilon_1}\log^{3l_3+7}(n))\sum_{e \in E_{G}(v,S_1,S_2,S_3)} \omega(e) \prod_{u \in e \setminus \{v\}} (1-d_{\omega^o}(v)).$$
\item for every open or closed $i$-permissible tuple $(v,S_1,S_2,S_3)$ with $v \in V(G')$ such that $\bigcup \{ e\setminus \{v\}: e \in E_{\mathcal{T}}(v,S_1,S_2,S_3)\} \subseteq K_{j+1}$, we have that $|E_{G'}(v,S_1,S_2,S_3)|=|E_{H_{1^*}[S]}(v,S_1,S_2,S_3)|$.
\end{enumerate}
\end{theo}
Note that in the inequalities with $\log^{500}(n)$ on the right hand side the $500$ is not important, but is an explicit (not tight) upper bound that all functions we are concerned with will easily satisfy.
\begin{proof}
We obtain $G'$ by running Theorem \ref{thm_weighted_egj} on $(G^o, \omega^o)$. We first claim that the hypotheses all hold in $(G^o, \omega^o)$, taking $(\eta_1, t_i)$ in place of $(\eta, \Delta)$.\footnote{We consider $n$ to be sufficiently large that $\Delta=t_i$ is sufficiently large, i.e. $t_i \geq \Delta_0$, the value given by Theorem \ref{thm_egj} given $\eta_1$.}. Note that all vertices, subsets and edges in $G^o$ are, by definition, $i$-reachable. Since this is the case, and $(G, \omega)$ is $i$-permissible, by (P3) we have $c_1, l_1 >0$ such that
$$\frac{c_1}{p_{\gr}^3t_i\log^{2}(n)} \leq \omega^o(e) \leq \frac{\log^{l_{1}}(n)}{p_{\gr}^3t_i},$$
for every $e \in G^o$, which implies $\omega^o(e) \geq t_i^{-1}$ and $\Delta_{\omega^o}^c(G^o)=\max_{e \in G^o} \omega^o(e) \leq t_i^{-\eta_1}$. It is also clear, since $G^o \subseteq \mathcal{T}[I_{t_i}]$, that $e(G^o) \leq \exp(t_i^{\epsilon_1^2/4})$. Recall that $\delta_{\omega^o}(G^o)=\min_{v \in V(G^o)} d_{\omega^o}(v)$.  
Since $(G, \omega)$ is $i$-permissible, and $(v, I_{t_{i}}\setminus I_{t_{i+1}})$ is an open $i$-reachable pair, we have by (P5) that $\min_{v \in V(G^o)} |E_{G}(v,I_{t_{i}}\setminus I_{t_{i+1}},*,*)| \geq \frac{|I_{t_{i}}\setminus I_{t_{i+1}}|p_{\gr}^3}{\log^{l_{3}}(n)} \geq \frac{t_ip_{\gr}^3}{\log^{l_{3}}(n)}$. Furthermore, since $\omega(e) \geq \frac{c_1}{p_{\gr}^3t_i\log^{2}(n)}$ for every edge $e \in E_{G}(v,I_{t_{i}\setminus I_{t_{i+1}}},*,*)$, we therefore have that
$$\delta_{\omega^o}(G^o)\geq \frac{c_1}{\log^{l_3+2}(n)} \geq \frac{t_i^{-\psi_1}}{1-t_i^{-1}}.$$
It remains to check that the necessary functions discussed above do indeed all satisfy (\ref{eq_wgt_egj_hyp}). Note also that the number of these functions we wish to keep an eye on is polynomial in $t_i$. The property we need each function $p_l$ to satisfy is:
\begin{equation}\label{eq_func}
\max_{v \in V(G^o)} p_l(\{v\}) \leq \frac{\sum_{\bar{v} \in \binom{V(G^o)}{l}} p_l(\bar{v})}{2\cdot 4^l t_i^{2\eta_1}},
\end{equation}
where we recall $r=4$ and $L=16$.
Note that for every $i$-valid $S$ which is $i$-reachable, since $\mathcal{T}[I_{t_i}]$ has depth $j$, we have that $|S|\geq |I_{k_{j+1}-1}\setminus I_{k_{j+1}}| \geq k_{j+1} \geq \frac{t_i}{\log^{2}(n)}$. 

In order to prove (K1), we wish to show that for every $i$-reachable $S \subseteq I_{t_{i+1}}$, the function $p_S$ described in Section \ref{sec_func_vert} with $\frac{\max_{v \in V(G^o[S])} p_S(v)}{\min_{v \in V(G^o[S])} p_S(v)} \leq \log^{500}(n)$ satisfies (\ref{eq_func}). This will give us that (\ref{eq_func_weighted_vert}) holds and applying Corollary \ref{cor_bound}, then yields (K1) for every $i$-reachable $S \subseteq I_{t_{i+1}}$. Now, when $f_S=1$, we have that $\max_{v \in V(G^o[S])} p_S(\{v\})= \max_{v \in V(G^o[S])} \mathbbm{1}_{v \in V(G^o[S])} \leq 1$, and by $i$-permissibility and the lower bound on $|S|$, 
$$\sum_{v \in V(G^o[S])} p_S(\{v\})=|V(G^o[S])|\geq \frac{|S|p_{\gr}}{\log^{l_{2}}(n)} \geq \frac{p_{\gr}t_i}{\log^{l_{2}+2}(n)},$$
for every $i$-reachable $S$, and so (\ref{eq_func}) holds. Considering some other function $f_S$, we have that 
$$\max_{v \in V(G^o[S])} p_S(\{v\})= \max_{v \in V(G^o[S])} f_S(v)\mathbbm{1}_{v \in V(G^o[S])} \leq \max_{v \in V(G^o[S])} f_S(v),$$
and that
\begin{multline*}
\sum_{v \in V(G^o[S])} p_S(\{v\})\geq |V(G^o[S])|\min_{v \in V(G^o[S])} f_S(v) \\ \geq \frac{|S|p_{\gr}\min_{v \in V(G^o[S])} f_S(v)}{\log^{l_{2}}(n)} \geq \frac{p_{\gr}t_i\min_{v \in V(G^o[S])} f_S(v)}{\log^{l_{2}+2}(n)}.
\end{multline*}
Then (\ref{eq_func}) holds, since $\frac{\max_{v \in V(G[S])} p_S(v)}{\min_{v \in V(G[S])} p_S(v)} \leq \log^{500}(n)$.

Similarly, to prove (K2) which only concerns $i$-reachable tuples, note that every open or closed $i$-permissible pair, uses $(S_1, S_2, S_3)$ which is $i$-valid and, by the same reasoning as above, also satisfies $|S_1|\geq \frac{t_i}{\log^{2}(n)}$.
Then we have that $q^{(1)}_{(G^o,v,S_1,S_2,S_3)}$, $q^{(2)}_{(G^o,v,S_1,S_2,S_3)}$, and $q^{(3)}_{(G^o,v,S_1,S_2,S_3)}$, as described in Section \ref{sec_func_deg}, with $f_v=1$, and every open or closed $i$-permissible tuple $(v,S_1, S_2, S_3)$ all satisfy
$$\max_{\bar{v} \in \binom{V(G^o)}{l}} q^{(l)}_{(G^o,v,S_1,S_2,S_3)}(\bar{v})\leq 1,$$
where $l \in [3]$, since they are all indicator functions. Furthermore note that 
$$q^{(l)}_{(G^o,v,S_1,S_2,S_3)}(V(G^o)) \geq |E_{G^o}(v, S_1,S_2,S_3)|$$ for every $l \in [3]$ since for every edge in $E_{G^o}(v, S_1,S_2,S_3)$ there is at least one tuple counted in the sum that is contained in the edge and therefore returns a $1$ in the indicator function. Again by $i$-permissibility and $|S_1|$ it follows that, 
\begin{equation} \label{eq_q}
\sum_{\bar{v} \in \binom{V(G^o)}{l}} q^{(l)}_{{(G^o,v,S_1,S_2,S_3)}}(\bar{v}) \geq |E_{G^o}(v, S_1,S_2,S_3)| \geq \frac{|S_1|p_{\gr}^3}{\log^{l_{3}}(n)} \geq \frac{p_{\gr}^3t_i}{\log^{l_{3}+2}(n)},
\end{equation}
for each $l \in [3]$ and so it follows that (\ref{eq_func}) also holds here. Then by Proposition \ref{prop_degs} and Corollary \ref{cor_bound}, the first statement of (K2) holds.

More generally considering the functions in Section \ref{sec_func_deg} so that additionally $\frac{\max_{e \in E_{G^o}(v,S_1,S_2,S_3)} f_v(e)}{\min_{e \in E_{G^o}(v,S_1,S_2,S_3)} f_v(e)} \leq \log^{500}(n)$, we have that 
$$\max_{\bar{v} \in \binom{V(G^o)}{l}} q^{(l)}_{(G^o,v,S_1,S_2,S_3)}(\bar{v}) \leq \max_{e \in G^o} f_v(e),$$
and, as in the previous case, $\sum_{\bar{v} \in \binom{V(G^o)}{l}} q^{(l)}_{(G^o,v,S_1,S_2,S_3)}(\bar{v}) \geq f_v(E_{G^o}(v, S_1,S_2,S_3))$ for every $l \in [3]$, since for every edge $e \in E_{G^o}(v, S_1,S_2,S_3)$ there is at least one tuple counted in the sum on the LHS that is contained in the edge $e$ and therefore contributes at least $f_v(e)$ to the LHS, and the RHS contributes exactly $f_v(e)$ for every such edge $e$. Hence
$$\sum_{\bar{v} \in \binom{V(G^o)}{l}} q^{(l)}_{(G^o,v,S_1,S_2,S_3)}(\bar{v}) \geq |E_{G^o}(v, S_1,S_2,S_3)|\min_{e \in G^o} f_v(e).$$
By (\ref{eq_q}) and since $\frac{\max_{e \in E_{G^o}(v,S_1,S_2,S_3)} f_v(e)}{\min_{e \in E_{G^o}(v,S_1,S_2,S_3)} f_v(e)} \leq \log^{500}(n)$ we again have that (\ref{eq_func}) holds here. Once again combining Proposition \ref{prop_degs} and Corollary \ref{cor_bound} we obtain the second statement of (K2). To see the final statement of (K2), consider a fixed vertex $v$ and let $\omega_v(e)=\omega(e)$ if $v \in e$ and $\omega_v(e)=0$ otherwise. 
Then $\frac{\max_{e \in E_{G^o}(v,S_1,S_2,S_3)} \omega_v(e)}{\min_{e \in E_{G^o}(v,S_1,S_2,S_3)} \omega_v(e)} \leq \frac{\log^{l_1+2}(n)}{c_1} \leq \log^{13}(n)$ and $\omega_v$ is a function as per the second statement of (K2). Thus
\begin{multline*}
\omega_v(E_{G'}(v,S_1,S_2,S_3))= \\(1 \pm t_i^{-\epsilon_1}\log^{3l_3+7}(n))\sum_{e \in E_{G}(v,S_1,S_2,S_3)} \omega_v(e) \prod_{u \in e \setminus \{v\}} (1-d_{\omega^o}(v)),
\end{multline*}
but by definition,
$\omega_v(E_{G'}(v,S_1,S_2,S_3))=\omega(E_{G'}(v,S_1,S_2,S_3))$ and so
$$\sum_{e \in E_{G}(v,S_1,S_2,S_3)} \omega_v(e) \prod_{u \in e \setminus \{v\}} (1-d_{\omega^o}(v))=\sum_{e \in E_{G}(v,S_1,S_2,S_3)} \omega(e) \prod_{u \in e \setminus \{v\}} (1-d_{\omega^o}(v)),$$
so the final statement of (K2) also holds. Hence, by (\ref{eq_func_vert}), Proposition \ref{prop_degs} and Theorem \ref{thm_weighted_egj} there exists a matching $M^o$ in $G^o$ such that in $G'=G[V(G)\setminus V(M^o)]$ the properties in (K1) and (K2) which refer to $i$-reachable vertices, edges and subsets contained in $I_{t_{i+1}}$ all hold. 

It remains to consider what happens to the other vertices, edges and sets considered in (K1) and (K3). Indeed, every subset $S \subseteq K_{j+1}$ satisfies $S \cap V(G^o) = \emptyset$, since sets in $K_{j+1}$ are not $i$-reachable, and $G^o$ is defined only to include vertices in sets which are $i$-reachable. Thus removing a matching $M^o \subseteq G^o$ removes no vertices from $V(G[S])$. Thus, by $i$-permissibility, $|V(G[S])|=|V(H_{1^*}[S])|$ for all such $S$. For (K3) we argue in the same way. In particular, by definition of $G^o$, the matching $M^o$ only uses edges containing a vertex in $I_{t_i}\setminus I_{t_{i+1}}$ and since $G \subseteq H_{1^*}[I_{t_i}]$, such edges do not contain any vertices within $K_{j+1}$. Hence, given that $v \in G'$, the number of edges it is contained in within a particular subset of $K_{j+1}$ is not affected from $G$ to $G'$. Thus, by $i$-permissibility, the claim follows, completing the proof.
\end{proof}

\begin{rem} \label{rem_allow}
Among other things, the proof of Theorem \ref{thm_key} shows that any function $f_S$ or $f_v$ as defined in statements (K1) and (K2) of the theorem is vertex or $v$-edge allowable for $(G, \omega, t_i, \eta_1)$ respectively. Furthermore it shows that $\omega$ is edge allowable for $(G, \omega, t_i, \eta_1)$.
\end{rem}

Before proceeding with details of the vortex, we include the following propositions concerning $d_{\omega^o}(v)$ and $\omega(E_{G}(v, I_{t_{i+1}}))$ for any $i$-permissible $(G, \omega)$, which will be useful for the subsequent corollary to Theorem \ref{thm_key} (Corollary \ref{cor_props_for_greedy}), and at various stages in the next section.

\begin{prop} \label{prop_o'}
Let $(G, \omega)$ be $i$-permissible such that $\omega$ is an almost-perfect fractional matching for $G$. We define 
$$\delta_{G,v}:=\omega(E_{G}(v, I_{t_{i+1}})),$$ 
and let $\delta_G:=\min_{v \in V(G[I_{t_{i+1}}])} \delta_{G,v}$. Suppose that $d_{\omega, G}(v) = 1-c_{G,v}$ for all $v \in V(G)$ and let $c_G:=\max_v c_{G,v}$. Then 
$$1-d_{\omega^o}(v)=\left(1 \pm \frac{c_{G,v}}{\delta_{G,v}}\right)\omega(E_{G}(v, I_{t_{i+1}}))=\left(1 \pm \frac{c_G}{\delta_{G}}\right)\omega(E_{G}(v, I_{t_{i+1}}))$$
for every $v \in V(G[I_{t_{i+1}}])$
\end{prop}

\begin{proof}
By $i$-permissibility we have that $0 \leq c_{G,v} \leq \log^{l_4}(n)t_i^{-\epsilon_1}$ for each $v \in V(G[I_{t_{i+1}}])$. By Proposition \ref{prop_io} we have for each $v \in V(G[I_{t_{i+1}}])$ that
\begin{multline*}
1-d_{\omega^o}(v)=\omega(E_{G}(v, I_{t_{i+1}})) + c_{G,v}  \\ = \left(1 + \frac{c_{G,v}}{\delta_{G,v}}\right)\omega(E_{G}(v, I_{t_{i+1}})) \leq \left(1 + \frac{c_{G}}{\delta_{G}}\right)\omega(E_{G}(v, I_{t_{i+1}})).
\end{multline*}
Additionally
\begin{multline*}
1-d_{\omega^o}(v) \geq \omega(E_{G}(v, I_{t_{i+1}})) \geq \left(1 - \frac{c_{G,v}}{\delta_{G,v}}\right)\omega(E_{G}(v, I_{t_{i+1}})) \\ \geq \left(1 - \frac{c_{G}}{\delta_{G}}\right)\omega(E_{G}(v, I_{t_{i+1}})),
\end{multline*}
proving the proposition. (Note that clearly $\delta_{G,v}, c_{G,v} \geq 0$ for all $v \in V(G[I_{t_{i+1}}])$ since $\omega$ is a fractional matching.) 
\end{proof}

\begin{prop}\label{prop_bound_err}
Given that $(G, \omega)$ is $i$-permissible, for $c_G$ and $\delta_G$ as defined in Proposition \ref{prop_o'} we have that
$$\frac{c_G}{\delta_G} \leq \frac{\log^{l_4+l_3+2}(n)}{c_1t_i^{\epsilon_1}}.$$
\end{prop}
\begin{proof}
We have by $i$-permissibility of $(G,\omega)$ that $c_G \leq \log^{l_4}(n)t_i^{-\epsilon_1}$ and by Corollary \ref{cor_bound} that $\delta_G=\min_{v \in V(G[I_{t_{i+1}}])} \omega(E_G(v, I_{t_{i+1}})) \geq \frac{c_1}{\log^{l_3+2}(n)}$. 
\end{proof}

We conclude this section with two more corollaries which will be extremely useful in the next section. Both will be useful for our `cover' step going from $H_i$ to $H_{i+1}$ which is `Step 3' of Plan \ref{alg_subgraphs} introduced at the start of the next section. 

\begin{cor}\label{cor_props_for_greedy}
Given that $(G, \omega)$ is $i$-permissible and we obtain $G'$ from $G$ as in Theorem \ref{thm_key}, the following properties hold for $G'$:
\begin{enumerate}[(G1)]
\item $|V(G'[I_{t_i}\setminus I_{t_{i+1}}])| \leq 2.1c_2\log^{l_4}(n)t_i^{1-\epsilon_1}p_{\gr}$.
\item For every $v \in V(G')$, $d_{G'}(v) \geq|E_{G'}(v, I_{t_{i+1}})| \geq \frac{0.9c_1^3t_ip_{\gr}^3}{\log^{4l_{3}+6}(n)}$.
\item For every $v \in V(G'[I_{t_{i+1}}])$, $|E_{G'}(v,I_{t_i}\setminus I_{t_{i+1}},*,*)|\leq 2c_3t_i^{1-\epsilon_1}\log^{l_{4}}(n)p_{\gr}^3$.
\end{enumerate}
\end{cor}

\begin{proof}
By (\ref{eq_psm}) we have that
$$|V(G'[I_{t_i}\setminus I_{t_{i+1}}])|=|V(G[I_{t_i}\setminus I_{t_{i+1}}])|-(1 \pm t_i^{-\epsilon_1})\sum_{v \in V(G[I_{t_i}\setminus I_{t_{i+1}}])} d_{w^o}(v).$$
By (P2) it follows that
\begin{eqnarray*}
|V(G'[I_{t_i}\setminus I_{t_{i+1}}])| &\leq& |V(G[I_{t_i}\setminus I_{t_{i+1}}])| \\
&~& ~~~~~~~~~~~~~~~~~~~~ -(1 - t_i^{-\epsilon_1})\sum_{v \in V(G[I_{t_i}\setminus I_{t_{i+1}}])}(1-\log^{l_4}(n)t_i^{-\epsilon_1}) \\
&\leq& 2.1c_2\log^{l_4}(n)t_i^{1-\epsilon_1}p_{\gr},
\end{eqnarray*}
where the last line follows using (P4) and that $|I_{t_i}\setminus I_{t_{i+1}}|\leq 2t_i$.

Now, we also know from (K2) that
\begin{multline*}
d_{G'}(v)=|E_{G'}(v,I_{t_i})| \geq |E_{G'}(v, I_{t_{i+1}})| \\
=(1 \pm \log^{3l_3+7}(n)t_i^{-\epsilon_1})\sum_{e \in E_{G}(v, I_{t_{i+1}})} \prod_{u \in e \setminus \{v\}} (1-d_{\omega^o}(u)),
\end{multline*}
noting that $e \setminus \{v\} \subseteq I_{t_{i+1}}$ for every $e \in E_{G}(v, I_{t_{i+1}})$. Furthermore, for each $e \in E_{G}(v,I_{t_{i+1}})$ and $u \in e \setminus \{v\}$ we have by Proposition \ref{prop_o'} that $1-d_{\omega^o}(u)=(1 \pm \frac{c_G}{\delta_G})\omega(E_{G}(u, I_{t_{i+1}}))$. Hence,
$$d_{G'}(v)\geq \left(1 - \log^{3l_3+7}(n)t_i^{-\epsilon_1}\right)\left(1 - 3.1\frac{c_G}{\delta_G}\right)\sum_{e \in E_{G}(v,I_{t_{i+1}})} \prod_{u \in e \setminus \{v\}} \omega(E_{G}(u,I_{t_{i+1}})).$$
Since $(G, \omega)$ is $i$-permissible we have by Corollary \ref{cor_bound} that $\omega(E_{G}(u,I_{t_{i+1}})) \geq  \frac{c_1}{\log^{l_3+2}(n)}$ for every $u \in V(G[I_{t_{i+1}}])$ and by (P5) that $|E_{G}(v,I_{t_{i+1}})|\geq \frac{p_{\gr}^3t_i}{\log^{l_{3}}(n)}$ for every $v \in V(G)$. This yields that
$$d_{G'}(v) \geq \left(1 - \log^{3l_3+7}(n)t_i^{-\epsilon_1}\right)\left(1 - 3.1\frac{c_G}{\delta_G}\right)\frac{c_1^3t_ip_{\gr}^3}{\log^{4l_{3}+6}(n)}.$$
By Proposition \ref{prop_bound_err} the second claim follows.
Finally, considering a vertex $v \in V(G'[I_{t_{i+1}}])$, by Proposition \ref{prop_degs} we have that
\begin{multline*}
|E_{G'}(v,I_{t_{i}}\setminus I_{t_{i+1}},*,*)| \leq \sum_{e \in E_{G}(v,I_{t_{i}}\setminus I_{t_{i+1}},*,*)} \prod_{u \in e \setminus \{v\}} (1-d_{w^o}(u)) \\+ O\left(t_i^{-\epsilon_1}|E_{G}(v,I_{t_{i}}\setminus I_{t_{i+1}},*,*)|\right).
\end{multline*}
We have that $O\left(t_i^{-\epsilon_1}|E_{G}(v,I_{t_{i}}\setminus I_{t_{i+1}},*,*)|\right)=O\left(t_i^{1-\epsilon_1}p_{\gr}^3\right)$ since from (P5) it also follows that $|E_{G}(v,I_{t_{i}}\setminus I_{t_{i+1}},*,*)| \leq 1.9c_3t_ip_{\gr}^3$. Furthermore, for every $e \in E_{G}(v,I_{t_{i}}\setminus I_{t_{i+1}},*,*)$, $e \setminus \{v\}$ contains at least one vertex $v_1 \in I_{t_{i}}\setminus I_{t_{i+1}}$. Thus $d_{\omega^o}(v_1)=d_{\omega, G}(v_1)$ and so $1-d_{\omega^o}(v_1)\leq \log^{l_{4}}(n)t_i^{-\epsilon_1}$. This gives
\begin{multline*}
\sum_{e \in E_{G}(v,I_{t_{i}}\setminus I_{t_{i+1}},*,*)} \prod_{u \in e \setminus \{v\}} (1-d_{w^o}(u)) \leq \log^{l_{4}}(n)t_i^{-\epsilon_1} |E_{G}(v,I_{t_{i}}\setminus I_{t_{i+1}},*,*)| \\
\leq \frac{\log^{l_{4}}(n)}{t_i^{\epsilon_1}}\cdot c_3|I_{t_i}\setminus I_{t_{i+1}}|p_{\gr}^3
\leq 1.9c_3t_i^{1-\epsilon_1}\log^{l_{4}}(n)p_{\gr}^3,
\end{multline*}
completing the proof. 
\end{proof}

\begin{cor}\label{cor_lower_bounds}
Given that $S$ is $i$-reachable and $S \subseteq I_{t_{i+1}}$, 
$$|V(G'[S])| \geq \frac{0.9c_1t_ip_{\gr}}{\log^{l_2+l_3+4}(n)}.$$
Furthermore, given that $(v, S_1, S_2, S_3)$ is an open or closed $i$-reachable tuple such that $S_1 \subseteq I_{t_{i+1}}$, we have that
$$|E_{G'}(v, S_1, S_2, S_3)| \geq \frac{0.9c_1^3t_ip_{\gr}^3}{\log^{4l_3+8}(n)}.$$
\end{cor}
\begin{proof}
First note that given that $G$ has depth $j$, we have that $S \subseteq K_{j-1} \setminus K_{j+1}$, and since $S$ is $i$-valid it follows that $|S| \geq k_{j+1}$. Since $k_{j-1} \geq t_i$, it follows that $|S| \geq \frac{t_i}{\log^{2}(n)}$. Similarly, for $(v, S_1, S_2, S_3)$ we have that $|S_1| \geq \frac{t_i}{\log^{2}(n)}$.  Furthermore, combining Propositions \ref{prop_o'} and \ref{prop_bound_err}, we have for every $v \in V(G[I_{t_{i+1}}])$ that $1-d_{\omega^o}(v)\geq 0.99\omega(E_G(v, I_{t_{i+1}}))$, and by Proposition \ref{prop_omega_permiss}, that $0.99\omega(E_G(v, I_{t_{i+1}})) \geq \frac{0.99c_1}{\log^{l_3+2}(n)}$. Additionally by $i$-permissibility we have that
$|V(G[S])| \geq \frac{t_ip_{\gr}}{\log^{l_2+2}(n)}$
and 
$|E_G(v, S_1, S_2, S_3)| \geq \frac{t_ip_{\gr}^3}{\log^{l_3+2}(n)}$.
Thus, by Theorem \ref{thm_key}, we find that 
$$|V(G'[S])| \geq \frac{0.9c_1t_ip_{\gr}}{\log^{l_2+l_3+4}(n)},$$
and
$$|E_{G'}(v, S_1, S_2, S_3)| \geq \frac{0.9c_1^3t_ip_{\gr}^3}{\log^{4l_3+8}(n)},$$
as required.
\end{proof}

%% file: sec_i.tex
\section{Reaching $L^*$} \label{sec_i}

In this section, starting from $(H_{1^*}, w_{1^*})$ as in Theorem \ref{thm_special_reweight}, we describe the process to reach $L^*$ via the vortex described in Section \ref{sec_overview}. We summarise our strategy roughly in the following plan. Each iteration of the plan starts with a weighted hypergraph $(H_i, w_i)$ and subsequently outputs a weighted hypergraph $(H_{i+1}, w_{i+1})$ which is used as the input for the next iteration. Crucially, we actually start with $(H_{1^*}, w_{1^*})$ but for the purposes of the plan below, by abuse of notation (since $H_1$ and $w_1$ are defined differently elsewhere), we relabel $(H_{1^*}, w_{1^*})$ as $(H_1, w_1)$, only for within the plan below.

\begin{plan} \label{alg_subgraphs}
~\\~
{\bf Initialise:} $i=1$. $(H_i, w_i)$, $i$-permissible, where in fact by $(H_1, w_1)$ we mean $(H_{1^*}, w^{1^*})$ as alluded to above.

{\bf Step 1:} Find a matching $M^o_i$ in $(H^o_i, w^o_i)$ via Theorem \ref{thm_weighted_egj} and define $H_{i.1}:=H_i[V(H_i) \setminus V(M^o_i)]$.

{\bf Step 2:} Obtain a weight function $w_{i.1}$ for $H_{i.1}$ such that $w_{i.1}|_{H_{i.1}[I_{t_{i+1}}]}$ is an almost-perfect fractional matching for $H_{i.1}[I_{t_{i+1}}]$.

{\bf Step 3:} Run a (random) greedy cover for vertices in $V(H_{i.1})\setminus I_{t_{i+1}}$ to obtain a matching $M^c_i$. Define $H_{i+1}=H_i[V(H_i) \setminus V(M_i)]$, where $M_i:= M^o_i \cup M^c_i$.

{\bf Step 4:} Define $w_{i+1}$ in terms of $w_i$ and $H_i$ so that $w_{i+1}$ is an almost-perfect fractional matching for $H_{i+1}$.

{\bf Step 5:} If $(H_{i+1}, w_{i+1})$ is not $(i+1)$-permissible, abort. If $V(H_i) \subseteq I_{n^{10^{-5}}}$ stop. Else, increase $i$ by $1$ and go to Step 1.

\end{plan}

Similarly to Section \ref{sec_1}, we show that Plan \ref{alg_subgraphs} does not abort prematurely, and subsequently that we can successfully reach $L^*$. We shall first explain the strategy to obtain $(H_{i+1}, w_{i+1})$ from $(H_i, w_i)$, filling in the details of a single iteration of Plan \ref{alg_subgraphs} given that $(H_i, w_i)$ is $i$-permissible. We then deduce, by strong induction, the properties of $(H_{i+1}, w_{i+1})$ in terms of $(H_{1^*}, w_{1^*})$ and hence show via backtracking that $(H_{i+1}, w_{i+1})$ is $(i+1)$-permissible, ensuring that the algorithm completes successfully. In the next section we go through one iteration of the Algorithm, filling in details for how we obtain the weight functions $w_{i.1}$ and $w_{i+1}$ and observing how $(H_{i+1}, w_{i+1})$ `looks' in terms of $(H_i, w_i)$. 

\subsection{One iteration of Plan \ref{alg_subgraphs}}

We go through the steps of Plan \ref{alg_subgraphs} one by one for one iteration, filling in the details of how we intend to get from $(H_i, w_i)$ to $(H_{i+1}, w_{i+1})$, where we assume throughout this section that $(H_i, w_i)$ is $i$-permissible.

\subsubsection{Step 1}

By Theorem \ref{thm_key}, since $(H_i, w_i)$ is $i$-permissible, we obtain $M_i^o$ and define $H_{i.1}$ as above. 

\subsubsection{Step 2}

Below we will define the new weighting $w_{i.1}(e):E(H_i) \rightarrow \mathbb{R}_{\geq 0}$. First let $w_{i.0}(e):E(H_i) \rightarrow \mathbb{R}_{\geq 0}$ be given by
$$w_{i.0}(e):=
\begin{cases}
\frac{w_{i}(e)}{\prod_{u \in e \cap I_{t_{i+1}}} (1- d_{w^o_{i}}(u))} & \mbox{~if $e$ is $i$-reachable~} \\
w_{1^*}(e) & \mbox{~otherwise.} \\
\end{cases}
.$$

\begin{prop}\label{prop_i.0}
For each $i$-reachable edge $e \in H_i$, 
$$w_i(e)\leq w_{i.0}(e) \leq \frac{1.1w_i(e)\log^{4l_3+8}(n)}{c_1^4}.$$
\end{prop}
\begin{proof}
Note that, since $(H_i, w_i)$ is $i$-permissible, by Proposition \ref{prop_o'}, for $u \in I_{t_{i+1}}$ we have that $1- d_{w_i^o}(u)=\left(1 \pm \frac{c_{H_i,v}}{\delta_{H_i,v}}\right)w_i(E_{H_i}(u,I_{t_{i+1}}))$, and by Proposition \ref{prop_bound_err}, $\frac{c_{H_i,v}}{\delta_{H_i,v}}\leq \frac{\log^{l_4+l_3+2}(n)}{c_1t_i^{\epsilon_1}}$. Furthermore, by Proposition \ref{prop_omega_permiss} we have that $\frac{c_1}{\log^{l_3+2}(n)} \leq w_i(E_{H_i}(v,I_{t_{i+1}}))$ since $t_i \leq |I_{t_{i+1}}|$, and $w_i(E_{H_i}(v,I_{t_{i+1}}))\leq d_{w_i, H_i}(v) \leq 1$. Thus clearly $\prod_{u \in e \cap I_{t_{i+1}}} (1- d_{w^o_{i}}(u)) \leq 1$
and
\begin{eqnarray*}
\prod_{u \in e \cap I_{t_{i+1}}} (1- d_{w^o_{i}}(u)) &=& \prod_{u \in e \cap I_{t_{i+1}}} \left(1 \pm \frac{\log^{l_4+l_3+2}(n)}{c_1t_i^{\epsilon_1}}\right)w_i(E_{H_i}(u,I_{t_{i+1}})) \\
&\geq& \left(1 - \frac{\log^{l_4+l_3+2}(n)}{c_1t_i^{\epsilon_1}}\right)^4 \frac{c_1^4}{\log^{4l_3+8}(n)} \\
&\geq& \left(1 - \frac{4.1\log^{l_4+l_3+2}(n)}{c_1t_i^{\epsilon_1}}\right)\frac{c_1^4}{\log^{4l_3+8}(n)}.
\end{eqnarray*}
The result follows.
\end{proof}

\begin{prop} \label{prop_i.0_allow}
$w_{i.0}$ is edge allowable for $(H_i, w_i, t_i, \eta_1)$.
\end{prop}
\begin{proof}
By Theorem \ref{thm_key} and the remark following it, if we can show that for each $i$-reachable tuple $(v,S_1,S_2,S_3)$ that $\frac{\max_{e \in E_{H^o_i}(v,S_1,S_2,S_3)} w_{i.0}(e)}{\min_{e \in E_{H^o_i}(v,S_1,S_2,S_3)} w_{i.0}(e)} \leq \log^{500}(n)$, then $w_{i.0}$ is indeed edge allowable for $(H_i, w_i, t_i, \eta_1)$. By Proposition \ref{prop_i.0} and (P3), for an $i$-reachable edge $e$ we have that $w_{i.0}(e) \geq w_i(e) \geq \frac{c_1}{p_{\gr}^3t_i\log^{2}(n)}$ and $w_{i.0}(e) \leq \frac{1.1\log^{l_1+4l_3+8}(n)}{c_1^4t_ip_{\gr}^3}$.\footnote{Note that any edge $e$ that is $i$-reachable but not $(i-1)$-reachable satisfies this by Theorem \ref{thm_special_reweight} combined with (P3).} Then for each $i$-reachable tuple $(v,S_1,S_2,S_3)$ we certainly have
$$\frac{\max_{e \in E_{H^o_i}(v,S_1,S_2,S_3)} w_{i.0}(e)}{\min_{e \in E_{H^o_i}(v,S_1,S_2,S_3)} w_{i.0}(e)} \leq \frac{2\log^{l_1+4l_3+8}(n)}{c_1^5} \leq \log^{500}(n),$$
as required.
\end{proof}

We modify $w_{i.0}$ to find a fractional matching for $H_{i.1}[I_{t_{i+1}}]$. Let 
$$d_{i.0}^*:=\max_{v \in V(H_{i.1})} w_{i.0}(E_{H_{i.1}}(v,I_{t_{i+1}}))$$ 
and define
$$w_{i.1}(e)=
\begin{cases}
\frac{w_{i.0}(e)}{d_{i.0}^*} & \mbox{~if $e$ is $i$-reachable and $d_{i.0}^*\geq 1$~} \\
w_{i.0}(e) & \mbox{~if $e$ is $i$-reachable and $d_{i.0}^*< 1$~} \\
w_{1^*}(e) & \mbox{~otherwise.} \\
\end{cases}
.$$

\begin{prop} \label{prop_d*_i}
$d_{i.0}^*\leq 1+\log^{3l_3+7}(n)t_i^{-\epsilon_1}$.
\end{prop}
\begin{proof}
Since $w_{i.0}$ is edge allowable for $(H_i, w_i, t_i, \eta_1)$, we have by (K2) that for every $v \in V(H_{i.1})$,
$$w_{i.0}(E_{H_{i.1}}(v,I_{t_{i+1}}))=(1 \pm \log^{3l_3+7}(n)t_i^{-\epsilon_1})\sum_{e \in E_{H_i}(v,I_{t_{i+1}})} w_{i.0}(e) \prod_{u \in e \setminus \{v\}}(1-d_{w_{i}^o}(u)).$$
Note that every edge $e \in E_{H_i}(v,I_{t_{i+1}})$ either has all four vertices (including $v$) contained in $I_{t_{i+1}}$, or has all vertices except for $v$ contained in $I_{t_{i+1}}$. Then for $v \in V(H_{i.1}[I_{t_i}\setminus I_{t_{i+1}}])$,
\begin{multline*}
\sum_{e \in E_{H_i}(v,I_{t_{i+1}})} w_{i.0}(e) \prod_{u \in e \setminus \{v\}}(1-d_{w_{i}^o}(u))= \sum_{e \in E_{H_i}(v,I_{t_{i+1}})} w_{i}(e) \\ =w_i(E_{H_i}(v,I_{t_{i+1}}))\leq 1,
\end{multline*} 
and for $v \in V(H_{i.1}[I_{t_{i+1}}])$,
\begin{equation} \label{eq_i.0}
\sum_{e \in E_{H_i}(v,I_{t_{i+1}})} w_{i.0}(e) \prod_{u \in e \setminus \{v\}}(1-d_{w_{i}^o}(u))=\frac{\sum_{e \in E_{H_i}(v,I_{t_{i+1}})} w_{i}(e)}{1-d_{w_{i}^o}(v)}=\frac{w_i(E_{H_i}(v,I_{t_{i+1}}))}{1-d_{w_{i}^o}(v)}.
\end{equation}
Now, by Proposition \ref{prop_io}, for $v \in H_i[I_{t_{i+1}}]$ we have that $\frac{w_i(E_{H_i}(v,I_{t_{i+1}}))}{1-d_{w_{i}^o}(v)}\leq 1$. It follows that $w_{i.0}(E_{H_{i.1}}(v, I_{t_{i+1}})) \leq 1+\log^{3l_3+7}(n)t_i^{-\epsilon_1}$ for every $v \in V(H_{i.1})$, as required.
\end{proof}

\begin{cor} \label{cor_i.1_allow}
$w_{i.1}$ is edge allowable for $(H_i, w_i, t_i, \eta_1)$.
\end{cor}
\begin{proof}
Just as in Proposition \ref{prop_i.0_allow}, it suffices to show that over all $i$-reachable edges $e$, that
$\frac{\max_{e} w_{i.1}(e)}{\min_{e} w_{i.1}(e)} \leq \log^{500}(n)$. Indeed, $\max_{e} w_{i.1}(e) \leq \max_e w_{i.0}(e)$ and by Proposition \ref{prop_d*_i}, $\min_{e} w_{i.1}(e) \geq \min_e w_{i.0}(e)/2$. Thus
$$\frac{\max_{e} w_{i.1}(e)}{\min_{e} w_{i.1}(e)} \leq \frac{4\log^{l_1+4l_3+8}(n)}{c_1^5} \leq \log^{500}(n),$$
as required.
\end{proof}
 
\begin{cor} \label{cor_fm_i}
$w_{i.1}$ is a fractional matching for $H_{i.1}[I_{t_{i+1}}]$ such that 
$$w_{i.1}(E_{H_{i.1}}(v,I_{t_{i+1}})) \geq  1 - \left(\frac{c_{H_i,v}}{\delta_{H_i,v}}+2.1\log^{3l_3+7}(n)t_i^{-\epsilon_1}\right)$$ 
for every $v \in V(H_{i.1}[I_{t_{i+1}}])$. Furthermore,
\begin{equation}\label{eq_i.1}
w_{i.1}(e)=
\begin{cases}
(1 \pm \log^{3l_3+7}(n)t_i^{-\epsilon_1})w_{i.0}(e) & \mbox{~if $e \in H_i$ is $i$-reachable~} \\
w_{1^*}(e) & \mbox{~otherwise.} \\
\end{cases}
\end{equation}
\end{cor}
\begin{proof}
That $w_{i.1}$ is a fractional matching for $H_{i.1}[I_{t_{i+1}}]$ follows immediately from the construction. From (\ref{eq_i.0}) we have that for every $v \in H_i[I_{t_{i+1}}]$,
$$w_{i.0}(E_{H_{i.1}}(v,I_{t_{i+1}}))=(1 \pm \log^{3l_3+7}(n)t_i^{-\epsilon_1})\frac{w_i(E_{H_i}(v, I_{t_{i+1}}))}{1-d_{w_i^o}(v)},$$
and by Proposition \ref{prop_o'} we have that $1-d_{w_i^o}(v)=\left(1 \pm \frac{c_{H_i, v}}{\delta_{H_i, v}}\right)w_i(E_{H_i}(v,I_{t_{i+1}}))$. Thus
$$w_{i.0}(E_{H_{i.1}}(v,I_{t_{i+1}}))=\frac{(1 \pm \log^{3l_3+7}(n)t_i^{-\epsilon_1})}{\left(1 \pm \frac{c_{H_i, v}}{\delta_{H_i, v}}\right)} \geq 1-\left(\frac{c_{H_i, v}}{\delta_{H_i, v}}+\log^{3l_3+7}(n)t_i^{-\epsilon_1}\right).$$
Since by Proposition \ref{prop_d*_i} we have that $w_{i.1}(E_{H_{i.1}}(v,I_{t_{i+1}})) \geq \frac{w_{i.0}(E_{H_{i.1}}(v,I_{t_{i+1}}))}{1+\log^{3l_3+7}(n)t_i^{-\epsilon_1}}$, the first claim follows.

Finally, also by Proposition \ref{prop_d*_i}, for each $e$ which is $i$-reachable, $w_{i.0}(e) \geq w_{i.1}(e) \geq \frac{w_{i.0}(e)}{1 +\log^{3l_3+7}(n)t_i^{-\epsilon_1}}$. This yields that $w_{i.1}(e)=(1 \pm \log^{3l_3+7}(n)t_i^{-\epsilon_1})w_{i.0}(e)$, as required.  
\end{proof}

\subsubsection{Step 3}

We now wish to cover all vertices remaining outside $I_{t_{i+1}}$ in $H_{i.1}$ in order to reach $H_{i+1}$. We do this via the following random greedy algorithm. Let $v_1, v_2, \ldots, v_{\chi}$ be an arbitrary enumeration of the vertices in $V(H_{i.1}[I_{t_i}\setminus I_{t_{i+1}}])$. We build a matching $M^c_i:=\{e_i: i \in [\chi]\}$ as follows. For every $i \in [\chi]$, one by one we choose an edge $e_i$ for $v_i$ so that $e_i \in E_{H_{i.1}}(v_i, I_{t_i})$ and $e_i$ is chosen uniformly at random from all such edges that are disjoint from all previous choices $\{e_j\}_{j<i}$. If there is no such choice available for $e_i$ for some $i \in [\chi]$ the algorithm aborts.

We want to track how this process affects degree-type properties and the density of vertices remaining in the graph. Our random greedy algorithm uses the properties given in Corollary \ref{cor_props_for_greedy}, that we have (G1)-(G3) with $H_{i.1}$ in place of $G$. 

Note that since there are at most $2.1c_2\log^{l_{4}}(n)t_i^{1-\epsilon_1}p_{\gr}$ vertices to cover via the random greedy algorithm above and each vertex is in at least $\frac{0.9c_1t_ip_{\gr}^3}{\log^{2l_3+2}(n)}$ suitable edges, where, recalling Definition \ref{def_const}, $p_{\gr}^2 \gg t_i^{-\epsilon_1}$, a greedy algorithm would successfully complete. Indeed, since the maximum pair degree in $H_{i.1}$ is at most $1$ (for every $i \geq 1$, since $\mathcal{T}[I_{t_1}]$ contains no wrap-around edges), choosing one edge destroys at most $4$ choices for the next vertex, so by the final choice we still have at least $\frac{0.9c_1^3t_ip_{\gr}^3}{\log^{4l_3+6}(n)}-8.4c_2\log^{l_{4}}(n)t_i^{1-\epsilon_1}p_{\gr} \geq \frac{0.89c_1^3t_ip_{\gr}^3}{\log^{4l_3+6}(n)}$. However we wish to run a {\it random} greedy algorithm to ensure `nice' properties remain in the graph at the end of the process, in particular those properties relating to permissibility of a weighted subgraph of $H_{1^*}$.

Let $p^i_v$ be the probability that a vertex $v \in V(H_{i.1}[I_{t_{i+1}}])$ is covered by the random greedy cover process.

\begin{prop} \label{prop_piv_gen}
Suppose that $8|V(H_{i.1}[I_{t_i}\setminus I_{t_{i+1}}])| < d_{H_{i.1}}(v)$ for every vertex \newline $v \in V(H_{i.1}[I_{t_i}\setminus I_{t_{i+1}}])$. Then for every $v \in V(H_{i.1}[I_{t_{i+1}}])$,
$$p^i_v \leq \frac{2|E_{H_{i.1}}(v,I_{t_{i}}\setminus I_{t_{i+1}},*,*)|}{\min_{v \in V(H_{i.1}[I_{t_i}\setminus I_{t_{i+1}}])} d_{H_{i.1}}(v)}.$$
\end{prop}
\begin{proof}
Since $v \in V(H_{i.1}[I_{t_{i+1}}])$, the probability that it is covered in the random greedy cover is the probability that it is in an edge chosen for one of the vertices in $V(H_{i.1}[I_{t_i}\setminus I_{t_{i+1}}])$. There are at most $|E_{H_{i.1}}(v,I_{t_{i}}\setminus I_{t_{i+1}},*,*)|$ instances where an edge containing $v$ might be chosen to cover a vertex in $V(H_{i.1}[I_{t_i}\setminus I_{t_{i+1}}])$. Every time an edge is chosen for a vertex $v_i \in V(H_{i.1}[I_{t_i}\setminus I_{t_{i+1}}])$ it reduces the possible choices for $v_{i+1}$ by at most $4$, (since the four vertices in $e_i$ are no longer available, and each of these could have been in at most one edge with $v_{i+1}$). Thus since every vertex in $V(H_{i.1}[I_{t_i}\setminus I_{t_{i+1}}])$ starts with at least $\min_{v \in V(H_{i.1}[I_{t_i}\setminus I_{t_{i+1}}])} d_{H_{i.1}}(v)$ choices, and there are $|V(H_{i.1}[I_{t_i}\setminus I_{t_{i+1}}])|$ such vertices to consider in the process, every vertex $v_i$ we wish to cover will have at least 
$$\min_{v \in V(H_{i.1}[I_{t_i}\setminus I_{t_{i+1}}])} d_{H_{i.1}}(v)-4|V(H_{i.1}[I_{t_i}\setminus I_{t_{i+1}}])|\geq \frac{1}{2}\min_{v \in V(H_{i.1}[I_{t_i}\setminus I_{t_{i+1}}])} d_{H_{i.1}}(v)$$ choices for the edge $e_i$ used to cover it. The proposition follows. 
\end{proof}

\begin{cor} \label{cor_piv}
For every $v \in V(H_{i.1}[I_{t_{i+1}}])$,
$$p^i_v < \frac{5c_3\log^{4l_3+l_4+6}(n)}{c_1^3t_i^{\epsilon_1}}.$$
\end{cor}
\begin{proof}
Using the bounds from Corollary \ref{cor_props_for_greedy} and Proposition \ref{prop_piv_gen}, we get
$$p^i_v \leq 2\cdot 2c_3\log^{l_4}(n)t_i^{1-\epsilon_1}p_{\gr}^3 \cdot \frac{\log^{4l_3+6}(n)}{0.9c_1^3t_ip_{\gr}^3} \leq \frac{5c_3\log^{4l_3+l_4+6}(n)}{c_1^3t_i^{\epsilon_1}},$$
as required.
\end{proof}

Let $p^i_e$ be the probability that an edge $e \in H_{i.1}[I_{t_{i+1}}]$ does not survive the greedy cover step. That is, the probability that at least one of the vertices in $e$ is in an edge that is used by the greedy cover step.
\begin{prop} \label{prop_pie_gen}
For every $e \in H_{i.1}[I_{t_{i+1}}]$,
$$p^i_e \leq 4\max_{v \in e} p^i_v.$$
\end{prop}
\begin{proof}
Let $e \in H_{i.1}[I_{t_{i+1}}]$. Then by a union bound we have $p^i_e=\bigcup_{v \in e} p^i_v \leq \sum_{v \in e} p^i_v \leq 4\max_{v \in e} p^i_v$ as required.
\end{proof}

\begin{cor}\label{cor_pie}
For every $e \in H_{i.1}[I_{t_{i+1}}]$,
$$p^i_e < \frac{20c_3\log^{4l_3+l_4+6}(n)}{c_1^3t_i^{\epsilon_1}}.$$
\end{cor}

We use the bounds from Corollaries \ref{cor_piv} and \ref{cor_pie} to look at properties in the graph remaining once the random greedy algorithm has completed. As previously discussed, since a greedy algorithm would not abort, it is clear that the random greedy cover will be able to cover all vertices in $V(H_{i.1}[I_{t_i}\setminus I_{t_{i+1}}])$, and so we obtain a matching $M^c_i$ covering $V(H_{i.1}[I_{t_i}\setminus I_{t_{i+1}}])$. Let 
$$H_{i+1}:=H_{i.1}[V(H_{i.1})\setminus V(M^c_i)].$$  

\begin{lemma} \label{lemma_vertex_freedman}
Let $S \subseteq V(\mathcal{T}[I_{t_{i+1}}])$ be $i$-reachable, and let $f_S$ be a function such that $\frac{\max_{v \in V(H_i[I_{t_{i+1}}])} f_S(v)}{\min_{v \in V(H_i[I_{t_{i+1}}])} f_S(v)} \leq \log^{500}(n)$. With high probability 
$$\sum_{v \in V(H_{i+1}[S])} f_S(v)=(1 \pm t_i^{-1.5\epsilon_1})\sum_{v \in V(H_{i.1}[S])}f_S(v)(1-p^i_v),$$
and in particular the number of vertices which survive the greedy cover process to $V(H_{i+1}[S])$ satisfies
$$|V(H_{i+1}[S])|=(1 \pm t_i^{-1.5\epsilon_1})\sum_{v \in V(H_{i.1}[S])}(1-p^i_v).$$
\end{lemma}
Note that we have $t_i^{-1.5\epsilon_1}$ where previous similar equations have included $t_i^{-\epsilon_1}$. The key detail here is that it is of the form $t_i^{-c\epsilon_1}$ for some $c>1$ to avoid a blow-up of error terms.
\begin{proof}
Let $X^{f,S}$ be the total weight removed from $V(H_{i.1}[S])$ with respect to $f_S$ as a result of the random greedy cover. Then we have that $\mathbb{E}(X^{f,S})=\sum_{v \in V(H_{i.1}[S])}f_S(v)p^i_v$. Furthermore, we may write $X^{f,S}=\sum_j X^{f,S}_j$ where $X^{f,S}_j$ is the weight removed from $V(H_{i.1}[S])$ with respect to $f_S$ as a result of the choice of edge $e_j$ for vertex $v_j$ in the random greedy algorithm. Let $\mathbb{E}'(X^{f,S}_j)$ be the conditional expectation of $X^{f,S}_j$ given that $e_1, \ldots, e_{j-1}$ have been revealed, and let $Y^{f,S}_j:=\sum_{k \leq j} (X^{f,S}_k-\mathbb{E}'(X^{f,S}_k))$. Write $f_S^{\max}:=\max_{v \in V(H_{i.1}[S])} f_S(v)$ and $f_S^{\min}:=\min_{v \in V(H_{i.1}[S])} f_S(v)$. Then $|Y^{f,S}_j - Y^{f,S}_{j-1}| \leq |X^{f,S}_j - \mathbb{E}'(X^{f,S}_j)| \leq 3f_S^{\max}$, since the edge $e_j$ chosen for $v_j$ contains at most $3$ vertices in $S$. Thus, by the Azuma-Hoeffding Inequality (Lemma \ref{azuma_mart}), we have that
$$\mathbb{P}(|Y^{f,S}_{\chi}|>f_S^{\min}t_i^{0.51}) \leq 2 \exp\left(-\frac{(f_S^{\min})^2t_i^{1.02}}{2\sum_{i=1}^{\chi} (3f_S^{\max})^2}\right).$$
Then since $\chi \leq 2.1c_2\log^{l_4}(n)t_i^{1-\epsilon_1}p_{\gr} \leq 2.1c_2\log^{l_4}(n)t_i^{1-\epsilon_1}$, and $\frac{f_S^{\max}}{f_S^{\min}}\leq \log^{500}(n)$ we have that whp
$$X^{f,S}=\left(\sum_{j \in [\chi]} \mathbb{E}'(X^{f,S}_j)\right) \pm f_S^{\min}t_i^{0.51}.$$
\begin{claim}
For each $j \in [\chi]$, we have that 
$$\mathbb{E}'(X^{f,S}_j) \leq \mathbb{E}(X^{f,S}_j)\left(1+\frac{15c_2\log^{4l_3+l_4+6}(n)}{c_1^3t_i^{\epsilon_1}p_{\gr}^2}\right),\mbox{~and~}
$$
$$\mathbb{E}'(X^{f,S}_j) \geq \mathbb{E}(X^{f,S}_j)-\frac{6.3f_S^{\max}c_2\log^{4l_3+l_4+6}(n)}{0.9c_1^3t_i^{\epsilon_1}p_{\gr}^2}.
$$
\end{claim}
\begin{proof}
Let $E_j$ be the set of edges that could be chosen for $e_j$ and let $E'_j$ be the set of edges that could be chosen for $e_j$ given the choices for $e_1, \ldots, e_{j-1}$. First note that $\mathbb{E}(X^{f,S}_j)=\sum_{e \in E_j} \frac{\sum_{v \in (e \cap S)} f_S(v)}{|E_j|}$ and that $\mathbb{E}'(X^{f,S}_j)=\sum_{e \in E_j'} \frac{\sum_{v \in (e \cap S)} f_S(v)}{|E_j'|}$. Then $\mathbb{E}'(X^{f,S}_j)\geq \frac{(\sum_{e \in E_j} \sum_{v \in (e \cap S)} f_S(v))-3f_S^{\max}\chi}{|E_j|}=\mathbb{E}(X^{f,S}_j)-\frac{3f_S^{\max}\chi}{|E_j|}$. Now by (G1) we have that $\chi \leq 2.1c_2\log^{l_4}(n)t_i^{1-\epsilon_1}p_{\gr}$ and by (G2) $|E_j| \geq \frac{0.9c_1^3t_ip_{\gr}^3}{\log^{4l_3+6}(n)}$ so that
$$\frac{3f_S^{\max}\chi}{|E_j|} \leq \frac{6.3f_S^{\max}c_2\log^{4l_3+l_4+6}(n)}{0.9c_1^3t_i^{\epsilon_1}p_{\gr}^2},$$
and the lower bound claim holds.

For the upper bound note that $\mathbb{E}'(X^{f,S}_j) \leq \sum_{e \in E_j} \frac{\sum_{v \in (e \cap S)} f_S(v)}{|E_j'|} = \frac{|E_j|}{|E_j'|}\mathbb{E}(X^{f,S}_j)$. Now $|E_j'| \geq |E_j|-3\chi$ and $|E_j| \gg \chi$. Thus $\mathbb{E}'(X^{f,S}_j) \leq \mathbb{E}(X^{f,S}_j)(1+\frac{6.1\chi}{|E_j|}) \leq \mathbb{E}(X^{f,S}_j)\left(1+\frac{15c_2\log^{4l_3+l_4+6}(n)}{c_1^3t_i^{\epsilon_1}p_{\gr}^2}\right)$, as stated.
\end{proof}
By this claim we have that
$$X^{f,S} \geq \mathbb{E}(X^{f,S}) - \left(\frac{15c_2^2f_S^{\max}\log^{4l_3+2l_4+6}(n)t_i^{1-2\epsilon_1}}{c_1^3p_{\gr}} + f_S^{\min}t_i^{0.51}\right),$$
and 
$$X^{f,S} \leq \mathbb{E}(X^{f,S}) + \frac{600c_2^2c_3f_S^{max}\log^{8l_3+2l_4+12}(n)t_i^{1-2\epsilon_1}}{c_1^6p_{\gr}}+ f_S^{\min}t_i^{0.51},$$
where 
\begin{multline*}
\mathbb{E}(X^{f,S})=\sum_{v \in V(H_{i.1}[S])}f_S(v)p^i_v \leq f_S^{\max}p^i_v|V(H_{i.1}[S])| \\ \leq \frac{40f_S^{\max}c_2c_3\log^{4l_3+l_4+6}(n)t_i^{1-\epsilon_1}p_{\gr}}{c_1^3}
\end{multline*} 
using (P4) with that $|V(H_{i.1}[S])| \leq |V(H_{i}[S])|$ and Corollary \ref{cor_piv}. This gives that 
$$X^{f,S}=\mathbb{E}(X^{f,S}) \pm O\left(f_S^{max}p_{\gr}^{-1}\log^{8l_3+2l_4+12}(n)t_i^{1-2\epsilon_1}\right).$$

Furthermore, $\sum_{v \in V(H_{i.1}[S])} f_S(v) \geq f_S^{\min}|V(H_{i.1}[S])| \geq \frac{0.9c_1f_S^{\min}t_ip_{\gr}}{\log^{l_2+l_3+4}(n)}$ by Corollary \ref{cor_lower_bounds}. Thus
\begin{multline*}
\sum_{v \in V(H_{i+1}[S])} f_S(v) =\sum_{v \in V(H_{i.1}[S])} f_S(v)-X^{f,S} \\
= \left(\sum_{v \in V(H_{i.1}[S])} f_S(v)(1 - p^i_v)\right) \pm O\left(f_S^{max}p_{\gr}^{-1}\log^{8l_3+2l_4+12}(n)t_i^{1-2\epsilon_1}\right)\\
= \left(1 \pm O\left(\frac{f_S^{\max}\log^{l_2+9l_3+2l_4+16}(n)t_i^{-2\epsilon_1}}{c_1f_S^{\min}p_{\gr}^2}\right)\right)\sum_{v \in V(H_{i.1}[S])} f_S(v)(1 - p^i_v).
\end{multline*}
Since $t_i^{-\epsilon_1/2}p_{\gr}^{-2} \ll 1$ for all $i$ and $\frac{f_S^{\max}}{f_S^{\min}}\leq \log^{500}(n)$, we have that 
$$O\left(\frac{f_S^{\max}\log^{l_2+9l_3+2l_4+16}(n)t_i^{-2\epsilon_1}}{c_1f_S^{\min}p_{\gr}^2}\right) \leq t_i^{-1.5\epsilon_1}$$ 
as required to complete the proof.
\end{proof} 

\begin{lemma} \label{lemma_degree_freedman}
Let $v$ be a vertex that survives the greedy cover. Let $(v,S_1,S_2,S_3)$ be an open or closed $i$-reachable tuple such that $S_1, S_2, S_3 \subseteq I_{t_{i+1}}$. Let $f$ be edge allowable for $(H_i, w_i, t_i, \eta_1)$ such that over all edges $e$ that are $i$-reachable $\frac{\max_e f(e)}{\min_e f(e)}\leq \log^{500}(n)$. Then 
$$f(E_{H_{i+1}}(v,S_1,S_2,S_3))=(1 \pm t_i^{-1.5\epsilon_1})\sum_{e \in E_{H_{i.1}}(v,S_1,S_2,S_3)}f(e)(1-p^i_e).$$
In particular, 
$$|E_{H_{i+1}}(v,S_1,S_2,S_3)|=(1 \pm t_i^{-1.5\epsilon_1})\sum_{e \in E_{H_{i.1}}(v,S_1,S_2,S_3)}(1-p^i_e).$$
\end{lemma}
\begin{proof}
The proof follows precisely the same strategy as that of Lemma \ref{lemma_vertex_freedman} with many details exactly the same. 

Let $X^{f}$ be the total weight removed from $f(E_{H_{i.1}}(v,S_1, S_2, S_3))$ as a result of the random greedy cover for some fixed $i$-reachable tuple $(v, S_1, S_2, S_3)$. Then we have that $\mathbb{E}(X^f)=\sum_{e \in E_{H_{i.1}}(v,S_1,S_2,S_3)}f(e)p^i_e$. Furthermore, we may write $X^f=\sum_j X^f_j$ where $X^f_j$ is the weight removed from $f(E_{H_{i.1}}(v,S_1, S_2, S_3))$ as a result of the choice of edge $e_j$ for vertex $v_j$ in the random greedy algorithm. Let $\mathbb{E}'(X^f_j)$ be the conditional expectation of $X^f_j$ given that $e_1, \ldots, e_{j-1}$ have been revealed, and let $Y^f_j:=\sum_{k \leq j} (X^f_k-\mathbb{E}'(X^f_k))$. Write $f^{\max}:=\max_{e \in E_{H_{i.1}}(v,S_1, S_2, S_3)} f(e)$ and $f^{\min}:=\min_{e \in E_{H_{i.1}}(v,S_1, S_2, S_3)} f(e)$. Then $|Y^f_j - Y^f_{j-1}| \leq |X^f_j - \mathbb{E}'(X^f_j)| \leq 3f^{\max}$, since the edge $e_j$ chosen for $v_j$ contains at most $3$ vertices in $I_{t_{i+1}}$ and each of these vertices can be in at most one edge that lies in $E_{H_{i.1}}(v,S_1, S_2, S_3)$. Thus, by Azuma-Hoeffding Inequality (Lemma \ref{azuma_mart}), we have that
$$\mathbb{P}(|Y^f_{\chi}|>f^{\min}t_i^{0.51}) \leq 2 \exp\left(-\frac{(f^{\min})^2t_i^{1.02}}{2\sum_{i=1}^{\chi} (3f^{\max})^2}\right).$$
Then since $\chi \leq 2.1c_2\log^{l_4}(n)t_i^{1-\epsilon_1}p_{\gr} \leq 2.1c_2\log^{l_4}(n)t_i^{1-\epsilon_1}$, and $\frac{f^{\max}}{f^{\min}}\leq \log^{500}(n)$ we have that whp
$$X^f=\left(\sum_{j \in [\chi]} \mathbb{E}'(X^f_j)\right) \pm f^{\min}t_i^{0.51}.$$
\begin{claim}
For each $j \in [\chi]$, we have that 
$$\mathbb{E}(X^f_j)\left(1+\frac{15c_2\log^{4l_3+l_4+6}(n)}{c_1^3t_i^{\epsilon_1}p_{\gr}^2}\right)\geq \mathbb{E}'(X^f_j)\geq \mathbb{E}(X^f_j)-\frac{6.3f^{\max}c_2\log^{4l_3+l_4+6}(n)}{0.9c_1^3t_i^{\epsilon_1}p_{\gr}^2}.$$ 
\end{claim}
\begin{proof}
Let $E_j$ be the set of edges that could be chosen for $e_j$ and let $E'_j$ be the set f edges that could be chosen for $e_j$ given the choices for $e_1, \ldots, e_{j-1}$. First note that 
$$\mathbb{E}(X^f_j)=\sum_{e \in E_j} \frac{\sum_{e' \in E_{H_{i.1}}(v,S_1, S_2, S_3): e \cap e' \neq \emptyset} f(e')}{|E_j|},$$ 
and that 
$$\mathbb{E}'(X^f_j)=\sum_{e \in E_j'} \frac{\sum_{e' \in E_{H_{i.1}}(v,S_1, S_2, S_3): e \cap e' \neq \emptyset} f(e')}{|E_j'|}.$$ Then $\mathbb{E}'(X^f_j)\geq \frac{(\sum_{e \in E_j} \sum_{e' \in E_{H_{i.1}}(v,S_1, S_2, S_3): e \cap e' \neq \emptyset} f(e'))-3f^{\max}\chi}{|E_j|}=\mathbb{E}(X^f_j)-\frac{3f^{\max}\chi}{|E_j|}$. Now by (G1) and (G2) we have that $\chi \leq 2.1c_2\log^{l_4}(n)t_i^{1-\epsilon_1}p_{\gr}$ and $|E_j| \geq \frac{0.9c_1^3t_ip_{\gr}^3}{\log^{4l_3+6}(n)}$ so that
$$\frac{3f^{\max}\chi}{|E_j|} \leq \frac{6.3f^{\max}c_2\log^{4l_3+l_4+6}(n)}{0.9c_1^3t_i^{\epsilon_1}p_{\gr}^2},$$
and the lower bound claim holds.

For the upper bound note that $\mathbb{E}'(X^f_j) \leq \sum_{e \in E_j} \frac{\sum_{e' \in E_{H_{i.1}}(v,S_1, S_2, S_3): e \cap e' \neq \emptyset} f(e')}{|E_j'|} = \frac{|E_j|}{|E_j'|}\mathbb{E}(X^f_j)$. Now $|E_j'| \geq |E_j|-3\chi$ and $|E_j| \gg \chi$. Thus $\mathbb{E}'(X^f_j) \leq \mathbb{E}(X^f_j)(1+\frac{6.1\chi}{|E_j|}) \leq \mathbb{E}(X^f_j)\left(1+\frac{15c_2\log^{4l_3+l_4+6}(n)}{c_1^3t_i^{\epsilon_1}p_{\gr}^2}\right)$, as stated.
\end{proof}
From the claim and the preceding statement it follows that
$$X^f \geq \mathbb{E}(X^f) - \left(\frac{15c_2^2f^{\max}\log^{4l_3+2l_4+6}(n)t_i^{1-2\epsilon_1}}{c_1^3p_{\gr}} + f^{\min}t_i^{0.51}\right),$$
and 
$$X^f \leq \mathbb{E}(X^f) + \frac{2400c_2c_3^2f^{max}\log^{8l_3+2l_4+12}(n)t_i^{1-2\epsilon_1}p_{\gr}}{c_1^6}+ f^{\min}t_i^{0.51},$$
where, letting $p^i_*:=\max_e p^i_e$,
\begin{multline*}
\mathbb{E}(X^f)=\sum_{e \in E_{H_{i.1}}(v,S_1, S_2, S_3)}f(e)p^i_e \\ \leq f^{\max}p^i_*|E_{H_{i}}(v,S_1, S_2, S_3)|  \leq \frac{160f^{\max}c_3^2\log^{4l_3+l_4+6}(n)t_i^{1-\epsilon_1}p_{\gr}^3}{c_1^3}
\end{multline*} 
using (P5) and that $|E_{H_{i.1}}(v,S_1, S_2, S_3)| \leq |E_{H_{i}}(v,S_1, S_2, S_3)|$ with $|S_1| \leq |I_{t_{i+1}}|$ and Corollary \ref{cor_pie}. This gives that 
$$X^f=\mathbb{E}(X^f) \pm O\left(f^{max}p_{\gr}^{-1}\log^{8l_3+2l_4+12}(n)t_i^{1-2\epsilon_1}\right).$$

Furthermore, $\sum_{e \in E_{H_{i.1}}(v, S_1, S_2, S_3)} f(e) \geq f^{\min}|E_{H_{i.1}}(v, S_1, S_2, S_3)| \geq \frac{0.9c_1^3f^{\min}t_ip_{\gr}^3}{\log^{4l_3+8}(n)}$ by Corollary \ref{cor_lower_bounds}. Thus,
\begin{multline*}
\sum_{E_{H_{i+1}}(v,S_1, S_2, S_3)} f(e) =\sum_{e \in E_{H_{i.1}}(v,S_1, S_2, S_3)} f(e)-X^f \\
= \left(\sum_{e \in E_{H_{i.1}}(v,S_1, S_2, S_3)} f(e)(1 - p^i_e)\right) \pm O\left(f^{max}p_{\gr}^{-1}\log^{8l_3+2l_4+12}(n)t_i^{1-2\epsilon_1}\right)\\
= \left(1 \pm O\left(\frac{f^{\max}\log^{12l_3+2l_4+20}(n)t_i^{-2\epsilon_1}}{c_1^3f^{\min}p_{\gr}^{4}}\right)\right)\sum_{e \in E_{H_{i.1}}(v,S_1, S_2, S_3)} f(e)(1 - p^i_e).
\end{multline*}
Since $t_i^{-\epsilon_1/2}p_{\gr}^{-4} \ll 1$ and $\frac{f^{\max}}{f^{\min}}\leq \log^{500}(n)$, we have that 
$$O\left(\frac{f^{\max}\log^{12l_3+2l_4+20}(n)t_i^{-2\epsilon_1}}{c_1^3f^{\min}p_{\gr}^4}\right) \leq t_i^{-1.5\epsilon_1}$$ as required to complete the proof.
\end{proof}

\subsubsection{Step 4}

We fix $M_i^c$ such that Lemmas \ref{lemma_vertex_freedman} and \ref{lemma_degree_freedman} both hold for all $i$-reachable $S$ and open and closed $i$-reachable tuples $(v,S_1,S_2,S_3)$. (It is clear this is possible by union bounds.)

Define $w_{i.2}: E(H_i) \rightarrow \mathbb{R}_{\geq 0}$ such that
$$w_{i.2}(e):=
\begin{cases}
\frac{w_{i.1}(e)}{1-p^i_e} & \mbox{~if $e \in H_{i}[I_{t_{i+1}}]$ is $i$-reachable~} \\
w_{i.1}(e) & \mbox{~otherwise.} \\
\end{cases}
.$$

Since $(H_i, w_i)$ is $i$-permissible, we have from Corollary \ref{cor_pie} that $\frac{1}{1-p^i_e}=1-o(1)$, for every $e \in H_{i.1}[I_{t_{i+1}}]$. 

\begin{prop} \label{prop_w_i.2}
For every edge $e \in H_i[I_{t_{i+1}}]$ that is $i$-reachable, we have that
$$\frac{c_1}{t_ip_{\gr}^3\log^{2}(n)}\leq w_{i.2}(e) \leq \frac{1.2\log^{l_1+4l_3+8}(n)}{c_1^4t_ip_{\gr}^3}.$$
\end{prop}
\begin{proof}
Since $(H_i, w_i)$ is $i$-permissible we have by (P3) that
$$\frac{c_1}{t_ip_{\gr}^3\log^{2}(n)} \leq w_i(e) \leq \frac{\log^{l_1}(n)}{t_ip_{\gr}^3},$$
for every $e \in H_i$ which is $(i-1)$-reachable, and $w_i(e)=w_{1^*}(e)$ otherwise. If $d_{i.0}^*\leq 1$ we have that $w_{i.2}(e)=\frac{w_{i.0}(e)}{1-p_e^i}$ and if $d_{i.0}^*> 1$ we have that $w_{i.2}(e)=\frac{w_{i.0}(e)}{d_{i.0}^*(1-p_e^i)}$. By Proposition \ref{prop_d*_i} and Corollary \ref{cor_pie} we have that $w_{i.2}(e)=(1 \pm o(1))w_{i.0}(e)$. Thus by Proposition \ref{prop_i.0} we have that 
$$w_i(e)\leq w_{i.2}(e) \leq \frac{1.2w_i(e)\log^{4l_3+8}(n)}{c_1^4}.$$
For those edges $e \in H_i[I_{t_{i+1}}]$ which are $(i-1)$-reachable, the result follows immediately. This covers all $i$-reachable edges unless $H_{i-1}$ had depth $j-1$ and $H_i$ has depth $j$ for some $j$. In this case all edges $e$ which are $i$-reachable but not $(i-1)$-reachable are of type $(\alpha, \beta, 0)_{j+1}$ with $\alpha \neq 0$. Thus by Theorem \ref{thm_special_reweight} we have $\frac{c^*_{1,1}}{k_{j+1}p_{\gr}^3\log(n)} \leq w_i(e)=w_{1^*}(e) \leq \frac{c^*_{1,2}}{k_{j+1}p_{\gr}^3\log(n)}$. Furthermore, in this case we have that $t_i=k_{j-1}= k_{j+1}\log^{2}(n)$, so in particular, $\frac{c^*_{1,1}\log(n)}{t_ip_{\gr}^3} \leq w_i(e)=w_{1^*}(e) \leq \frac{c^*_{1,2}\log(n)}{t_ip_{\gr}^3}$ and the result still holds. 
\end{proof}

Then we have the following corollary to Lemma \ref{lemma_degree_freedman}:

\begin{cor} \label{cor_i+1_to_i.1}
Let $v \in V(H_{i+1})$ and $S_1, S_2, S_3 \subseteq I_{t_{i+1}}$. Let $(v,S_1,S_2,S_3)$ be an open or closed $i$-reachable tuple. Then 
$$w_{i.2}(E_{H_{i+1}}(v,S_1,S_2,S_3))=(1 \pm t_i^{-1.5\epsilon_1})w_{i.1}(E_{H_{i.1}}(v,S_1,S_2,S_3)).$$
\end{cor}
\begin{proof}
By Proposition \ref{prop_w_i.2}, $w_{i.2}$ is edge allowable for $(H_i, w_i, t_i, \eta_1)$ and $\frac{\max_e w_{i.2}(e)}{\min_e w_{i.2}(e)} \leq \log^{500}(n)$ over all $i$-reachable edges. Then by Lemma \ref{lemma_degree_freedman} we have that for $(v, S_1, S_2, S_3)$ an open or closed $i$-reachable tuple such that $S_1,S_2,S_3 \subseteq I_{t_{i+1}}$, 
\begin{eqnarray} \label{eq_i.2}
w_{i.2}(E_{H_{i+1}}(v,S_1,S_2,S_3)) &=& (1 \pm t_i^{-1.5\epsilon_1})\sum_{e \in E_{H_{i.1}}(v,S_1,S_2,S_3)} w_{i.2}(e)(1-p^i_e) \nonumber \\
&=& (1 \pm t_i^{-1.5\epsilon_1})\sum_{e \in E_{H_{i.1}}(v,S_1,S_2,S_3)} \frac{1-p^i_e}{1-p^i_e}w_{i.1}(e) \nonumber \\
&=& (1 \pm t_i^{-1.5\epsilon_1})\sum_{e \in E_{H_{i.1}}(v,S_1,S_2,S_3)} w_{i.1}(e) \nonumber \\
&=& (1 \pm t_i^{-1.5\epsilon_1})w_{i.1}(E_{H_{i.1}}(v,S_1,S_2,S_3)).
\end{eqnarray}
\end{proof}

Now, since $V(H_{i+1}) \subseteq I_{t_{i+1}}$, we have that $d_{w_{i.2}, H_{i+1}}(v)=w_{i.2}(E_{H_{i+1}}(v,I_{t_{i+1}}))=(1 \pm t_i^{-1.5\epsilon_1})w_{i.1}(E_{H_{i.1}}(v,I_{t_{i+1}}))$. By Corollary \ref{cor_fm_i} $1 \geq w_{i.1}(E_{H_{i.1}}(v,I_{t_{i+1}})) \geq 1-\left(\frac{c_{H_i,v}}{\delta_{H_i,v}}+2.1\log^{3l_3+7}(n)t_i^{-\epsilon_1}\right)$, so $d_{w_{i.2}, H_{i+1}}(v)\leq 1+t_i^{-1.5\epsilon_1}$ for every $v \in V(H_{i+1})$. Let 
\begin{equation}\label{eq_d_i*}
d_i^*:=\max_{v \in V(H_{i+1})} d_{w_{i.2}, H_{i+1}}(v).
\end{equation} 
We rescale to obtain $w_{i+1}: E(H_i) \rightarrow \mathbb{R}_{\geq 0}$ an almost-perfect fractional matching for $H_{i+1}$ as follows:
$$w_{i+1}(e):=
\begin{cases}
\frac{w_{i.2}(e)}{d_i^*} & \mbox{~if $e \in H_i[I_{t_{i+1}}]$ is $i$-reachable~} \\
w_{i.1}(e) & \mbox{~otherwise.} \\
\end{cases}
$$

\begin{prop}\label{prop_apfm}
$w_{i+1}$ is an almost-perfect fractional matching for $H_{i+1}$ such that
$$d_{w_{i+1}, H_{i+1}}(v) \geq 1 - \left(\frac{c_{H_i, v}}{\delta_{H_i, v}} +2.2\log^{3l_3+7}(n)t_i^{-\epsilon_1}\right),$$
and for every open and closed $i$-reachable tuple $(v,S_1,S_2,S_3)$ we have 
\begin{multline*}
w_{i+1}(E_{H_{i+1}}(v,S_1,S_2,S_3)) = \\(1 \pm 1.1\log^{3l_3+7}(n)t_i^{-\epsilon_1})\sum_{e \in E_{H_{i}}(v,S_1,S_2,S_3)} w_{i.0}(e) \prod_{u \in e\setminus \{v\}} (1-d_{w_i^o}(u)).
\end{multline*}
\end{prop}
\begin{proof}
That $w_{i+1}$ is a fractional matching for $H_{i+1}$ follows by construction and the fact that $(H_i, w_i)$ is $i$-permissible. It is clear that no weights fall below $0$ and all weights are at most $1$ in the reweighting from $w_i$, and we normalised to ensure that $d_{w_{i+1}, H_{i+1}}(v) \leq 1$ for all $v \in V(H_{i+1})$. Given that $w_{i+1}$ is a fractional matching for $H_{i+1}$ it remains to consider a lower bound for $d_{w_{i+1}, H_{i+1}}(v)$ for each $v \in V(H_{i+1})$. Now, by definition of $w_{i+1}$ and $d_i^*$, we have that $d_{w_{i+1}, H_{i+1}}(v) = w_{i+1}(E_{H_{i+1}}(v, I_{t_{i+1}}))=\frac{w_{i.2}(E_{H_{i+1}}(v, I_{t_{i+1}}))}{d_i^*} \geq \frac{w_{i.2}(E_{H_{i+1}}(v, I_{t_{i+1}}))}{1+t_i^{-1.5\epsilon_1}}$. Then by Corollary \ref{cor_i+1_to_i.1} we have that 
$$d_{w_{i+1}, H_{i+1}}(v) \geq \frac{1-t_i^{-1.5\epsilon_1}}{1+t_i^{-1.5\epsilon_1}}w_{i.1}(E_{H_{i.1}}(v, I_{t_{i+1}})),$$ and by Corollary \ref{cor_fm_i},
$$d_{w_{i+1}, H_{i+1}}(v) \geq \frac{1-t_i^{-1.5\epsilon_1}}{1+t_i^{-1.5\epsilon_1}}\left(1- \frac{c_{H_i, v}}{\delta_{H_i, v}} -2.1\log^{3l_3+7}(n)t_i^{-\epsilon_1}\right).$$
By Proposition \ref{prop_bound_err}, it follows that
$$d_{w_{i+1}, H_{i+1}}(v) \geq 1 - \left(\frac{c_{H_i, v}}{\delta_{H_i, v}} +2.2\log^{3l_3+7}(n)t_i^{-\epsilon_1}\right),$$
and in particular $d_{w_{i+1}, H_{i+1}}(v) = 1-o(1)$, so $w_{i+1}$ is an almost-perfect fractional matching for $H_{i+1}$.

Using (\ref{eq_i.2}), (\ref{eq_i.1}), (K2) and the fact that we now only consider $i$-reachable $S_1,S_2,S_3 \subseteq I_{t_{i+1}}$,
\begin{multline*}
w_{i+1}(E_{H_{i+1}}(v,S_1,S_2,S_3)) = \sum_{e \in E_{H_{i+1}}(v,S_1,S_2,S_3)} w_{i+1}(e) \\
= \sum_{e \in E_{H_{i+1}}(v,S_1,S_2,S_3)} \frac{w_{i.2}(e)}{d_i^*}
= (1 \pm 1.1t_i^{-1.5\epsilon_1})\sum_{e \in E_{H_{i+1}}(v,S_1,S_2,S_3)} w_{i.2}(e) \\
= (1 \pm 2.2t_i^{-1.5\epsilon_1})\sum_{e \in E_{H_{i.1}}(v,S_1,S_2,S_3)} w_{i.1}(e) \\
= (1 \pm 2.2t_i^{-1.5\epsilon_1})(1 \pm \log^{3l_3+7}(n)t_i^{-\epsilon_1})\sum_{e \in E_{H_{i.1}}(v,S_1,S_2,S_3)} w_{i.0}(e) \\
= (1 \pm 1.1\log^{3l_3+7}(n)t_i^{-\epsilon_1})\sum_{e \in E_{H_{i}}(v,S_1,S_2,S_3)} w_{i.0}(e) \prod_{u \in e\setminus\{v\}} (1-d_{w_{i}^o}(u)).
\end{multline*}
\end{proof}

\subsubsection{$(H_{i+1}, w_{i+1})$ in terms of $(H_i, w_i)$}

As per the strategy to show that $(H_{i+1}, w_{i+1})$ is $(i+1)$-permissible, we wish to understand $(H_{i+1}, w_{i+1})$ in terms of $(H_i, w_i)$. In the previous steps there are many properties of $(H_{i+1}, w_{i+1})$ described in terms of $H_{i.1}$, the graph obtained from $H_i$ after Step 1 of Plan \ref{alg_subgraphs}. In this section we shift to understanding how such variables and properties `look' in terms of $(H_i, w_i)$. In particular, we start by upper bounding $p^i_v$ and $p^i_e$, before giving a more complete list of properties comparable to those considered in the definition of graph permissibility.

\begin{prop}\label{prop_p_props}
Given that $(H_i, w_i)$ is $i$-permissible, we have that
$$p^i_v \leq \frac{4.1\max\{c_{H_i}, \log(n)t_i^{-\epsilon_1}\}|E_{H_{i}}(v,I_{t_i}\setminus I_{t_{i+1}},*,*)|}{(\delta_{H_i})^3|E_{H_{i}}(v,I_{t_{i+1}})|},$$
for every $v \in V(H_{i.1}[I_{t_{i+1}}])$, and
$$p^i_e \leq \max_{v \in e}\frac{16.1\max\{c_{H_i}, \log(n)t_i^{-\epsilon_1}\}|E_{H_{i}}(v,I_{t_i}\setminus I_{t_{i+1}},*,*)|}{(\delta_{H_i)^3}|E_{H_{i}}(v,I_{t_{i+1}})|},$$
for every $e \in H_{i.1}[I_{t_{i+1}}]$. Additionally we have that
$$\frac{1}{1-p^i_e} \leq 1+1.1p^i_e.$$  
\end{prop}
\begin{proof}
We have from Proposition \ref{prop_piv_gen} that
$$p^i_v \leq \frac{2|E_{H_{i.1}}(v,I_{t_i}\setminus I_{t_{i+1}},*,*)|}{\min_{v \in V(H_{i.1}[I_{t_{i}}\setminus I_{t_{i+1}}])} d_{H_{i.1}}(v)}.$$
Now, we have that $d_{H_{i.1}}(v) \geq |E_{H_{i.1}}(v,I_{t_{i+1}})|$ for each $v \in V(H_{i.1}[I_{t_{i}}\setminus I_{t_{i+1}}])$, and by Theorem \ref{thm_key}(K2) and Proposition \ref{prop_o'},
\begin{multline*}
|E_{H_{i.1}}(v,I_{t_{i+1}})| = (1 \pm \log^{3l_3+7}(n)t_i^{-\epsilon_1})\sum_{e \in E_{H_{i}}(v,I_{t_{i+1}})} \prod_{u \in e \setminus \{v\}} (1-d_{w_i^o}(u)) \\
= \left(1 \pm \left(\frac{3.1c_{H_i}}{\delta_{H_i}}+\log^{3l_3+7}(n)t_i^{-\epsilon_1}\right)\right)\sum_{e \in E_{H_{i}}(v,I_{t_{i+1}})} \prod_{u \in e \setminus \{v\}} w_i(E_{H_{i}}(u,I_{t_{i+1}})) \\
\geq \left(1 \pm \left(\frac{3.1c_{H_i}}{\delta_{H_i}}+\log^{3l_3+7}(n)t_i^{-\epsilon_1}\right)\right)\delta_{H_i}^3|E_{H_{i}}(v,I_{t_{i+1}})|.
\end{multline*}
Recalling Proposition \ref{prop_degs}, we also have that
\begin{multline*}
|E_{H_{i.1}}(v,I_{t_i}\setminus I_{t_{i+1}},*,*)| =\sum_{e \in E_{H_{i}}(v,I_{t_i}\setminus I_{t_{i+1}},*,*)} \prod_{u \in e \setminus \{v\}} (1-d_{w_i^o}(v)) \\
\pm O\left(t_i^{-\epsilon_1}|E_{H_{i}}(v,I_{t_i}\setminus I_{t_{i+1}},*,*)|\right)\\
\leq \left(1 \pm \frac{2.1c_{H_i}}{\delta_{H_i}}\right)c_{H_i} \sum_{e \in E_{H_{i}}(v,I_{t_i}\setminus I_{t_{i+1}},*,*)} \prod_{u \in (e \setminus \{v\}) \cap I_{t_{i+1}}} w_i(E_{H_{i}}(u,I_{t_{i+1}})) \\
 \pm O\left(t_i^{-\epsilon_1}|E_{H_{i}}(v,I_{t_i}\setminus I_{t_{i+1}},*,*)|\right) \\
\leq \left(1 \pm \frac{2.1c_{H_i}}{\delta_{H_i}}\right)2\max\{c_{H_i}, \log(n)t_i^{-\epsilon_1}\}|E_{H_{i}}(v,I_{t_i}\setminus I_{t_{i+1}},*,*)|.
\end{multline*}
It follows that
\begin{eqnarray*}
p^i_v &\leq & \frac{4(1 \pm \frac{2.1c_{H_i}}{\delta_{H_i}})\max\{c_{H_i}, \log(n)t_i^{-\epsilon_1}\}|E_{H_{i}}(v,I_{t_i}\setminus I_{t_{i+1}},*,*)|}{(1 \pm (\frac{3.1c_{H_i}}{\delta_{H_i}}+\log^{3l_3+7}(n)t_i^{-\epsilon_1}))\delta_{H_i}^3|E_{H_{i}}(v,I_{t_{i+1}})|} \\
& \leq & \frac{4.1\max\{c_{H_i}, \log(n)t_i^{-\epsilon_1}\}|E_{H_{i}}(v,I_{t_i}\setminus I_{t_{i+1}},*,*)|}{\delta_{H_i}^3|E_{H_{i}}(v,I_{t_{i+1}})|},
\end{eqnarray*}
since by Proposition \ref{prop_bound_err}, $\left(\frac{3.1c_{H_i}}{\delta_{H_i}}+\log^{3l_3+7}(n)t_i^{-\epsilon_1}\right) \leq \frac{4\log^{l_4+l_3+2}(n)}{c_1t_i^{\epsilon_1}}$. 

That 
$$p^i_e \leq \max_{v \in e}\frac{16.1\max\{c_{H_i}, \log(n)t_i^{-\epsilon_1}\}|E_{H_{i}}(v,I_{t_i}\setminus I_{t_{i+1}},*,*)|}{\delta_{H_i}^3|E_{H_{i}}(v,I_{t_{i+1}})|}$$ 
for every $e \in H_{i.1}[I_{t_{i+1}}]$ follows from the bound on $p^i_v$ and Proposition \ref{prop_pie_gen}. Finally, by the upper bound on $p^i_e$ given by Corollary \ref{cor_pie} we have that
$$1 \leq \frac{1}{1-p^i_e} = 1+\frac{p^i_e}{1-p^i_e} \leq 1+1.1p^i_e$$
as required.
\end{proof}

We introduce the following notation so that subsequent equations become less cumbersome. Let 
$$E^i_{out}:=\max_{v \in V(H_i[I_{t_{i+1}}])}|E_{H_{i}}(v,I_{t_i}\setminus I_{t_{i+1}},*,*)|$$ and
$$E^i_{in}:=\min_{v \in V(H_i[I_{t_{i+1}}])}|E_{H_{i}}(v,I_{t_{i+1}})|,$$ 
and define 
$$m_i:=\frac{\max\{c_{H_i}, \log(n)t_i^{-\epsilon_1}\}E^i_{out}}{\delta_{H_i}^3E^i_{in}}.$$ 
We also suppose that $H_{i+1}$ has depth $j'$. (Note that given $H_i$ has depth $j$, that $j' \in \{j, j+1\}$.)

In the following lemma, some properties follow immediately from the strategy or afore mentioned results, but we list all the properties here to account for all properties we care to understand for permissibility of $(H_{i+1}, w_{i+1})$.

\begin{lemma}\label{lemma_to_next_it}
$(H_{i+1}, w_{i+1})$ has the following properties:
\begin{enumerate}[(i)]
\item $H_{i+1} \subseteq H_{1^*}[I_{t_{i+1}}]$,
\begin{enumerate}
	\item if $v \in V(H_{i+1})$ is $i$-reachable then
	$$d_{w_{i+1}, H_{i+1}}(v) \geq 1 - \left(\frac{c_{H_i, v}}{\delta_{H_i, v}} +2.2\log^{3l_3+7}(n)t_i^{-\epsilon_1}\right),$$ 
and in particular, $c_{H_{i+1}, v}\leq \frac{c_{H_i,v}}{\delta_{H_i,v}} + 2.2\log^{3l_3+7}(n)t_i^{-\epsilon_1}$. 
	\item If $v \in V(H_{i+1})$ is not $i$-reachable then
$$d_{w_{i+1}, H_{i+1}}(v)=d_{w_{1^*}, H_{1^*}}(v) \geq 1-t_1^{-\epsilon_1},$$
and $c_{H_{i+1}, v} \leq t_1^{-\epsilon_1}$.
	\end{enumerate}
\item for every edge $e \in H_{i+1}$ which is $i$-reachable,
$$w_{i+1}(e)=\left(1 \pm \left(17.1m_i + 2.2\log^{3l_3+7}(n)t_i^{-\epsilon_1} \right)\right) \frac{w_i(e)}{\prod_{u \in e}(1-d_{w_i^o}(u))},$$
and every edge $e \in H_{i+1}$ which is not $i$-reachable satisfies $w_{i+1}(e)=w_{1^*}(e)$.
\item for every $(i+1)$-valid subset $S$ which is $i$-reachable, and for polynomially many $p_S$ as defined in Theorem \ref{thm_key},
\begin{multline*}
\sum_{v \in V(H_{i+1}[S])} p_S(v)= \\
\left(1 \pm \left(4.1m_i +1.1c_1^{-1}\log^{l_3+2}(n)t_i^{-\epsilon_1}\right)\right)\sum_{v \in V(H_{i}[S])} p_S(v)(1-d_{w_i^o}(v)),
\end{multline*}
and in particular
$$
|V(H_{i+1}[S])|=\left(1 \pm \left(4.1m_i +1.1c_1^{-1}\log^{l_3+2}(n)t_i^{-\epsilon_1}\right)\right)\sum_{v \in V(H_{i}[S])} (1-d_{w_i^o}(v)).
$$
Furthermore, we have that $|V(H_{i+1}[S])|=|V(H_{1^*}[S])|$ for every $(i+1)$-valid subset $S \subseteq K_{j'+1}$.
\item For every open or closed $i$-reachable tuple $(v,S_1, S_2, S_3)$ which is an $(i+1)$-permissible tuple, we have that
\begin{multline*}
w_{i+1}(E_{H_{i+1}}(v,S_1,S_2,S_3)) = (1 \pm 1.1\log^{3l_3+7}(n)t_i^{-\epsilon_1})\frac{w_i(E_{H_{i}}(v,S_1,S_2,S_3))}{1-d_{w_i^o}(v)}\\
= \left(1 \pm \left(\frac{1.1c_{H_i,v}}{\delta_{H_i,v}} + 1.2\log^{3l_3+7}(n)t_i^{-\epsilon_1}\right)\right)\frac{w_i(E_{H_i}(v,S_1,S_2,S_3))}{w_i(E_{H_i}(v,I_{t_{i+1}}))},
\end{multline*}
and for every open or closed $i+1$-permissible tuple $(v,S_1, S_2, S_3)$ with $\bigcup \{ e\setminus \{v\}: e \in E_{\mathcal{T}}(v,S_1,S_2,S_3)\} \subseteq K_{j'+1}$,
$$w_{i+1}(E_{H_{i+1}}(v,S_1,S_2,S_3))=w_{1^*}(E_{H_{1^*}}(v,S_1,S_2,S_3)).$$
\item For every open or closed $i$-reachable tuple $(v,S_1, S_2, S_3)$ which is an $(i+1)$-permissible tuple, and polynomially many $f$ which are edge allowable for $(H_i, w_i, t_i, \eta_1)$ and as in Theorem \ref{thm_key}, we have that
\begin{multline*}
f(E_{H_{i+1}}(v,S_1,S_2,S_3)) = \\
\left(1 \pm (16.1m_i+1.1\log^{3l_3+7}(n)t_i^{-\epsilon_1})\right)\sum_{e \in E_{H_{i}}(v,S_1,S_2,S_3)} f(e) \prod_{u \in e \setminus \{v\}} (1-d_{w_i^o}(u)),
\end{multline*}
and in particular,
\begin{multline*}
|E_{H_{i+1}}(v,S_1,S_2,S_3)| = \\\left(1 \pm (16.1m_i+1.1\log^{3l_3+7}(n)t_i^{-\epsilon_1})\right)\sum_{e \in E_{H_{i}}(v,S_1,S_2,S_3)} \prod_{u \in e \setminus \{v\}} (1-d_{w_i^o}(u)).
\end{multline*}
\end{enumerate}
\end{lemma}
\begin{proof}
We have that $H_{i+1} \subseteq H_{1^*}[I_{t_{i+1}}]$ by the strategy, since $H_i \subseteq H_{1^*}[I_{t_{i}}]$ and Plan \ref{alg_subgraphs} ensures that we cover all vertices in $I_{t_i}\setminus I_{t_{i+1}}$ to reach $H_{i+1}$. Furthermore, Proposition \ref{prop_apfm} gives that $d_{w_{i+1}, H_{i+1}}(v) \geq 1 - \left(\frac{c_{H_i, v}}{\delta_{H_i, v}} +2.2\log^{3l_3+7}(n)t_i^{-\epsilon_1}\right)$ for all $i$-reachable $v \in V(H_{i+1})$. Supposing that $v$ is not $i$-reachable, then none of the edges in $E_{H_i}(v, I_{t_i})$ are $i$-reachable, and thus we have that 
$d_{w_{i+1}, H_{i+1}}(v)=w_{i+1}(E_{H_{i+1}}(v, I_{t_{i+1}}))=w_{1^*}(E_{H_{1^*}}(v, I_{t_{i+1}}))=w_{1^*}(E_{H_{1^*}}(v, I_{t_{1}}))=d_{w_{1^*}, H_{1^*}}(v).$

In general it is clear, by nature of Plan \ref{alg_subgraphs}, that the lemma holds for all properties we consider when  variables are not $i$-reachable. Considering $e \in H_{i+1}$ that is $i$-reachable, we have by (\ref{eq_d_i*}) and the definitions of $w_{i+1}$ and $w_{i.2}$ that $w_{i+1}(e)=
\frac{w_{i.2}(e)}{d_i^*}= (1 \pm 1.1t_i^{-1.5\epsilon_1})w_{i.2}(e)=(1 \pm 1.1t_i^{-1.5\epsilon_1})\frac{w_{i.1}(e)}{1-p^i_e}$. By Proposition \ref{prop_p_props}, then, we have that $w_{i+1}(e)=(1 \pm 1.1t_i^{-1.5\epsilon_1})(1 \pm 1.1p^i_e)w_{i.1}(e)$. Then by Corollary \ref{cor_fm_i},
$w_{i+1}(e)=(1 \pm (1.1p^i_e+2.2\log^{3l_3+7}(n)t_i^{-\epsilon_1}))w_{i.0}(e)$. Since $e \subseteq I_{t_{i+1}}$ and $e$ is $i$-reachable, we further have that 
$w_{i+1}(e)=(1 \pm (1.1p^i_e+2.2\log^{3l_3+7}(n)t_i^{-\epsilon_1}))\frac{w_{i}(e)}{\prod_{u \in e} (1-d_{w_i^o}(u))}$. Then by the upper bound on $p^i_e$ given in Proposition \ref{prop_p_props}, we have that (ii) follows.

To see (iii) for each (i+1)-valid subset which is $i$-reachable, we have by Lemma \ref{lemma_vertex_freedman} that $\sum_{v \in V(H_{i+1}[S])} p_S(v)=(1 \pm t_i^{-1.5\epsilon_1})\sum_{v \in V(H_{i.1}[S])} p_S(v)(1-p^i_v)=(1 \pm t_i^{-1.5\epsilon_1})(1 \pm \max_v p^i_v)\sum_{v \in V(H_{i.1}[S])} p_S(v)$. Then by Theorem \ref{thm_key}(K1) we have that 
$$\sum_{v \in V(H_{i.1}[S])} p_S(v)=(1 \pm c_1^{-1}\log^{l_3+2}(n)t_i^{-\epsilon_1})\sum_{v \in V(H_i[S])}p_S(v)(1 - d_{w_i^o}(v)).$$ Thus using the upper bound on $p^i_v$ given by Proposition \ref{prop_p_props} we have that (iii) follows. In particular, taking $p_S=\mathbbm{1}_{v \in H_i[S]}$, we have that 
$$|V(H_{i+1}[S])|=(1 \pm (4.1m_i + 1.1c_1^{-1}\log^{l_3+2}(n)t_i^{-\epsilon_1})) \sum_{v \in V(H_i[S])}(1 - d_{w_i^o}(v)),$$ 
as claimed.

For (iv) we have by Proposition \ref{prop_apfm} that 
\begin{multline*}
w_{i+1}(E_{H_{i+1}}(v,S_1,S_2,S_3))= \\(1 \pm 1.1\log^{3l_3+7}(n)t_i^{-\epsilon_1})\sum_{e \in E_{H_{i}}(v,S_1,S_2,S_3)} w_{i.0}(e) \prod_{u \in e\setminus \{v\}} (1-d_{w_i^o}(u)).
\end{multline*}
Since the edges considered all have their vertices contained in $I_{t_{i+1}}$, but are also $i$-reachable, we have that $w_{i.0}(e) \prod_{u \in e\setminus \{v\}} (1-d_{w_i^o}(u)) = \frac{w_i(e)}{1-d_{w_i^o}(v)}$ for each $e \in E_{H_{i}}(v,S_1,S_2,S_3)$, and hence 
$$\sum_{e \in E_{H_{i}}(v,S_1,S_2,S_3)} w_{i.0}(e) \prod_{u \in e\setminus \{v\}} (1-d_{w_i^o}(u))=\frac{w_i(E_{H_{i}}(v,S_1,S_2,S_3))}{1-d_{w_i^o}(v)}.$$
Then by Proposition \ref{prop_o'}, since $v \in V(H_{i+1})$,
$$\frac{w_i(E_{H_{i}}(v,S_1,S_2,S_3))}{1-d_{w_i^o}(v)}=\left(1 \pm \frac{1.1c_{H_i, v}}{\delta_{H_i, v}}\right)\frac{w_i(E_{H_{i}}(v,S_1,S_2,S_3))}{w_i(E_{H_{i}}(v,I_{t_{i+1}}))},$$ and the result follows. 

Finally, to see (v), we have from Lemma \ref{lemma_degree_freedman} that
\begin{multline*}
f(E_{H_{i+1}}(v,S_1,S_2,S_3)) = \left(1 \pm t_i^{-1.5\epsilon_1}\right)\sum_{e \in E_{H_{i.1}}(v,S_1,S_2,S_3)} f(e) (1-p^i_e) \\= (1 \pm t_i^{-1.5\epsilon_1})(1 \pm \max_e p^i_e)f(E_{H_{i.1}}(v,S_1,S_2,S_3)).
\end{multline*} 
Then by Theorem \ref{thm_key}~(K2), we have that 
\begin{multline*}
f(E_{H_{i.1}}(v,S_1,S_2,S_3)) = \\ (1 \pm \log^{3l_3+7}(n)t_i^{-\epsilon_1})\sum_{e \in E_{H_i}(v, S_1, S_2, S_3)} f(e) \prod_{u \in e \setminus \{v\}} (1- d_{w_i^o}(u)).
\end{multline*}
Thus 
\begin{multline*}
f(E_{H_{i+1}}(v,S_1,S_2,S_3)) = \\ \left(1 \pm (\max_e p_e^i + 1.1\log^{3l_3+7}(n)t_i^{-\epsilon_1})\right)\sum_{e \in E_{H_i}(v, S_1, S_2, S_3)} f(e) \prod_{u \in e \setminus \{v\}} (1- d_{w_i^o}(u)),
\end{multline*}
and by Proposition \ref{prop_p_props} and the bound given for $p^i_e$, the general result follows, and taking $f=1$ yields the result for $|E_{H_{i+1}}(v,S_1,S_2,S_3)|$. This completes the proof.
\end{proof}

We complete this section by noting that we only make claims about the error terms in the statement of Lemma \ref{lemma_to_next_it}, such as $m_i$ {\it given} $i$-permissibility of $(H_i, w_i)$. In particular, to ensure that these are sufficiently small $o(1)$ terms, as we require to be the case, we'll wish to understand $E^i_{out}$ and $E^i_{in}$ (among other variables) in terms of $H_{1^*}$ and $w_{1^*}$. This is the content of the next section, from which we shall be able to conclude that Plan \ref{alg_subgraphs} completes to reach $L^*$.

\subsection{$(H_{i+1},w_{i+1})$ in terms of $(H_{1^*}, w_{1^*})$} \label{sec_i_to_1}

Recall by Lemma \ref{lemma_1_permiss} that we know, assuming Theorem \ref{thm_H_1}, that $(H_{1^*}, w_{1^*})$ is $1$-permissible. We now suppose that $(H_i, w_i)$ is $i$-permissible for every $i\leq l$, and obtain bounds for the properties of $(H_{l+1}, w_{l+1})$ in terms of $(H_{1^*}, w_{1^*})$, and show, subsequently, that $(H_{l+1}, w_{l+1})$ is $(l+1)$-permissible. Note that all properties remain the same as in $(H_{1^*}, w_{1^*})$ until the vertices and edges relating to such properties become reachable. Once they are reachable, provided that Plan \ref{alg_subgraphs} does not abort, there are fewer than $40\log\log(n)$ iterations of Plan \ref{alg_subgraphs} before all of these vertices have been covered by the process. Indeed, a vertex or edge containing a vertex in $H_l[I_{t_l}\setminus I_{t_{l+1}}]$, where $H_l$ has depth $j$, is no longer present after iteration $l$ and first became reachable at iteration $l'$ such that $l'=k_{j-3}$ (where $k_{j-1} \geq l > k_j$). For each $j \in [c_g]$ let $t^*(k_j):=i$ such that $t_{i}=t_{k_j}$. Then note that if $H_l$ has depth $j$, (so that $t_l > k_j$), we have that $l-t^*(k_{j-3}) < a$, where $k_j=c_{\vor}^a k_{j-3}$, yielding $l-t^*(k_{j-3}) < 40\log\log(n)$. 

\begin{lemma}\label{lemma_l+1_to_1}
Suppose that $(H_i, w_i)$ is $i$-permissible for every $i\leq l$, and that $H_l$ has depth $j$. Let $(v,S_1,S_2,S_3)$ be an open or closed $l$-reachable tuple such that $(v,S_1,S_2,S_3)$ is $(l+1)$-valid. Then
\begin{multline*}
w_{l+1}(E_{H_{l+1}}(v,S_1,S_2,S_3))= \\
\frac{w_{1^*}(E_{H_{1^*}}(v,S_1,S_2,S_3))}{w_{1^*}(E_{H_{1^*}}(v,I_{t_{l+1}}))}\prod_{t^*(k_{j-3}) \leq i \leq l}\left(1 \pm 2.1\left(\frac{1.1c_{H_i,v}}{\delta_{H_i,v}}+1.2\log^{3l_2+7}(n)t_i^{-\epsilon_1}\right)\right).
\end{multline*}
\end{lemma}
\begin{proof}
Consider an open or closed $l$-reachable tuple that is $(l+1)$-valid. Such a tuple was certainly not reachable before $K_{j-3}$, and so up until reaching $H_{t^*(k_{j-3})}$ had the same properties as in $(H_{1^*}, w_{1^*})$. Thus by Lemma \ref{lemma_to_next_it}(iv)
\begin{multline*}
w_{i+1}(E_{H_{i+1}}(v,S_1,S_2,S_3))= \\
\left(1 \pm \left(\frac{1.1c_{H_i,v}}{\delta_{H_i,v}}+1.2\log^{3l_2+7}(n)t_i^{-\epsilon_1}\right)\right)\frac{w_{i}(E_{H_{i}}(v,S_1,S_2,S_3))}{w_i(E_{H_i}(v,I_{t_{i+1}}))}
\end{multline*}
for every $t^*(k_{j-3}) \leq i \leq l$, and 
$$w_{t^*(k_{j-3})}(E_{H_{t^*(k_{j-3})}}(v,S_1,S_2,S_3))=w_{1^*}(E_{H_{1^*}}(v,S_1,S_2,S_3)).$$
Iterating to get $\frac{w_{i}(E_{H_{i}}(v,S_1,S_2,S_3))}{w_i(E_{H_i}(v,I_{t_{i+1}}))}$ in terms of $\frac{w_{i-1}(E_{H_{i-1}}(v,S_1,S_2,S_3))}{w_{i-1}(E_{H_{i-1}}(v,I_{t_{i+1}}))}$, which is possible for all $t^*(k_{j-3})+1 \leq i \leq l$ by assumption that $(H_i, w_i)$ is $i$-permissible in each of these cases, we get that 
\begin{small}
\begin{multline*}
\frac{w_{i}(E_{H_{i}}(v,S_1,S_2,S_3))}{w_{i}(E_{H_{i}}(v,I_{t_{i+1}}))} = \\ 
\frac{\left(1 \pm \left(\frac{1.1c_{H_{i-1},v}}{\delta_{H_{i-1},v}}+1.2\log^{3l_2+7}(n)t_{i-1}^{-\epsilon_1}\right)\right)w_{i-1}(E_{H_{i-1}}(v,S_1,S_2,S_3))w_{i-1}(E_{H_{i-1}}(v,I_{t_{i}}))}{\left(1 \pm \left(\frac{1.1c_{H_{i-1},v}}{\delta_{H_{i-1},v}}+1.2\log^{3l_2+7}(n)t_{i-1}^{-\epsilon_1}\right)\right)w_{i-1}(E_{H_{i-1}}(v,I_{t_{i}}))w_{i-1}(E_{H_{i-1}}(v,I_{t_{i+1}}))} \\ = \left(1 \pm 2.1\left(\frac{1.1c_{H_{i-1},v}}{\delta_{H_{i-1},v}}+1.2\log^{3l_2+7}(n)t_{i-1}^{-\epsilon_1}\right)\right) \frac{w_{i-1}(E_{H_{i-1}}(v,S_1,S_2,S_3))}{w_{i-1}(E_{H_{i-1}}(v,I_{t_{i+1}}))}.
\end{multline*}
\end{small}
This yields the result.
\end{proof}

We now obtain bounds for $c_{H_i,v}$ and $\delta_{H_i, v}$ for every $i$ and $v$. Recall (from Proposition \ref{prop_o'}), that $c_{G,v}$ and $\delta_{G, v}$ are defined for $(G, \omega)$ an $i$-permissible pair for some $i$, and in this case, $\delta_{G, v}:= \omega(E_G(v, I_{t_{i+1}}))$, and $c_{G,v}:=1-d_{\omega, G}(v)=1-\omega(E_G(v, I_{t_{i}}))$. We also have that $\delta_{G}:= \min_{v \in V(G[I_{t_{i+1}}])} \delta_{G, v}$, and $c_{G}:=\max_{v \in V(G)} c_{G,v}$. We wish to show that both $c_{H_i, v}$ and $\frac{c_{H_i,v}}{\delta_{H_i,v}}$ do not blow up through the process for all $i$ and $v$.

Recall further Definition \ref{def_delta_*-}:
$$\delta_*^-:=\min\left\{\min_{i \in [c_h]}\frac{w_{1^*}(E_{H_{1^*}}(v,I_{t_{i+1}}))}{w_{1^*}(E_{H_{1^*}}(v,I_{t_{i}}))}, \delta_{1^*}^-\right\}.$$ We also define 
$$\delta_*^+:=\max\left\{\max_{i \in [c_h]}\frac{w_{1^*}(E_{H_{1^*}}(v,I_{t_{i+1}}))}{w_{1^*}(E_{H_{1^*}}(v,I_{t_{i}}))}, \delta_{1^*}^+\right\}$$ and note by Corollary \ref{cor_delta_const} and Lemma \ref{lemma_1_permiss}, we have that $1 \geq \delta_*^{\pm}>c$ where $c>0$ is an absolute constant. Furthermore, note that subsequently we have $40\log(\frac{2}{\delta_*^-})>3l_3+7=28$. 

\begin{prop}\label{prop_delta_c}
There exists a constant $1 \leq m^* \leq 80\log\left(\frac{2}{\delta_*^-}\right)$ such that for any $l$ the following holds: Suppose that $(H_i, w_i)$ is $i$-permissible for every $1^* \leq i \leq l$.  Then for every $1^* \leq i \leq l+1$ and $v \in V(H_i)$,
$\frac{\delta_{*}^-}{2} \leq \delta_{H_i, v} \leq 2\delta_{*}^+$, and $c_{H_i, v} \leq \log^{m^*}(n)t_i^{-\epsilon_1}$.
\end{prop}
\begin{proof}
We prove the proposition via strong induction. First note that the base case is satisfied by Lemma \ref{lemma_1_permiss}. We now assume that $\frac{\delta_{*}^-}{2} \leq \delta_{H_i, v} \leq 2\delta_{*}^+$, and $c_{H_i, v} \leq \log^{m^*}(n)t_i^{-\epsilon_1}$ for every $i \leq l'$ for some $l' \leq l$, and that $H_{l'}$ has depth $j$. By Lemma \ref{lemma_l+1_to_1}, we have since $\delta_{H_{l'+1}, v} \equiv w_{l'+1}(E_{H_{l'+1}}(v, I_{t_{l'+2}}))$ that
\begin{multline*}
\delta_{H_{l'+1}, v}= \\
\frac{w_{1^*}(E_{H_{1^*}}(v,I_{t_{l'+2}}))}{w_{1^*}(E_{H_{1^*}}(v,I_{t_{l'+1}}))}\prod_{t^*(k_{j-3}) \leq i \leq l'}\left(1 \pm 2.1\left(\frac{1.1c_{H_i,v}}{\delta_{H_i,v}}+1.2\log^{3l_3+7}(n)t_i^{-\epsilon_1}\right)\right).
\end{multline*}
Now, we have that $\delta_*^- \leq \frac{w_{1^*}(E_{H_{1^*}}(v,I_{t_{l'+2}}))}{w_{1^*}(E_{H_{1^*}}(v,I_{t_{l'+1}}))} \leq \delta_*^+$ and, by induction, that 
$\frac{1.1c_{H_i,v}}{\delta_{H_i,v}}\leq \frac{2.2\log^{m^*}(n)t_i^{-\epsilon_1}}{\delta_{*}^-}$, for every $t^*(k_{j-3}) \leq i \leq l'$ and every $v \in V(H_{l'+1})$, where $1 \geq \delta_{*}^->0$ is an absolute constant. In particular we have that
\begin{multline*}
\prod_{t^*(k_{j-3}) \leq i \leq l'}\left(1 \pm 2.1\left(\frac{1.1c_{H_i,v}}{\delta_{H_i,v}}+1.2\log^{3l_3+7}(n)t_i^{-\epsilon_1}\right)\right) = \\
\left(1 \pm c\log^{m^*}(n)t_{l'}^{-\epsilon_1}\right)^{40\log\log(n)} =(1 \pm c\log^{m^*+1}(n)t_{l'}^{-\epsilon_1}),
\end{multline*} 
where we can take $c=\frac{5}{\delta_{*}^-}+1$ and have that $c>0$ is an absolute constant. In particular, this yields that
$$\delta_*^-/2 \leq (1-c\log^{m^*+1}(n)t_{l'}^{-\epsilon_1})\delta_*^- \leq \delta_{H_{l'+1}, v} \leq (1+c\log^{m^*+1}(n)t_{l'}^{-\epsilon_1})\delta_*^+ \leq 2\delta_*^+,$$
 as required. It remains to upper bound $c_{H_{l'+1}, v}$ for every $v \in V(H_{l'+1})$. For those vertices $v$ which are not $l'$-reachable, we have that $c_{H_{l'+1}, v}\leq t_1^{-\epsilon_1}\leq \log^{m^*}(n)t_{l'+1}^{-\epsilon_1}$. By Lemma \ref{lemma_to_next_it} we have that, given $(H_i, w_i)$ is $i$-permissible, $c_{H_{i+1}, v} \leq \frac{c_{H_{i,v}}}{\delta_{H_i, v}}+2.2\log^{3l_3+7}(n)t_i^{-\epsilon_1}$. Now, by induction we have that $\delta_{H_i, v}\geq \delta_*^-/2$ for every $i \leq l'$ and so we may write $c_{H_{i+1}, v}\leq \frac{2c_{H_{i,v}}}{\delta_*^-}+2.2\log^{3l_3+7}(n)t_{i}^{-\epsilon_1}$ for every $i \leq l'$. In particular, this yields that
$$c_{H_{l'+1}, v}\leq \left(\frac{2}{\delta_*^-}\right)^{l'-t^*(k_{j-3})}c_{H_{t^*(k_{j-3}), v}}+ \sum_{i=1}^{l'-t^*(k_{j-3})} \left(\frac{2}{\delta_*^-}\right)^i\cdot 2.2t_{l'-i}^{-\epsilon_1}.$$
Furthermore, since $H_{l'}$ has depth $j$, we have that $c_{H_{t^*(k_{j-3}), v}}=c_{H_{1^*, v}}\leq t_1^{-\epsilon_1}$ for every $v \in H_{t^*(k_{j-3})}$ and $2.2\log^{3l_3+7}(n)t_{l'-i}^{-\epsilon_1} \leq 2.2\log^{3l_3+7}(n)t_{l'}^{-\epsilon_1}$ for every $i \in [1, l'-t^*(k_{j-3})]$, where  
$l'-t^*(k_{j-3})\leq 40\log\log(n)$. Thus we find that
\begin{eqnarray*}
c_{H_{l'+1}, v} &\leq & \left(\frac{2}{\delta_*^-}\right)^{l'-t^*(k_{j-3})}t_1^{-\epsilon_1} +2.2\log^{3l_3+7}(n)t_{l'}^{-\epsilon_1} \sum_{i=1}^{l'-t^*(k_{j-3})} \left(\frac{2}{\delta_*^-}\right)^i \\
&\leq& \log^{40\log(2/\delta_*^-)}(n)t_1^{-\epsilon_1} + 40\log\log(n)\log^{40\log(2/\delta_*^-)+3l_3+7}(n)t_{l'}^{-\epsilon_1}.
\end{eqnarray*}
In particular, $c_{H_{l'+1}, v} \leq \log^{80\log(2/\delta_*^-)}(n)t_{l'}^{-\epsilon_1}$, completing the proof.
\end{proof}

The following two propositions will be useful for the subsequent lemmas and corollaries.

\begin{prop}\label{prop_cde_err}
Suppose that $(H_i, w_i)$ is $i$-permissible for every $i \leq l$. Then
$$\frac{c_{H_i}}{\delta_{H_i}} \leq \frac{2\log^{m^*}(n)t_l^{-\epsilon_1}}{\delta_*^-}$$
and
$$m_i=\frac{\max\{c_{H_i}, \log(n)t_i^{-\epsilon_1}\}E^i_{out}}{\delta_{H_i}E^i_{in}} \leq \frac{8c_3}{(\delta_*^-)^3}\log^{m^*+l_3}(n)t_l^{-\epsilon_1}.$$
\end{prop} 
\begin{proof}
The first claim is a direct corollary of Proposition \ref{prop_delta_c} (noting that $t_i^{-\epsilon_1} \leq t_l^{-\epsilon_1}$ for every $i \leq l$). For the second claim, it is clear from the first claim that $m_i \leq \frac{8\log^{m^*}(n)t_l^{-\epsilon_1}E^i_{out}}{(\delta_*^-)^3E^i_{in}}$. So it remains to bound $\frac{\max_{v \in H_i[I_{t_{i+1}}]}|E_{H_i}(v, I_{t_i}\setminus I_{t_{i+1}}, *,*)|}{\min_{v \in H_i[I_{t_{i+1}}]}|E_{H_i}(v, I_{t_{i+1}})|}$. Now, by $i$-permissibility of $(H_i, w_i)$ we have from (P5) both that $|E_{H_i}(v, I_{t_i}\setminus I_{t_{i+1}}, *,*)|\leq c_3|I_{t_i}\setminus I_{t_{i+1}}|p_{\gr}^3$ and $|E_{H_i}(v, I_{t_{i+1}})|\geq \frac{|I_{t_{i+1}}|p_{\gr}^3}{\log^{l_3}(n)}$, and so 
$$\frac{|E_{H_i}(v, I_{t_i}\setminus I_{t_{i+1}}, *,*)|}{|E_{H_i}(v, I_{t_{i+1}})|} \leq c_3\log^{l_3}(n)$$ 
for every $i \leq l$ and every $v \in V(H_{i})$. The second claim follows.
\end{proof}

\begin{prop} \label{prop_w1*}
Suppose that $H_l$ has depth $j$ and $u \in K_{j-1} \setminus K_{j+1}$. Then
$$w_{1^*}(E_{H_{1^*}}(u, I_{t_{l+1}})) \geq \frac{c_{1,5}c_{1,1}^*}{\log^2(n)}.$$
\end{prop}
\begin{proof}
By Theorem \ref{thm_H_1} we have that $|E_{H_1}(u, I_{t_{l+1}})| \geq c_{1,5}k_jp_{\gr}^3$. Furthermore, by Theorem \ref{thm_special_reweight}, each edge in $E_{H_1}(u, I_{t_{l+1}})$ has weight at least $\frac{c_{1,1}^*}{k_jp_{\gr}^3 \log^{2}(n)}$. The proposition follows.
\end{proof}

It remains to check that $(H_{l+1}, w_{l+1})$ satisfies all of the properties that ensure it is $(l+1)$-permissible, given that $(H_i, w_i)$ is $i$-permissible for every $i \leq l$. Note that in this case we reach $(H_{l+1}, w_{l+1})$ by Plan \ref{alg_subgraphs}, and certainly have $H_{l+1} \subseteq H_{1^*}[I_{t_{l+1}}]$, so (P1) holds. By Proposition \ref{prop_apfm} we have that $w_{l+1}$ is a fractional matching for $H_{l+1}$, so $1 \geq d_{w_{l+1}, H_{l+1}}(v)$ for every $v \in V(H_{l+1})$. Furthermore, we have that $d_{w_{l+1}, H_{l+1}}(v)=1-c_{H_{l+1}, v} \geq 1-\log^{m^*}(n)t_{l+1}^{-\epsilon_1}$ by Proposition \ref{prop_delta_c} for all $v \in V(H_{l+1})$ and for all $v$ which were not $l$-reachable, it is clear that $d_{w_{l+1}, H_{l+1}}(v)=d_{w_{1^*}, H_{1^*}}(v)$, so (P2) also holds. We consider the remaining properties, (P3)-(P5) of the definition of $(l+1)$-permissibility, in the following series of results, starting with (P3).  

\begin{lemma} \label{lemma_wl+1}
Suppose that $(H_i, w_i)$ is $i$-permissible for every $i \leq l$. Further suppose that $H_l$ has depth $j$ and that $e \in H_{l+1}$ is $l$-reachable. Then
$$w_{l+1}(e)=\left(1 \pm \frac{137.1c_3}{(\delta_*^-)^3}\log^{m^*+4l_3+9}(n)t_l^{-\epsilon_1} \right)\frac{w_{1^*}(e)}{\prod_{u \in e} w_{1^*}(E_{H_{1^*}}(u, I_{t_{l+1}}))}.$$
\end{lemma}
\begin{proof}
By Lemma \ref{lemma_to_next_it}(ii) we have that
$$w_{l+1}(e)=\left(1 \pm \left(17.1m_i + 2.2\log^{3l_3+7}(n)t_l^{-\epsilon_1} \right)\right) \frac{w_l(e)}{\prod_{u \in e}(1-d_{w_l^o}(u))},$$
and so by Proposition \ref{prop_cde_err}, 
$$w_{l+1}(e)=\left(1 \pm \frac{137c_3}{(\delta_*^-)^3}\log^{m^*+4l_3+8}(n)t_l^{-\epsilon_1}\right) \frac{w_l(e)}{\prod_{u \in e}(1-d_{w_l^o}(u))}.$$
Since $e \in E_{H_{l+1}}$, we have that for every $u \in e$ that $u \in I_{t_{l+1}}$. Thus by Proposition \ref{prop_o'} we may replace $\prod_{u \in e}(1-d_{w_l^o}(u))$ by $\left(1 \pm \frac{c_{H_l}}{\delta_{H_l}}\right)^3\prod_{u \in e} w_l(E_{H_l}(u, I_{t_{l+1}}))$ so that, again using Proposition \ref{prop_cde_err}, 
$$w_{l+1}(e)=\left(1 \pm \frac{137.1c_3}{(\delta_*^-)^3}\log^{m^*+4l_3+8}(n)t_l^{-\epsilon_1}\right) \frac{w_l(e)}{\prod_{u \in e}w_l(E_{H_l}(u, I_{t_{l+1}}))}.$$
We claim that for every $i \leq l$, 
\begin{multline*}
\frac{w_i(e)}{\prod_{u \in e}w_{i}(E_{H_{i}}(u, I_{t_{l+1}}))}= \\
\left(1 \pm \frac{137.1c_3}{(\delta_*^-)^3}\log^{m^*+4l_3+8}(n)t_l^{-\epsilon_1}\right)\frac{w_{i-1}(e)}{\prod_{u \in e}w_{i-1}(E_{H_{i-1}}(u, I_{t_{l+1}}))}.
\end{multline*}
Indeed, by Lemma \ref{lemma_to_next_it}(iv),
\begin{equation} \label{eq_wid}
w_i(E_{H_i}(u, I_{t_{l+1}}))=(1 \pm 1.1\log^{3l_3+7}(n)t_i^{-\epsilon_1})\frac{w_{i-1}(E_{H_{i-1}}(u, I_{t_{l+1}}))}{1-d_{w_{i-1}^o}(u)},
\end{equation}
for each $u \in e$. Then using Lemma \ref{lemma_to_next_it} (ii) again the claim follows. Subsequently, we see that
\begin{multline*}
\frac{w_l(e)}{\prod_{u \in e}(1-d_{w_l^o}(u))}= \\
\prod_{t^*(k_{j-3})\leq i \leq l} \left(1 \pm \frac{137.1c_3}{(\delta_*^-)^3}\log^{m^*+4l_3+8}(n)t_l^{-\epsilon_1}\right) \frac{w_{1^*}(e)}{\prod_{u \in e}(1-d_{w_{1^*}^o}(u))}.
\end{multline*}
Since $H_l$ has depth $j$, we have that $l-t^*(k_{j-3})\leq 40\log\log(n)$. Thus,
\begin{multline*}
w_{l+1}(e) =  \\
\left(1 \pm \frac{137.1c_3}{(\delta_*^-)^3}\log^{m^*+4l_3+8}(n)t_l^{-\epsilon_1} \right)^{l-t^*(k_{j-3})+1}\frac{w_{1^*}(e)}{\prod_{u \in e} w_{1^*}(E_{H_{1^*}}(u, I_{t_{l+1}}))} \\
= \left(1 \pm \frac{137.1c_3}{(\delta_*^-)^3}\log^{m^*+4l_3+9}(n)t_l^{-\epsilon_1} \right)\frac{w_{1^*}(e)}{\prod_{u \in e} w_{1^*}(E_{H_{1^*}}(u, I_{t_{l+1}}))}.
\end{multline*}
\end{proof}

\begin{cor} \label{cor_e_permiss_l}
Suppose that $(H_i, w_i)$ is $i$-permissible for every $i \leq l$. Furthermore, suppose that $H_l$ has depth $j$ and that $e \in H_{l+1}$ is $l$-reachable. Then, recalling that $l_1=10$, we have that
$$\frac{c_{1,1}^*}{2t_{l+1}p_{\gr}^3\log^{2}(n)} \leq w_{l+1}(e) \leq \frac{2c_{1,2}^*\log^{9}(n)}{(c_{1,5}c^*_{1,1})^4t_{l+1}p_{\gr}^3}.$$
\end{cor}
\begin{proof}
By Lemma \ref{lemma_wl+1} we have that 
$$w_{l+1}(e)= (1 \pm o(1))\frac{w_{1^*}(e)}{\prod_{u \in e} w_{1^*}(E_{H_{1^*}}(u,I_{t_{l+1}}))}.$$
Each $w_{1^*}(E_{H_{1^*}}(u,I_{t_{l+1}}))\leq 1$, so certainly $w_{l+1}(e) \geq w_{1^*}(e)/2$. Now, given that $e$ is $l$-reachable, and $H_l$ has depth $j$, we have that $e$ is of type $(4,0,0)_{j-1}$, $(\alpha,\beta,0)_j$ or $(\alpha,\beta,0)_{j+1}$ for $\alpha \neq 0$. By Theorem \ref{thm_special_reweight} it follows that $\frac{c_{1,1}^*}{k_{j-1}p_{\gr}^3\log(n)}\leq w_{1^*}(e) \leq \frac{c_{1,2}^*}{k_{j+1}p_{\gr}^3\log(n)}$ for every $l$-reachable $e \in H_{l+1}$. Thus we have that 
$$w_{l+1}(e) \geq \frac{c_{1,1}^*}{2k_{j-1}p_{\gr}^3\log(n)}\geq \frac{c_{1,1}^*}{2t_{l+1}p_{\gr}^3\log^{2}(n)}.$$

Considering upper bounds, we have that $w_{1^*}(E_{H_{1^*}}(u,I_{t_{l+1}}))=d_{w_{1^*}, H_{1^*}}(u) \geq 1-t_1^{-\epsilon_1}$ for each $u \in K_{j+1}$. Otherwise $u \in K_{j-1}\setminus K_{j+1}$ and so by Proposition \ref{prop_w1*} we have that $w_{1^*}(E_{H_{1^*}}(u,I_{t_{l+1}})) \geq \frac{c_{1,5}c_{1,1}^*}{\log^{2}(n)}$.  

It follows that $w_{l+1}(e) \leq \frac{c_{1,2}^*}{k_{j+1}p_{\gr}^3\log(n)} \cdot \frac{2\log^{8}(n)}{(c_{1,5}c_{1,1}^*)^4} \leq \frac{2c_{1,2}^*\log^{9}(n)}{(c_{1,5}c_{1,1}^*)^4t_{l+1}p_{\gr}^3},$ as required.
\end{proof}

This gives (P3). To show (P4) and (P5), we first include the following proposition:

\begin{prop}\label{cor_allowable}
Suppose that $H_l$ has depth $j$. Let 
$$p_{S}(v):=\frac{w_{i-1}(E_{H_{i-1}}(v, I_{t_{l+1}}))}{1-d_{w_{i-1}^o}(v)}\mathbbm{1}_{v \in V(H_{i-1}[S])},$$ 
and 
$$f_v(e):=\left(\prod_{u \in e\setminus \{v\}} \frac{w_{i-1}(E_{H_{i-1}}(u, I_{t_{l+1}}))}{1-d_{w_{i-1}^o}(u)}\right)\mathbbm{1}_{e \in E_{H_{i-1}}(v,S_1, S_2, S_3)}$$
where $t^*(k_{j-3}) \leq i-1 \leq l$. Then whenever $S$ is $i$-reachable, or $(v, S_1, S_2, S_3)$ is an $i$-reachable tuple with $v \in V(H_{i-1}[I_{t_i}])$, we have that $p_S$ is vertex allowable for $(H_{i-1}, w_{i-1}, t_{i-1}, \eta_1)$ and $f_v$ is $v$-edge allowable for $(H_{i-1}, w_{i-1}, t_{i-1}, \eta_1)$.
\end{prop}
\begin{proof}
By the permissibility of $(H_{i-1}, w_{i-1})$ and using Propositions \ref{prop_omega_permiss} and \ref{prop_o'}, in each case we have that 
$$\frac{w_{i-1}(E_{H_{i-1}}(v, I_{t_{l+1}}))}{1-d_{w_{i-1}^o}(v)} \leq \frac{8c_3\log^{l_1+l_3+2}(n)t_{l+1}}{c_1t_i}$$ and that $$\frac{w_{i-1}(E_{H_{i-1}}(v, I_{t_{l+1}}))}{1-d_{w_{i-1}^o}(v)} \geq \frac{c_1t_{l+1}}{8c_3\log^{l_1+l_3+2}(n)t_i}.$$ Thus
$$\frac{\max \frac{w_{i-1}(E_{H_{i-1}}(v, I_{t_{l+1}}))}{1-d_{w_{i-1}^o}(v)}}{\min \frac{w_{i-1}(E_{H_{i-1}}(v, I_{t_{l+1}}))}{1-d_{w_{i-1}^o}(v)}} \leq \frac{64c_3^2\log^{2l_1+2l_3+4}(n)}{c_1^2} \leq \log^{500}(n).$$
By Remark \ref{rem_allow} it follows that $p_S$ is vertex allowable. Furthermore, we see immediately also that 
$$\frac{\max_e \prod_{u \in e\setminus\{v\}}\frac{w_{i-1}(E_{H_{i-1}}(u, I_{t_{l+1}}))}{1-d_{w_{i-1}^o}(u)}}{\min_e \prod_{u \in e\setminus\{v\}}\frac{w_{i-1}(E_{H_{i-1}}(u, I_{t_{l+1}}))}{1-d_{w_{i-1}^o}(u)}} \leq \frac{64^3c_3^6\log^{6l_1+6l_3+12}(n)}{c_1^6} \leq \log^{500}(n),$$
so similarly $f_v$ is $v$-edge allowable, as required.
\end{proof}

\begin{prop} \label{prop_vertex_l_to_1}
Suppose that $(H_i, w_i)$ is $i$-permissible for every $i\leq l$, that $H_l$ has depth $j$, and $S$ is $(l+1)$-valid and $l$-reachable. Then
$$|V({H_{l+1}[S]})|=\left(1 \pm \frac{33c_3}{(\delta_*^-)^3}\log^{m^*+2l_3+4}(n)t_l^{-\epsilon_1} \right)\sum_{v \in V({H_{1^*}[S]})} w_{1^*}(E_{H_{1^*}}(v,I_{t_{l+1}})).$$
\end{prop}
\begin{proof}
The proof is similar to that of Lemma \ref{lemma_wl+1}. By Lemma \ref{lemma_to_next_it}(iii) we have that
$$
|V(H_{l+1}[S])|=\left(1 \pm \left(4.1m_l +1.1c_1^{-1}\log^{l_3+2}(n)t_l^{-\epsilon_1}\right)\right)\sum_{v \in V(H_{l}[S])} (1-d_{w_l^o}(v)),$$
and so by Proposition \ref{prop_cde_err}, 
$$|V(H_{l+1}[S])|=\left(1 \pm \frac{32.8c_3}{(\delta_*^-)^3}\log^{m^*+2l_3+3}(n)t_l^{-\epsilon_1}\right) \sum_{v \in V(H_{l}[S])} (1-d_{w_l^o}(v)).$$
Since $S \subseteq I_{t_{l+1}}$, we have by Proposition \ref{prop_o'} that we may replace $1-d_{w_l^o}(v)$ by $\left(1 \pm \frac{c_{H_l}}{\delta_{H_l}}\right) w_l(E_{H_l}(u, I_{t_{l+1}}))$, so that 
$$|V(H_{l+1}[S])|=\left(1 \pm \frac{33c_3}{(\delta_*^-)^3}\log^{m^*+2l_3+3}(n)t_l^{-\epsilon_1}\right) \sum_{v \in V(H_{l}[S])} w_l(E_{H_l}(v, I_{t_{l+1}})).$$
We claim that for every $i \leq l$,
\begin{multline*}
\sum_{v \in V(H_{i}[S])} w_i(E_{H_i}(v, I_{t_{l+1}}))= \\ \left(1 \pm \frac{33c_3}{(\delta_*^-)^3}\log^{m^*+2l_3+3}(n)t_l^{-\epsilon_1}\right)\sum_{v \in V(H_{i-1}[S])} w_{i-1}(E_{H_{i-1}}(v, I_{t_{l+1}})).
\end{multline*}

Indeed we may use (\ref{eq_wid}) and note that by Proposition \ref{cor_allowable} we have that $\frac{w_{i-1}(E_{H_{i-1}}(v, I_{t_{l+1}}))}{1-d_{w_{i-1}^o}(v)}$ is vertex allowable for $(H_{i-1}, w_{i-1}, t_{i-1}, \eta_1)$. 
Then again by Lemma \ref{lemma_to_next_it}(iii)
\begin{multline*}
\sum_{v \in V(H_{i}[S])} w_i(E_{H_i}(v, I_{t_{l+1}}))= \\ \left(1 \pm \frac{33c_3}{(\delta_*^-)^3}\log^{m^*+2l_3+3}(n)t_l^{-\epsilon_1}\right)\sum_{v \in V(H_{i-1}[S])} w_{i-1}(E_{H_{i-1}}(v, I_{t_{l+1}})),
\end{multline*}
as claimed. Then
\begin{multline*}
\sum_{v \in V(H_{l}[S])} w_l(E_{H_l}(v, I_{t_{l+1}}))= \\ 
\prod_{t^*(k_{j-3})\leq i \leq l} \left(1 \pm \frac{33c_3}{(\delta_*^-)^3}\log^{m^*+2l_3+3}(n)t_l^{-\epsilon_1}\right)\sum_{v \in V(H_{1^*}[S])} w_{1^*}(E_{H_{1^*}}(v, I_{t_{l+1}})),
\end{multline*}
and so
$$|V(H_{l+1}[S])|=\left(1 \pm \frac{33c_3}{(\delta_*^-)^3}\log^{m^*+2l_3+4}(n)t_l^{-\epsilon_1}\right)\sum_{v \in V(H_{1^*}[S])} w_{1^*}(E_{H_{1^*}}(v, I_{t_{l+1}})).$$
\end{proof}

\begin{cor} \label{cor_vertex_permiss_l}
Suppose that $(H_i, w_i)$ is $i$-permissible for every $i\leq l$, and that $H_l$ has depth $j$, and $S$ is $(l+1)$-valid and $l$-reachable. Then
$$\frac{c_{1,3}c_{1,5}c_{1,1}^*|S|p_{\gr}}{2\log^{2}(n)} \leq |V({H_{l+1}[S]})| \leq c_{1,4}|S|p_{\gr}.$$
\end{cor}
\begin{proof}
Clearly $|V({H_{l+1}[S]})| \leq |V({H_{1^*}[S]})|$. By Theorems \ref{thm_H_1} and \ref{thm_special_reweight} this gives the upper bound. Furthermore, by Proposition \ref{prop_vertex_l_to_1} a lower bound is given by 
$$\frac{|V({H_{1^*}[S]})|\min_{v \in V({H_{1^*}[S]})} w_{1^*}(E_{H_{1^*}}(v, I_{t_{l+1}}))}{2}.$$ 
Using Proposition \ref{prop_w1*} along with Theorems \ref{thm_H_1} and \ref{thm_special_reweight} yields the given lower bound.
\end{proof}

Corollary \ref{cor_vertex_permiss_l} addresses (P4) in the definition of $l+1$ permissibility. It remains to address (P5). 

\begin{prop} \label{prop_deg_l_to_1}
Suppose that $(H_i, w_i)$ is $i$-permissible for every $i\leq l$, that $H_l$ has depth $j$, and $(v,S_1,S_2,S_3)$ is an open or closed $l$-reachable tuple which is $(l+1)$-permissible. Then
\begin{multline*}
|E_{H_{l+1}}(v, S_1,S_2,S_3)| = \\
\left(1 \pm \frac{129c_3}{(\delta_*^-)^3}\log^{m^*+4l_3+9}(n)t_l^{-\epsilon_1} \right)\sum_{e \in E_{H_{1^*}}(v,S_1,S_2,S_3)} \prod_{u \in e \setminus\{v\}} w_{1^*}(E_{H_{1^*}}(u,I_{t_{l+1}})).
\end{multline*}
\end{prop}
\begin{proof}
Once again the proof uses the same strategy as that of Lemma \ref{lemma_wl+1} and Proposition \ref{prop_vertex_l_to_1}. By Lemma \ref{lemma_to_next_it}(v) we have that
\begin{multline*}
|E_{H_{l+1}}(v, S_1,S_2,S_3)|= \\ \left(1 \pm \left(16.1m_l +1.1\log^{3l_3+7}(n)t_l^{-\epsilon_1}\right)\right)\sum_{e \in E_{H_{l}}(v, S_1,S_2,S_3)} \prod_{u \in e \setminus \{v\}}(1-d_{w_l^o}(u)),
\end{multline*}
and so by Proposition \ref{prop_cde_err}, 
\begin{multline*}
|E_{H_{l+1}}(v, S_1,S_2,S_3)|= \\ \left(1 \pm \frac{128.9c_3}{(\delta_*^-)^3}\log^{m^*+4l_3+8}(n)t_l^{-\epsilon_1}\right) \sum_{e \in E_{H_{l}}(v, S_1,S_2,S_3)} \prod_{u \in e \setminus \{v\}}(1-d_{w_l^o}(u)).
\end{multline*}
Since $S_1, S_2, S_3 \subseteq I_{t_{l+1}}$ and $u \in I_{t_{l+1}}$ for every $u \in e \in E_{H_{l}}(v, S_1,S_2,S_3)$, we have by Proposition \ref{prop_o'} we may replace each $(1-d_{w_l^o}(u))$ by $\left(1 \pm \frac{c_{H_l}}{\delta_{H_l}}\right) w_l(E_{H_l}(u, I_{t_{l+1}}))$ so that 
\begin{multline*}
|E_{H_{l+1}}(v, S_1,S_2,S_3)|= \\
\left(1 \pm \frac{129c_3}{(\delta_*^-)^3}\log^{m^*+4l_3+8}(n)t_l^{-\epsilon_1}\right)  \sum_{e \in E_{H_{l}}(v, S_1,S_2,S_3)} \prod_{u \in e \setminus \{v\}} w_l(E_{H_l}(u, I_{t_{l+1}})).
\end{multline*}
We claim that for every $i \leq l$,
\begin{multline*} 
\sum_{e \in  E_{H_{i}}(v, S_1,S_2,S_3)} \prod_{u \in e \setminus \{v\}} w_i(E_{H_i}(u, I_{t_{l+1}})) =\\
\left(1 \pm \frac{129c_3}{(\delta_*^-)^3}\log^{m^*+4l_3+8}(n)t_l^{-\epsilon_1}\right)\sum_{e \in E_{H_{i-1}}(v, S_1,S_2,S_3)} \prod_{u \in e \setminus \{v\}} w_{i-1}(E_{H_{i-1}}(u, I_{t_{l+1}})).
\end{multline*} 
Indeed by Proposition \ref{cor_allowable}, we have that $\prod_{u \in e \setminus \{v\}} \frac{w_{i-1}(E_{H_{i-1}}(u, I_{t_{l+1}}))}{1-d_{w_{i-1}^o}(u)}$ is an allowable weight function for $(H_{i-1}, w_{i-1})$ and by (\ref{eq_wid}) we have
$$w_i(E_{H_i}(u, I_{t_{l+1}}))=(1 \pm 1.1\log^{3l_3+7}(n)t_i^{-\epsilon_1})\frac{w_{i-1}(E_{H_{i-1}}(u, I_{t_{l+1}}))}{1-d_{w_{i-1}^o}(u)}$$
for each $u$. Using Lemma \ref{lemma_to_next_it}(v) again, we have
\begin{multline*} \sum_{e \in E_{H_{i}}(v, S_1,S_2,S_3)} \prod_{u \in e \setminus \{v\}} w_i(E_{H_i}(u, I_{t_{l+1}})) =\\
\left(1 \pm \frac{129c_3}{(\delta_*^-)^3}\log^{m^*+4l_3+8}(n)t_l^{-\epsilon_1}\right)\sum_{e \in E_{H_{i-1}}(v, S_1,S_2,S_3)} \prod_{u \in e \setminus \{v\}} w_{i-1}(E_{H_{i-1}}(u, I_{t_{l+1}})),
\end{multline*}
as claimed. Then
\begin{multline*} \sum_{e \in E_{H_{l}}(v, S_1,S_2,S_3)} \prod_{u \in e \setminus \{v\}} w_l(E_{H_l}(u, I_{t_{l+1}}))= \\
\prod_{t^*(k_{j-3})\leq i \leq l} \left(1 \pm \frac{129c_3}{(\delta_*^-)^3}\log^{m^*+4l_3+8}(n)t_l^{-\epsilon_1}\right) \cdot \\ \sum_{e \in E_{H_{1^*}}(v, S_1,S_2,S_3)} \prod_{u \in e \setminus \{v\}} w_{1^*}(E_{H_{1^*}}(u, I_{t_{l+1}})),
\end{multline*}
and so
\begin{multline*}
|E_{H_{l+1}}(v, S_1,S_2,S_3)|= \\
\left(1 \pm \frac{129c_3}{(\delta_*^-)^3}\log^{m^*+4l_3+9}(n)t_l^{-\epsilon_1}\right)\sum_{e \in E_{H_{1^*}}(v, S_1,S_2,S_3)} \prod_{u \in e \setminus \{v\}} w_{1^*}(E_{H_{1^*}}(u, I_{t_{l+1}})).
\end{multline*}
\end{proof}

\begin{cor} \label{cor_deg_permiss_l}
Suppose that $(H_i, w_i)$ is $i$-permissible for every $i\leq l$, that $H_l$ has depth $j$, and $(v,S_1,S_2,S_3)$ is an open or closed $l$-reachable tuple which is $(l+1)$-permissible. Then
$$\frac{c_{1,5}^4(c_{1,1}^*)^3|S_1|p_{\gr}^3}{2\log^{6}(n)} \leq |E_{H_{l+1}}(v, S_1,S_2,S_3)| \leq 2c_{1,6}|S_1|p_{\gr}^3.$$
\end{cor}
\begin{proof}
By Proposition \ref{prop_deg_l_to_1}, we have that 
\begin{multline*}
\frac{1}{2}|E_{H_{1^*}}(v,S_1,S_2,S_3)|\left(\min_{v \in V({H_{1^*}[I_{t_{l+1}}]})} w_{1^*}(E_{H_{1^*}}(v, I_{t_{l+1}}))\right)^3 \\ \leq |E_{H_{l+1}}(v, S_1,S_2,S_3)|
\end{multline*}
and that $|E_{H_{l+1}}(v, S_1,S_2,S_3)| \leq 2|E_{H_{1^*}}(v,S_1,S_2,S_3)|.$ By Proposition \ref{prop_w1*} we also have that $\min_{v \in V({H_{1^*}[I_{t_{l+1}}]})} w_{1^*}(E_{H_{1^*}}(v, I_{t_{l+1}})) \geq \frac{c_{1,5}c_{1,1}^*}{\log^{2}(n)}$. The Corollary follows using Theorems \ref{thm_H_1}(iv) and \ref{thm_special_reweight}.
\end{proof}

Since the above corollary shows that (P5) is satisfied (as $l_3=7$) we have, given $(H_i, w_i)$ is $i$-permissible for every $i \in [l]$, that $(H_{l+1}, w_{l+1})$ is $(l+1)$-permissible for every $l \in [c_h]$. Hence since $(H_{1^*}, w_{1^*})$ is $1$-permissible, we have by induction that Plan \ref{alg_subgraphs} completes to reach $L^*$. In particular, we have shown that, assuming Theorem \ref{thm_H_1}, the \emph{vortex} completes, allowing us to obtain a matching covering all vertices in $V(H_{\gr})\setminus A^*$ but a qualifying leave $L^*$. We discussed how a qualifying leave is absorbed by $A^*$ in Section \ref{sec_mainres}, thus we have proved Lemma \ref{finish_PM}. Hence in order to prove Theorem \ref{thm_main}, all remains is to prove Theorem \ref{thm_H_1}, which is done in the following section. 

%% file: sec_initial_steps.tex
\section{Initial steps} \label{sec_H}

It remains to bridge the gap between Theorem \ref{thm_H} and Theorem \ref{thm_H_1}. Recall that $H$ has the following properties (as per Theorem \ref{thm_H}):
\begin{enumerate}[(i)]
\item every $\mathcal{T}$-valid subset $S \subseteq V(\mathcal{T})$ satisfies 
$$|V(H[S])|=(1 \pm \alpha_G)|S|p_{\gr},$$
\item for every $v \in V(H)$ and every open or closed $\mathcal{T}$-valid tuple $(v,S_1,S_2,S_3)$, we have
$$|E_H(v,S_1, S_2, S_3)|=(1 \pm \alpha_G)|E_{\mathcal{T}}(v,S_1, S_2, S_3)|p_{\gr}^3,$$
\item for every $i \in [c_g]$, 
$$|\mathcal{Z}^+_{i,e,H}(\alpha, \beta, \gamma)|:=
\begin{cases}
(1 \pm \alpha_G)|\mathcal{Z}_{i,e,\mathcal{T}}^{+}(\alpha, \beta, \gamma)|p_{\gr}^{12} & \mbox{~if $e$ is a bad edge}, \\
O\left(k_it_1p_{\gr}^{12} \right) & \mbox{~if $\alpha=0$, $\beta=1$, $\gamma=3$},\\
0 & \mbox{~otherwise}.\\
\end{cases}
$$
$$|\mathcal{Z}^-_{i,e,H}(\alpha, \beta, \gamma)|:=
\begin{cases}
(1 \pm \alpha_G)|\mathcal{Z}_{i,e,\mathcal{T}}^{-}(\alpha, \beta, \gamma)|p_{\gr}^{12} & \mbox{~if $\alpha \neq 0$ and $\gamma=0$}, \\
O\left(j_ik_ip_{\gr}^{12} \right) & \mbox{~if $\alpha=0$, $\beta=0$, $\gamma=4$},\\
O\left(j_ik_ip_{\gr}^{12} \right) & \mbox{~if $\alpha=0$, $\beta=1$, $\gamma=3$},\\
0 & \mbox{~otherwise}.\\
\end{cases}
$$
Finally, for every bad edge $e$,
$$|\mathcal{Z}^2_{i,e,H}|=O\left(k_it_1p_{\gr}^{12}\right).$$
\item $|V_{O}^{X+Y}(H)|=|V_{O}^{X-Y}(H)|$, and, furthermore $|V_O^{J}(H)|=(1 \pm 2\alpha_G)|V_E^{J}(H)|$ for every $J \in \{X,Y,X+Y,X-Y\}$. Additionally, $|V_{O/E}^{J_1}(H[S])|=(1 \pm 2\alpha_G)|V_{O/E}^{J_2}(H[S])|$ for every valid layer interval $S$, and $J_1,J_2 \in \{X,Y,X+Y,X-Y\}$.
\item For every $J \in \{X,Y,X+Y,X-Y\}$ and every valid $J$-layer interval $I^J$ and $v \notin J$,
$$|E_H(v, I^J, O/E)|=(1 \pm \alpha_G)|E_\mathcal{T}(v, I^J, O/E)|p_{\gr}^3.$$
\end{enumerate}

 As previously mentioned, the strategy is similar to that of Plan \ref{alg_subgraphs}, but for some key differences. We start by assigning a weight function $w_H$ to $H$ as in Theorem \ref{thm_H}, such that $w_H$ is an almost-perfect fractional matching for $H$ and we use Theorem \ref{thm_weighted_egj} in a similar form to Theorem \ref{thm_key} on $(H^o, w^o_H)$, the graph containing all edges in $H$ with a vertex in $V(H) \setminus I_{t_0}$, and all vertices induced by this collection of edges, to find a matching $M^o_H$, as in Step 1 of Plan \ref{alg_subgraphs}. Using Theorem \ref{thm_weighted_egj} to obtain the matching $M^o_H$ means that we may have used unequal numbers of edges of each wrap-around type. Recalling that we have certain parity requirements of $L^*$ at the end of the process, we wish to correct for any unevenness at this stage. (Note that this is only necessary when $n$ is odd. Due to the nature of our process removing disjoint edges, when $n$ is even removal of a matching cannot affect these particular parity constraints.) On the other hand, in place of our random greedy cover step, over the two iterations for which we need to consider the parity constraints, it is also possible to simultaneously use a (deterministic) greedy strategy to cover any vertices remaining outside $I_{t_0}$ after removing $M^o_H$, so whilst an additional step is required to consider parity, we combine it with a greedy cover strategy which is more straightforward than that of Step 3 in Plan \ref{alg_subgraphs}. To obtain $w_0$ from $w_H$ follows a similar but simpler strategy to that in Plan \ref{alg_subgraphs}. Before proceeding with the details for these steps we make one final note that, due to the nature of wrap-around edges, our management of parity issues is slightly different moving from $H$ to $H_0$ than from $H_0$ to $H_1$ and hence we describe the steps one by one. The reason for the difference will become clear as we detail the strategy.

\subsection{$H$ to $H_0$}

We assign a weighting $w_H$ to $H$, which is an almost-perfect fractional matching: 
\begin{defn}[$w_H$, $(H^o, w_H^o)$] \label{defs_H}
We set $w_H(e)=\frac{1}{D}$ for every $e \in H$, where $(1 - \alpha_G)np_{\gr}^3 \leq D=\max_{v \in V(H)} d_H(v) \leq (1 + \alpha_G)np_{\gr}^3$. Then $\max_{v \in V(H)} d_{w_H, H}(v)=\max_{v \in V(H)} \sum_{e \ni v} w_H(e)=1$ and $d_{w_H, H}(v) \geq \frac{1-\alpha_G}{1+\alpha_G} \geq 1-2\alpha_G$ for every $v \in V(H)$. We also define $(H^o, w_H^o)$ to be the weighted hypergraph where $V(H^o):=\{v \in V(H): \exists e \in H, e \cap V(\mathcal{T})\setminus I_{t_0} \neq \emptyset\}$, $E(H^o)=E(H[V(H^o)]) \setminus E(H[I_{t_0}])$, and $w_H^o=w_H|_{e \in H^o}$.
\end{defn}

\begin{prop} \label{prop_H}
Given $(H, w_H)$ and $(H^o, w_H^o)$ as above we have that:
\begin{enumerate}[(i)]
\item $w_H(E_H(v, I_{t_0})) \geq \frac{1}{3}-O(\alpha_G)$ for every $v \in V(H)$,
\item $d_{w_H^o}(v)=1 \pm 2\alpha_G$ for every $v \in V(H)\setminus V(H[I_{t_0}])$,
\item $d_{w_H, H}(v)=d_{w_H^o}(v)+w_H(E_H(v, I_{t_0}))=1 \pm 2\alpha_G$ for every $v \in V(H[I_{t_0}])$,
\item $1- d_{w_H^o}(v)=(1 \pm 6.1\alpha_G) w_H(E_H(v, I_{t_0}))$ for every $v \in V(H[I_{t_0}])$
\end{enumerate}
\end{prop}
\begin{proof}
By Fact \ref{fact_basic2}(i) we have that $|E_{\mathcal{T}}(v, I_{t_0})| \geq \frac{n}{3}-2$ for all $v \in V(\mathcal{T})$, where $n=2t_0+1$. Thus by Theorem \ref{thm_H}(ii), and the definition of $w_H$, (i) follows. It is straightforward to see that (ii) and (iii) follow from the fact that $1-2\alpha_G \leq d_{w_H, H}(v) \leq 1$ combined with the definitions.  
We deduce (iv) from (iii) and (i). For $v \in V(H[I_{t_0}])$,
\begin{eqnarray*}
 1-d_{w_H^o}(v) &=& w_H(E_H(v, I_{t_0})) \pm 2 \alpha_G \\
&=& w_H(E_H(v, I_{t_0}))\left(1 \pm \frac{2\alpha_G}{1/3-O(\alpha_G)}\right) 
= (1 \pm 6.1\alpha_G)w_H(E_H(v, I_{t_0})).
\end{eqnarray*}
\end{proof}

We start the process by running Theorem \ref{thm_weighted_egj} on $(H^o, w_H^o)$, to obtain $M^o_H$.

\begin{prop} \label{prop_M_H}
There exists a matching $M^o_H$ in $(H^o, w_H^o)$ such that letting $H^c:=H[V(H) \setminus V(M_H^o)]$, we have that
\begin{enumerate}[(i)]
\item every $\mathcal{T}$-valid subset $S \subseteq I_{t_0}$ satisfies 
$$|V(H^c[S])|=(1 \pm O(t_0^{-\epsilon}))\sum_{v \in V(H[S])} (1-d_{w_H^o}(v)),$$
and for any vertex allowable function $p_S(v):V(H) \rightarrow \mathbb{R}_{\geq 0}$ for $(H, w_H, t_0, \eta)$ such that $p_S(v)=f_S(v)\mathbbm{1}_{v \in V(H[S])}$,
$$\sum_{v \in V(H^c[S])} p_S(v)=(1 \pm O(t_0^{-\epsilon}))\sum_{v \in V(H[S])} p_S(v)(1-d_{w_H^o}(v)),$$
\item every open and closed $\mathcal{T}$-valid tuple $(v,S_1,S_2,S_3)$ such that $S_1, S_2, S_3 \subseteq I_{t_0}$ satisfies 
$$|E_{H^c}(v, S_1,S_2,S_3)|=(1 \pm O(t_0^{-\epsilon}))\sum_{e \in E_{H}(v, S_1,S_2,S_3)} \prod_{u \in e \setminus \{v\}} (1-d_{w_H^o}(u)),$$
and for any $v$-edge allowable function $f_v:E(H) \rightarrow \mathbb{R}_{\geq 0}$ for $(H, w_H, t_0, \eta)$ such that $f_v(e)=0$ wherever $v \notin e$,
$$f_v(E_{H^c}(v, S_1,S_2,S_3))=(1 \pm O(t_0^{-\epsilon}))\sum_{e \in E_{H}(v, S_1,S_2,S_3)} f_v(e)\prod_{u \in e \setminus \{v\}} (1-d_{w_H^o}(u)),$$
\item for every $i \in [c_g]$, 
$$|\mathcal{Z}^+_{i,e,H^c}(\alpha, \beta, \gamma)|:=$$
$$
\begin{cases}
(1 \pm O(t_0^{-\epsilon}))\sum_{z \in \mathcal{Z}^{+}_{i,e,H}(\alpha, \beta, \gamma)} \prod_{u \in z\setminus e} (1-d_{w_H^o}(u)) & \mbox{~if $e$ is a bad edge}, \\
O\left(k_it_1p_{\gr}^{12} \right) & \mbox{~if $\alpha=0$, $\beta=1$, $\gamma=3$},\\
0 & \mbox{~otherwise}.\\
\end{cases}
$$
$$|\mathcal{Z}^-_{i,e,H^c}(\alpha, \beta, \gamma)|:=$$
$$
\begin{cases}
(1 \pm O(t_0^{-\epsilon}))\sum_{z \in \mathcal{Z}^{-}_{i,e,H}(\alpha, \beta, \gamma)} \prod_{u \in z\setminus e} (1-d_{w_H^o}(u)) & \mbox{~if $\alpha \neq 0$ and $\gamma=0$}, \\
O\left(j_ik_ip_{\gr}^{12} \right) & \mbox{~if $\alpha=0$, $\beta=0$, $\gamma=4$},\\
O\left(j_ik_ip_{\gr}^{12} \right) & \mbox{~if $\alpha=0$, $\beta=1$, $\gamma=3$},\\
0 & \mbox{~otherwise}.\\
\end{cases}
$$
Finally, for every bad edge $e$,
$$|\mathcal{Z}^2_{i,e,H^c}|=O\left(k_it_1p_{\gr}^{12}\right).$$
\item For every $J \in \{X,Y, X+Y, X-Y\}$ and every valid $J$-layer interval $S \subseteq I_{t_0}$,
$$|V_{O/E}^J(H^c[S])|=(1 \pm O(t_0^{-\epsilon}))\sum_{v \in V(H_{O/E}^J[S])} (1-d_{w_H^o}(v)),$$
\item For every $J \in \{X, Y X+Y, X-Y\}$ and every valid $J$-layer interval $I^J \subseteq I_{t_0}$ and $v \notin J$ with $v \in V(H^c)$,
$$|E_{H^c}(v, I^J, O/E)|=(1 \pm O(t_0^{-\epsilon}))\sum_{e \in E_H(v, I^J, O/E)} \prod_{u \in e\setminus\{v\}} (1-d_{w_H^o}(u)),$$
\end{enumerate}
where $\epsilon = 10^{-8}/204800$.
\end{prop}  
\begin{proof}
The proof follows the same strategy as the proof of Theorem \ref{thm_key}. In particular, we obtain $H^c$ by running Theorem \ref{thm_weighted_egj} on $(H^o, w_H^o)$ with $\Delta=t_0$ and $\eta=10^{-8}$. Comparing (i), (ii), (iv) and (v) to the proof of Theorem \ref{thm_key} (i) and (ii), the key difference is that in place of having $S$ such that $|S| \geq \frac{t_i}{\log^{2}(n)}$, we have that $|S|$ is of size $\Omega(\min\{t_h, \alpha_G^2n\})=\Omega(n^{10^{-5}})$. Additionally, in place of calling Corollary \ref{cor_bound}, we have by Proposition \ref{prop_H} that $1-d_{w_H^o}(v) \geq 1/4$ for every $v \in I_{t_0}$. Then we have from Theorem \ref{thm_H} and the definition of $H^o$ that $\sum_{v \in V(H^o[S])} p_S({v})=|V(H^o[S])| \geq \frac{|S|p_{\gr}}{2}=\Omega(n^{10^{-5}-10^{-25}})$ and $\max_{v \in V(H^o[S])} p_S({v})=1$. Thus since $\eta=10^{-8}$, we see that (\ref{eq_func}) still holds easily in this setting. Similarly, for degree-type properties we get that $|E_{H^o}(v, \mathcal{S})| \geq \Omega(n^{10^{-5}}p_{\gr}^3)$ and still (\ref{eq_func}) holds with $\eta=10^{-8}$.

 For (iii) the details are similar. First note that we only need to consider bad edges and edges of type $(\alpha, \beta, 0)_i$ where $\alpha \neq 0$ and for every $i \in [c_g]$, since the process only removes zero-sum configurations and by Theorem \ref{thm_H} the claims for all other edge types are already satisfied. Then, recalling the notation from Section \ref{sec_zsfunc},  $\max_{\bar{v} \in \binom{V(H^o)}{l}} f^{(l)}_{Z^{\pm}_{i,e,H}(\alpha, \beta, \gamma)}(\bar{v})=O(t_1)$ and $\sum_{\bar{v} \in \binom{V(H^o)}{l}} f^{(l)}_{Z^{\pm}_{i,e,H}(\alpha, \beta, \gamma)}(\bar{v})=\Omega(t_1j_ip_{\gr}^{12})$ for each edge $e$ under consideration. Then since $j_ip_{\gr}^{12}\gg t_0^{2\eta}$ for all $i$ we still have that (\ref{eq_func}) holds. 
\end{proof}

\begin{cor} \label{cor_M_H}
In $H^c:=H[V(H) \setminus V(M_H^o)]$, we have that
\begin{enumerate}[(i)]
\item every $0$-valid subset $S \subseteq \mathcal{T}$ satisfies 
$$|V(H^c[S])|=\Theta\left(|S|p_{\gr}\right),$$
\item every open and closed $0$-valid tuple $(v,S_1,S_2,S_3)$ satisfies 
$$|E_{H^c}(v, S_1,S_2,S_3)|=\Theta\left(|E_{\mathcal{T}}(v, S_1,S_2,S_3)|p_{\gr}^3\right)$$
\item for every $i \in [c_g]$, 
$$|\mathcal{Z}^+_{i,e,H^c}(\alpha, \beta, \gamma)|:=
\begin{cases}
\Theta\left(|\mathcal{Z}^{+}_{i,e,\mathcal{T}}(\alpha, \beta, \gamma)|p_{\gr}^{12}\right) & \mbox{~if $e$ is a bad edge}, \\
O\left(k_it_1p_{\gr}^{12} \right) & \mbox{~if $\alpha=0$, $\beta=1$, $\gamma=3$},\\
0 & \mbox{~otherwise}.\\
\end{cases}
$$
$$|\mathcal{Z}^-_{i,e,H^c}(\alpha, \beta, \gamma)|:=
\begin{cases}
\Theta\left(|\mathcal{Z}^{-}_{i,e,\mathcal{T}}(\alpha, \beta, \gamma)|p_{\gr}^{12}\right) & \mbox{~if $\alpha \neq 0$ and $\gamma=0$}, \\
O\left(j_ik_ip_{\gr}^{12} \right) & \mbox{~if $\alpha=0$, $\beta=0$, $\gamma=4$},\\
O\left(j_ik_ip_{\gr}^{12} \right) & \mbox{~if $\alpha=0$, $\beta=1$, $\gamma=3$},\\
0 & \mbox{~otherwise}.\\
\end{cases}
$$
Finally, for every bad edge $e$,
$$|\mathcal{Z}^2_{i,e,H^c}|=O\left(k_it_1p_{\gr}^{12}\right).$$
\item For every $J \in \{X, Y, X+Y, X-Y\}$ and every valid $J$-layer interval $S \subseteq I_{t_0}$,
$$|V_{O/E}^J(H^c[S])|=\Theta\left(|V(\mathcal{T}_{O/E}^J[S])|p_{\gr}\right),$$
\item For every $J \in \{X, Y X+Y, X-Y\}$ and every valid $J$-layer interval $I^J \subseteq I_{t_0}$ and $v \notin J$ with $v \in V(H^c)$,
$$|E_{H^c}(v, I^J, O/E)|=\Theta\left(|E_{\mathcal{T}}(v, I^J, O/E)|p_{\gr}^3\right).$$
\end{enumerate} 
\end{cor}
\begin{proof}
We again use that by Proposition \ref{prop_H}, $1 \geq 1-d_{w_H^o}(u) \geq 1/4$ for every $u \in V(H[I_{t_0}])$. Combining this with Theorem \ref{thm_H} and Proposition \ref{prop_M_H}, and noting that we only consider subsets $S \subseteq V(\mathcal{T})$ either contained within $I_{t_0}$, or such that $|S \cap I_{t_0}|= \Theta(t_0)$, the claim follows.
\end{proof}

We shall now modify $M^o_H$ to balance out parities in what remains of the $X+Y$ and $X-Y$ parts, and cover the remaining vertices in $V(H^c)\setminus V(H[I_{t_0}])$.

\subsection{Parity modifications and the greedy cover to reach $H_0$}

First note that, as previously mentioned, these parity modifications are only required when $n$ is odd. In this case, in the first two steps we will need to modify $M^o_H$ and the subsequent matching we shall obtain from $H_0$, $M^o_0$, to ensure that the leave $L^*$ at the end of the process satisfies parity requirements for a particular `zero-summing' strategy (i.e. that of Proposition \ref{zero-summing}).

We start by noting the distribution of odd and even vertices in parts $|V^{X \pm Y}_{O/E}(H^c)|$. First note (trivially) that $|V^{X+Y}_{O}(H^c)| + |V^{X+Y}_E(H^c)|=|V^{X-Y}_{O}(H^c)| + |V^{X-Y}_E(H^c)|$ with no error terms, since we only `lose' vertices by removing disjoint edges, each of which removes exactly one from each part. Let $P_{G}:=||V^{X+Y}_{O}(G)|-|V^{X-Y}_{O}(G)||$, the absolute difference in number of odd vertices in the $X+Y$ and $X-Y$ parts in some $G \subseteq \mathcal{T}$. We refer to $P_G$ as the {\it parity disparity of $G$}. To calculate the parity disparity of $H^c$, we shall use the valid layer intervals we have been keeping track of in Theorem \ref{thm_H} (iv) and which were originally introduced in Definition \ref{def_layer}, which will enable us to break up particular summations into small subsets in a useful way. 

\begin{prop} \label{prop_parity_H}
The parity disparity of $H^c$, $P_{H^c}$, satisfies
$$P_{H^c} \leq 2.1t_0^{1-\epsilon}p_{\gr}.$$ 
\end{prop}
\begin{proof}
We have from (\ref{eq_func_vert}) that
$$|V^{X \pm Y}_{O}(H^c)|=\sum_{v \in V^{X\pm Y}_{O}(H)} (1-d_{w_H^o}(v)) \pm t_0^{-\epsilon} \sum_{v \in V^{X\pm Y}_{O}(H)} d_{w_H^o}(v).$$
\begin{claim} \label{claim_parity}
$$\sum_{v \in V^{X-Y}_{O}(H)} (1-d_{w_H^o}(v))=(1 \pm O(t_0^{-2\epsilon}))\sum_{v \in V^{X+Y}_{O}(H)} (1-d_{w_H^o}(v)).$$
\end{claim}
\begin{proof}[Proof of claim] 
Let $\mathcal{I}$ be a partition of $[-\frac{n-1}{2}, \frac{n-1}{2}]$ into consecutive intervals of size $t_0^{1-2\epsilon}$. Then for any interval $I \in \mathcal{I}$, by Theorem \ref{thm_H} (iv), we have that $|V^{X+Y}_{O}(H[I])|=(1 \pm 2\alpha_G)|V^{X-Y}_{O}(H[I])|$. Now let $v_I$ be the vertex in the middle of interval $I$ and let $d_I:=(1-d_{w_H^o}(v_I))$. Then for every $v \in I$, by Theorem \ref{thm_H} (ii) we have that $(1-d_{w_H^o}(v))=(1 \pm O(t_0^{-2\epsilon}))d_I$. 
Finally, for two vertices $v_{X+Y}$ and $v_{X-Y}$ with the same index in parts $X+Y$ and $X-Y$, we have that $d_{w_{\mathcal{T}}^o}(v_{X+Y})=d_{w_{\mathcal{T}}^o}(v_{X-Y})$, and hence, by Theorem \ref{thm_H} (ii), $d_{w_{H}^o}(v_{X+Y})=(1 \pm 2\alpha_G)d_{w_{H}^o}(v_{X-Y})$. Thus, since $t_0^{-2\epsilon}\gg \alpha_G$, we find that
\begin{multline*}
\sum_{v \in V^{X-Y}_{O}(H)} (1-d_{w_H^o}(v)) = \sum_{I \in \mathcal{I}} \sum_{v \in V^{X-Y}_O(H[I])} (1-d_{w_H^o}(v)) \\
= \sum_{I \in \mathcal{I}} \sum_{v \in V^{X-Y}_O(H[I])} (1 \pm O(t_0^{-2\epsilon}))d_I 
=  \sum_{I \in \mathcal{I}} |V^{X-Y}_O(H[I])| (1 \pm O(t_0^{-2\epsilon}))d_I \\
=  (1 \pm O(t_0^{-2\epsilon}))\sum_{I \in \mathcal{I}} |V^{X+Y}_O(H[I])|d_I 
=  (1 \pm O(t_0^{-2\epsilon}))\sum_{I \in \mathcal{I}}\sum_{v \in V^{X+Y}_O(H[I])} (1-d_{w_H^o}(v)) \\
= (1 \pm O(t_0^{-2\epsilon}))\sum_{v \in V^{X+Y}_{O}(H)} (1-d_{w_H^o}(v)).
\end{multline*}
\end{proof}
Thus, since $t_0^{-\epsilon} \sum_{v \in V^{X\pm Y}_{O}(H)} d_{w_H^o}(v) \leq t_0^{1-\epsilon}$, we have that
$$|V^{X+Y}_{O}(H^c)|=(1 \pm O(t_0^{-2\epsilon}))\sum_{v \in V^{X-Y}_{O}(H)} (1-d_{w_H^o}(v)) \pm (1 \pm \alpha_G)t_0^{1-\epsilon}p_{\gr},$$
and 
$$|V^{X-Y}_{O}(H^c)|=\sum_{v \in V^{X\pm Y}_{O}(H)} (1-d_{w_H^o}(v)) \pm (1 \pm \alpha_G)t_0^{1-\epsilon}p_{\gr}$$
and the result follows.
\end{proof}

Without loss of generality, assume that there are more odd parity vertices in $V(H^c[X-Y])$ than $V(H^c[X+Y])$, and thus fewer even vertices in $V(H^c[X-Y])$ than $V(H^c[X+Y])$. Hence to fix the parity issues, we wish to greedily form $M_H' \supseteq M_H^o$ by adding $P_{H^c}$ disjoint edges from $H^c$, such that they have even and odd vertices in the $X+Y$ and $X-Y$ parts respectively. We simultaneously try to cover any vertices outside $I_{t_0}$, using these additional edges as much as possible to balance the parity issues, and then deal with any remaining parity issues after this cover step. 

Let $U_{H^c}^X$ and $U^Y_{H^c}$ denote the sets of vertices in $(V(H^c)\setminus I_{t_0}) \cap X$ and $(V(H^c)\setminus I_{t_0}) \cap Y$ respectively. We'll drop the subscript $H^c$ when it is clear from context. (Since $I_{t_0}$ covers $V(H^c) \cap X \pm Y$, there are no other vertices to consider outside of the target interval.)

\begin{prop} \label{prop_cover}
We have that 
$$|U^{X/Y}|\leq t_0^{1-\epsilon}p_{\gr}.$$
\end{prop}
\begin{proof}
By (\ref{eq_func_vert}) we have that 
\begin{multline*}
|U^{X/Y}|=\sum_{v \in V(H[X/Y])\setminus I_{t_0}} (1-d_{w^o_H}(v)) \pm |V(H[X/Y])\setminus I_{t_0}|t_0^{-\epsilon} \\ \leq |V(H[X/Y])\setminus I_{t_0}| (2\alpha_G + t_0^{\epsilon}),
\end{multline*}
since we know from Proposition \ref{prop_H} that $d_{w_H^o}(v)\geq 1-2\alpha_G$ for every $v \in V(H) \setminus I_{t_0}$. Then since $|V(\mathcal{T}[X/Y]) \setminus I_{t_0}|=\frac{2t_0}{3}$, by Theorem \ref{thm_H}(i) we have that, 
$$|V(H[X/Y])\setminus I_{t_0}| \leq \frac{2.1t_0p_{\gr}}{3}.$$
Since $\alpha_G \ll t_0^{-\epsilon}$, the claim follows.
\end{proof}

Now for a vertex $v \in U^X \cup U^Y$ we want to lower bound the number of feasible edges that we may cover $v$ by in $H^c$, which also help to balance out the parity disparity. We deliberately restrict to edges which avoid vertices in $I_{t_{20}}$. This is so that any properties that this process affects are only to vertices with index of size $\Theta(n)$. In this way we are helping to preserve `nice' properties of smaller order vertices (which we need to last longer for the process to be successful) in a straightforward way.

\begin{prop} \label{prop_parity_deg}
Every vertex $v \in (V(H^c)\setminus I_{t_1}) \cap (X \cup Y)$ satisfies
$$|E_{H^c}(v, I_{t_0}\setminus I_{t_{20}}) \cap E_{AB}(\mathcal{T})|\geq \frac{t_0p_{\gr}^3}{1500},$$
for $AB \in \{OE, EO\}$.
\end{prop}
\begin{proof}
By Proposition \ref{prop_M_H} we have that for each $v \in V(H^c)$ that
$$|E_{H^c}(v, I_{t_0}\setminus I_{t_{20}})| = (1 \pm O(t_0^{-\epsilon})) \sum_{e \in E_{H}(v, I_{t_0}\setminus I_{t_{20}})} \prod_{u \in e \setminus \{v\}} (1- d_{w_H^o}(u)).$$
Recall from Proposition \ref{prop_H} that $1- d_{w_H^o}(u) \geq (1-6.1\alpha_G)w_H(E_H(u, I_{t_0})) \geq 1/4$ for every $u \in V(H[I_{t_0}])$. In addition, by Fact \ref{fact_basic2}(iii) for every $v \in V(\mathcal{T}\setminus I_{t_1})$, we have that $|E_{\mathcal{T}}(v, I_{t_0}\setminus I_{t_{20}}) \cap E_{AB}(\mathcal{T})|\geq \frac{t_0}{20}$, and so by Theorem \ref{thm_H}, $|E_{H}(v, I_{t_0}\setminus I_{t_{20}}) \cap E_{AB}(\mathcal{T})|\geq \frac{t_0p_{\gr}^3}{21}$. Thus, by Proposition \ref{prop_M_H}~(v) and Proposition \ref{prop_H}~(i) the proposition follows.
\end{proof}

We now show that we can cover $U^X \cup U^Y$ and obtain $H_0 \subseteq H_1$ by removing a matching $M^c_H$ such that $\bigcup M^c_H \cap I_{t_{20}} = \emptyset$ and such that $P_{H_0}=0$.

\begin{prop}
There exists a matching $M^c_H \subseteq H^c$ such that $\bigcup M^c_H \cap I_{t_{20}} = \emptyset$ and setting $H_0:= H[V(H) \setminus V(M^c_H \cup M^o_H)]$ we have that $V(H_0) \subseteq I_{t_0}$, $P_{H_0}=0$ and $|M^c_H|=O(t_0^{1-\epsilon}p_{\gr})$.
\end{prop}
\begin{proof}
Noting that $\frac{t_0p_{\gr}^3}{1500} = \Omega(n^{1-(3 \times 10^{-25})})$ and $2.1t_0^{1-\epsilon}p_{\gr}=O(n^{1-10^{-14}})$, it follows from Propositions \ref{prop_cover} and \ref{prop_parity_deg}, that we can cover all the vertices in $U^X \cup U^Y$ greedily using a collection of disjoint edges only of the correct parity type to reduce $P_{H^c}$, and avoiding vertices in $I_{t_{20}}$. Thus we first greedily pair vertices in $U^X$ with vertices in $Y$, and vertices in $U^Y$ with vertices in $X$ to dictate wrap around edges of the right parity so that, updating $P_{H^c}$ as we go along, we always reduce $P_{H^c}$ until either $P_{H^c} = 0$ or $U^X \cup U^Y$ have all been covered, where each choice avoids previous vertices thus yielding a matching $M^c_1$. In the former case, having reduced $P_{H^c}$ to $0$ and having some remaining vertices uncovered in $U^X \cup U^Y$, we enumerate these remaining vertices $v_1, v_2, \ldots$ and greedily choose edge $i$ for $v_i$ so that $e_i$ is of even-odd parity when $i$ is odd, and odd-even parity when $i$ is even, each time continuing to avoid all vertices already used in the process to obtain a matching $M^c_2$. If the enumeration was even then we are done. Indeed, letting $M^c_H:=M^c_1 \cup M^c_2$ we have an edge in $M^c_H$ for each of the $O(t_0^{1-\epsilon}p_{\gr})$ vertices in $U^X \cup U^Y$ and we have that $H_0:=H[V(H) \setminus V(M^c_H \cup M^o_H)]$ satisfies $P_{H_0}=0$ and $V(H_0) \subseteq I_{t_0}$, as desired. If however, the enumeration was odd, we find that we have now increased $P_{H^c} = 1$. Equally, if we are in the latter case, we also still have $P_{H_c}>0$ and $U^X \cup U^Y$ covered by a matching, $M^c_1$. In both of these cases we still have that $P_{H_c}=O(t_0^{1-\epsilon}p_{\gr})$ and we must find a matching such that removing the matching obtains $P_{H_c}=0$. For the current updated value of $P_{H_c}$, choose $\lfloor \frac{P'_{H^c}}{2} \rfloor$ vertices with largest modulus in $(V(H^c)\setminus V(M^c_1)) \cap X$ to form $W^X$, and $\lceil \frac{P'_{H^c}}{2} \rceil$ vertices with largest modulus in $(V(H^c)\setminus V(M^c_1)) \cap Y$ to form $W^Y$. By Proposition \ref{prop_parity_deg}, their appropriate parity degree avoiding $I_{t_{20}}$ is large enough to greedily reduce $P_{H^c}$ to $0$ by adding edges of the appropriate parity to cover $W^X \cup W^Y$ whilst avoiding all previous choices. This yields a matching $M^c_2$ such that taking $M_H^c:=M^c_1 \cup M^c_2$ we have $|M_H^c|=O(t_0^{1-\epsilon}p_{\gr})$, and for $H_0:=H[V(H) \setminus V(M^c_H \cup M^o_H)]$ we also have $P_{H_0}=0$ and $V(H_0) \subseteq I_{t_0}$ as desired, completing the proof.
\end{proof}

We have now reached $H_0 \subseteq H$, where parity requirements in the $X \pm Y$ parts are satisfied. We now show that the properties of Proposition \ref{prop_M_H} are not overly affected by this process.

\begin{lemma}\label{lemma_key_props_small}
There exists an absolute constant $C_{0}$ such that $H_0$ satisfies the following:
\begin{enumerate}[(i)]
\item every $0$-valid subset $S \subseteq \mathcal{T}$ satisfies 
$$|V(H_0[S])|=(1 \pm C_{0}t_0^{-\epsilon})\sum_{v \in V(H[S])} (1-d_{w_H^o}(v)),$$
and for any vertex allowable function $p_S(v):V(H) \rightarrow \mathbb{R}_{\geq 0}$ for $(H, w_H, t_0, \eta)$ such that $p_S(v)=f_S(v)\mathbbm{1}_{v \in V(H[S])}$,
$$\sum_{v \in V(H_0[S])} p_S(v)=(1 \pm C_0t_0^{-\epsilon})\sum_{v \in V(H[S])} p_S(v)(1-d_{w_H^o}(v)),$$
\item every open and closed $0$-valid tuple $(v,S_1,S_2,S_3)$ satisfies 
$$|E_{H_0}(v,S_1,S_2,S_3)|=(1 \pm C_{0}p_{\gr}^{-2}t_0^{-\epsilon})\sum_{e \in E_{H}(v,S_1,S_2,S_3)} \prod_{u \in e\setminus\{v\}} (1-d_{w_H^o}(u)),$$
furthermore, for any $v$-edge allowable function $f_v$ for $(H, w_H, t_0, \eta)$  
we have that
$$f(E_{H_0}(v,S_1,S_2,S_3))=(1 \pm C_{0}p_{\gr}^{-2}t_0^{-\epsilon})\sum_{e \in E_{H}(v,S_1,S_2,S_3)} f(e) \prod_{u \in e\setminus\{v\}} (1-d_{w_H^o}(u)),$$
\item for every $i \in [c_g]$, 
$$|\mathcal{Z}^+_{i,e,H_0}(\alpha, \beta, \gamma)|:=$$
$$
\begin{cases}
(1 \pm C_{0}p_{\gr}^{-11}t_0^{-\epsilon})\sum_{z \in \mathcal{Z}^{+}_{i,e,H}(\alpha, \beta, \gamma)} \prod_{u \in z \setminus \{e\}} (1-d_{w_H^o}(u)) & \mbox{~if $e$ is a bad edge}, \\
O\left(k_it_1p_{\gr}^{12} \right) & \mbox{~if $\alpha=0$, $\beta=1$, $\gamma=3$},\\
0 & \mbox{~otherwise}.\\
\end{cases}
$$
$$|\mathcal{Z}^-_{i,e,H_0}(\alpha, \beta, \gamma)|:=$$
$$
\begin{cases}
(1 \pm C_{0}p_{\gr}^{-11}t_0^{-\epsilon})\sum_{z \in \mathcal{Z}^{-}_{i,e,H}(\alpha, \beta, \gamma)} \prod_{u \in z \setminus \{e\}} (1-d_{w_H^o}(u)) & \mbox{~if $\alpha \neq 0$ and $\gamma=0$}, \\
O\left(j_ik_ip_{\gr}^{12} \right) & \mbox{~if $\alpha=0$, $\beta=0$, $\gamma=4$},\\
O\left(j_ik_ip_{\gr}^{12} \right) & \mbox{~if $\alpha=0$, $\beta=1$, $\gamma=3$},\\
0 & \mbox{~otherwise}.\\
\end{cases}
$$
Finally, for every bad edge $e$,
$$|\mathcal{Z}^2_{i,e,H_0}|=O\left(k_it_1p_{\gr}^{12}\right).$$
\item For every $J \in \{X,Y,X+Y,X-Y\}$ and every $J$-layer interval $S$ of size at least $t_0^{1-\epsilon/2}$,
$$|V^J_{O/E}(H_0[S])|= (1-C_{0}t_0^{-\epsilon/2})\sum_{v \in V^J_{O/E}(H[S])} (1-d_{w_H^o}(v)),$$
\item For every $J \in \{X, Y X+Y, X-Y\}$ and every $J$-layer interval $I^J$ of size at least $t_0^{1-\epsilon/2}$ and $v \notin J$ with $v \in V(H_0)$,
$$|E_{H_0}(v,I^J,O/E)|=(1 \pm C_{0}p_{\gr}^{-2}t_0^{-\epsilon/2})\sum_{e \in E_{H}(v,I^J,O/E)} \prod_{u \in e\setminus\{v\}} (1-d_{w_H^o}(u)).$$
\end{enumerate}
\end{lemma}
\begin{proof}
We consider the effect of the parity modification and greedy cover on Proposition \ref{prop_M_H}. In total, the process removed at most $2.5t_0^{1-\epsilon}p_{\gr}$ vertices from each of the four parts. Any properties relating to vertices, edges and subsets all contained within $I_{t_{20}}$ remain the same as in $H^c$. Otherwise, for property (i), and every $0$-valid subset $S \subseteq \mathcal{T}$, we have from Corollary \ref{cor_M_H} that $|V(H^c[S])|=\Theta\left(t_0p_{\gr}\right)$, since $S \not\subseteq I_{t_{20}}$. Thus we have that 
$$|V(H^c[S])| \geq |V(H_0[S])| \geq |V(H^c[S])|-2.5t_0^{1-\epsilon}p_{\gr},$$
where $|V(H^c[S])|-2.5t_0^{1-\epsilon}p_{\gr}= (1-O(t_0^{-\epsilon}))|V(H^c[S])|$.
Thus, $|V(H_0[S])|=(1 \pm O(t_0^{-\epsilon}))|V(H^c[S])|$. That is, there exists a constant $C_{0,1}$ such that 
$$|V(H_0[S])|=(1 \pm C_{0,1}t_0^{-\epsilon})\sum_{v \in V(H[S])} (1-d_{w_H^o}(v)),$$
for every $0$-valid subset $S$.

For (ii)-(v) we argue in a similar way. For (iv), we consider $J$-layer intervals of size at least $t_0^{1-\epsilon/2}$. We have from Corollary \ref{cor_M_H} that $|V^J_{o/e}(H^c[S])|\geq\Theta\left(t_0^{1-\epsilon/2}p_{\gr}\right)$. Hence $|V^J_{o/e}(H^c[S])|-2.5t_0^{1-\epsilon}p_{\gr}= (1-O(t_0^{-\epsilon/2}))|V^J_{o/e}(H^c[S])|$. So there exists a constant $C_{0,4}$ such that 
$$|V^J_{o/e}(H_0[S])|= (1-C_{0,4}t_0^{-\epsilon/2})\sum_{v \in V^J_{o/e}(H[S])} (1-d_{w_H^o}(v)),$$
for every $J$-layer interval of size at least $t_0^{1-\epsilon/2}$.

Now, for (ii), we have that for every open and closed $0$-valid tuple where $(S_1, S_2, S_3)$ are not contained in $I_{t_{20}}$, that $|S_1|=\Theta(t_0)$. By Corollary \ref{cor_M_H}, it follows that $|E_{H^c}(v,S_1,S_2,S_3)|=\Theta\left(|E_{\mathcal{T}}(v,S_1,S_2,S_3)|p_{\gr}^3\right)=\Theta\left(t_0p_{\gr}^3\right)$. Now, since $2.5t_0^{1-\epsilon}p_{\gr}$ vertices may have been removed in each part, and each such vertex could feasibly dictate one unique edge in $E_{H^c}(v,S_1,S_2,S_3)$, it follows that 
$$|E_{H^c}(v,S_1,S_2,S_3)| \geq |E_{H_0}(v,S_1,S_2,S_3)| \geq |E_{H^c}(v,S_1,S_2,S_3)|- 10t_0^{1-\epsilon}p_{\gr}.$$
From above, we have that $|E_{H^c}(v,S_1,S_2,S_3)|- 10t_0^{1-\epsilon}p_{\gr} = (1-O(p_{\gr}^{-2}t_0^{-\epsilon}))$. It follows that there exists an absolute constant $C_{0,2}$ such that 
$$|E_{H_0}(v,S_1,S_2,S_3)|=(1 \pm C_{0,2}p_{\gr}^{-2}t_0^{-\epsilon})\sum_{e \in E_{H}(v,S_1,S_2,S_3)} \prod_{u \in e\setminus\{v\}} (1-d_{w_H^o}(u)),$$
for every $0$-valid tuple $(v,S_1,S_2,S_3)$. Similarly, for an edge allowable function $f$ where $f(e)$ has the same order for every $e \in H$, it is clear that the relative impact has the same order.

The same argument holds for (v), however, again an $J$-layer interval may have size $t_0^{1-\epsilon/2}$, and hence as for (iv), we have that 
$$|E_{H_0}(v,I^J,O/E)|=(1 \pm C_{0,5}p_{\gr}^{-2}t_0^{-\epsilon/2})\sum_{e \in E_{H}(v,I^J,O/E)} \prod_{u \in e\setminus\{v\}} (1-d_{w_H^o}(u)),$$
for $I^J$ an $J$-layer interval, $v \notin J$, and $C_{0,5}$ some absolute constant.

Finally, for (iii), note that every additional vertex removed from $H^c$ to $H_0$ is in at most $O(j_i)$ relevant zero-sum configurations for an $i$-bad edge $e$ and every $i \in [c_g]$, and $O(t_1)$ configurations for an edge of type $(\alpha, \beta, 0)_i$ with $\alpha \neq 0$ and $i \in [c_g]$. 
Additionally, using fact \ref{fact_Z} and Theorem \ref{thm_H}, we have that these types of edge are in $\Theta\left(j_it_0p_{\gr}^{12}\right)$ and $\Theta\left(t_0^2p_{\gr}^{12}\right)$ relevant configurations respectively. Thus in total the at most $10t_0^{1-\epsilon}p_{\gr}$ vertices which have been removed to reach $H_0$ remove at most $O(j_it_0^{1-\epsilon}p_{\gr})$ and $O(t_1t_0^{1-\epsilon}p_{\gr})$ configurations, respectively, and the result follows. 
\end{proof}

\subsection{Reweighting}

We now wish to update the weighting $w_H$ for edges remaining in $H_0$. The strategy for this is very similar to the strategy for reweighting for subsequent steps described in Section \ref{sec_i}, however we only require one intermediate step, rather than two. We first define a new weighting $w_{H.0}$ as follows:
$$w_{H.0}(e):= \frac{w_H(e)}{\prod_{v \in e \cap I_{t_0}} (1-d_{w_H^o}(v))},$$
for every $e \in E(H)$. Supposing now that $d_{w_{H.0}, H_0}(v) >1$ for some $v \in V(H_0)$, let $d_H:= \max_{v \in V(H_0)} d_{w_{H.0}, H_0}(v)$. Then define
$$w_{0}(e):=\frac{w_{H.0}(e)}{d_H},$$
for every $e \in E(H)$. 

\begin{prop}
$w_{H.0}$ is edge allowable for $(H, w_H, t_0, \eta)$.
\end{prop}
\begin{proof}
Recall that $1 \geq 1-d_{w_H^o}(v) \geq 1/4$ for every $v \in V(H[I_{t_0}])$. Thus $w_{H.0}(e)$ has the same order as $w_H(e)$ for each $e \in E(H)$. In particular we have that $w_{H.0}(e)=\Theta\left(t_0p_{\gr}^3\right)$ for every $e \in E(H)$. Thus by Remark \ref{rem_allow} it follows that $w_{H.0}$ is edge allowable for $(H, w_H, t_0, \eta)$.
\end{proof}

\begin{prop} \label{prop_d_H}
$d_H \leq 1 + 1.1C_{0}p_{\gr}^{-2}t_0^{-\epsilon}$.
\end{prop}
\begin{proof}
Since $w_{H.0}$ is edge allowable for $(H, w_H, t_0, \eta)$, using Lemma \ref{lemma_key_props_small}, we have that
\begin{multline*}
d_{w_{H.0}, H_0}(v)= \\
w_{H.0}(E_{H_0}(v, I_{t_0}))=(1 \pm C_{0}p_{\gr}^{-2}t_0^{-\epsilon})\sum_{e \in E_H(v, I_{t_0})} w_{H.0}(e)\prod_{u \in e \setminus \{v\}} (1-d_{w_H^o}(u)),
\end{multline*}
for every $v \in V(H_0)$. 
Note that 
$$\sum_{e \in E_H(v, I_{t_0})} w_{H.0}(e)\prod_{u \in e \setminus \{v\}} (1-d_{w_H^o}(u)) = \sum_{e \in E_H(v, I_{t_0})} \frac{w_H(e)}{1-d_{w_H^o}(v)} = \frac{w_{H}(E_H(v, I_{t_0}))}{1-d_{w_H^o}(v)}.$$ 
By Proposition \ref{prop_H} we have that  $1-d_{w_H^o}(v)=(1 \pm 6.1 \alpha_G)w_{H}(E_H(v, I_{t_0}))$ and hence $\sum_{e \in E_H(v, I_{t_0})} w_{H.0}(e)\prod_{u \in e \setminus \{v\}} (1-d_{w_H^o}(u)) = 1 \pm 12.3\alpha_G$. This yields that $d_{w_{H.0}, H_0}(v)=1 \pm 1.1C_{0}p_{\gr}^{-2}t_0^{-\epsilon}$ for every $v \in V(H_0)$ and so in particular $d_H \leq 1 + 1.1C_{0}p_{\gr}^{-2}t_0^{-\epsilon}$ as required.
\end{proof}

\begin{cor}
$w_0$ is edge allowable for $(H, w_H, t_0, \eta)$. 
\end{cor}
\begin{proof}
The proof follows from Proposition \ref{prop_d_H} in the same way as Corollary \ref{cor_i.1_allow} follows from Proposition \ref{prop_d*_i}.
\end{proof}

\begin{prop}\label{prop_w_0}
We have that
$$w_{0}(e)=(1 \pm 1.1C_{0}p_{\gr}^{-2}t_0^{-\epsilon})w_{H.0}(e)$$ 
and
$$w_{H.0}(e)=(1 \pm 1.1C_{0}p_{\gr}^{-2}t_0^{-\epsilon})w_{0}(e)$$ 
for every $e \in H$. Furthermore, for every $e \in H_0$ we have that
$$(1 - 1.1C_{0}p_{\gr}^{-2}t_0^{-\epsilon})w_H(e) \leq w_0(e)\leq 82w_H(e).$$
\end{prop}
\begin{proof}
Now, since by Proposition \ref{prop_d_H}, $d_{w_{H.0}, H_0}(v) \leq 1 + 1.1C_{0}p_{\gr}^{-2}t_0^{-\epsilon}$ for every $v \in V(H_0)$, we have that
$$w_{H.0}(e) \geq w_0(e) \geq \frac{w_{H.0}(e)}{1 + 1.1C_{0}p_{\gr}^{-2}t_0^{-\epsilon}},$$
for every edge $e \in H_0$.
In particular, this gives the first two statements. By Proposition \ref{prop_H} we have that $1 \geq (1-d_{w_H^o}(u)) \geq \frac{1}{3}-O(\alpha_G)$ for every $u \in V(H_0)$. Hence $w_H(e) \leq w_{H.0}(e) \leq 81.1w_H(e)$. It then follows from the second statement of the proposition that
$$(1 - 1.1C_{0}p_{\gr}^{-2}t_0^{-\epsilon})w_H(e) \leq w_0(e)\leq 82w_H(e)$$ 
for every $e \in H_0$, as claimed.
\end{proof}

\begin{prop} \label{prop_deg_rel} 
$w_0$ is a fractional matching for $H_0$ such that 
$$d_{w_0, H_0}(v) \geq 1 -2.2C_{0}p_{\gr}^{-2}t_0^{-\epsilon}$$ 
for every $v \in V(H_0)$. Furthermore, given a $0$-valid tuple $(v, S_1, S_2, S_3)$, we have that
$$w_0(E_{H_0}(v, S_1,S_2,S_3))=(1 \pm 2.2C_{0}p_{\gr}^{-2}t_0^{-\epsilon})\frac{w_H(E_{H}(v, S_1,S_2,S_3))}{w_H(E_H(v, I_{t_0}))}.$$
\end{prop}
\begin{proof}
That $w_0$ is a fractional matching for $H_0$ follows from the construction of $w_0$. Considering a $0$-valid tuple $(v, S_1, S_2, S_3)$ we have that 
$$w_0(E_{H_0}(v, S_1,S_2,S_3)) = (1 \pm 1.1C_{0}p_{\gr}^{-2}t_0^{-\epsilon})\sum_{e \in E_{H_0}(v, S_1,S_2,S_3)} w_{H.0}(e)$$
by Proposition \ref{prop_w_0}. Then since $w_{H.0}$ is edge allowable for $(H, w_H, t_0, \eta)$ we have by Lemma \ref{lemma_key_props_small}, that 
$$\sum_{e \in E_{H_0}(v, S_1,S_2,S_3)} w_{H.0}(e)=(1 \pm C_{0}p_{\gr}^{-2}t_0^{-\epsilon})\sum_{e \in E_{H}(v, S_1,S_2,S_3)}w_{H.0}(e) \prod_{u \in e\setminus \{v\}} (1-d_{w_H^o}(u)).$$
Since for $v \ni e$, and $e \in H[I_{t_0}]$ we have that $w_{H.0}(e) \prod_{u \in e\setminus \{v\}} (1-d_{w_H^o}(u)) = \frac{w_H(e)}{1-d_{w_H^o}(v)}$. Using Proposition \ref{prop_H}~(iv) we get that 
\begin{multline*}
w_0(E_{H_0}(v, S_1,S_2,S_3))= \\(1 \pm 1.1C_{0}p_{\gr}^{-2}t_0^{-\epsilon})(1 \pm C_{0}p_{\gr}^{-2}t_0^{-\epsilon})(1 \pm 6.1\alpha_G)\frac{w_H(E_{H}(v, S_1,S_2,S_3))}{w_H(v, I_{t_0})},\end{multline*}
and the final claim follows. In particular, this gives that 
$$w_0(E_{H_0}(v, I_{t_0}))=(1 \pm 2.2C_{0}p_{\gr}^{-2}t_0^{-\epsilon})\frac{w_H(v, I_{t_0})}{w_H(v, I_{t_0})},$$
that is, $d_{w_0, H_0}(v)=w_0(E_{H_0}(v, I_{t_0}))\geq 1 - 2.2C_{0}p_{\gr}^{-2}t_0^{-\epsilon}$, as required.
\end{proof}

In particular, we note that every edge $e$ has $w_0(e)=\Theta(w_H(e))$, and so it follows that for every $v$, $d_{H_0}(v)=\Theta(d_H(v))$, since $d_{w_0, H_0}(v)=(1 \pm o(1))d_{w_H, H}(v)$ (as $w_H$ and $w_0$ are both almost-perfect fractional matchings for $H$ and $H_0$ respectively). 

\subsection{$H_0$ to $H_1$}

We now repeat the process above, this time with $(H_0, w_0)$ in place of $(H, w_H)$, and $I_{t_1}$ as the new `target interval' in place of $I_{t_0}$. We start with some useful properties to note for this process.

\begin{prop} \label{prop_w_0_lb}
For each $v \in V(H_0)$, $w_0(E_{H_0}(v, I_{t_1}))\geq (1-O(p_{\gr}^{-2}t_0^{-\epsilon}))\frac{1}{20}$.
\end{prop}
\begin{proof}
We have by Proposition \ref{prop_deg_rel} that for each $v \in V(H_0)$,
$$w_0(E_{H_0}(v, I_{t_1}))=(1 \pm 2.2C_{0}p_{\gr}^{-2}t_0^{-\epsilon})\frac{w_H(E_{H}(v, I_{t_1}))}{w_H(E_H(v, I_{t_0}))}.$$
Now, from properties of $\mathcal{T}$, in particular Fact \ref{fact_basic2}, and Theorem \ref{thm_H}, we know that $w_H(E_H(v, I_{t_0})) \leq \frac{2}{3}+O(\alpha_G)$, and that $w_H(E_{H}(v, I_{t_1}))\geq \frac{1}{30}-O(\alpha_G)$. Thus, the result follows.
\end{proof}
  
\begin{cor}\label{cor_w_0_relate}
For each $v \in V(H_0[I_{t_1}])$, we have that
$$1-d_{w_0^o}(v)=(1 \pm 44.1C_{0}p_{\gr}^{-2}t_0^{-\epsilon})w_0(E_{H_0}(v, I_{t_1})).$$
\end{cor}
\begin{proof}
By Proposition \ref{prop_w_0_lb}, we have that $w_0(E_{H_0}(v, I_{t_1})) \geq (1-o(1))\frac{1}{20}$. Furthermore, we have that for each $v \in V(H_0[I_{t_1}])$ that $d_{w_0^o}(v)+w_0(E_{H_0}(v,I_{t_1}))=1 \pm  2.2C_{0}p_{\gr}^{-2}t_0^{-\epsilon}$. Thus,
$$1-d_{w_0^o}(v)=w_0(E_{H_0}(v,I_{t_1})) \pm  2.2C_{0}p_{\gr}^{-2}t_0^{-\epsilon} = (1 \pm 44.1C_{0}p_{\gr}^{-2}t_0^{-\epsilon})w_0(E_{H_0}(v,I_{t_1})),$$
as required.
\end{proof}

\begin{prop} \label{prop_M_H_0}
There exists a matching $M^o_0$ in $(H^o_0, w_0^o)$ such that letting $H^c_0:=H_0[V(H_0) \setminus V(M_0^o)]$, we have that
\begin{enumerate}[(i)]
\item every $0$-valid subset $S \subseteq I_{t_1}$ satisfies 
$$|V(H^c_0[S])|=(1 \pm O(t_0^{-\epsilon}))\sum_{v \in V(H_0[S])} (1-d_{w_0^o}(v)),$$
and for any vertex allowable function $p_S(v):V(H_0) \rightarrow \mathbb{R}_{\geq 0}$ for $(H_0, w_0, t_0, \eta)$ such that $p_S(v)=f_S(v)\mathbbm{1}_{v \in V(H_0[S])}$, 
$$\sum_{v \in V(H_0^c[S])} p_S(v)=(1 \pm O(t_0^{-\epsilon}))\sum_{v \in V(H_0[S])} p_S(v)(1-d_{w_{H_0^o}}(v)),$$
\item every open and closed $\mathcal{T}$-valid tuple $(v,S_1,S_2,S_3)$ with $S_1, S_2, S_3 \subseteq I_{t_1}$ satisfies 
$$|E_{H^c_0}(v, S_1,S_2,S_3)|=(1 \pm O(t_0^{-\epsilon}))\sum_{e \in E_{H_0}(v, S_1,S_2,S_3)} \prod_{u \in e \setminus \{v\}} (1-d_{w_0^o}(u)),$$
and for any $v$-edge allowable function $f_v:E(H_0) \rightarrow \mathbb{R}_{\geq 0}$ for $(H, w_H, t_0, \eta)$ such that $f_v(e)=0$ wherever $v \notin e$, 
$$f_v(E_{H_0^c}(v, S_1,S_2,S_3))=(1 \pm O(t_0^{-\epsilon}))\sum_{e \in E_{H_0}(v, S_1,S_2,S_3)} f_v(e)\prod_{u \in e \setminus \{v\}} (1-d_{w_{H_0^o}}(u)),$$
\item for every $i \in [c_g]$, 
$$|\mathcal{Z}^+_{i,e,H_0^c}(\alpha, \beta, \gamma)|:=$$
$$
\begin{cases}
(1 \pm O(t_0^{-\epsilon}))\sum_{z \in \mathcal{Z}^{+}_{i,e,H_0}(\alpha, \beta, \gamma)} \prod_{u \in z\setminus e} (1-d_{w_0^o}(u)) & \mbox{~if $e$ is a bad edge}, \\
O\left(k_it_1p_{\gr}^{12} \right) & \mbox{~if $\alpha=0$, $\beta=1$, $\gamma=3$},\\
0 & \mbox{~otherwise}.\\
\end{cases}
$$
$$|\mathcal{Z}^-_{i,e,H_0^c}(\alpha, \beta, \gamma)|:=$$
$$
\begin{cases}
(1 \pm O(t_0^{-\epsilon}))\sum_{z \in \mathcal{Z}^{-}_{i,e,H_0}(\alpha, \beta, \gamma)} \prod_{u \in z\setminus e} (1-d_{w_0^o}(u)) & \mbox{~if $\alpha \neq 0$ and $\gamma=0$}, \\
O\left(j_ik_ip_{\gr}^{12} \right) & \mbox{~if $\alpha=0$, $\beta=0$, $\gamma=4$},\\
O\left(j_ik_ip_{\gr}^{12} \right) & \mbox{~if $\alpha=0$, $\beta=1$, $\gamma=3$},\\
0 & \mbox{~otherwise}.\\
\end{cases}
$$
Finally, for every bad edge $e$,
$$|\mathcal{Z}^2_{i,e,H_0^c}|=O\left(k_it_1p_{\gr}^{12}\right).$$
\item For every $J \in \{X,Y,X+Y,X-Y\}$ and every $J$-layer interval $S \subseteq I_{t_1}$ of size at least $t_0^{1-\epsilon/2}$,
$$|V_{O/E}^J(H^c_0[S])|=(1 \pm O(t_0^{-\epsilon}))\sum_{v \in V_{O/E}^J(H_0[S])} (1-d_{w_0^o}(v)),$$
\item For every $J \in \{X, Y X+Y, X-Y\}$ and every $J$-layer interval $I^J$ contained in $I_{t_1}$ of size at least $t_0^{1-\epsilon/2}$ and $v \notin J$ with $v \in V(H^c_0)$,
$$|E_{H^c_0}(v, I^J, O/E)|=(1 \pm O(t_0^{-\epsilon}))\sum_{e \in E_{H_0}(v, I^J, O/E)} \prod_{u \in e\setminus\{v\}} (1-d_{w_0^o}(u)).$$
\item When $n$ is odd, for each $AB \in \{OE, EO\}$,
$$|M_0^o \cap E_{AB}(H_0^o[I_{t_0}\setminus I_{t_{20}}])|=(1 \pm t_0^{-\epsilon})\sum_{e \in E_{AB}(H_0^o[I_{t_0}\setminus I_{t_{20}}])}w_0(e).$$ 
\end{enumerate} 
\end{prop}  
\begin{proof}
The proof of (i)-(v) again follows the same strategy as the proof of Theorem \ref{thm_key}, using Proposition \ref{prop_w_0_lb} in place of Corollary \ref{cor_bound} to lower bound $1-d_{w_0^o}(u)$ for each $u \in I_{t_1}$. For (vi) we appeal directly to Theorem \ref{thm_weighted_egj}. By the same reasoning as for Theorem \ref{thm_key}, we know that the hypotheses are satisfied. Let $q_{AB}:E(H_0^o)\rightarrow \mathbb{R}_{\geq 0}$ be defined by $q_{AB}(e)=\mathbbm{1}_{e \in E_{AB}(H_0^o[I_{t_0}\setminus I_{t_{20}}])}$. Then $q_{AB}(E(H_0^o))=|E_{AB}(H_0^o[I_{t_0}\setminus I_{t_{20}}])|$ and $\max_{e \in E(H_0^o)} q(e)=1$ for every $AB \in \{OE, EO\}$. It suffices to show that $|E_{AB}(H_0^o[I_{t_0}\setminus I_{t_{20}}])| \geq \frac{t_0^{1+\eta}}{1-t_0^{-1}}$ for each $AB \in \{OE, EO\}$. First note that by Fact \ref{fact_basic2} (iii), for each $v \in (I_{t_0}\setminus I_{t_1}) \cap (X \cup Y)$, there are at least $n/50$ wrap-around edges of each relevant type which contain only other vertices in $I_{t_1}\setminus I_{t_{20}}$. That is, $|E_{\mathcal{T}}(v, I^*_v, O/E)|\geq \frac{t_0}{50}$ for each $v \in (V(\mathcal{T})\setminus I_{t_1}) \cap (X \cup Y)$, where $I^*_v$ is the relevant layer interval that counts only edges with all other vertices in $I_{t_1}\setminus I_{t_{20}}$. Then by Lemma \ref{lemma_key_props_small}(v), there exists a constant $c_1$ such that $|E_{H_0^o}(v, I^*_v, O/E)|=|E_{H_0}(v, I^*_v, O/E)|\geq c_1t_0p_{\gr}^3$ for every $v \in (V(H_0)\setminus I_{t_1}) \cap X$. Furthermore, by Lemma \ref{lemma_key_props_small}(iv), there exists a constant $c_2$ such that $|(V(H_0^o)\setminus I_{t_1}) \cap X|=|(V(H_0)\setminus I_{t_1}) \cap X|\geq c_2t_0p_{\gr}$. Then 
\begin{equation}\label{eq_AB}
|E_{AB}(H_0^o[I_{t_0}\setminus I_{t_{20}}])| \geq \sum_{v \in (V(H_0^o)\setminus I_{t_1}) \cap X} |E_{H_0^o}(v, I^*_v, O/E)| \geq c_1c_2t_0^2p_{\gr}^4,
\end{equation}
where $c_1c_2t_0^2p_{\gr}^4 \gg \frac{t_0^{1+\eta}}{1-t_0^{-1}}$, as required to satisfy (\ref{eq_edge_egj}). Thus by Theorem \ref{thm_weighted_egj} there exists a matching $M_0^o$ such that (as well as (i)-(v)), we have that 
\begin{eqnarray*}
|M_0^o \cap E_{AB}(H_0^o[I_{t_0}\setminus I_{t_{20}}])| &=& q_{AB}(M_0^o)=(1 \pm t_0^{-\epsilon})\sum_{e \in E(H)} q_{AB}(e)w(e) \\
&=& (1 \pm t_0^{-\epsilon})\sum_{e \in E_{AB}(H_0^o[I_{t_0}\setminus I_{t_{20}}])}w_0(e)
\end{eqnarray*} 
for each $AB \in \{OE, EO\}$, as claimed.
\end{proof}

Now in $H_0$ we still have wrap-around edges, and thus for $n$ odd we still have additional parity requirements to adjust for depending on the types of edges in $M^o_0$.

\subsection{Parity to reach $H_1$} \label{sec_parity_H_1}

To adjust for any parity disparity in $H^c$, we were able to simply balance out the disparity by {\it adding} wrap-around edges of appropriate types. Now to adjust for any disparity in $H^c_0$ (again presuming $n$ is odd) we cannot proceed in this way since vertices remaining to be covered may not have sufficiently large wrap around degree in $I_{t_1}$. We now, however, have that every vertex should have sufficiently large non-wrap around degree into $I_{t_1}$. (This was not the case for the first step, hence why we do it differently in each step.) Note, also, that any parity disparity can only be caused by the existence of edges of the relevant wrap-around type appearing in $M^o_0$, thus we can try to fix the disparity by first {\it removing} sufficiently many edges from $M^o_0$ to reduce the disparity to $0$.

\begin{prop} \label{prop_parity_0}
There exists a constant $C'_{0}$ such that the parity disparity of $H_0^c$, $P_0^c$, satisfies
$$P_0^c \leq C'_{0}p_{\gr}t_0^{1-\epsilon/2}.$$
\end{prop}
\begin{proof}
We have from (\ref{eq_func_vert}) that
$$|V^{X\pm Y}_{O/E}(H_0^c)| = \sum_{v \in V^{X\pm Y}_{O/E}(H_0)} 1-d_{w_0^o}(v) \pm t_0^{\epsilon}\sum_{v \in V^{X\pm Y}_{O/E}(H_0)} d_{w_0^o}(v).$$
We split the sum to consider vertices in $I_{t_1}$ and $I_{t_0}\setminus I_{t_1}$ separately. In particular, we know for $v \in I_{t_0}\setminus I_{t_1}$ that $1-d_{w_0^o}(v)\leq 2.2p_{\gr}^{-2}t_0^{-\epsilon}$, and furthermore, that $|V^{X\pm Y}_{O/E}(H_0[I_{t_0} \setminus I_{t_1}])|\leq |V^{X\pm Y}_{O/E}(H[I_{t_0} \setminus I_{t_1}])| (1 +o(1))\frac{t_0p_{\gr}}{5}$. Then since $\sum_{v \in V^{X\pm Y}_{O/E}(H_0)} d_{w_0^o}(v)=O(t_0p_{\gr})$, the difference in remaining vertices of odd parity outside $I_{t_1}$ over both the $X+Y$ and $X-Y$ parts is at most 
$$O(t_0^{1-\epsilon}p_{\gr})+2.2p_{\gr}^{-2}t_0^{-\epsilon}(1 +o(1))\frac{4t_0p_{\gr}}{5} \leq 2p_{\gr}^{-1}t_0^{1-\epsilon}.$$
Now, to consider the maximum disparity inside $I_{t_1}$, recall from Corollary \ref{cor_w_0_relate} that for each $v \in V(H_0[I_{t_1}])$, 
$$1-d_{w_0^o}(v)=(1 \pm 44.1C_{0}p_{\gr}^{-2}t_0^{-\epsilon})w_0(E_{H_0}(v, I_{t_1})),$$
and by Proposition \ref{prop_deg_rel}, that
$$w_0(E_{H_0}(v, I_{t_1})) = (1 \pm 2.2C_{0}p_{\gr}^{-2}t_0^{-\epsilon})\frac{w_H(E_H(v, I_{t_1}))}{w_H(E_H(v, I_{t_0}))}.$$
Then
\begin{eqnarray*}
|V^{X\pm Y}_{O/E}(H_0^c[I_{t_1}])| &=& (1 \pm t_0^{-\epsilon})\sum_{v \in V^{X\pm Y}_{O/E}(H_0[I_{t_1}])} 1-d_{w_0^o}(v) \\
&=&(1 \pm 46.3C_{0}p_{\gr}^{-2}t_0^{-\epsilon})\sum_{v \in V^{X\pm Y}_{O/E}(H_0[I_{t_1}])} \frac{w_H(E_H(v, I_{t_1}))}{w_H(E_H(v, I_{t_0}))} \\
&=& (1 \pm 46.4C_{0}p_{\gr}^{-2}t_0^{-\epsilon})\sum_{v \in V^{X\pm Y}_{O/E}(H_0[I_{t_1}])} \frac{|E_{\mathcal{T}}(v, I_{t_1})|}{|E_{\mathcal{T}}(v, I_{t_0})|},
\end{eqnarray*}
where the last equality holds using Theorem \ref{thm_H}, and the fact that $w_H$ is a uniform weight function.
Now we may proceed as in Proposition \ref{prop_parity_H}, splitting the summation across consecutive intervals, this time of size $t_0^{1-\epsilon/2}$. Letting $\mathcal{I}$ be a partition of $I_{t_1} \cap X+Y$ into intervals of size $t_0^{1-\epsilon/2}$, then for $v_I$ the vertex in the middle of $I \in \mathcal{I}$, let $d^i_I=|E_{\mathcal{T}}(v, I_{t_i})|$ for $i \in \{0,1\}$. Then we have that $|E_{\mathcal{T}}(u, I_{t_i})|=d_I^i \pm O(t_0^{1-\epsilon/2})$ for every $u \in I$ and $i \in \{0,1\}$. Noting that for every vertex $u \in I_{t_1}$ and $i \in \{0,1\}$ we have that $|E_{\mathcal{T}}(u, I_{t_i})|=\Theta(t_0)$ we thus find that
\begin{eqnarray*}
\sum_{v \in V^{X+Y}_{O}(H_0[I_{t_1}])} \frac{|E_{\mathcal{T}}(v, I_{t_1})|}{|E_{\mathcal{T}}(v, I_{t_0})|} &=& (1 \pm O(t_0^{-\epsilon/2})) \sum_{I \in \mathcal{I}} \sum_{v \in V^{X+Y}_{O}(H_0[I])} \frac{d^1_I}{d^0_I} \\
&=& (1 \pm O(t_0^{-\epsilon/2})) \sum_{I \in \mathcal{I}} \frac{d^1_I}{d^0_I}|V^{X+Y}_{O}(H_0[I])|,
\end{eqnarray*}
where, by Claim \ref{claim_parity} and Lemma \ref{lemma_key_props_small}(iv), we have that $|V^{X+Y}_{O}(H_0[I])|=(1 \pm O(t_0^{-\epsilon/2}))|V^{X-Y}_{O}(H_0[I])|$. From this it follows that
\begin{eqnarray*}
\sum_{v \in V^{X+Y}_{O}(H_0[I_{t_1}])} \frac{|E_{\mathcal{T}}(v, I_{t_1})|}{|E_{\mathcal{T}}(v, I_{t_0})|}
&=& (1 \pm O(t_0^{-\epsilon/2})) \sum_{I \in \mathcal{I}} \frac{d^1_I}{d^0_I}|V^{X-Y}_{O}(H_0[I])| \\
&=& (1 \pm O(t_0^{-\epsilon/2})) \sum_{I \in \mathcal{I}} \sum_{v \in V^{X-Y}_{O}(H_0[I])} \frac{d^1_I}{d^0_I} \\
&=& (1 \pm O(t_0^{-\epsilon/2}))\sum_{v \in V^{X-Y}_{O}(H_0[I_{t_1}])} \frac{|E_{\mathcal{T}}(v, I_{t_1})|}{|E_{\mathcal{T}}(v, I_{t_0})|}.
\end{eqnarray*}
In particular, then, we have that 
$|V^{X+Y}_{O}(H_0^c[I_{t_1}])|=(1 \pm O(t_0^{-\epsilon/2}))|V^{X-Y}_{O}(H_0^c[I_{t_1}])|$. Hence, as $|V^{X+Y}_{O}(H_0^c[I_{t_1}])|=O(t_0p_{\gr})$ the parity disparity is at most $O(t_0^{1-\epsilon/2}p_{\gr})$, since $t_0^{-\epsilon/2} \ll p_{gr}^2$. That is there exists a constant $C'_{0}$ such that $P_0^c \leq C'_{0}p_{\gr}t_0^{1-\epsilon/2}$, as required.
\end{proof}

Now, as noted before Proposition \ref{prop_parity_0}, we can only obtain such disparity from such edges which we wish to balance out having been used in the matching $M_0^o$. 

\begin{prop} \label{prop_wrap_edge}
There exists a constant $c$ such that at least $ct_0p_{\gr}$ edges of $M_0^o$ are wrap-around edges of each type which avoid $I_{t_{20}}$.
\end{prop}
\begin{proof}
By Proposition \ref{prop_M_H_0}~(vi) 
we have that the number of wrap-around edges of each type used in $M_0^o$ is
$$|M_0^o \cap E_{AB}(H_0^o[I_{t_0}\setminus I_{t_{20}}])|=(1 \pm t_0^{-\epsilon})\sum_{e \in E_{AB}(H_0^o[I_{t_0}\setminus I_{t_{20}}])}w_0(e),$$
for $AB \in \{OE, EO\}$. 
Recall that by Proposition \ref{prop_w_0}, and since $w_H(e)=(1 \pm O(\alpha_G))\frac{1}{np_{\gr}^3}$, we have that $w_0(e) \geq \frac{1}{2np_{\gr}^3}$ for every $e \in E_{AB}(H_0^o[I_{t_0}\setminus I_{t_{20}}])$ and each $AB \in \{OE, EO\}$. Furthermore by (\ref{eq_AB}) we have a constant $c'$ such that $E_{AB}(H_0^o[I_{t_0}\setminus I_{t_{20}}])|\geq c't_0^2p_{\gr}^4$. Hence it follows that
$$|M_0^o \cap E_{AB}(H_0^o[I_{t_0}\setminus I_{t_{20}}])|\geq \frac{c'}{4}t_0p_{\gr},$$
so the claim holds.
\end{proof}

Since $P_0^c = O(p_{\gr}t_0^{1-\epsilon/2})$ and $M_0^o$ contains $\Theta\left(t_0p_{\gr}\right)$ of each wrap around parity avoiding $I_{t_{20}}$, we may greedily remove edges of the relevant wrap-around parity from $M_0^o$ and thus reduce $P_0$ to $0$, only affecting properties of vertices, edges and subsets which are not within $I_{t_{20}}$. Let $M_0' \subseteq M_0^o$ be the matching obtained in this way. We will then greedily cover all remaining vertices outside the target interval, including those which are now uncovered as a result of the wrap-around edge removal, without reintroducing any parity disparity. Let $H_0':=H_0[V(H_0)\setminus V(M_0')]$. Note that any lower bounds for degree and interval properties of $H_0^c$ remain valid for $H_0'$, since we have only {\it added} edges back in to go from $H_0^c$ to $H_0'$.

We now run the greedy cover, ensuring that we don't introduce any parity problems. Let $U_0 = U_0^X \cup U_0^Y \cup U_0^{X+Y} \cup U_0^{X-Y}$ denote the set of uncovered vertices remaining outside the target interval.

\begin{prop} \label{prop_U_0}
The number of vertices remaining to cover in the greedy cover step, $U_0$, satisfies
$$|U_0| = O(p_{\gr}t_0^{1-\epsilon/2}).$$
\end{prop}
\begin{proof}
Since $P_0^c =O(p_{\gr}t_0^{1-\epsilon/2})$, we have added at most $O(p_{\gr}t_0^{1-\epsilon/2})$ additional vertices to $U_0$ than those which already needed covering after removing $M_0^o$.
We have from (\ref{eq_func_vert}) that 
$$|V(H_0^c[I_{t_0}\setminus I_{t_1}])|=\sum_{v \in V(H_0[I_{t_0}\setminus I_{t_1}])} (1-d_{w^o_0}(v)) \pm  t_0^{-\epsilon}\sum_{V(H_0[I_{t_0}\setminus I_{t_1}])} d_{w^o_0}(v),$$ and for every $v \in H_0[I_{t_0}\setminus I_{t_1}]$, we have that $1-d_{w^o_0}(v)=1-d_{w_0, H_0}(v) \leq 2.2p_{\gr}^{-2}t_0^{-\epsilon}$. Furthermore, we have that $|V(H_0[I_{t_0}\setminus I_{t_1}])| \leq \frac{4t_0p_{\gr}}{3}$. It follows that $|V(H_0^c[I_{t_0}\setminus I_{t_1}])|=O(p_{\gr}^{-1}t_0^{1-\epsilon}),$ and thus since $t_0^{-\epsilon/2}\ll p_{\gr}^2$, we have that
$|U_0| = O(p_{\gr}t_0^{1-\epsilon/2}),$ as required.
\end{proof}

\begin{prop} \label{prop_greedy_cover}
There exists an absolute constant $c>0$ such that every vertex $v \in U_0$ is contained in at least $ct_0p_{\gr}^3$ edges in $H_0'$ which avoid vertices in $I_{t_{20}}$ and do not wrap-around.
\end{prop}
\begin{proof}
In $\mathcal{T}$, as per Fact \ref{fact_basic2}(iv), such a vertex is in at least $\frac{t_0}{150}$ such edges. Furthermore, for each such vertex $v$, one can lower bound the collection of non-wrap around edges containing $v$ via $E_{H_0'}(v, I^J, O) \cup E_{H_0'}(v, I^J, E)$ for some $J$-layer interval of size $\Theta(t_0)$ with $J$ such that $v \notin J$. Thus by Theorem \ref{thm_H} we have that in $H$ the vertex $v$ is in at least $(1 - \alpha_G)\frac{t_0p_{\gr}^3}{150}$ such edges. Since $H_0' \supseteq H_0^c$, $1-d_{w^o_0}(u)\geq 1/4$ and $1-d_{w^o_H}(u) \geq 1/21$ for every $u \in I_{t_1}$, by Proposition \ref{prop_M_H_0} (v) and Lemma \ref{lemma_key_props_small} (v) the result follows. 
\end{proof}

By Proposition \ref{prop_greedy_cover} since $|U_0|=O\left(p_{\gr}t_0^{1-\epsilon/2}\right)$ we are able to cover all of the vertices in $U_0$ greedily, without causing any parity problems. Let $M_0^c$ be the matching obtained from such a greedy cover. 
Then setting 
$$M_0:=M_0' \cup M_0^c,$$ 
we let 
$$H_1:=H_0[V(H_0)\setminus V(M_0)].$$ 
Next we will define a weight function $w_1$ satisfying the requirements of Theorem \ref{thm_H_1}.
Let
$$w_{0.0}(e)=\frac{w_0(e)}{\prod_{u \in (e \cap I_{t_1})} 1-d_{w_0^o}(u)},$$
for every $e \in H_0$ and set 
$$w_1(e):=\frac{w_{0.0}(e)}{\max_{v \in V(H_1)} d_{w_{0.0}, H_1}(v)}.$$ 
It remains to prove that the statements of Theorem \ref{thm_H_1} indeed hold.

\begin{prop} \label{prop_step_1_apfm} 
$w_1$ is a fractional matching for $H_1$ such that $d_{w_1, H_1}(v) \geq 1-O(p_{\gr}^{-2}t_0^{-\epsilon/2})$ for every $v \in V(H_1)$. Furthermore, $w_1(e)=(1 \pm O(p_{\gr}^{-2}t_0^{-\epsilon/2}))w_{0.0}(e)$ for every $e \in H_0$.
\end{prop}
\begin{proof}
The proof follows via the same strategy used to prove Propositions \ref{prop_d_H} \ref{prop_w_0} and \ref{prop_deg_rel},  the facts about $w_0$ and $d_{w_0, H_0}(v)$. In particular, for each $v \in V(H_1)$, we have that 
$$d_{w_{0.0}, H_1}(v)=\sum_{e \in E_{H_1}(v, I_{t_1})} w_{0.0}(e) \leq \left(\sum_{e \in E_{H_0^c}(v, I_{t_1})} w_{0.0}(e)\right) + O(p_{\gr}^{-2}t_0^{-\epsilon/2}),$$ since by Proposition \ref{prop_parity_0} we have from $H_0^c$ to $H_1$ that we added at most $P_0^c=O(p_{\gr}t_0^{1-\epsilon/2})$ edges, and the weight on each such edge has order $\Theta(p_{\gr}^{-3}t_0^{-1})$. Then by Proposition \ref{prop_M_H_0} we have that 
$$\sum_{e \in E_{H_0^c}(v, I_{t_1})} w_{0.0}(e)=(1 \pm O(t_0^{-\epsilon}))\sum_{e \in E_{H_0}(v, I_{t_1})} w_{0.0}(e) \prod_{u \in e\setminus \{v\}} (1- d_{w_0^o}(u)),$$
where $\sum_{e \in E_{H_0}(v, I_{t_1})} w_{0.0}(e) \prod_{u \in e\setminus \{v\}} (1- d_{w_0^o}(u)) =\frac{w_0(E_{H_0}(v, I_{t_1}))}{1-d_{w_0^o}(v)}=1 \pm O(p_{\gr}^{-2}t_0^{-\epsilon})$, where the last equality holds by Corollary \ref{cor_w_0_relate}.  
So we find that $d_{w_{0.0}, H_1}(v) \leq 1 + O(p_{\gr}^{-2}t_0^{-\epsilon/2})$. It follows then that $w_1=(1 \pm O(p_{\gr}^{-2}t_0^{-\epsilon/2}))w_{0.0}(e)$. Now, that $w_1$ is a fractional matching for $H_1$ follows by construction. Furthermore, we have that $d_{w_{1}, H_1}(v)=\sum_{e \in E_{H_1}(v, I_{t_1})} w_{1}(e)=(1 \pm O(p_{\gr}^{-2}t_0^{-\epsilon/2}))\sum_{e \in E_{H_1}(v, I_{t_1})} w_{0.0}(e)$ so from the above, we have that $d_{w_{1}, H_1}(v) \geq 1 - O(p_{\gr}^{-2}t_0^{-\epsilon/2})$, for every $v \in V(H_1)$, as claimed.
\end{proof}

\begin{proof}[Proof of Theorem \ref{thm_H_1}]
We recall that $\epsilon_1=\epsilon/10$, and crucially that $t_1^{-\epsilon_1} \gg p_{\gr}^{-2}t_0^{-\epsilon/2}$. Thus the first statement holds by Proposition \ref{prop_step_1_apfm}. We also have that (iv) holds as a direct result of our parity modifications to obtain $(H_1, w_1)$. For (ii) we note that $w_1(e)=(1 \pm O(p_{\gr}^{-2}t_0^{-\epsilon/2})\frac{w_0(e)}{\prod_{u \in e} (1-d_{w_0^o}(u))}$ for every $e \in H_1$. Thus using the last statement from Proposition \ref{prop_w_0} it is clear that (ii) follows provided that $1-d_{w_0^o}(v)=\Theta(1)$ for every $v \in V(H_1)$. This follows from Proposition \ref{prop_w_0_lb} and Corollary \ref{cor_w_0_relate}: indeed, $1 \geq 1-d_{w_0^o}(v) \geq \frac{1}{21}$. Claims (iii), (iv) and (v) follow using this along with Proposition \ref{prop_M_H_0}, provided that the edges added and removed from $H_0^c$ to $H_1$ do not have a significant impact. Note that this is true by default for all subsets inside $I_{t_{20}}$. Any $1$-valid subsets containing vertices outside $I_{t_{20}}$ have size $\Theta(t_0)$, and this means, for (iv) and (v) that the counts related to the vertex subsets and degree-type properties of $H_0^c$ have size $\Theta(p_{\gr}t_0)$ and $\Theta(p_{\gr}^{3}t_0)$ respectively, and the impact of gaining or losing $O(p_{\gr}t_0^{1-\epsilon/2})$ is a $o(1)$ relative term to both. For (v), the upper bounds all hold by Theorem \ref{thm_H}, since $H_1 \subseteq H$. For the lower bound, we have that removing $O(p_{\gr}t_0^{1-\epsilon/2})$ vertices from $I_{t_1}\setminus I_{t_{20}}$ can only remove $O(p_{\gr}j_it_0^{1-\epsilon/2})$ $i$-legal zero-sum configurations containing a fixed bad edge $e$ and at most $O(p_{\gr}t_0^{2-\epsilon/2})$ containing a fixed edge of type $(\alpha, \beta,0)_i$ with $\alpha \neq 0$. In particular, in both cases this is a $O(p_{\gr}^{-11}t_0^{-\epsilon/2})=o(1)$ fraction of the total value, and so the lower bounds in $H_1$ are a constant proportion of those in $H$, as required.
\end{proof}

%% file: ch_classical.tex
\chapter{Classical queens and concluding remarks} \label{ch_classical}

\section{The classical $n$-queens problem}

We now turn to Theorem \ref{thm_class}, considering a lower bound for $Q(n)$ rather than $T(n)$. Throughout, we have considered subgraphs of $\mathcal{T}(n)$ for all $n$ sufficiently large though our main result, Theorem \ref{thm_main}, only concerns $n \equiv 1,5\mymod 6$. It is easy to miss why our proof counting the number of perfect matchings in $\mathcal{T}(n)$ only applies to the cases where $n \equiv 1,5\mymod 6$, since for most of the sub results leading to Theorem \ref{thm_main} this condition is not required. The key is that our count uses that (at least) one perfect matching exists in $\mathcal{T}(n)$, or rather that ${\bf 1} \in \mathcal{L}(\mathcal{T}(n))$, which we know is true when $n \equiv 1,5\mymod 6$ by P\'{o}lya's~\cite{polya} observations (or equivalently by Corollary \ref{cor_lattice_t}). However, when $n$ is divisible by $2$ or $3$, and so no perfect matching exists in $\mathcal{T}(n)$, we can modify our strategy to one that lower bounds perfect matchings for some $\mathcal{T}^*(n) \subseteq \mathcal{T}(n)$, where a perfect matching has size $n'$, such that $\mathcal{T}^*(n) \setminus \mathcal{T}(n)$ has a collection of $n-n'$ edges which amount to a collection of queens placed on the $n \times n$ board in such a way that, whilst they may attack toroidally, they do not attack classically. Then the union of a perfect matching in $\mathcal{T}^*(n)$ and the fixed collection of $n- n'$ edges in $\mathcal{T}(n) \setminus \mathcal{T}^*(n)$ translates to a placement of $n$ non-attacking queens on the $n \times n$ classical board, where the only toroidal attacks are among queens in positions corresponding to the edges used from $\mathcal{T}(n) \setminus \mathcal{T}^*(n)$. 

\begin{proof}[Proof of Theorem \ref{thm_class}]
Start by considering $V(\mathcal{T}(n))$ where we index the vertices in each part by $\{1, \ldots, n\}$. We split into cases based on divisibility of $n$. For all cases we take $a_i, b_i, c_i, d_i, x_i, y_i, w_i, z_i$ for $i \in [3]$ to be $24$ distinct elements of $[n]$ such that $a_i+b_i \in [n/2]$, $x_i+y_i=a_i+b_i+n$, $c_i-d_i \in [n/2]$, $w_i-z_i=c_i-d_i-n$ for every $i \in [3]$. Furthermore, we require that $a_i+b_i$, $c_i+d_i$ and $w_i+z_i$ are $9$ distinct elements $\mymod n$, and also that $a_i-b_i$, $x_i-y_i$, and $c_i-d_i$ are $9$ distinct elements $\mymod n$. Then we have that queens placed on the squares of the $n \times n$ chessboard associated with $(a_i, b_i)$ and $(x_i, y_i)$ attack toroidally but not classically in the $X+Y$ diagonal, and queens placed on $(c_i, d_i)$ and $(w_i, z_i)$ attack toroidally but not classically in the $X-Y$ diagonal for every $i \in [3]$, and there are no other attacks between the 12 queens placed on the toroidal $n \times n$ board. We split into three cases for divisibility of $n$.

First we consider when $n$ is even and $3~|~n$. We define $W \subseteq V(\mathcal{T}(n))$ as follows. As well as the conditions above on $\{a_i, b_i, x_i, y_i, c_i, d_i, w_i, z_i\}_{i \in [3]}$, we additionally require that 
$$2+a+2\sum_{i \in [3]} (a_i+b_i+c_i-d_i) \in 12 \mathbb{Z},$$
where $n \equiv a \mymod 12$. 
Then we let 
\begin{multline*}
W^+= \\ 
\{a_i^X, x_i^X, b_i^Y, y_i^X (a_i+b_i)^{X+Y}, (a_i+b_i+n/6)^{X+Y}, (a_i-b_i)^{X-Y}, (x_i-y_i)^{X-Y}\}_{i \in [3]}
\end{multline*}
and 
\begin{multline*}
W^-= \\
\{c_i^X, w_i^X, d_i^Y, z_i^X (c_i+d_i)^{X+Y}, (w_i+z_i)^{X+Y}, (c_i-d_i)^{X-Y}, (c_i-d_i+n/6)^{X-Y} \}_{i \in [3]}
\end{multline*} 
and let $W:=W^+ \cup W^-$. 

When $n$ is even and $3\not|~n$ we instead additionally ensure that 
$$2+a+2\sum_{i \in [3]}(a_i+b_i+c_i-d_i) \in 4\mathbb{Z}$$ 
where $n \equiv a \mymod 12$. Then we define $W:=W^+ \cup W^-$ via 
\begin{multline*}
W^+= \\
\{a_i^X, x_i^X, b_i^Y, y_i^X (a_i+b_i)^{X+Y}, (a_i+b_i+n/2)^{X+Y}, (a_i-b_i)^{X-Y}, (x_i-y_i)^{X-Y}\}_{i \in [3]}
\end{multline*}
and 
\begin{multline*}
W^-= \\
\{c_i^X, w_i^X, d_i^Y, z_i^X (c_i+d_i)^{X+Y}, (w_i+z_i)^{X+Y}, (c_i-d_i)^{X-Y}, (c_i-d_i+n/2)^{X-Y} \}_{i \in [3]}.
\end{multline*}

When $n$ is odd and $3~|~n$ our additional constraint is that 
$$1+2\sum_{i \in [3]}(a_i+b_i+c_i-d_i) \in 3\mathbb{Z},$$ and we define $W:=W^+ \cup W^-$ via 
\begin{multline*}
W^+= \\
\{a_i^X, x_i^X, b_i^Y, y_i^X (a_i+b_i)^{X+Y}, (a_i+b_i+n/3)^{X+Y}, (a_i-b_i)^{X-Y}, (x_i-y_i)^{X-Y}\}_{i \in [3]}
\end{multline*}
and 
\begin{multline*}
W^-= \\
\{c_i^X, w_i^X, d_i^Y, z_i^X (c_i+d_i)^{X+Y}, (w_i+z_i)^{X+Y}, (c_i-d_i)^{X-Y}, (c_i-d_i+n/3)^{X-Y} \}_{i \in [3]}.
\end{multline*}

Note that Theorem \ref{thm_class} holds for $n \equiv 1,5 \mod 6$ as an immediate corollary to Theorem \ref{thm_main}, so this covers all remaining cases. 
Now let $\mathcal{T}^*:=\mathcal{T}[V(\mathcal{T})\setminus W]$. We claim that $\mathcal{T}^*$ has at least $\left((1+o(1))\frac{n}{e^3}\right)^n$ perfect matchings, and that each of these extends to a distinct placement of $n$ non-attacking queens on the $n \times n$ classical board such that at most $6$ pairs of queens attack toroidally. Indeed, supposing that $\mathcal{T}^*$ has at least $\left((1+o(1))\frac{n}{e^3}\right)^n$ perfect matchings, this translates to a placement of $n-12$ queens on the $n \times n$ toroidal board so that no two queens can attack each other. Then adding the 12 queens dictated by $\{(a_i, b_i), (c_i, d_i), (x_i, y_i), (w_i, z_i)\}_{i \in [3]}$, these queens cannot attack any of the $n-12$ previously placed queens on the toroidal board (and therefore nor on the classical board), and these $12$, by construction, do not attack classically and divide into three pairs which attack toroidally on the $X+Y$ diagonal, and three pairs which attack toroidally on the $X-Y$ diagonal. Thus, to prove the theorem, it remains to show that $\mathcal{T}^*$ has at least $\left((1+o(1))\frac{n}{e^3}\right)^n$ perfect matchings. 

The proof of this is exactly the proof of the lower bound for Theorem \ref{thm_main}, that when $n \equiv 1,5 \mymod 6$, $\mathcal{T}(n)$ has at least $\left((1+o(1))\frac{n}{e^3}\right)^n$ perfect matchings, but starting from $\mathcal{T}^*(n)$ in place of $\mathcal{T}(n)$. Note that the constant number of vertices removed to obtain $\mathcal{T}^*$ from $\mathcal{T}$ neither affect parity issues, nor can they have a significant effect on the other properties we track throughout the process. The only aspect of the proof that needs reverifying for $\mathcal{T}^*(n)$ in place of $\mathcal{T}(n)$ is that ${\bf 1} \in \mathcal{L}(\mathcal{T}^*)$. In particular, this is the key element that ensures that $L^*$, what is left to be absorbed at the end of the process is a qualifying leave for the absorber $A^*$ taken out at the beginning. Let ${\bf v_{L^*}}$ be the support vector of $L^*$. Then we require that ${\bf v_{L^*}} \in \mathcal{L}(\mathcal{T}^*)$ to ensure that $L^*$ is a qualifying leave. Since $L^*$ is obtained by removing a matching from $\mathcal{T}^*(n)$, it follows that showing that ${\bf 1} \in \mathcal{L}(\mathcal{T}^*)$ will complete the proof. 
 
By Lemma \ref{lem_lattice_t} we have that the vector ${\bf v}$ corresponding to weight $1$ on all vertices in $V(\mathcal{T}) \setminus W$ and weight $0$ on vertices in $W$ satisfies ${\bf v} \in \mathcal{L}(\mathcal{T})$. (This is seen simply by verifying that ${\bf v}$ satisfies (i)-(iv) when $n$ is odd and $3~|~n$, and (a)-(d) when $n$ is even. We give more details of these calculations below the proof.) This in fact also implies that ${\bf 1} \in \mathcal{L}(\mathcal{T}^*)$ as follows. Consider an integer collection of edges $E_1 \subseteq \mathcal{T}$ whose vertex shadow yields the vector corresponding to ${\bf v}$. If $e \cap W = \emptyset$ for every $e \in E_1$ then $E_1 \subseteq \mathcal{T}^*$ so ${\bf 1} \in \mathcal{L}(\mathcal{T}^*)$. Supposing this is not the case,  
let $V^*(E_1):=\{v_1, \ldots, v_{\chi}\}$ be an enumeration of the vertices with multiplicity and sign $v \in W$ such that there exists $e \in E_1$ with $e \ni v$. Without loss of generality, suppose that $v_1$ appears in an edge $e \in E_1$ with negative sign. Since we know that $v_1$ has total weight $0$ we may also choose another edge $e' \in E_1$ with positive sign such that $v_1 \in e'$. Then we form a zero-sum configuration $z$ containing $e$ with positive sign and $e'$ with negative sign and only other vertices in $V(\mathcal{T}^*)$. (Since we have one degree of freedom left to dictate $z$ and only a constant number of vertices - those in $W$ - to avoid, this is possible.) Updating $E_1$ to $E_2$ by adding $z$ removes the edges $e, e'$ from $E_1$ and adds edges only using new vertices outside $W$. Thus we have $|V^*(E_2)| < |V^*(E_1)|$. Repeating the process we eventually obtain $E^*$ such that $V^*(E^*) = \emptyset$. Then $E^*$ yields ${\bf 1} \in \mathcal{L}(\mathcal{T}^*)$, as required.  
\end{proof}

To see that our vector ${\bf v}$ with a $1$ on every element in $V(\mathcal{T}^*) \setminus W$ and $0$ on every element in $W$ is in $\mathcal{L}(\mathcal{T})$ simply requires checking (i)-(iv) or (a)-(d) of Lemma \ref{lem_lattice_t}. It is easy to verify (i)-(iii) and (a)-(c). We give some intermediate details of the calculations for $(iv)$ and $(d)$ here for transparency.

When $n$ is odd (and $3~|~n$), to satisfy (iv) in Lemma \ref{lem_lattice_t} reduces, after cancelling and regrouping terms, to showing that
$$\frac{2n^2}{3}+n+\frac{1}{3} +\frac{2n}{3}+\frac{2}{3}\sum_{i \in [3]}(a_i+b_i+c_i-d_i) \in \mathbb{Z}.$$
Then since $3~|~n$, this reduces to having
$$1+2\sum_{i \in [3]}(a_i+b_i+c_i-d_i) \in 3\mathbb{Z},$$
which is the requirement given in the proof of Theorem \ref{thm_class}.

Similarly, when $n$ is even and $3~|~n$, to satisfy (d) and (e) in Lemma \ref{lem_lattice_t} reduces to showing that
$$\frac{n^2}{3}+\frac{n}{2}+\frac{1}{6}+\frac{n}{12}+\frac{1}{6}\sum_{i \in [3]}(a_i+b_i+c_i-d_i) \in \mathbb{Z},$$
which is equivalent to showing that
$$2+n+2\sum_{i \in [3]}(a_i+b_i+c_i-d_i) \in 12 \mathbb{Z}.$$

When $n$ is even and $3\not|~n$ satisfying (d) and (e) in Lemma \ref{lem_lattice_t} reduces to showing that
$$\frac{n^2}{3}+\frac{n}{2}+\frac{1}{6}+\frac{3n}{4}+\frac{1}{2}\sum_{i \in [3]}(a_i+b_i+c_i-d_i) \in \mathbb{Z},$$  
and since in this case we have $n^2 \equiv 1\mymod 3$, this is equivalent to showing that
$$2+n+2\sum_{i \in [3]}(a_i+b_i+c_i-d_i) \in 4 \mathbb{Z}.$$
These conditions are all satisfied by the constraints on $a_i, b_i, c_i, d_i$ given in the proof of Theorem \ref{thm_class}. 

We could, of course, give explicit collections for the $12$ queens taken out and chosen to attack toroidally (but not classically), but the general description above shows that there are many choices for every $n$ sufficiently large.

\section{Concluding remarks}

Theorem \ref{thm_main}, our main result, asymptotically answers an open question of P\'{o}lya~\cite{polya} from 1918, as well as settling conjectures of Rivin, Vardi and Zimmerman~\cite{UB} and Luria~\cite{luria}. The proof of Theorem \ref{thm_main} uses the upper bounds of Luria~\cite{luria} and the lower bound is our main contribution. Recall that previously there was no known non-trivial lower bound for all $n \equiv 1,5 \mymod 6$ and the best lower bound for some $n$ was due to Luria~\cite{luria} using very different methods. 
Together with Theorem \ref{thm_class} and the upper bounds of Luria both for the toroidal and classical case, we completely settle Conjecture \ref{conj_rvz}, but we also recall that the classical case of the conjecture has been independently settled by the recent lower bound of Luria and Simkin~\cite{lslb} matching our lower bound in Theorem \ref{thm_class}. One difference in our results, other than the vastly different strategies, is that whilst their lower bound obtains `almost-toroidal' $n$-queens configurations in the sense that there are at most $o(n)$ toroidal attacks for each classical configuration, our result produces the same count for an even stronger structure where we obtain the same lower bound, but counting only those configurations where at most some constant $C \leq 12$ toroidal attacks occur for each classical configuration counted. We note that the three pairs of toroidal attacks on the $X+Y$ diagonal and three pairs of toroidal attacks on the $X-Y$ diagonal used in the proof of Theorem \ref{thm_class} is not necessarily best possible for all $n$, but does indeed work for all $n$. For example, when $n$ is even and $3\not|~n$, our constructions in the proof of Theorem \ref{thm_class} work taking only $i=1$ and disregarding $i=2,3$ so that we then have only one pair attacking toroidally along the $X+Y$ diagonal and only one pair attacking toroidally along the $X-Y$ diagonal.

We also make some remarks concerning the $n$ semi-queens problem.  
Whilst the toroidal semi-queens problem was settled by Eberhard, Manners and Mrazovi\'{c}~\cite{semi-queens}, which also gives a lower bound for classical semi-queens when $n$ is odd, there does not seem to have been any work on considering the classical version separately, or extending the lower bound from the toroidal setting to when $n$ is even. A follow-up paper by Eberhard~\cite{eberhard} considers how the result of Eberhard, Manners and Mrazovi\'{c}, which is more generally on additive triples of bijections than just the toroidal semi-queens problem, can be extended to more abelian groups than just those of order $n$ where $n$ is odd. Perhaps the ideas used here could also give a lower bound for $Q'(n)$. Alternatively, adapting our methods used to prove the toroidal $n$-queens result it should be possible to prove that $T'(n) \geq ((1+o(1))\frac{n}{e^2})^n$ and Luria's upper bound also matches this. On the one hand this in itself is not very interesting since this bound is a weaker form than that already given by Eberhard, Manners and Mrazovi\'{c}. On the other hand, we note that our methods should also then adapt to yield a lower bound for the classical $n$ semi-queens problem, that
$Q'(n) \geq ((1+o(1))\frac{n}{e^2})^n$ for all $n$ sufficiently large. Additionally, one could obtain an upper bound for $Q'(n)$ using a fairly straight forward application of the entropy method, as done by Luria for an upper bound on $Q(n)$. However given Simkin's~\cite{simkin} new upper and lower bounds for $Q(n)$, we presume that neither bound obtained this way would be tight for $Q'(n)$. 
Perhaps Simkin's~\cite{simkin} methods for $Q(n)$ could similarly further improve the accuracy of approximation for bounds for $Q'(n)$. However, as well as determining $Q'(n)$, it is still open as to determining a value for $Q(n)$ which is as accurate as our result for $T(n)$, so perhaps new ideas beyond those of Simkin's are required to close these gaps completely.

The methods we have used are inspired by powerful and recently developed tools in probabilistic combinatorics including most notably the methods of randomised algebraic construction and iterative absorption. However, these tools cannot be directly applied to the $n$-queens problem and so new ideas were needed to find variant forms applicable in our setting,  
making use of the combinatorial and algebraic structure embedded in the problem. There are several generalisations of the toroidal and classical $n$-queens problem discussed in Sections 6 and 8 of the survey by Bell and Stevens~\cite{survey}, including generalisations to higher dimensions. It would be interesting to consider whether some of the methods used here would enable progress on the open problems in this area.